\documentclass[11pt, reqno]{amsart}
\usepackage[dvipsnames]{xcolor}
\usepackage{amscd, amsmath, amssymb, amsthm, yfonts, mathrsfs, tikz, rotating, stmaryrd}
\usetikzlibrary{arrows, decorations.markings, patterns}
\tikzset{>=latex}
\usepackage[centering,margin=1.2in]{geometry}

\global\binoppenalty=10000

\newtheorem{theorem}{Theorem}[section]
\newtheorem{lemma}[theorem]{Lemma}
\newtheorem{prop}[theorem]{Proposition}
\newtheorem{cor}[theorem]{Corollary}
\newtheorem{conjecture}[theorem]{Conjecture}

\theoremstyle{definition}
\newtheorem{defn}[theorem]{Definition}
\newtheorem{remark}[theorem]{Remark}
\newtheorem{example}[theorem]{Example}
\newtheorem{notation}[theorem]{Notation}

\numberwithin{equation}{section}

\DeclareMathAlphabet{\mathpzc}{OT1}{pzc}{m}{it}

\def\beq{\begin{equation}}
\def\eeq{\end{equation}}

\def\longra{\longrightarrow}
\def\hooklongra{\lhook\joinrel\longrightarrow}

\newcommand{\hr}[1]{\left(#1\right)} 
\newcommand{\hm}[1]{\left|#1\right|} 
\newcommand{\ha}[1]{\left\langle#1\right\rangle} 
\newcommand{\hs}[1]{\left[#1\right]} 
\newcommand{\hc}[1]{\left\{#1\right\}} 

\def\le{\leqslant}
\def\ge{\geqslant}

\def\Ac{\mathcal A}
\def\Ad{\operatorname{Ad}}

\def\bgt{\mathfrak b}
\def\Bc{\mathcal B}

\def\C{\mathbb C}
\def\Cc{\mathcal C}

\def\Dc{\mathcal D}
\def\Dgt{\mathfrak D}
\def\Dcox{\Dc_n^{\mathrm{cox}}}
\def\Dsph{\Dc_n^{\mathrm{sph}}}
\def\Dsq{\Dc_n^{\mathrm{sq}}}
\def\Dstd{\Dc_n^{\mathrm{std}}}
\def\Dsym{\Dc_n^{\mathrm{sym}}}

\def\eps{\varepsilon}

\def\Fc{\mathcal F}
\def\Frac{\operatorname{Frac}}

\def\hgt{\mathfrak h}
\def\Hc{\mathcal H}

\def\i{\mathbf i}
\def\icox{\i_{\mathrm{cox}}}
\def\ifull{\i_{\mathrm{full}}}
\def\ihalf{\i_{\mathrm{half}}}
\def\isf{\i_{\mathrm{sf}}}
\def\isym{\i_{\mathrm{sym}}}

\def\la{\lambda}
\def\La{\Lambda}

\def\Mc{\mathcal{M}}

\def\Nc{\mathcal N}

\def\pr{\mathrm{pr}}
\def\Pc{\mathcal P}

\def\Qc{\mathcal Q}
\def\Qcox{\Qc_n^{\mathrm{cox}}}
\def\Qfull{\Qc_n^{\mathrm{full}}}
\def\Qhalf{\Qc_n^{\mathrm{half}}}
\def\Qsph{\Qc_n^{\mathrm{sph}}}
\def\Qsq{\Qc_n^{\mathrm{sq}}}
\def\Qstd{\Qc_n^{\mathrm{std}}}
\def\Qsym{\Qc_n^{\mathrm{sym}}}

\def\rk{\operatorname{rank}}
\def\R{\mathbb R}

\def\sl{\mathfrak{sl}}

\def\tLambda{\widetilde{\Lambda}}

\def\Tc{\mathcal T}

\def\Vc{\mathcal V}

\def\wdt{\widetilde}
\def\wh{\widehat}

\def\Xc{\mathcal X}

\def\Z{\mathbb Z}

\begin{document}

\title[Continuous tensor categories from quantum groups I]{Continuous tensor categories from quantum groups I: \\ algebraic aspects.}
\author{Gus Schrader}
\author{Alexander Shapiro}

\begin{abstract}
We describe the algebraic ingredients of a proof of the conjecture of Frenkel and Ip that the category of positive representations $\mathcal{P}_\lambda$ of the quantum group $U_q(\mathfrak{sl}_{n+1})$ is closed under tensor products. Our results generalize those of Ponsot and Teschner in the rank~1 case of $U_q(\mathfrak{sl}_2)$. In higher rank, many nontrivial features appear, the most important of these being a surprising connection to the quantum integrability of the open Coxeter-Toda lattice. We show that the closure under tensor products follows from the orthogonality and completeness of the Toda eigenfunctions (i.e. the $q$-Whittaker functions), and obtain an explicit construction of the Clebsch-Gordan intertwiner giving the decomposition of $\mathcal{P}_{\lambda}\otimes\mathcal{P}_{\mu}$ into irreducibles. 
\end{abstract}

\maketitle

\section{Introduction}
A remarkable class of infinite-dimensional representations of the quantum group $U_q(\mathfrak{sl}_2)$ has been studied by Bytsko, Ponsot, and Teschner, see~\cite{PT99, PT01,BT03}. These representations, which we refer to as {\em positive representations}, are defined for $q=e^{\pi i \hbar^2}$, $\hbar \in \R_{>0} \setminus \mathbb{Q}$, and are labelled by points $s\in \R_{\ge 0}$ of a Weyl chamber of $\mathfrak{sl}_2$. The positive representation $\mathcal{P}_s$ is modeled on a dense subspace of the Hilbert space $L^2(\R)$, on which the Chevalley-Serre generators $E,F,K$ of $U_q(\mathfrak{sl}_2)$ act by positive, essentially self-adjoint operators. The positive representations have several unusual features. Firstly, they possess a remarkable {\em modular duality} property: setting $q^\vee=e^{\pi i/\hbar^2}$, the space $\mathcal{P}_s$ carries an action of $U_{q^\vee}(\mathfrak{sl}_2)$ commuting with that of $U_q(\mathfrak{sl}_2)$. Thus, the positive representations can be regarded as representations of the modular double of the quantum group introduced by Faddeev in~\cite{Fad99}. Furthermore, the category of positive representations of $U_q(\mathfrak{sl}_2)$ turns out to be closed under taking tensor products. More precisely, in \cite{PT01} it was proved that the tensor product $\mathcal{P}_{s_1}\otimes \mathcal{P}_{s_2}$ admits a direct integral decomposition
\beq
\label{sl2-closure}
\mathcal{P}_{s_1}\otimes \mathcal{P}_{s_2} \simeq \int^{\oplus}_{\R_{\ge0}}\mathcal{P}_s~dm(s),
\eeq
where the integration measure $dm$ on the Weyl chamber $\R_{\ge0}$ is given by
$$
dm(s)=4\sinh(2\pi\hbar s)\sinh(2\pi\hbar^{-1}s).
$$
 Thus, the positive representations of $U_q(\mathfrak{sl}_2)$ can be said to form a ``continuous tensor category''.

Subsequently, a definition of positive representations of the quantum group $U_q(\mathfrak{sl}_{n+1})$ was proposed by Frenkel and Ip in~\cite{FI13}, and later extended by Ip to the other finite Dynkin types, see~\cite{Ip12a,Ip12b}. The positive representations $\mathcal{P}_{\lambda}$ of $U_q(\mathfrak{g})$ are again parameterized by points $\lambda\in \mathcal{C}^+$ of a Weyl chamber, and have the property that all Chevalley generators act by positive, essentially self adjoint operators. They also exhibit modular duality, in complete analogy with the rank 1 case. Establishing that the categories of higher rank positive representations are closed under tensor product, however, has remained a challenging open problem.

To give a sense of the difficulties involved, let us briefly recall how the decomposition \eqref{sl2-closure} was constructed by Ponsot and Teschner. Both for $\mathfrak{sl}_2$ and in higher rank, the positive representations are defined by providing an explicit embedding of the quantum group into a quantum torus algebra; that is, an algebra generated by symbols $X_j$ subject to the skew-commutativity relations $q^{\eps_{jk}}X_jX_k=q^{\eps_{kj}}X_kX_j$ for some skew-symmetric matrix $\eps$. The positive representations then arise as pullbacks of quantum torus algebra representations where the generators $X_j$ act by positive self-adjoint operators. Since the positive representations of $U_q(\mathfrak{sl}_2)$ are determined by their central character, the strategy of Ponsot and Teschner is to analyze the action of the Casimir $\Omega$ of $U_q(\mathfrak{sl}_2)$ on $\mathcal{P}_{s_1}\otimes \mathcal{P}_{s_2}$, which is a dense subspace in $L^2(\R^2)$. They exhibit an explicit unitary transformation of $L^2(\R^2)$ bringing $\Delta(\Omega)$ to $H \otimes 1$, where $H$ is the unbounded operator
\beq
\label{kashaev-operator}
H = e^{2\pi \hbar x} + e^{2\pi \hbar p}+e^{-2\pi \hbar p},
\eeq
on $L^2(\R)$, with $p=\frac{1}{2\pi i}\frac{\partial}{\partial x}$. This operator coincides with Kashaev's geodesic length operator from quantum Teichm\"uller theory. In particular, it is self-adjoint with simple spectrum $[2,\infty)$, and has eigenfunctions $\psi(x|s)$ satisfying
$$
H \cdot \psi(x|s) = (e^{2\pi \hbar s}+e^{-2\pi \hbar s})\psi(x|s), \quad s \in \R_{\ge0}. 
$$
Moreover, the eigenfunctions $\psi(x|s)$ are delta-function orthogonalized and complete in $L^2(\R)$:
\beq
\label{sl2-orthog}
\int_{\R}\psi(x|s)\psi(x|s')dx=\delta(s-s'),
\eeq
\beq
\label{sl2-plancherel}
\int_{\R_{\ge0}}\psi(x|s)\psi(x'|s)dm(s)=\delta(x-x').
\eeq
Thus, the transformation of the form
$$
f(x,y)\mapsto \hat{f}(y,s)=\int_{\R}f(x,y)\psi(x|s)dx
$$
delivers an explicit unitary equivalence \eqref{sl2-closure}.

Let us consider the ingredients involved in extending this argument to the higher rank positive representations. The first step is an essentially algebraic one. The center of $U_q(\sl_{n+1})$ is generated by $n$ Casimir elements that can be taken, for instance, to be the quantum traces $\Omega_k$ of the fundamental representations of $\sl_{n+1}$. One could hope to analyze the diagonal action of these elements on $\mathcal{P}_\lambda \otimes \mathcal{P}_{\mu}$, and attempt to find a unitary transformation under which they become the Hamiltonians $H_k$ of some simpler and hopefully known quantum integrable system, in the same way that Ponsot and Teschner arrive at Kashaev's geodesic length operator in the rank 1 case. 

The second step is analytic in nature and involves determination of the joint spectrum and eigenfunctions of the commuting quantum Hamiltonians $H_k$. Namely, one must prove analogs of the orthogonality and completeness relations \eqref{sl2-orthog} and \eqref{sl2-plancherel}. 

Unfortunately, for $n>1$, the algebraic part of the problem appears to be highly nontrivial. Indeed, even for $n=2$ the fundamental quantum Casimirs $\Omega_1,\Omega_2$ acting on $\mathcal{P}_\lambda\otimes\mathcal{P}_\mu$ both become noncommutative polynomials with 192 terms when written in terms of Frenkel and Ip's original quantum torus realization. It seems unfeasible to analyze these operators by direct calculation as was done by Ponsot and Teschner, and thus new ideas are needed. 

In the present paper, we address this algebraic problem for $\mathfrak{g}=\mathfrak{sl}_{n+1}$ using the tools of quantum cluster algebras and higher Teichm\"uller theory. The general theory of quantum cluster $\mathcal{X}$-varieties has been developed by Fock and Goncharov in~\cite{FG06a,FG06b,FG09}. The quantum analog of cluster mutation in direction $k$ can be realized as the algebra automorphism of conjugation by the {\em non-compact quantum dilogarithm} $\Phi^\hbar(x_k)$, where $x_k$ is related to the cluster variable $X_k$ by $X_k=e^{2\pi \hbar x_k}$. In particular, if $x_k$ acts in some representation by a self-adjoint operator $\hat{x}_k$, then the operator $\Phi^\hbar(\hat{x}_k)$ is unitary. Thus, quantum cluster transformations provide a large supply of unitary equivalences, which we shall exploit in our study of the positive representations of $U_q(\mathfrak{sl}_{n+1})$. 

Our approach is based on the cluster realization of $U_q(\mathfrak{sl}_{n+1})$ obtained in \cite{SS16}. In turn, that realization is formulated in terms of the quantum cluster structure associated to moduli spaces of framed $PGL_{n+1}$-local systems on a marked surface, see~\cite{FG06a}. Cluster charts on these varieties can be obtained from an ideal triangulation of the surface by `amalgamating' certain simpler cluster charts associated to each triangle. In the case of moduli spaces of $PGL_{n+1}$-local systems, a flip of a triangulation corresponds to a sequence of $\binom{n+2}{3}$ cluster mutations. 

Taking a particular cluster chart on the moduli space associated to a triangulation of the punctured disk $D_{2,1}$ with two marked points on its boundary, we obtain a quantum torus algebra~$\Dgt_n$ and an explicit embedding of $U_q(\mathfrak{sl}_{n+1})$ into $\Dgt_n$. Moreover, the representations of $U_q(\mathfrak{sl}_{n+1})$ obtained by pulling back positive representations of $\Dgt_n$ are precisely the positive representations of Frenkel and Ip, see~\cite{Ip16}.

Again following \cite{SS16}, a cluster realization of the tensor product $\mathcal{P}_\lambda\otimes \mathcal{P}_\mu$ can be given by considering a quantum cluster chart on the moduli space of framed $PGL_{n+1}$-local systems on a twice punctured disk $D_{2,2}$ with two boundary marked points. Now, the basic geometric idea behind our approach can be described as follows. As shown in Figure~\ref{fig-pants}, the marked surface $D_{2,2}$ can be visualized as a pair of pants with two marked points on one of its boundary components. Consider a loop $\gamma$ in $D_{2,2}$ winding once around each puncture. In Figure~\ref{fig-pants} it is drawn in orange and is labelled by $\nu \in \mathcal C^+$. In this geometric setting, the diagonal action of the fundamental Casimirs of $U_q(\mathfrak{sl}_{n+1})$ on $\mathcal{P}_\lambda\otimes \mathcal{P}_\mu$ is given by the operators $H_k$ quantizing the elementary symmetric functions of the eigenvalues of the local system's monodromy around $\gamma$.

\begin{figure}[h]
\begin{tikzpicture}[x=0.5cm,y=0.5cm]

\draw[thick] (0,3) to (8,3);
\draw[thick] (8,3) to [bend left = 15] (8,1);
\draw[thick] (8,3) to [bend right = 15] (8,1);
\draw[thick,domain=90:270] plot ({8+3*cos(\x)}, {sin(\x)});
\draw[thick] (8,-1) to [bend left=15] (8,-3);
\draw[thick] (8,-1) to [bend right=15] (8,-3);
\draw[thick] (8,-3) to (0,-3);
\draw[thick] (0,-3) to [bend left=15] (0,3);
\draw[thick, dashed] (0,-3) to [bend right=15] (0,3);
\draw[BurntOrange, thick] (3,-3) to [bend left = 15] (3,3);
\draw[BurntOrange, thick, dashed] (3,-3) to [bend right = 15] (3,3);
\draw[thick,fill] (0,3) circle (0.15);
\draw[thick,fill] (0,-3) circle (0.15);

\node at (8.7,2) {$\lambda$};
\node at (8.7,-2) {$\mu$};
\node[BurntOrange] at (3,3.5) {$\nu$};

\draw[ultra thick] (9.5,0.5) to (12.5,0.5);
\draw[ultra thick] (9.5,-0.5) to (12.5,-0.5);

\draw[thick, shift={(16,0)}] (-15:2cm) to [out=180,in=0] (-2,-3);
\draw[thick, shift={(16,0)}] (-2,-3) to [bend left=15] (-2,3);
\draw[thick, shift={(16,0)}, dashed] (-2,-3) to [bend right=15] (-2,3);
\draw[thick, shift={(16,0)}] (-2,3) to [out=0,in=180] (15:2cm);
\draw[thick, shift={(16,0)}, BurntOrange] (15:2cm) to [bend left=15] (-15:2cm);
\draw[thick, shift={(16,0)}, BurntOrange] (15:2cm) to [bend right=15] (-15:2cm);
\draw[thick,fill] (14,3) circle (0.15);
\draw[thick,fill] (14,-3) circle (0.15);

\node[BurntOrange] at (19.9,1.6) {$\nu$};

\draw[ultra thick] (20.6,0.5) to (20.6,0);
\draw[ultra thick,domain=0:180] plot ({21+cos(\x)*2/5}, {-sin(\x)/2});
\draw[ultra thick] (21.4,0) to (21.4,0.5);

\node[BurntOrange] at (21,-1.2) {$\nu$};

\draw[thick, shift={(26,0)}] (195:2cm) to [out=0,in=120] (-75:2cm);
\draw[thick, shift={(26,0)}] (-75:2cm) to [bend left=15] (-45:2cm);
\draw[thick, shift={(26,0)}] (-75:2cm) to [bend right=15] (-45:2cm);
\draw[thick, shift={(26,0)}] (-45:2cm) to [out=120,in=-120] (45:2cm);
\draw[thick, shift={(26,0)}] (45:2cm) to [bend left=15] (75:2cm);
\draw[thick, shift={(26,0)}] (45:2cm) to [bend right=15] (75:2cm);
\draw[thick, shift={(26,0)}] (75:2cm) to [out=-120,in=0] (165:2cm);
\draw[thick, shift={(26,0)}, BurntOrange, dashed] (165:2cm) to [bend left=15] (195:2cm);
\draw[thick, shift={(26,0)}, BurntOrange] (165:2cm) to [bend right=15] (195:2cm);

\node[BurntOrange] at (22.1,1.6) {$\nu$};
\node at (28.2,4.2) {$\lambda$};
\node at (28.2,-4.2) {$\mu$};

\end{tikzpicture}
\caption{Cutting out a pair of pants: $D_{2,2} = D_{2,1} \cup_{\gamma} S_3$.}
\label{fig-pants}
\end{figure}
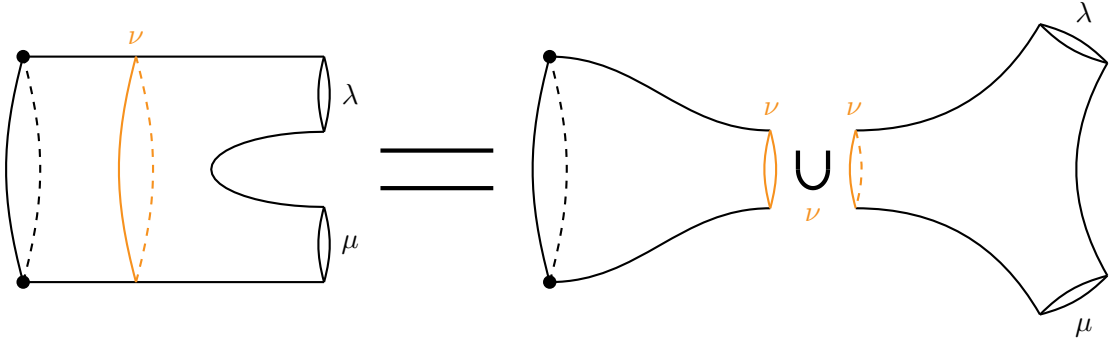

The essential idea is that the decomposition of the tensor product $P_\lambda\otimes P_\mu$ of two positive representations arises from the isomorphism of quantum higher Teichm\"uller spaces corresponding to cutting along the simple closed curve $\gamma$ as shown in Figure~\ref{fig-pants}. The right hand side of the figure corresponds to a ``fiber product'' of two quantum Teichm\"uller spaces: the diagonally acting copy of $U_q(\mathfrak{sl}_{n+1})$ represented by $D_{2,1}$, a punctured disk with two marked points on its boundary, and the quantum Teichm\"uller space on a thrice-punctured sphere $S_3$, where the eigenvalues of two monodromies are specified to $\lambda$ and $\mu$ respectively. The fiber product is taken over the spectrum of the unspecified monodromy $\nu$ around the loop $\gamma$ which is associated to the remaining third puncture of the sphere, and parameterizes the central characters of $U_q(\mathfrak{sl}_{n+1})$.

At the level of classical higher Teichm\"uller spaces, the existence of such an isomorphism is a theorem~\cite[Theorem 7.6]{FG06a} of Fock and Goncharov. Its quantum counterpart is more subtle, and can be viewed as a particular case of the ``modular functor conjecture'' of the same authors, see \cite{FG09}.

Our first step towards establishing this isomorphism is to pass to a triangulation of $ D_{2,2}$ which is well-adapted to this cutting and gluing decomposition. We explain the sequence of flips required to achieve this in Section~\ref{subsec-full-toda}; the upshot is that in the resulting cluster chart the $H_k$ are identified with the quantization of conjugation-invariant functions on the reduced open double Bruhat cell $PGL_{n+1}^{w_0,w_0}/\Ad_H$. 
 
Next, we observe that these operators can be simplified further by performing an additional sequence of cluster transformations of a more refined nature than those corresponding to flips of triangulation. These mutations, which we detail in Section~\ref{sec-muts}, are related to the combinatorics of factorization in the double Weyl group $W\times W$. This being done, we arrive at a cluster chart in which our monodromy operators are given by the Hamiltonians of a genuine quantum integrable system, namely the quantum open Coxeter-Toda chain of type~$A_n$. 

This completes the algebraic side of the story: the problem of decomposing the tensor product $\mathcal{P}_\lambda\otimes \mathcal{P}_\mu$ is now reduced to that of understanding the spectrum of the quantum Coxeter-Toda Hamiltonians. The eigenfunctions of these operators, known as $q$-{\em Whittaker functions}, have been determined and studied in \cite{KLS02}. The only missing ingredient is the general analog of the Plancherel inversion formula~\eqref{sl2-plancherel}, established by Kashaev in the rank~1 case in~\cite{Kas01}. More precisely, in Section 9, we prove the following main theorem of the paper:
\begin{theorem}
Suppose that the Plancherel inversion formula~\eqref{planch-formula} for the $q$-Whittaker functions holds. Then there is an isomorphism of $U_q(\mathfrak{sl}_{n+1})$-modules
$$
\mathfrak{I}  \ \colon  \ \Pc_\lambda \otimes \Pc_\mu \longrightarrow \int^{\oplus}_{\mathcal{C}^+}\Pc_{\nu}\otimes \Mc_{\lambda,\mu}^{\nu}~dm(\nu),
$$
where the measure $dm(\nu)$ on the Weyl chamber $\mathcal{C}^+$ is the Sklyanin measure defined in~\eqref{sklyanin-mes}, and $U_q(\mathfrak{sl}_{n+1})$ acts trivially on $\Mc_{\lambda,\mu}^{\nu}$. In turn, the tensor factor $\Mc_{\lambda,\mu}^\nu$ is a positive representation of the quantum torus algebra associated to the moduli space of $PGL_{n+1}$-local systems on a thrice-punctured sphere, where eigenvalues of the three monodromies are specified to $\lambda$, $\mu$, and $\nu$.
\end{theorem}

The analytic problem of establishing the Plancherel formula~\eqref{planch-formula} will be addressed in the sequel~\cite{SS17} to the present article. Let us conclude the introduction by mentioning some interesting directions for future work suggested by our results. 

\begin{itemize}
\item
\emph{Modular functor conjecture.}
We hope to be able to establish the ``cutting and gluing'' part of the modular functor conjecture of~\cite{FG09} by similar methods to the ones described here. Indeed, as explained in~\cite{Tes05} in addition to the operation of cutting a marked pair of pants we have already considered, one must also treat the case of cutting a punctured torus along a non-separating cycle. We will return to this problem in a future work. 

\item
\emph{Peter--Weyl theorem for the modular double.} We believe that the techniques of this paper, combined with the analytic ones developed by Ip in the $\mathfrak{sl}_2$ case \cite{Ip13}, should allow us to establish a positive analog of the Peter-Weyl theorem for the modular double of $U_q(\mathfrak{sl}_{n+1})$.  

\item
\emph{Relation with the results of Gekhtman--Shapiro--Vainshtein.}
The cluster chart on the moduli space of $PGL_{n+1}$ local systems on $D_{2,2}$ whose quiver is illustrated in Figure~\ref{Q4-sym} is remarkably similar to the one discovered in~\cite{GSV16} that determines the cluster structure on the dual Poisson-Lie group $PGL_{n+1}^*$. The precise connection between the two cluster structures will be made explicit in a separate publication. 

\item
\emph{$3j$-- and $6j$-- symbols for positive and finite dimensional representations.}
It was shown in~\cite{Ip15} that the Clebsch-Gordon maps for finite-dimensional representations of $U_q(\sl_2)$ can be recovered from those of the positive representations. It is thus natural to ask whether similar results can be obtained in higher rank using our construction of the Clebsch-Gordan intertwiners for $U_q(\sl_{n+1})$. Similarly, the results presented here open the door to an explicit calculation of the Racah-Wigner coefficients (also known as the $6j$-symbols) for positive representations of $U_q(\sl_{n+1})$. One could thus hope to obtain those for finite-dimensional representations by a similar procedure to that outlined in~\cite{Ip15}. 

\item
\emph{Quantum monodromies and higher rank AGT.} Quantum monodromy operators similar to the ones corresponding to the action of the center of $U_q(\mathfrak{sl}_{n+1})$ on $\mathcal{P}_\lambda\otimes \mathcal{P}_\mu$ have recently emerged as an important tool in the program to extend the AGT correspondence to class $\mathcal{S}$ theories of type $A_{n}$ \cite{CGT15}. We hope that the algebraic techniques presented here can help elucidate this construction.

\end{itemize}

The paper is organized as follows. Sections~\ref{sec-cl-alg} and~\ref{sec-std-real} contain some background and technical preliminaries that we shall use extensively. Section~\ref{sec-cl-alg} presents the basic setup of quantum cluster $\mathcal{X}$-varieties and their representations following \cite{FG09}. Section~\ref{sec-std-real} describes the quantum cluster algebras associated to moduli spaces of $G$-local systems on marked surfaces \cite{FG06a}. We also recall the results of \cite{SS16} on the cluster realization of $U_q(\mathfrak{sl}_{n+1})$, and give the precise definition of its positive representations $\mathcal{P}_\lambda$. Section~\ref{sec-net} gives an exposition of the cluster combinatorics of directed networks on the disk and punctured disk. As far as we can tell, the network formalism has not previously been applied to quantum cluster algebras in the literature, so we provide self-contained proofs of various results that we use later in the paper. We apply the network formalism to re-express the cluster embedding of $U_q(\mathfrak{sl}_{n+1})$ in different cluster charts obtained from the initial one by flips of triangulation. Section~\ref{sec-muts} recalls the Poisson geometry and cluster structure of double Bruhat cells $G^{u,v}$ in a complex semisimple Lie group of adjoint type, and presents several important mutation sequences that we use extensively in the paper. Using these sequences, we analyze the action of $U_q(\mathfrak{sl}_{n+1})$ on the tensor product $\mathcal{P}_\lambda\otimes\mathcal{P}_\mu$ of two positive representations, and obtain combinatorial formulas for the diagonal action of the Chevalley generators of the quantum group. In Section~\ref{sec-Toda}, we review the definition of the $q$-deformed open Coxeter-Toda chain and its eigenfunctions, and use the combinatorial formulas from Section~\ref{sec-muts} to express the action of the quantum group generators in terms of the quantum Toda Hamiltonians. In Section~\ref{sec-int} we explain how completeness of the Toda eigenfunctions implies the fact that the tensor product $\mathcal{P}_\lambda\otimes\mathcal{P}_\mu$ decomposes as a direct integral of positive representations. Finally, in Section~\ref{sec-comp}, we compare our results with those obtained by Ponsot and Teschner in the case $n=1$. Throughout the paper, we illustrate our constructions via examples for $n=4$.

\section*{Acknowledgements}

We are very grateful to Vladimir Fock for many excellent discussions and observations. We thank Nicolai Reshetikhin for his support and encouragement throughout the course of this work. Our special thanks go to Michael Gekhtman and Michael Shapiro for many helpful explanations and suggestions. We are grateful to Igor Frenkel and Ivan Ip for explaining their results. Finally, this work benefited enormously from many discussions with our colleagues: we would like to thank Alexander Braverman, Davide Gaiotto, Alexander Goncharov, Joel Kamnitzer, and Ian Le. The second author has been supported by the University of Toronto Faculty of Arts and Science Fellowship, by the NSF Postdoctoral Fellowship DMS-1703183, and by the RFBR grant 17-01-00585.

\section{Quantum cluster algebras}
\label{sec-cl-alg}

\subsection{Quantum cluster $\mathcal{X}$-seeds and their mutations}

In this section we recall a few basic facts about cluster tori and their quantization following~\cite{FG09}. We shall only need the quantum cluster algebras related to quantum groups of type A, and we incorporate this in the definition of a cluster seed.

\begin{defn}
A cluster seed\footnote{Our definition differs from the general one by requiring the form $(\cdot,\cdot)$ to be $\Z/2$-valued on frozen vectors rather than just $\mathbb Q$-valued, and by setting all multipliers $d_i=1$.} is a datum $\Theta=\hr{\La, (\cdot,\cdot),\hc{e_i},I_0}$ where
\begin{itemize}
\item $\La$ is a lattice;
\item $(\cdot,\cdot)$ is a skew-symmetric $\Z/2$-valued form on $\La$;
\item $\hc{e_i}$ is a basis of the lattice $\La$;
\item $I_0$ is a subset of $I = \hc{1,2,\dots, \rk(\La)}$
\end{itemize}
and the following integrality conditions are satisfied:
$$
\eps_{ij} = \hr{e_i,e_j} \in \Z \quad\text{unless}\quad (i,j) \in I_0 \times I_0.
$$
\end{defn}

\begin{notation}
In what follows, given a graph $\Gamma$ we denote the set of its faces by $F(\Gamma)$, the set of its edges by $E(\Gamma)$, and the set of its vertices by $V(\Gamma)$.
\end{notation}

To a seed $\Theta$, we can associate a quiver $\Qc$ with vertices labelled by the set $I$ and arrows given by the adjacency matrix $\eps = \hr{\eps_{ij}}$. Then, a vertex $i \in V(\Qc)$ corresponds to the basis vector $e_i$, which gives rise to a lattice as $i$ runs through $I$; the adjacency matrix defines the form $(\cdot, \cdot)$; finally, we draw vertices $i \in I \setminus I_0$ as circles, while vertices $i \in I_0$ are depicted as squares and are referred to as $\emph{frozen vertices}$. 
 
The pair $\hr{\Lambda,(\cdot, \cdot)}$ determines a \emph{quantum torus algebra} $\mathcal{T}_\Lambda$, which is the free $\Z[q^{\pm1/2}]$-module spanned by $X_{\lambda}$, $\lambda\in \Lambda$, with $X_0 = 1$ and the multiplication defined\footnote{This definition gives the opposite algebra structure to the one defined in~\cite{FG09}} by
\beq
q^{(\lambda,\mu)}X_\lambda X_\mu = X_{\lambda+\mu}.
\eeq 
The seed also gives rise to a distinguished system of generators for the quantum torus algebra $\mathcal{T}_\Lambda$, namely the elements $X_i=X_{e_i}$. The generators $X_i$ for $i\in I_0$ are called \emph{frozen variables}.

Let $\Theta=(\Lambda, (\cdot,\cdot), \hc{e_i}, I_0)$ be a seed, and $k \in I \setminus I_0$ a non-frozen vertex of the corresponding quiver $\Qc$. Then one obtains a new seed, $\mu_k(\Theta)$, called the \emph{mutation of $\Theta$ in direction $k$}, by changing the basis $\hc{e_i}$ while the rest of the data remains the same. The new basis $\{e_i'\}$ is defined by
\beq
\label{eq-basis}
e'_i = 
\begin{cases}
-e_k &\text{if} \; i=k, \\
e_i + [\eps_{ik}]_+e_k &\text{if} \; i \ne k,
\end{cases}
\eeq
where $[a]_+=\max(a,0)$. We remark that bases the $\hc{e_i}$ and $\hc{\mu_k^2(e_i)}$ do not necessarily coincide, although the seeds $\Theta$ and $\mu_k^2(\Theta)$ are isomorphic.

For each mutation $\mu_k$ we define an algebra automorphism of the skew field $\Frac(\mathcal{T}_\Lambda)$, which by abuse of notation we call a \emph{quantum mutation} and denote by the same symbol $\mu_k$, as follows. Recall the compact quantum dilogarithm function
\beq
\label{q-dilog}
\Psi^q(z) = \prod_{k=0}^{\infty}\frac1{(1+q^{2k+1}z)}.
\eeq
The quantum mutation $\mu_k$ is the automorphism of $\mathcal{T}_\Lambda$ defined by
\beq
\label{mut}
\mu_k = \Ad_{\Psi^q\hr{X_{e'_k}}}.
\eeq
The fact that $(e_i,e_j)$ is integral unless $(i,j) \in I_0 \times I_0$ guarantees that conjugation by the formal power series $\Psi^q$ yields a genuine birational automorphism. For example,
$$
\mu_k\hr{X_{e_i}} =
\begin{cases}
X_{e'_i}\hr{1+qX_{e'_k}} &\text{if} \quad \eps_{ki}=1, \\
X_{e'_i}\hr{1+qX_{e'_k}^{-1}}^{-1} &\text{if} \quad \eps_{ki}=-1.
\end{cases}
$$
Equivalently, for $\eps_{ki}=1$ we have
\beq
\label{reg-mut}
\mu_k(X_{e_i}) = X_{e_i'}+ X_{e_i'+e_k'},
\eeq
while if $\eps_{ki}=-1$ we have
\beq
\label{denom-mut}
\mu_k(X_{e_i}) = X_{e_i'+e_k'}\hr{1+q^{-1}X_{e_k'}}^{-1}.
\eeq

Let us remark that the quantum mutation~\eqref{mut} is inverse to that defined in~\cite{SS16}, and the change of basis~\eqref{eq-basis} corresponds to the inverse of the monomial transformation $\mu'_k$ in the decomposition $\mu_k = \mu_k^\sharp \circ \mu'_k$, see~\cite[Section 1]{SS16}. Finally, we would like to recall that the semi-classical limit of the quantum mutations~\eqref{mut} can be viewed as gluing maps between \emph{cluster charts} $\Xc_\Theta\simeq \hr{\C^*}^{|I|}$, whose ring of regular functions is the quasi-classical limit of the quantum torus algebra $\mathcal T_\Lambda$.

\subsection{Positive representations of quantum tori}
\label{positive-rep-sec}
Suppose that $\mathcal{T}_\Lambda$ is a quantum torus algebra equipped with a choice of generators $\{X_j\}$, and $\hbar \in \R$. Then we can consider an associated Heisenberg $*$-algebra $\mathcal{H}_\Lambda$. It is a topological $*$-algebra over $\C$ generated by elements $\{x_i\}$ satisfying
\beq
\label{heis-gens}
[x_j,x_k]=\frac{1}{2\pi i} \eps_{jk} \qquad\text{and}\qquad *x_j=x_j. 
\eeq
Then the assignments
$$
X_j = e^{2\pi \hbar x_j} \qquad\text{and}\qquad q=e^{\pi i \hbar^2}
$$
define an embedding of algebras $\Tc_\La \hookrightarrow \Hc_\La$. Write $\Lambda_{\R}=\Lambda\otimes\R$, and let $Z_{\Lambda}\subset \Lambda_{\R}$ be the kernel of the skew form $(\cdot,\cdot)$. The algebra $\Hc_\La$ has a family of irreducible $*$-representations $V_{\Lambda,\chi}$ parameterized by central characters $\chi\in \Lambda_{\R}^*$, in which the generators $x_j$ act by unbounded operators in a Hilbert space. To obtain an explicit realization of $V_{\Lambda,\chi}$ one should pick a symplectic basis $(p_j,q_j)$ in the symplectic vector space $\Lambda_{\R}/Z_{\Lambda}$; then the Hilbert space $V_{\Lambda,\chi}$ can be taken as the space of $L^2$ functions on the Lagrangian subspace spanned by the $q_j$, where operators $q_j$ act by multiplication while operators $p_j$ act as $\partial/\partial q_j$. The realizations corresponding to different choices of symplectic basis are all unitarily equivalent, the interwiners being provided by the action of the metaplectic group $Mp\hr{\Lambda_{\R}/Z_\Lambda}$. The space $V_{\Lambda,\chi}$ is a positive representation of the quantum torus algebra $\mathcal{T}_{\Lambda}$, in the sense that the generators $X_j$ act by positive, essentially self-adjoint, unbounded operators. 


\subsection{Non-compact quantum dilogarithm}
\label{h-dilog-section}

In order to extend the notion of mutation to positive representations, one must replace the quantum dilogarithm \eqref{q-dilog} by its non-compact analog $\Phi^\hbar$. Let us introduce the quantities
$$
c_{\hbar} = \hbar + \hbar^{-1} \qquad\text{and}\qquad \rho_\hbar = e^{\pi i \hr{\frac{\hbar^2+\hbar^{-2}}{6}+ \frac{1}{2}}}.
$$
Consider the function $G_\hbar(z)$ defined for $0<Re(z)<c_\hbar$ by
$$
G_\hbar(z):= \overline{\rho_\hbar}~\exp\hr{-\frac{1}{4}\int_{C}\frac{e^{\pi t z}}{\hr{e^{\pi\hbar t}-1}\hr{e^{\pi\hbar^{-1}t}-1}}\frac{dt}{t}},
$$
where the contour of integration $C$ runs along the real axis, bypassing the pole of the integrand at $t=0$ from above. The non-compact quantum dilogarithm $\Phi^\hbar(z)$ is then defined as
$$
\Phi^\hbar(z) := G_\hbar\hr{\frac{c_\hbar}{2}-iz}.
$$
The functions $G_\hbar$ and $\Phi^\hbar$ can be analytically continued to meromorphic functions on $\C$, and have many remarkable properties, see \cite{FG09,FK94,Kas01}. For us, the most important of these are gathered in the following lemma.

\begin{lemma}
\label{dilog-properties}
The non-compact quantum dilogarithm $\Phi^\hbar(z)$ has the following properties:
\begin{enumerate}
\item (Modular duality) $\Phi^\hbar$ is self-dual under the modular transformation $\hbar\mapsto \hbar^{-1}$: 
$$
\Phi^\hbar(z) = \Phi^{\hbar^{-1}}(z).
$$

\item (Difference equations)
$\Phi^\hbar$ satisfies the dual pair of difference equations
\beq
\label{dilog-diff-eqs}
\begin{aligned}
\Phi^{\hbar}(z+i\hbar) &= (1+qe^{2\pi \hbar z})\Phi^{\hbar}(z),\\
\Phi^{\hbar}(z+i\hbar^{-1}) &= (1+qe^{2\pi \hbar^{-1} z})\Phi^{\hbar}(z).
\end{aligned}
\eeq

\item
\label{dilog-unitary}
(Unitarity) If $\hbar\in\R$, then we have
\beq
\overline{\Phi^\hbar(z)}=\frac{1}{\Phi^\hbar(\overline{z})}.
\eeq
\end{enumerate}
\end{lemma}
Now suppose that $\Theta$ is a cluster seed, and $V_{\chi,\Lambda}$ is a positive representation of the corresponding quantum torus algebra $\mathcal{T}_{\Lambda}$. Then the Heisenberg algebra generators $\{x_k\}$ from~\eqref{heis-gens} act by essentially self-adjoint operators on $V_{\chi,\Lambda}$. Thus, by property~\eqref{dilog-unitary} from Lemma~\ref{dilog-properties}, we have unitary operators $\Phi^\hbar\hr{x_j}$ acting on $V_{\chi,\Lambda}$ for each generator $x_j$. Let us now consider the quantum torus algebra generators $X_j = e^{2\pi\hbar x_j}$, which act on $V_{\chi,\Lambda}$ by positive essentially self-adjoint operators. Then in view of the difference equations~\eqref{dilog-diff-eqs}, for each non-frozen variable $X_k$ we have the equalities
$$
\mu_k\hr{X_j} = \Ad_{\Phi^\hbar\hr{X_k^{-1}}}\hr{X_j}.
$$
Thus, the mutation in direction $k$ is realized in the representation $V_{\chi,\Lambda}$ via the unitary operator $\Phi^{\hbar}(x_k)$. In particular, the positive representations of $\mathcal{T}_\Lambda$ obtained from an initial one by applying mutation automorphisms are all unitary equivalent. We will exploit this fact extensively in the sequel, as we apply various mutation sequences to analyze the action of $U_q(\mathfrak{sl}_{n+1})$ on its positive representations. 

\subsection{Concatenation of representations}

Suppose we have a direct sum decomposition of lattices $\Lambda = \Lambda_{-1} \oplus \Lambda_{0} \oplus \Lambda_{1}$, such that $\Lambda_{-1}$ is orthogonal to $\Lambda_1$ with respect to the form $(\cdot,\cdot)_{\Lambda}$, and $\Lambda_0$ is isotropic. Thus any $\lambda \in \Lambda$ can be uniquely written $\lambda = \lambda_{-1} + \lambda_0 + \lambda_1$, with $(\lambda_{-1},\lambda_1)=0$. Set 
$$
\Lambda_{\le0}=\Lambda_{-1}\oplus \Lambda_0, \qquad \Lambda_{\ge0}=\Lambda_{0}\oplus \Lambda_1.
$$
Now consider the lattice $\Lambda_{\le0}\oplus\Lambda_{\ge0}$, equipped with the skew-form 
$$((\lambda,\mu),(\lambda',\mu'))_{\Lambda_{\le0}\oplus\Lambda_{\ge0}} = (\lambda,\lambda')_{\Lambda} + (\mu,\mu')_{\Lambda}.
$$
 The positive representations of the corresponding Heisenberg algebra $\mathcal{H}_{\Lambda_{\le0}\oplus\Lambda_{\ge0}}$ are Hilbert space tensor products of positive representations of $\mathcal{H}_{\Lambda_{\le0}}$ and those of $\mathcal{H}_{\Lambda_{\ge0}}$:
 $$
 V_{\chi_{\le 0}\oplus\chi_{\ge0}} \simeq V_{\chi_{\le 0}}\otimes V_{\chi_{\ge 0}}.
 $$ 
There is an isometric embedding of lattices 
$$
\Lambda\hookrightarrow \Lambda_{\le0}\oplus\Lambda_{\ge0}, \quad \lambda \mapsto (\lambda_{-1}+\lambda_0)\oplus (\lambda_{0}+\lambda_1). 
$$
which induces an embedding of the corresponding Heisenberg algebras $\mathcal{H}_\Lambda\hookrightarrow \mathcal{H}_{\Lambda_{\le0}\oplus\Lambda_{\ge0}}$. Then we have the following lemma, which is a simple exercise in linear algebra.

\begin{lemma} 
\label{rep-concatenation}
Suppose that $Z_\Lambda = \Lambda\cap Z_{\Lambda_{\le0}\oplus\Lambda_{\ge0}}$, and that the induced map
$$
\Lambda/Z_\Lambda \hookrightarrow \Lambda_{\le0}/Z_{\Lambda_{\le0}}\oplus\Lambda_{\ge0}/Z_{\Lambda_{\ge0}}
$$
is an isomorphism. Then each positive representation $V_{\Lambda,\chi}$ is isomorphic to the Hilbert space tensor product
$$
V_{\Lambda,\chi} \simeq V_{\Lambda_{\le0},\chi_{\le0}}\otimes V_{\Lambda_{\ge0},\chi_{\ge0}}.
$$
\end{lemma}

\section{Cluster realization of $U_q(\mathfrak{sl}_{n+1})$ and its positive representations}
\label{sec-std-real}

\subsection{Moduli spaces of framed local systems and their cluster structure}
\label{locsys-intro}

We now recall some basics of the theory of quantum character varieties following~\cite{FG06a}. Let $\wh S$ be a decorated surface --- that is, a topological surface $S$ with boundary $\partial S$, equipped with a finite collection of marked points $x_1,\ldots, x_r\in \partial S$ and punctures $p_1,\ldots, p_s$. In \cite{FG06a}, the moduli space $\mathcal{X}_{\wh S,PGL_m}$ of $PGL_m$-local systems on $S$ with reductions to Borel subgroups at each marked point $x_i$ and each puncture $p_i$, was defined and shown to admit the structure of a cluster $\mathcal{X}$-variety. In particular, suppose that~$T$ is an ideal triangulation of~$S$: recall that this means that all vertices of~$T$ are at either marked points or punctures. Then it was shown in \cite{FG06a} that for each such ideal triangulation, one can produce a cluster $\mathcal{X}$-chart on $\mathcal{X}_{\wh S,PGL_m}$. Moreover, the Poisson algebra of functions on such a chart admits a canonical quantization, whose construction we shall now recall. 
 

The first step is to describe the quantum cluster $\mathcal{X}$-chart associated to a single triangle. To do this, consider a triangle $ABC$ given by the equation $x+y+z=m$, $x,y,z\ge0$ and intersect it with lines $x=p$, $y=p$, and $z=p$ for all $0 < p < m$, $p \in \Z$. The resulting picture is called the $m$-triangulation of the triangle~$ABC$. Let us now color the triangles of the $m$-triangulation in black and white, as in Figure~\ref{fig-triang} so that triangles adjacent to vertices $A$, $B$, or $C$ are black, and two triangles sharing an edge are of different color. We shall also orient the edges of white triangles counterclockwise. Finally, we connect the vertices of the $m$-triangulation lying on the same side of the triangle $ABC$ by dashed arrows in the clockwise direction. The resulting graph is shown in Figure~\ref{fig-triang}. Note that the vertices on the boundary of $ABC$ are depicted by squares. Throughout the text we will use square vertices for frozen variables. All dashed arrows will be of weight $\frac12$, that is a dashed arrow $v_i \to v_j$ denotes the commutation relation $X_iX_j = q^{-1}X_jX_i$.

\begin{figure}[h]
\begin{tikzpicture}[thick, y=0.866cm, x=0.5cm]

\node at (-0.5,-0.1) {A};
\node at (5,5.4) {B};
\node at (10.5,-0.1) {C};

\draw (0,0) -- (10,0);
\draw (1,1) -- (9,1);
\draw (2,2) -- (8,2);
\draw (3,3) -- (7,3);
\draw (4,4) -- (6,4);

\draw (0,0) -- (5,5);
\draw (2,0) -- (6,4);
\draw (4,0) -- (7,3);
\draw (6,0) -- (8,2);
\draw (8,0) -- (9,1);

\draw (5,5) -- (10,0);
\draw (4,4) -- (8,0);
\draw (3,3) -- (6,0);
\draw (2,2) -- (4,0);
\draw (1,1) -- (2,0);

\fill[pattern=north east lines] (0,0) -- (1,1) -- (2,0);
\fill[pattern=north east lines] (2,0) -- (3,1) -- (4,0);
\fill[pattern=north east lines] (4,0) -- (5,1) -- (6,0);
\fill[pattern=north east lines] (6,0) -- (7,1) -- (8,0);
\fill[pattern=north east lines] (8,0) -- (9,1) -- (10,0);

\fill[pattern=north east lines] (1,1) -- (2,2) -- (3,1);
\fill[pattern=north east lines] (3,1) -- (4,2) -- (5,1);
\fill[pattern=north east lines] (5,1) -- (6,2) -- (7,1);
\fill[pattern=north east lines] (7,1) -- (8,2) -- (9,1);

\fill[pattern=north east lines] (2,2) -- (3,3) -- (4,2);
\fill[pattern=north east lines] (4,2) -- (5,3) -- (6,2);
\fill[pattern=north east lines] (6,2) -- (7,3) -- (8,2);

\fill[pattern=north east lines] (3,3) -- (4,4) -- (5,3);
\fill[pattern=north east lines] (5,3) -- (6,4) -- (7,3);

\fill[pattern=north east lines] (4,4) -- (5,5) -- (6,4);

\end{tikzpicture}
\qquad\qquad
\begin{tikzpicture}[every node/.style={inner sep=0, minimum size=0.4cm, thick}, thick, x=0.5cm, y=0.866cm]

\node at (2,-0.2) {};

\node (1) at (1,1) [draw] {\scriptsize{1}};
\node (2) at (2,0) [draw] {\scriptsize{2}};
\node (3) at (2,2) [draw] {\scriptsize{3}};
\node (4) at (3,1) [circle, draw] {\scriptsize{4}};
\node (5) at (4,0) [draw] {\scriptsize{5}};
\node (6) at (3,3) [draw] {\scriptsize{6}};
\node (7) at (4,2) [circle, draw] {\scriptsize{7}};
\node (8) at (5,1) [circle, draw] {\scriptsize{8}};
\node (9) at (6,0) [draw] {\scriptsize{9}};
\node (10) at (4,4) [draw] {\scriptsize{10}};
\node (11) at (5,3) [circle, draw] {\scriptsize{11}};
\node (12) at (6,2) [circle, draw] {\scriptsize{12}};
\node (13) at (7,1) [circle, draw] {\scriptsize{13}};
\node (14) at (8,0) [draw] {\scriptsize{14}};
\node (15) at (6,4) [draw] {\scriptsize{15}};
\node (16) at (7,3) [draw] {\scriptsize{16}};
\node (17) at (8,2) [draw] {\scriptsize{17}};
\node (18) at (9,1) [draw] {\scriptsize{18}};

\draw[->, thick] (1) -- (2);
\draw[->, thick] (2) -- (4);
\draw[->, thick] (4) -- (7);
\draw[->, thick] (7) -- (11);
\draw[->, thick] (11) -- (15);
\draw[->, thick] (15) -- (10);
\draw[->, thick] (10) -- (11);
\draw[->, thick] (11) -- (12);
\draw[->, thick] (12) -- (13);
\draw[->, thick] (13) -- (14);
\draw[->, thick] (14) -- (18);
\draw[->, thick] (18) -- (13);
\draw[->, thick] (13) -- (8);
\draw[->, thick] (8) -- (4);
\draw[->, thick] (4) -- (1);

\draw[->, thick] (3) -- (4);
\draw[->, thick] (4) -- (5);
\draw[->, thick] (5) -- (8);
\draw[->, thick] (8) -- (12);
\draw[->, thick] (12) -- (16);
\draw[->, thick] (16) -- (11);
\draw[->, thick] (11) -- (6);
\draw[->, thick] (6) -- (7);
\draw[->, thick] (7) -- (8);
\draw[->, thick] (8) -- (9);
\draw[->, thick] (9) -- (13);
\draw[->, thick] (13) -- (17);
\draw[->, thick] (17) -- (12);
\draw[->, thick] (12) -- (7);
\draw[->, thick] (7) -- (3);

\draw[->, thick, dashed] (1) -- (3);
\draw[->, thick, dashed] (3) -- (6);
\draw[->, thick, dashed] (6) -- (10);
\draw[->, thick, dashed] (15) -- (16);
\draw[->, thick, dashed] (16) -- (17);
\draw[->, thick, dashed] (17) -- (18);
\draw[->, thick, dashed] (14) -- (9);
\draw[->, thick, dashed] (9) -- (5);
\draw[->, thick, dashed] (5) -- (2);

\end{tikzpicture}
\caption{Cluster $\Xc$-coordinates on the configuration space of 3 flags and 3 lines.}
\label{fig-triang}
\end{figure}
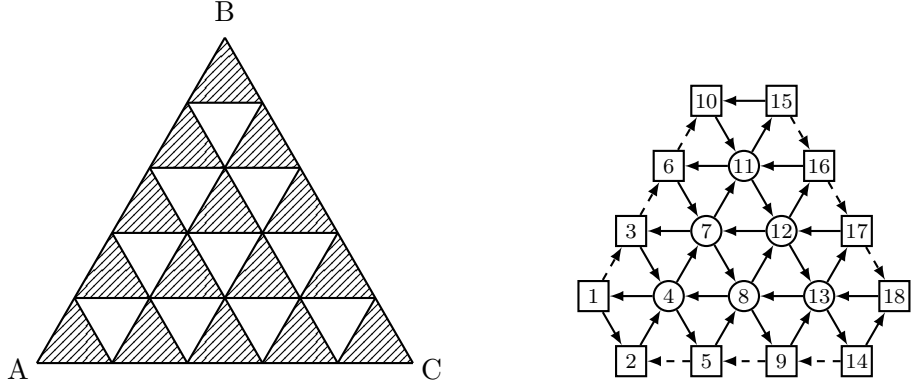

Now, let us recall the procedure of \emph{amalgamating} two quivers by a subset of frozen variables, following~\cite{FG06b}. Simply put, amalgamation is nothing but the gluing of two quivers by identifying a number of frozen vertices. More formally, let $\Qc_1$, $\Qc_2$ be a pair of quivers, and $I_1$, $I_2$ be certain subsets of frozen variables in $\Qc_1$, $\Qc_2$ respectively. Given a bijection $\phi \colon I_1 \to I_2$ we can amalgamate quivers $\Qc_1$ and $\Qc_2$ by the subsets $I_1$, $I_2$ along $\phi$. The result is a new quiver $\Qc$ constructed in the following three steps:
\begin{enumerate}
\item unfreeze vertices $v_i \in \Qc_1$ and $v_j \in \Qc_2$ for all $i \in I_1$ and $j \in I_2$;
\item for any $i \in I_1$ identify vertices $v_i \in \Qc_1$ and $v_{\phi(i)} \in \Qc_2$ in the union $\Qc_1 \sqcup \Qc_2$;
\item for any pair $i,j \in I_1$ with an arrow $v_i \to v_j$ in $\Qc_1$ labelled by $\eps_{ij}$ and an arrow $v_{\phi(i)} \to v_{\phi(j)}$ in $\Qc_2$ labelled by $\eps_{\phi(i),\phi(j)}$, label the arrow between corresponding vertices in $\Qc$ by $\eps_{ij} + \eps_{\phi(i),\phi(j)}$
\end{enumerate}
Amalgamation of a pair of quivers $\Qc_1$, $\Qc_2$ into a quiver $\Qc$ induces an embedding $\Xc \to \Xc_1 \otimes \Xc_2$ of the corresponding cluster $\Xc$-tori:
$$
X_i \mapsto
\begin{cases}
X_i \otimes 1, &\text{if} \; i \in \Qc_1 \setminus I_1, \\
1 \otimes X_i, &\text{if} \; i \in \Qc_2 \setminus I_2, \\
X_i \otimes X_{\phi(i)}, &\text{otherwise.}
\end{cases}
$$

An example of amalgamation is shown in Figure~\ref{fig-flip}. There, the left quiver is obtained by amalgamating a triangle $ABC$ from Figure~\ref{fig-triang} with a similar triangle along the side $BC$ (or more precisely, along frozen vertices 15, 16, 17, and 18 on the edge $BC$). 

As explained in~\cite{FG06a}, in order to construct the cluster $\Xc$-coordinate chart on $\Xc_{\wh S,PGL_m}$
corresponding to an ideal triangulation $T$ of $\wh S$, one performs the following procedure:
\begin{enumerate}
\item $m$-triangulate each of the ideal triangles in $T$;
\item for any pair of ideal triangles in $T$ sharing an edge, amalgamate the corresponding pair of quivers by this edge.
\end{enumerate}
In general, different ideal triangulations of $\wh S$ result in different quivers, and hence different cluster $\Xc$-tori. However, any triangulation can be transformed into any other by a sequence of \emph{flips} that replace one diagonal in an ideal 4-gon with the other one. Each flip corresponds to the following sequence of cluster mutations that we shall recall in the $m=5$ example shown in Figure~\ref{fig-flip}. There, a flip is obtained in four steps. First, mutate at vertices 15, 16, 17 18; second, mutate at vertices 11, 12, 13, 20, 21, 22; third, at vertices 7, 8, 16, 17, 25, 26; finally, mutate at 4, 12, 21, 29. Note, that the order of mutations within one step does not matter. In general, a flip in an $m$-triangulated 4-gon consists of $m-1$ steps. On the $i$-th step, one should do the following. First, inscribe an $i$-by-$(m-i)$ rectangle in the 4-gon, such that vertices of the rectangle coincide with boundary vertices of the $m$-triangulation and the side of the rectangle of length $m-i$ goes along the diagonal of a 4-gon. Second, divide the rectangle into $i (m-i)$ squares and mutate at the center of each square. As in the example, the order of mutations within a single step does not matter.

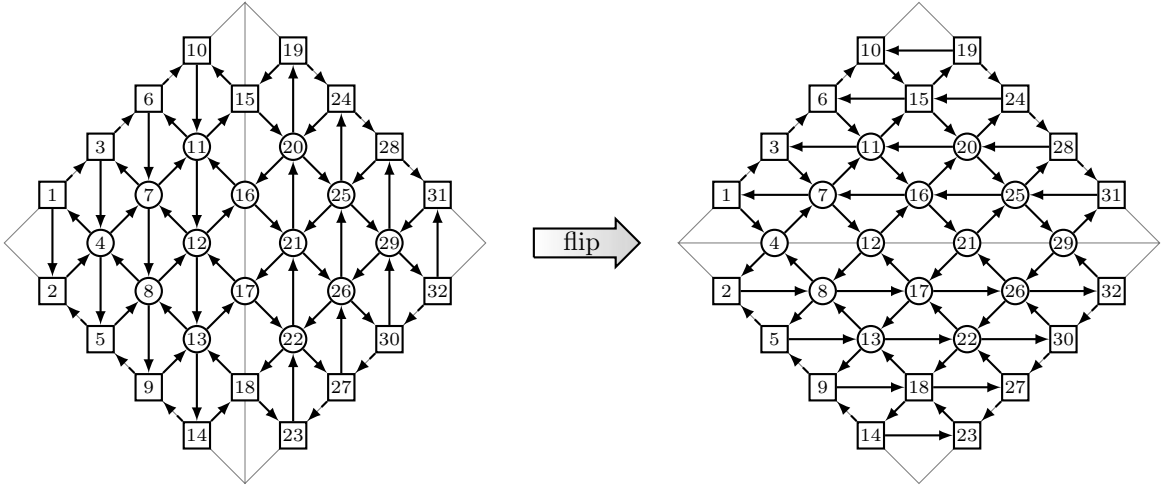
\begin{figure}[h]
\begin{tikzpicture}[every node/.style={inner sep=0, minimum size=0.35cm, thick, draw, fill=white}, x=0.64cm, y=0.64cm]

\draw [gray] (-7,5) to (-2,0) to (-7,-5) to (-12,0) to (-7,5) to (-7,-5);

\node (1) at (-11,1) {\tiny{1}};
\node (2) at (-11,-1) {\tiny{2}};
\node (3) at (-10,2) {\tiny{3}};
\node (4) at (-10,0) [circle] {\tiny{4}};
\node (5) at (-10,-2) {\tiny{5}};
\node (6) at (-9,3) {\tiny{6}};
\node (7) at (-9,1) [circle] {\tiny{7}};
\node (8) at (-9,-1) [circle] {\tiny{8}};
\node (9) at (-9,-3) {\tiny{9}};
\node (10) at (-8,4) {\tiny{10}};
\node (11) at (-8,2) [circle] {\tiny{11}};
\node (12) at (-8,0) [circle] {\tiny{12}};
\node (13) at (-8,-2) [circle] {\tiny{13}};
\node (14) at (-8,-4) {\tiny{14}};
\node (15) at (-7,3) {\tiny{15}};
\node (16) at (-7,1) [circle] {\tiny{16}};
\node (17) at (-7,-1) [circle] {\tiny{17}};
\node (18) at (-7,-3) {\tiny{18}};
\node (19) at (-6,4) {\tiny{19}};
\node (20) at (-6,2) [circle] {\tiny{20}};
\node (21) at (-6,0) [circle] {\tiny{21}};
\node (22) at (-6,-2) [circle] {\tiny{22}};
\node (23) at (-6,-4) {\tiny{23}};
\node (24) at (-5,3) {\tiny{24}};
\node (25) at (-5,1) [circle] {\tiny{25}};
\node (26) at (-5,-1) [circle] {\tiny{26}};
\node (27) at (-5,-3) {\tiny{27}};
\node (28) at (-4,2) {\tiny{28}};
\node (29) at (-4,0) [circle] {\tiny{29}};
\node (30) at (-4,-2) {\tiny{30}};
\node (31) at (-3,1) {\tiny{31}};
\node (32) at (-3,-1) {\tiny{32}};

\draw[->, thick] (1) -- (2);
\draw[->, thick] (2) -- (4);
\draw[->, thick] (4) -- (7);
\draw[->, thick] (7) -- (11);
\draw[->, thick] (11) -- (15);
\draw[->, thick] (15) -- (10);
\draw[->, thick] (10) -- (11);
\draw[->, thick] (11) -- (12);
\draw[->, thick] (12) -- (13);
\draw[->, thick] (13) -- (14);
\draw[->, thick] (14) -- (18);
\draw[->, thick] (18) -- (13);
\draw[->, thick] (13) -- (8);
\draw[->, thick] (8) -- (4);
\draw[->, thick] (4) -- (1);

\draw[->, thick] (3) -- (4);
\draw[->, thick] (4) -- (5);
\draw[->, thick] (5) -- (8);
\draw[->, thick] (8) -- (12);
\draw[->, thick] (12) -- (16);
\draw[->, thick] (16) -- (11);
\draw[->, thick] (11) -- (6);
\draw[->, thick] (6) -- (7);
\draw[->, thick] (7) -- (8);
\draw[->, thick] (8) -- (9);
\draw[->, thick] (9) -- (13);
\draw[->, thick] (13) -- (17);
\draw[->, thick] (17) -- (12);
\draw[->, thick] (12) -- (7);
\draw[->, thick] (7) -- (3);

\draw[->, thick, dashed] (14) -- (9);
\draw[->, thick, dashed] (9) -- (5);
\draw[->, thick, dashed] (5) -- (2);
\draw[->, thick, dashed] (1) -- (3);
\draw[->, thick, dashed] (3) -- (6);
\draw[->, thick, dashed] (6) -- (10);

\draw[->, thick] (32) -- (31);
\draw[->, thick] (31) -- (29);
\draw[->, thick] (29) -- (26);
\draw[->, thick] (26) -- (22);
\draw[->, thick] (22) -- (18);
\draw[->, thick] (18) -- (23);
\draw[->, thick] (23) -- (22);
\draw[->, thick] (22) -- (21);
\draw[->, thick] (21) -- (20);
\draw[->, thick] (20) -- (19);
\draw[->, thick] (19) -- (15);
\draw[->, thick] (15) -- (20);
\draw[->, thick] (20) -- (25);
\draw[->, thick] (25) -- (29);
\draw[->, thick] (29) -- (32);

\draw[->, thick] (30) -- (29);
\draw[->, thick] (29) -- (28);
\draw[->, thick] (28) -- (25);
\draw[->, thick] (25) -- (21);
\draw[->, thick] (21) -- (17);
\draw[->, thick] (17) -- (22);
\draw[->, thick] (22) -- (27);
\draw[->, thick] (27) -- (26);
\draw[->, thick] (26) -- (25);
\draw[->, thick] (25) -- (24);
\draw[->, thick] (24) -- (20);
\draw[->, thick] (20) -- (16);
\draw[->, thick] (16) -- (21);
\draw[->, thick] (21) -- (26);
\draw[->, thick] (26) -- (30);

\draw[->, thick, dashed] (19) -- (24);
\draw[->, thick, dashed] (24) -- (28);
\draw[->, thick, dashed] (28) -- (31);
\draw[->, thick, dashed] (32) -- (30);
\draw[->, thick, dashed] (30) -- (27);
\draw[->, thick, dashed] (27) -- (23);

\draw [thick, shade, top color = gray!35, shading angle=-90] (-1,-0.3) -- (-1,0.3) -- (0.7,0.3) -- (0.7,0.5) -- (1.2,0) -- (0.7,-0.5) -- (0.7,-0.3) -- (-1,-0.3);

\node [draw=none, fill=none] at (0,0) {\small{flip}};

\draw [gray] (2,0) to (7,5) to (12,0) to (7,-5) to (2,0) to (12,0);

\node (1) at (3,1) {\tiny{1}};
\node (2) at (3,-1) {\tiny{2}};
\node (3) at (4,2) {\tiny{3}};
\node (4) at (4,0) [circle] {\tiny{4}};
\node (5) at (4,-2) {\tiny{5}};
\node (6) at (5,3) {\tiny{6}};
\node (7) at (5,1) [circle] {\tiny{7}};
\node (8) at (5,-1) [circle] {\tiny{8}};
\node (9) at (5,-3) {\tiny{9}};
\node (10) at (6,4) {\tiny{10}};
\node (11) at (6,2) [circle] {\tiny{11}};
\node (12) at (6,0) [circle] {\tiny{12}};
\node (13) at (6,-2) [circle] {\tiny{13}};
\node (14) at (6,-4) {\tiny{14}};
\node (15) at (7,3) {\tiny{15}};
\node (16) at (7,1) [circle] {\tiny{16}};
\node (17) at (7,-1) [circle] {\tiny{17}};
\node (18) at (7,-3) {\tiny{18}};
\node (19) at (8,4) {\tiny{19}};
\node (20) at (8,2) [circle] {\tiny{20}};
\node (21) at (8,0) [circle] {\tiny{21}};
\node (22) at (8,-2) [circle] {\tiny{22}};
\node (23) at (8,-4) {\tiny{23}};
\node (24) at (9,3) {\tiny{24}};
\node (25) at (9,1) [circle] {\tiny{25}};
\node (26) at (9,-1) [circle] {\tiny{26}};
\node (27) at (9,-3) {\tiny{27}};
\node (28) at (10,2) {\tiny{28}};
\node (29) at (10,0) [circle] {\tiny{29}};
\node (30) at (10,-2) {\tiny{30}};
\node (31) at (11,1) {\tiny{31}};
\node (32) at (11,-1) {\tiny{32}};

\draw[->, thick] (19) -- (10);
\draw[->, thick] (10) -- (15);
\draw[->, thick] (15) -- (20);
\draw[->, thick] (20) -- (25);
\draw[->, thick] (25) -- (29);
\draw[->, thick] (29) -- (31);
\draw[->, thick] (31) -- (25);
\draw[->, thick] (25) -- (16);
\draw[->, thick] (16) -- (7);
\draw[->, thick] (7) -- (1);
\draw[->, thick] (1) -- (4);
\draw[->, thick] (4) -- (7);
\draw[->, thick] (7) -- (11);
\draw[->, thick] (11) -- (15);
\draw[->, thick] (15) -- (19);

\draw[->, thick] (24) -- (15);
\draw[->, thick] (15) -- (6);
\draw[->, thick] (6) -- (11);
\draw[->, thick] (11) -- (16);
\draw[->, thick] (16) -- (21);
\draw[->, thick] (21) -- (25);
\draw[->, thick] (25) -- (28);
\draw[->, thick] (28) -- (20);
\draw[->, thick] (20) -- (11);
\draw[->, thick] (11) -- (3);
\draw[->, thick] (3) -- (7);
\draw[->, thick] (7) -- (12);
\draw[->, thick] (12) -- (16);
\draw[->, thick] (16) -- (20);
\draw[->, thick] (20) -- (24);

\draw[->, thick, dashed] (1) -- (3);
\draw[->, thick, dashed] (3) -- (6);
\draw[->, thick, dashed] (6) -- (10);
\draw[->, thick, dashed] (19) -- (24);
\draw[->, thick, dashed] (24) -- (28);
\draw[->, thick, dashed] (28) -- (31);

\draw[->, thick] (14) -- (23);
\draw[->, thick] (23) -- (18);
\draw[->, thick] (18) -- (13);
\draw[->, thick] (13) -- (8);
\draw[->, thick] (8) -- (4);
\draw[->, thick] (4) -- (2);
\draw[->, thick] (2) -- (8);
\draw[->, thick] (8) -- (17);
\draw[->, thick] (17) -- (26);
\draw[->, thick] (26) -- (32);
\draw[->, thick] (32) -- (29);
\draw[->, thick] (29) -- (26);
\draw[->, thick] (26) -- (22);
\draw[->, thick] (22) -- (18);
\draw[->, thick] (18) -- (14);

\draw[->, thick] (9) -- (18);
\draw[->, thick] (18) -- (27);
\draw[->, thick] (27) -- (22);
\draw[->, thick] (22) -- (17);
\draw[->, thick] (17) -- (12);
\draw[->, thick] (12) -- (8);
\draw[->, thick] (8) -- (5);
\draw[->, thick] (5) -- (13);
\draw[->, thick] (13) -- (22);
\draw[->, thick] (22) -- (30);
\draw[->, thick] (30) -- (26);
\draw[->, thick] (26) -- (21);
\draw[->, thick] (21) -- (17);
\draw[->, thick] (17) -- (13);
\draw[->, thick] (13) -- (9);

\draw[->, thick, dashed] (32) -- (30);
\draw[->, thick, dashed] (30) -- (27);
\draw[->, thick, dashed] (27) -- (23);
\draw[->, thick, dashed] (14) -- (9);
\draw[->, thick, dashed] (9) -- (5);
\draw[->, thick, dashed] (5) -- (2);

\end{tikzpicture}
\caption{A pair of triangles amalgamated by 1 side.}
\label{fig-flip}
\end{figure}

\subsection{Cluster realization of $U_q(\mathfrak{sl}_{n+1})$}
\label{subsec-embed}

In what follows, we consider the complex simple Lie algebra $\sl_{n+1}=\sl_{n+1}(\C)$, equipped with a pair of opposite Borel subalgebras $\bgt_\pm$ and a Cartan subalgebra $\hgt = \bgt_+ \cap \bgt_-$. The corresponding root system $\Delta$ is equipped with a polarization $\Delta =\Delta_+ \sqcup \Delta_-$, consistent with the choice of Borel subalgebras $\bgt_\pm$, and a set of simple roots $\hc{\alpha_1,\ldots, \alpha_n} \subset \Delta_+$. We denote by $(\cdot,\cdot)$ the unique symmetric bilinear form on $\hgt^*$ invariant under the Weyl group $W$, such that $(\alpha,\alpha)=2$ for all roots $\alpha\in \Delta$. Entries of the Cartan matrix are denoted $a_{ij} = (\alpha_i, \alpha_j)$. 

Let $q$ be a formal parameter, consider an associative $\C(q)$-algebra $\Dgt_n$ generated by elements\footnote{The generators $E_i$ and $F_i$ that we consider here coincide with those in~\cite{SS16} multiplied by $\sqrt{-1}$.}

$$
\hc{E_i, F_i, K_i, K'_i \,|\, i=1,\dots, n},
$$
subject to the relations
\begin{align*}
&K_i E_j =q^{a_{ij}} E_j K_i, &\qquad K_iK_j = K_jK_i, \\
 &K'_i E_j =q^{-a_{ij}} E_j K'_i, &\qquad K'_iK_j = K_jK'_i, \\
&K_i F_j = q^{-a_{ij}} F_j K_i, &\qquad K'_iK'_j = K'_jK'_i, \\
&K'_i F_j =q^{a_{ij}} F_j K'_i,
\end{align*}
the relations
$$
[E_i,F_j] =\delta_{ij}\hr{q-q^{-1}}\hr{K'_i-K_i},
$$
and the quantum Serre relations
\begin{align*}
&E_i^2 E_{i\pm1} - (q+q^{-1})E_i E_{i\pm1} E_i + E_{i\pm1} E_i^2 = 0, \\
&F_i^2 F_{i\pm1} - (q+q^{-1})F_i F_{i\pm1} F_i + F_{i\pm1} F_i^2 = 0, \\[2pt]
&[E_i,E_j] = [F_i, F_j] = 0 \quad\text{if}\quad |i-j|>1.
\end{align*}
The algebra $\Dgt_n$ is a Hopf algebra, with the comultiplication
\begin{align*}
&\Delta(E_i) = E_i\otimes 1+K_i\otimes E_i, & &\Delta(K_i) = K_i \otimes K_i, \\
&\Delta(F_i) = F_i\otimes K'_i+1\otimes F_i, & &\Delta(K'_i) = K'_i \otimes K'_i,
\end{align*}
the antipode
\begin{align*}
&S(E_i)=-K_i^{-1}E_i, & &S(K_i)=K_i^{-1}, \\
&S(F_i)=-F_iK_i, & &S(K'_i) = (K'_i)^{-1},
\end{align*}
and the counit
$$
\epsilon(K_i)=\epsilon(K'_i) = 1, \qquad \epsilon(E_i) = \epsilon(F_i)=0.
$$

The quantum group $U_q(\sl_{n+1})$ is defined as the quotient 
$$
U_q(\sl_{n+1}) = \Dgt_n / \ha{K_iK_i'=1 \,|\, i=1, \dots, n}.
$$
Note that $U_q(\sl_{n+1})$ inherits a well-defined Hopf algebra structure from~$\Dgt_n$. 
The subalgebra $U_q(\bgt) \subset \Dgt_n$ generated by all $K_i, E_i$ is a Hopf subalgebra in~$\Dgt_n$. The algebra $U_q(\bgt)$ is isomorphic to its image under the projection onto $U_q(\sl_{n+1})$ and is called the quantum Borel subalgebra of $U_q(\sl_{n+1})$. The Hopf algebra $\Dgt_n$ is nothing but the Drinfeld double of $U_q(\bgt)$.

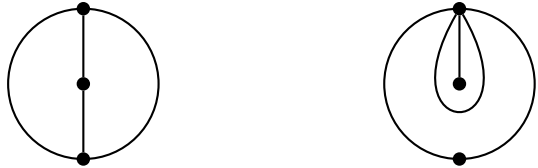
\begin{figure}[h]
\begin{tikzpicture}[every node/.style={inner sep=0, minimum size=0.15cm, circle, draw, fill=black}, x=0.25cm,y=0.25cm]

\node [thick] (t) at (-10,4) {};
\node [thick] (c) at (-10,0) {};
\node [thick] (b) at (-10,-4) {};

\draw[thick] (-10,0) circle (4);
\draw[thick] (t) -- (c) -- (b);

\node [thick] (t) at (10,4) {};
\node [thick] (c) at (10,0) {};
\node [thick] (b) at (10,-4) {};

\draw[thick] (10,0) circle (4);
\draw[thick] (t) -- (c);
\draw[thick] (t) to [out=-60,in=-120,min distance = 2cm] (t);

\end{tikzpicture}
\caption{Standard triangulation (left) and self-folded triangulation (right) of $D_{2,1}$, the punctured disk with a pair of marked points on its boundary.}
\label{triangulations}
\end{figure}

Let us recall the results of \cite{SS16} on the cluster algebraic realization of $U_q(\mathfrak{sl}_{n+1})$. We denote by $D_{2,1}$ a punctured disk with two marked points on its boundary, and consider the quantized moduli space $\mathcal{X}_{PGL_{n+1},D_{2,1}}$ of framed $PGL_{n+1}$-local systems on $D_{2,1}$. Let $\Theta_n^{\mathrm{std}}$ be the quantum cluster seed determined by the ideal triangulation of $D_{2,1}$ consisting of a pair of triangles glued by two sides; this triangulation is referred to as the standard one in Figure~\ref{triangulations}. The resulting quiver for $n=4$ is shown in Figure~\ref{fig-D4}. We shall write $\Dstd$ for the corresponding quantum torus algebra. 

\begin{theorem} \cite{SS16}
\label{original-embed}
There is an embedding of algebras
$$
\iota \colon \Dgt_n \longra \Dstd,
$$
such that for each Chevalley generator $E_i, F_i, K_i$ and $K'_i$ of $\Dgt_n$, there exists a quantum cluster chart, mutation-equivalent to $\Theta_n^{\mathrm{std}}$, in which that Chevalley generator is mapped under $\iota$ to a cluster monomial.
\end{theorem}

For a sequence of cluster variables $(X_1, \cdots, X_m)$ let us define
\begin{align*}
S(X_1, \dots, X_m) &= X_1(1+qX_2(1+qX_3(\dots(1+qX_m)))), \\
P(X_1, \dots, X_m) &= q^{m-1} X_1 \dots X_m.
\end{align*}
Then the explicit formulas for $\iota$ in case $n=4$ are as follows:
$$
E_i \longmapsto S_{[n^2,n^2+2n-1]} \qquad\text{and}\qquad K_i \longmapsto P_{[n^2,n^2+2n]}
$$
where
$$
S_{[i,j]} = S(X_i,X_{i+1},\dots,X_j), \qquad P_{[i,j]} = P(X_i,X_{i+1},\dots,X_j).
$$
Formulas for $F_{n+1-i}$ and $K'_{n+1-i}$ are obtained respectively from those for $E_i$ and $K_i$ by rotating the corresponding quiver $180^\circ$, for example, we have
\begin{align*}
&E_1 \longmapsto S(X_1,X_2) & &K_1 \longmapsto P(X_1,X_2,X_3) \\
&F_4 \longmapsto S(X_{24},X_{25}) & &K'_4 \longmapsto P(X_{24},X_{25},X_{16}).
\end{align*}

\begin{figure}[h]
\begin{tikzpicture}[every node/.style={inner sep=0, minimum size=0.4cm, thick, draw}, x=0.75cm, y=0.5cm]

\node (1) at (-4,-3) {\scriptsize{1}};
\node (2) at (0,-7) [circle] {\scriptsize{2}};
\node (3) at (4,-3) {\scriptsize{3}};
\node (4) at (-4,-1) {\scriptsize{4}};
\node (5) at (-3,-2) [circle] {\scriptsize{5}};
\node (6) at (0,-5) [circle] {\scriptsize{6}};
\node (7) at (3,-2) [circle] {\scriptsize{7}};
\node (8) at (4,-1) {\scriptsize{8}};
\node (9) at (-4,1) {\scriptsize{9}};
\node (10) at (-3,0) [circle] {\scriptsize{10}};
\node (11) at (-2,-1) [circle] {\scriptsize{11}};
\node (12) at (0,-3) [circle] {\scriptsize{12}};
\node (13) at (2,-1) [circle] {\scriptsize{13}};
\node (14) at (3,0) [circle] {\scriptsize{14}};
\node (15) at (4,1) {\scriptsize{15}};
\node (16) at (-4,3) {\scriptsize{16}};
\node (17) at (-3,2) [circle] {\scriptsize{17}};
\node (18) at (-2,1) [circle] {\scriptsize{18}};
\node (19) at (-1,0) [circle] {\scriptsize{19}};
\node (20) at (0,-1) [circle] {\scriptsize{20}};
\node (21) at (1,0) [circle] {\scriptsize{21}};
\node (22) at (2,1) [circle] {\scriptsize{22}};
\node (23) at (3,2) [circle] {\scriptsize{23}};
\node (24) at (4,3) {\scriptsize{24}};
\node (25) at (0,7) [circle] {\scriptsize{25}};
\node (26) at (0,5) [circle] {\scriptsize{26}};
\node (27) at (0,3) [circle] {\scriptsize{27}};
\node (28) at (0,1) [circle] {\scriptsize{28}};

\draw [->, thick] (1) -- (2);
\draw [->, thick] (2) -- (3);

\draw [->, thick] (4) -- (5);
\draw [->, thick] (5) -- (6);
\draw [->, thick] (6) -- (7);
\draw [->, thick] (7) -- (8);

\draw [->, thick] (9) -- (10);
\draw [->, thick] (10) -- (11);
\draw [->, thick] (11) -- (12);
\draw [->, thick] (12) -- (13);
\draw [->, thick] (13) -- (14);
\draw [->, thick] (14) -- (15);

\draw [->, thick] (16) -- (17);
\draw [->, thick] (17) -- (18);
\draw [->, thick] (18) -- (19);
\draw [->, thick] (19) -- (20);
\draw [->, thick] (20) -- (21);
\draw [->, thick] (21) -- (22);
\draw [->, thick] (22) -- (23);
\draw [->, thick] (23) -- (24);

\draw [->, thick] (24) -- (25);
\draw [->, thick] (25) -- (16);

\draw [->, thick] (15) -- (23);
\draw [->, thick] (23) -- (26);
\draw [->, thick] (26) -- (17);
\draw [->, thick] (17) -- (9);

\draw [->, thick] (8) -- (14);
\draw [->, thick] (14) -- (22);
\draw [->, thick] (22) -- (27);
\draw [->, thick] (27) -- (18);
\draw [->, thick] (18) -- (10);
\draw [->, thick] (10) -- (4);

\draw [->, thick] (3) -- (7);
\draw [->, thick] (7) -- (13);
\draw [->, thick] (13) -- (21);
\draw [->, thick] (21) -- (28);
\draw [->, thick] (28) -- (19);
\draw [->, thick] (19) -- (11);
\draw [->, thick] (11) -- (5);
\draw [->, thick] (5) -- (1);

\draw [->, thick, dashed] (1) -- (4);
\draw [->, thick, dashed] (4) -- (9);
\draw [->, thick, dashed] (9) -- (16);
\draw [->, thick, dashed] (24) -- (15);
\draw [->, thick, dashed] (15) -- (8);
\draw [->, thick, dashed] (8) -- (3);

\draw [->, thick] (2) -- (5);
\draw [->, thick] (5) -- (10);
\draw [->, thick] (10) -- (17);
\draw [->, thick] (17) -- (25);
\draw [->, thick] (25) -- (23);
\draw [->, thick] (23) -- (14);
\draw [->, thick] (14) -- (7);
\draw [->, thick] (7) -- (2);

\draw [->, thick] (6) -- (11);
\draw [->, thick] (11) -- (18);
\draw [->, thick] (18) -- (26);
\draw [->, thick] (26) -- (22);
\draw [->, thick] (22) -- (13);
\draw [->, thick] (13) -- (6);

\draw [->, thick] (12) -- (19);
\draw [->, thick] (19) -- (27);
\draw [->, thick] (27) -- (21);
\draw [->, thick] (21) -- (12);

\end{tikzpicture}
\caption{The quiver $\Qc^{\mathrm{std}}_4$.}
\label{fig-D4}
\end{figure}

\begin{remark} Figure~\ref{fig-D4} contains four rhombi: the first one with vertices 1, 2, 3, 28, the second one with vertices 4, 6, 8, 27, the third one with vertices 9, 12, 15, 26, and finally the fourth one with vertices 16, 20, 24, 25. Note that for each simple root $\alpha_i$, the image of the corresponding $\sl_2$-triple depends only on cluster coordinates lying on the $i$-th rhombus.
\end{remark}

\subsection{Positive representations of $U_q\hr{\mathfrak{sl}_{n+1}}$}
\label{pos-rep-single}
The positive representations of $U_q(\mathfrak{sl}_{n+1})$ are obtained by restricting positive representations of the quantum torus algebra $\Dstd$. As explained in Section~\ref{positive-rep-sec}, positive representations of the algebra $\Dstd$ are labelled by its central characters. The center of $\Dstd$ can be described as follows. The quiver $\Qc_n^{\mathrm{std}}$ contains $n$ distinguished counter-clockwise oriented cycles $C_1, \dots, C_{n}$; in the notations of Figure~\ref{fig-D4} they are
\begin{align*}
C_1 &= \hc{2, 5, 10, 17, 25, 23, 14, 7}; &
C_3 &= \hc{12, 19, 27, 21}; \\
C_2 &= \hc{6, 11, 18, 26, 22, 13}; &
C_4 &= \hc{20, 28}.
\end{align*}
The 2-cycle $C_n$ is a degenerate one: its arrows $v \to w$ and $w \to v$ are cancelled out when the two triangles are amalgamated. For $j = 1, \dots, n$, we define cluster monomials $\Omega_j\in\Dstd$ by
\beq
\label{single-center-gens}
\Omega_{j} = X_{\sum_{k\in C_j}e_k}.
\eeq

\begin{lemma}
The center of the quantum torus algebra $\Dstd$ is the Laurent polynomial ring generated by the elements
\beq
\label{center-1}
\Omega_{j} \qquad\text{and}\qquad \iota\hr{K_jK_j'}
\eeq
for all $1 \le j \le n$.
\end{lemma}

\begin{proof}
The statement of the lemma follows from counting the rank of the adjacency matrix of the quiver $\Qc_n^{\mathrm{std}}$, and an observation that the monomials~\eqref{center-1} are primitive and independent.
\end{proof}

\begin{defn}
\label{pos-rep-def}
Let
$$
\mathfrak{h}_n = \hc{\hr{\lambda_0,\lambda_1,\ldots,\lambda_n}\in\R^{n+1} \bigg | \sum_{k=0}^n\lambda_k=0}
$$
Then for $\lambda\in\mathfrak{h}_n$, the positive representation $\Pc_\lambda$ of $U_q(\sl_{n+1})$ is the restriction to $U_q(\sl_{n+1})$ of the positive representation of the quantum torus algebra $\Dstd$ where for all $1\le j\le n$ we have
$$
\Omega_j\longmapsto e^{2\pi \hbar \hr{\lambda_j-\lambda_{j-1}}}, \qquad
\iota\hr{K_jK_j'}\longmapsto 1.
$$
\end{defn}

\subsection{Diagonally embedded quantum group}
\label{diag-embed-std}

\begin{figure}[b]
\begin{tikzpicture}[every node/.style={inner sep=0, minimum size=0.4cm, thick, draw, circle}, x=0.75cm, y=0.5cm]

\node[rectangle] (1) at (-8,-3) {\scriptsize{1}};
\node (2) at (-4,-7) {\scriptsize{2}};
\node (3) at (0,-3) {\scriptsize{3}};
\node (4) at (4,-7) {\scriptsize{4}};
\node[rectangle] (5) at (8,-3) {\scriptsize{5}};
\node[rectangle] (6) at (-8,-1) {\scriptsize{6}};
\node (7) at (-7,-2) {\scriptsize{7}};
\node (8) at (-4,-5) {\scriptsize{8}};
\node (9) at (-1,-2) {\scriptsize{9}};
\node (10) at (0,-1) {\scriptsize{10}};
\node (11) at (1,-2) {\scriptsize{11}};
\node (12) at (4,-5) {\scriptsize{12}};
\node (13) at (7,-2) {\scriptsize{13}};
\node[rectangle] (14) at (8,-1) {\scriptsize{14}};
\node[rectangle] (15) at (-8,1) {\scriptsize{15}};
\node (16) at (-7,0) {\scriptsize{16}};
\node (17) at (-6,-1) {\scriptsize{17}};
\node (18) at (-4,-3) {\scriptsize{18}};
\node (19) at (-2,-1) {\scriptsize{19}};
\node (20) at (-1,0) {\scriptsize{20}};
\node (21) at (0,1) {\scriptsize{21}};
\node (22) at (1,0) {\scriptsize{22}};
\node (23) at (2,-1) {\scriptsize{23}};
\node (24) at (4,-3) {\scriptsize{24}};
\node (25) at (6,-1) {\scriptsize{25}};
\node (26) at (7,0) {\scriptsize{26}};
\node[rectangle] (27) at (8,1) {\scriptsize{27}};
\node[rectangle] (28) at (-8,3) {\scriptsize{28}};
\node (29) at (-7,2) {\scriptsize{29}};
\node (30) at (-6,1) {\scriptsize{30}};
\node (31) at (-5,0) {\scriptsize{31}};
\node (32) at (-4,-1) {\scriptsize{32}};
\node (33) at (-3,0) {\scriptsize{33}};
\node (34) at (-2,1) {\scriptsize{34}};
\node (35) at (-1,2) {\scriptsize{35}};
\node (36) at (0,3) {\scriptsize{36}};
\node (37) at (1,2) {\scriptsize{37}};
\node (38) at (2,1) {\scriptsize{38}};
\node (39) at (3,0) {\scriptsize{39}};
\node (40) at (4,-1) {\scriptsize{40}};
\node (41) at (5,0) {\scriptsize{41}};
\node (42) at (6,1) {\scriptsize{42}};
\node (43) at (7,2) {\scriptsize{43}};
\node[rectangle] (44) at (8,3) {\scriptsize{44}};
\node (45) at (4,7) {\scriptsize{45}};
\node (46) at (4,5) {\scriptsize{46}};
\node (47) at (4,3) {\scriptsize{47}};
\node (48) at (4,1) {\scriptsize{48}};
\node (49) at (-4,7) {\scriptsize{49}};
\node (50) at (-4,5) {\scriptsize{50}};
\node (51) at (-4,3) {\scriptsize{51}};
\node (52) at (-4,1) {\scriptsize{52}};

\draw [->, thick] (1) -- (2);
\draw [->, thick] (2) -- (3);
\draw [->, thick] (3) -- (4);
\draw [->, thick] (4) -- (5);

\draw [->, thick] (6) -- (7);
\draw [->, thick] (7) -- (8);
\draw [->, thick] (8) -- (9);
\draw [->, thick] (9) -- (10);
\draw [->, thick] (10) -- (11);
\draw [->, thick] (11) -- (12);
\draw [->, thick] (12) -- (13);
\draw [->, thick] (13) -- (14);

\draw [->, thick] (15) -- (16);
\draw [->, thick] (16) -- (17);
\draw [->, thick] (17) -- (18);
\draw [->, thick] (18) -- (19);
\draw [->, thick] (19) -- (20);
\draw [->, thick] (20) -- (21);
\draw [->, thick] (21) -- (22);
\draw [->, thick] (22) -- (23);
\draw [->, thick] (23) -- (24);
\draw [->, thick] (24) -- (25);
\draw [->, thick] (25) -- (26);
\draw [->, thick] (26) -- (27);

\draw [->, thick] (28) -- (29);
\draw [->, thick] (29) -- (30);
\draw [->, thick] (30) -- (31);
\draw [->, thick] (31) -- (32);
\draw [->, thick] (32) -- (33);
\draw [->, thick] (33) -- (34);
\draw [->, thick] (34) -- (35);
\draw [->, thick] (35) -- (36);
\draw [->, thick] (36) -- (37);
\draw [->, thick] (37) -- (38);
\draw [->, thick] (38) -- (39);
\draw [->, thick] (39) -- (40);
\draw [->, thick] (40) -- (41);
\draw [->, thick] (41) -- (42);
\draw [->, thick] (42) -- (43);
\draw [->, thick] (43) -- (44);

\draw [->, thick] (44) -- (45);
\draw [->, thick] (45) -- (36);
\draw [->, thick] (36) -- (49);
\draw [->, thick] (49) -- (28);

\draw [->, thick] (27) -- (43);
\draw [->, thick] (43) -- (46);
\draw [->, thick] (46) -- (37);
\draw [->, thick] (37) -- (21);
\draw [->, thick] (21) -- (35);
\draw [->, thick] (35) -- (50);
\draw [->, thick] (50) -- (29);
\draw [->, thick] (29) -- (15);

\draw [->, thick] (14) -- (26);
\draw [->, thick] (26) -- (42);
\draw [->, thick] (42) -- (47);
\draw [->, thick] (47) -- (38);
\draw [->, thick] (38) -- (22);
\draw [->, thick] (22) -- (10);
\draw [->, thick] (10) -- (20);
\draw [->, thick] (20) -- (34);
\draw [->, thick] (34) -- (51);
\draw [->, thick] (51) -- (30);
\draw [->, thick] (30) -- (16);
\draw [->, thick] (16) -- (6);

\draw [->, thick] (5) -- (13);
\draw [->, thick] (13) -- (25);
\draw [->, thick] (25) -- (41);
\draw [->, thick] (41) -- (48);
\draw [->, thick] (48) -- (39);
\draw [->, thick] (39) -- (23);
\draw [->, thick] (23) -- (11);
\draw [->, thick] (11) -- (3);
\draw [->, thick] (3) -- (9);
\draw [->, thick] (9) -- (19);
\draw [->, thick] (19) -- (33);
\draw [->, thick] (33) -- (52);
\draw [->, thick] (52) -- (31);
\draw [->, thick] (31) -- (17);
\draw [->, thick] (17) -- (7);
\draw [->, thick] (7) -- (1);

\draw [->, thick] (18) -- (31);
\draw [->, thick] (31) -- (51);
\draw [->, thick] (51) -- (33);
\draw [->, thick] (33) -- (18);

\draw [->, thick] (8) -- (17);
\draw [->, thick] (17) -- (30);
\draw [->, thick] (30) -- (50);
\draw [->, thick] (50) -- (34);
\draw [->, thick] (34) -- (19);
\draw [->, thick] (19) -- (8);

\draw [->, thick] (2) -- (7);
\draw [->, thick] (7) -- (16);
\draw [->, thick] (16) -- (29);
\draw [->, thick] (29) -- (49);
\draw [->, thick] (49) -- (35);
\draw [->, thick] (35) -- (20);
\draw [->, thick] (20) -- (9);
\draw [->, thick] (9) -- (2);

\draw [->, thick] (24) -- (39);
\draw [->, thick] (39) -- (47);
\draw [->, thick] (47) -- (41);
\draw [->, thick] (41) -- (24);

\draw [->, thick] (12) -- (23);
\draw [->, thick] (23) -- (38);
\draw [->, thick] (38) -- (46);
\draw [->, thick] (46) -- (42);
\draw [->, thick] (42) -- (25);
\draw [->, thick] (25) -- (12);

\draw [->, thick] (4) -- (11);
\draw [->, thick] (11) -- (22);
\draw [->, thick] (22) -- (37);
\draw [->, thick] (37) -- (45);
\draw [->, thick] (45) -- (43);
\draw [->, thick] (43) -- (26);
\draw [->, thick] (26) -- (13);
\draw [->, thick] (13) -- (4);

\draw [->, thick, dashed] (1) -- (6);
\draw [->, thick, dashed] (6) -- (15);
\draw [->, thick, dashed] (15) -- (28);

\draw [->, thick, dashed] (44) -- (27);
\draw [->, thick, dashed] (27) -- (14);
\draw [->, thick, dashed] (14) -- (5);

\end{tikzpicture}
\caption{Quivers $\Qc_4^{\mathrm{sq}}$.}
\label{Z4}
\end{figure}

Let $\Dsq \subset \Dstd \otimes \Dstd$ be the quantum torus obtained by amalgamating two copies of $\Dstd$ as shown in Figure~\ref{Z4}. Consider a homomorphism $\tau \colon \Dgt_n \longra \Dstd \otimes \Dstd$ defined by
\beq
\label{tau}
\tau = (\iota \otimes \iota)\Delta.
\eeq
As explained in~\cite{SS16}, the image $\tau(\Dgt_n)$ lies in the subalgebra $\Dsq$. Let us describe this embedding
$$
\tau \colon \Dgt_n \longra \Dsq
$$
via the running example for when $n=4$. In the notations of Figure~\ref{Z4} we have
$$
E_i \longmapsto S_{[2n^2-n,2n^2+3n-1]} \qquad\text{and}\qquad K_i \longmapsto P_{[2n^2-n,2n^2+3n]}
$$
Formulas for $F_{n+1-i}$ and $K'_{n+1-i}$ are obtained respectively from those for $E_i$ and $K_i$ by rotating the corresponding quiver $180^\circ$, for example, we have
\begin{align*}
&E_1 \longmapsto S(X_1,X_2,X_3,X_4) &
&K_1 \longmapsto P(X_1,X_2,X_3,X_4,X_5) \\
&F_4 \longmapsto S(X_{44},X_{45},X_{36},X_{49}) &
&K'_4 \longmapsto P(X_{44},X_{45},X_{36},X_{49},X_{28}).
\end{align*}

\begin{lemma}
The center of the quantum torus algebra $\Dsq$ is the Laurent polynomial ring generated by the elements 
$$
\Omega_j^1, \qquad
\Omega_j^2, \qquad\text{and}\qquad
\tau\hr{K_jK_j'}
$$
for $j = 1, \dots, n$. Here $\Omega_j^1$ and $\Omega_k^2$ are the central elements~\eqref{single-center-gens} corresponding to the left and right $\Qstd$-subquivers of the quiver $\Qsq$.
\end{lemma}

\begin{proof}
Again, the proof follows from calculating the rank of the adjacency matrix of $\Qsq$.
\end{proof}

Now let $\mathcal{P}_\lambda, \mathcal{P}_\mu$ be two positive representations of $U_q(\mathfrak{sl}_{n+1})$, see Definition~\ref{pos-rep-def}. Then their tensor product $\mathcal{P}_{\lambda}\otimes\mathcal{P}_\mu$ is the restriction to $U_q\hr{\mathfrak{sl}_{n+1}}$ of the positive representation of $\Dsq$ where for all $1\le j\le n$ we have
$$
\Omega^1_{j} \longmapsto e^{2\pi \hbar \hr{\lambda_j-\lambda_{j-1}}}, \qquad
\Omega^2_{j} \longmapsto e^{2\pi \hbar \hr{\mu_j-\mu_{j-1}}}, \qquad
\tau\hr{K_jK_j'} \longmapsto 1.
$$

\section{Cluster combinatorics of directed networks}
\label{sec-net}

\subsection{Planar directed networks}

Let $D_k$ be a closed disk with $k$ boundary points removed. Recall~\cite{Pos06,GSV12, GSV09} that a directed network in $D_k$ is a planar oriented graph drawn in $D_k$ with vertices of three types: \emph{black} vertices of valency at least three, which have exactly one incoming edge, \emph{white} vertices of valency at least three, which have exactly one outgoing edge, and \emph{terminal} vertices of valency one, which we regard as lying on the boundary~$\partial D_k$. We will refer to terminal vertices as sources and sinks depending on the orientation of the adjacent edge. Sometimes we will omit drawing terminal vertices, and simply draw hanging edges instead. For example, Figure~\ref{fig-dir-net} shows a directed network without terminal vertices. Here we only consider networks that satisfy the following two conditions, which we incorporate in the definition of a network:
\begin{enumerate}
\item there are no oriented cycles;
\item a terminal vertex can only share an edge with a white vertex.
\end{enumerate}
Given a directed network $\mathcal{N}$, one can recover a quiver $\Qc = \Qc\hr{\Nc}$ in the following way:
\begin{itemize}
\item draw a vertex of the quiver $\Qc$ inside each face of the network $\Nc$ that has vertices of different colors;
\item for each edge $e \in E(\Nc)$ with vertices of different colors, draw an arrow of the quiver~$\Qc$ that crosses $e$ in such a way that the white vertex lies on the right.
\item for each $e \in E(\Nc)$ joining a terminal vertex with an interior (white) vertex, draw a \emph{dotted} arrow of weight $\frac{1}{2}$ crossing $e$ in such a way that the interior vertex lies on the right. 
\end{itemize}
The quiver $\Qc$ obtained by this procedure is a planar directed graph in $D_k$ with the property that all of its faces are oriented either clockwise (those dual to white vertices) or counterclockwise (those dual to black ones). We regard the nodes of $\Qc$ dual to faces that meet a boundary component of $D_k$ as being frozen. It is important to note that $\Qc$ does not actually depend on direction of the edges of the network $\Nc$, but only on the underlying bi-colored graph. Moreover, the quiver $\Qc$ is unchanged if we contract or create edges in $\Nc$ between vertices of the same color, as shown on Figure~\ref{net-id}. For this reason, we consider two networks to be equivalent if they can be obtained from each other by applying such transformations.

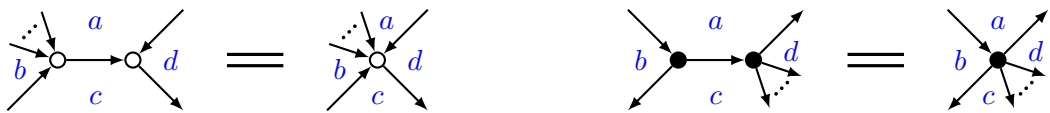
\begin{figure}[t]
\begin{tikzpicture}[every node/.style={inner sep=0, minimum size=0.2cm, circle, draw, thick}, x=0.25cm,y=0.25cm]

\node[draw=none] (a1') at (0,1) {};
\node[draw=none] (a1) at (2,3) {};
\node[draw=none] (a2) at (0,-3) {};
\node[fill=white] (a3) at (3,0) {};
\node[fill=white] (a4) at (7,0) {};
\node[draw=none] (a5) at (10,3) {};
\node[draw=none] (a6) at (10,-3) {};

\node[fill, minimum size=0.02cm] at (1.2,1.2) {};
\node[fill, minimum size=0.02cm] at (1.5,1.5) {};
\node[fill, minimum size=0.02cm] at (1.8,1.8) {};
\node[draw=none, text=blue] at (1,-0.5) {$b$};
\node[draw=none, text=blue] at (5,2) {$a$};
\node[draw=none, text=blue] at (9,0) {$d$};
\node[draw=none, text=blue] at (5,-2) {$c$};

\draw[thick, ->] (a1') to (a3);
\draw[thick, ->] (a1) to (a3);
\draw[thick, ->] (a2) to (a3);
\draw[thick, ->] (a3) to (a4);
\draw[thick, ->] (a5) to (a4);
\draw[thick, ->] (a4) to (a6);

\node[draw=none] (a7') at (17,1) {};
\node[draw=none] (a7) at (19,3) {};
\node[draw=none] (a8) at (17,-3) {};
\node[fill=white] (a9) at (20,0) {};
\node[draw=none] (a10) at (23,3) {};
\node[draw=none] (a11) at (23,-3) {};

\node[fill, minimum size=0.02cm] at (18.2,1.2) {};
\node[fill, minimum size=0.02cm] at (18.5,1.5) {};
\node[fill, minimum size=0.02cm] at (18.8,1.8) {};
\node[draw=none, text=blue] at (18,-0.5) {$b$};
\node[draw=none, text=blue] at (20.5,2) {$a$};
\node[draw=none, text=blue] at (22,0) {$d$};
\node[draw=none, text=blue] at (20,-2) {$c$};

\draw[thick, ->] (a7') to (a9);
\draw[thick, ->] (a7) to (a9);
\draw[thick, ->] (a8) to (a9);
\draw[thick, ->] (a10) to (a9);
\draw[thick, ->] (a9) to (a11);

\draw [-, very thick] (12,-0.3) to (15,-0.3);
\draw [-, very thick] (12,0.3) to (15,0.3);

\node[draw=none] (a1) at (33,3) {};
\node[draw=none] (a2) at (33,-3) {};
\node[fill] (a3) at (36,0) {};
\node[fill] (a4) at (40,0) {};
\node[draw=none] (a5) at (43,3) {};
\node[draw=none] (a6) at (43,-1) {};
\node[draw=none] (a6') at (41,-3) {};

\node[fill, minimum size=0.02cm] at (41.8,-1.2) {};
\node[fill, minimum size=0.02cm] at (41.5,-1.5) {};
\node[fill, minimum size=0.02cm] at (41.2,-1.8) {};
\node[draw=none, text=blue] at (34,0) {$b$};
\node[draw=none, text=blue] at (38,2) {$a$};
\node[draw=none, text=blue] at (42,0.5) {$d$};
\node[draw=none, text=blue] at (38,-2) {$c$};

\draw[thick, ->] (a1) to (a3);
\draw[thick, ->] (a3) to (a2);
\draw[thick, ->] (a3) to (a4);
\draw[thick, ->] (a4) to (a5);
\draw[thick, ->] (a4) to (a6);
\draw[thick, ->] (a4) to (a6');

\node[draw=none] (a7) at (50,3) {};
\node[draw=none] (a8) at (50,-3) {};
\node[fill] (a9) at (53,0) {};
\node[draw=none] (a10) at (56,3) {};
\node[draw=none] (a11) at (56,-1) {};
\node[draw=none] (a11') at (54,-3) {};

\node[fill, minimum size=0.02cm] at (54.8,-1.2) {};
\node[fill, minimum size=0.02cm] at (54.5,-1.5) {};
\node[fill, minimum size=0.02cm] at (54.2,-1.8) {};
\node[draw=none, text=blue] at (51,0) {$b$};
\node[draw=none, text=blue] at (53,2) {$a$};
\node[draw=none, text=blue] at (55,0.5) {$d$};
\node[draw=none, text=blue] at (52.5,-2) {$c$};

\draw[thick, ->] (a7) to (a9);
\draw[thick, ->] (a9) to (a8);
\draw[thick, ->] (a9) to (a10);
\draw[thick, ->] (a9) to (a11);
\draw[thick, ->] (a9) to (a11');

\draw [-, very thick] (45,-0.3) to (48,-0.3);
\draw [-, very thick] (45,0.3) to (48,0.3);

\end{tikzpicture}
\caption{Network identities.}
\label{net-id}
\end{figure}

\begin{figure}[b]
\begin{tikzpicture}[every node/.style={inner sep=0, minimum size=0.2cm, draw, thick, circle}, x=0.3cm, y=0.346cm]

\draw[gray] (0,0) to (10,15) to (20,0) to (0,0);

\foreach \i in {4,8,12,16}
	\node[blue, thick, rectangle, inner sep=0, minimum size=0.35cm, fill=white] (\i_0) at (\i,0) {};
\foreach \i in {2,18}
	\node[blue, thick, rectangle, inner sep=0, minimum size=0.35cm, fill=white] (\i_3) at (\i,3) {};
\foreach \i in {4,16}
	\node[blue, thick, rectangle, inner sep=0, minimum size=0.35cm, fill=white] (\i_6) at (\i,6) {};
\foreach \i in {6,14}
	\node[blue, thick, rectangle, inner sep=0, minimum size=0.35cm, fill=white] (\i_9) at (\i,9) {};
\foreach \i in {8,12}
	\node[blue, thick, rectangle, inner sep=0, minimum size=0.35cm, fill=white] (\i_12) at (\i,12) {};
\foreach \i in {6,10,14}
	\node[blue, thick, inner sep=0, minimum size=0.35cm, fill=white] (\i_3) at (\i,3) {};
\foreach \i in {8,12}
	\node[blue, thick, inner sep=0, minimum size=0.35cm, fill=white] (\i_6) at (\i,6) {};
\foreach \i in {10}
	\node[blue, thick, inner sep=0, minimum size=0.35cm, fill=white] (\i_9) at (\i,9) {};

\draw[blue, thick,->] (4_0) to (6_3);
\draw[blue, thick,->] (6_3) to (8_6);
\draw[blue, thick,->] (8_6) to (10_9);
\draw[blue, thick,->] (10_9) to (12_12);
\draw[blue, thick,->] (12_12) to (8_12);
\draw[blue, thick,->] (8_12) to (10_9);
\draw[blue, thick,->] (10_9) to (12_6);
\draw[blue, thick,->] (12_6) to (14_3);
\draw[blue, thick,->] (14_3) to (16_0);
\draw[blue, thick,->] (16_0) to (18_3);
\draw[blue, thick,->] (18_3) to (14_3);
\draw[blue, thick,->] (14_3) to (10_3);
\draw[blue, thick,->] (10_3) to (6_3);
\draw[blue, thick,->] (6_3) to (2_3);
\draw[blue, thick,->] (2_3) to (4_0);

\draw[blue, thick,->] (4_6) to (6_3);
\draw[blue, thick,->] (6_3) to (8_0);
\draw[blue, thick,->] (8_0) to (10_3);
\draw[blue, thick,->] (10_3) to (12_6);
\draw[blue, thick,->] (12_6) to (14_9);
\draw[blue, thick,->] (14_9) to (10_9);
\draw[blue, thick,->] (10_9) to (6_9);
\draw[blue, thick,->] (6_9) to (8_6);
\draw[blue, thick,->] (8_6) to (10_3);
\draw[blue, thick,->] (10_3) to (12_0);
\draw[blue, thick,->] (12_0) to (14_3);
\draw[blue, thick,->] (14_3) to (16_6);
\draw[blue, thick,->] (16_6) to (12_6);
\draw[blue, thick,->] (12_6) to (8_6);
\draw[blue, thick,->] (8_6) to (4_6);

\draw[blue, thick,dashed,->] (16_0) to (12_0);
\draw[blue, thick,dashed,->] (12_0) to (8_0);
\draw[blue, thick,dashed,->] (8_0) to (4_0);
\draw[blue, thick,dashed,->] (2_3) to (4_6);
\draw[blue, thick,dashed,->] (4_6) to (6_9);
\draw[blue, thick,dashed,->] (6_9) to (8_12);
\draw[blue, thick,dashed,->] (12_12) to (14_9);
\draw[blue, thick,dashed,->] (14_9) to (16_6);
\draw[blue, thick,dashed,->] (16_6) to (18_3);

\foreach \i in {2,6,10,14,18}
	\draw[thick] (\i,1) to (\i,0) {};
\foreach \i in {2,4,6,8,10}
	\draw[thick] (\i,3/2*\i-2) to (\i-1,3/2*\i-1.5) {};
\foreach \i in {10,12,14,16,18}
	\draw[thick] (\i,28-3/2*\i) to (\i+1,28.5-3/2*\i) {};
	
\foreach \i in {4,8,12,16}
{
	\draw[thick] (\i,2) to (\i,4) {};
	\draw[thick] (\i,2) to (\i-2,1) {};
	\draw[thick] (\i,2) to (\i+2,1) {};
}
\foreach \i in {6,10,14}
{
	\draw[thick] (\i,5) to (\i,7) {};
	\draw[thick] (\i,5) to (\i-2,4) {};
	\draw[thick] (\i,5) to (\i+2,4) {};
}
\foreach \i in {8,12}
{
	\draw[thick] (\i,8) to (\i,10) {};
	\draw[thick] (\i,8) to (\i-2,7) {};
	\draw[thick] (\i,8) to (\i+2,7) {};
}
\foreach \i in {10}
{
	\draw[thick] (\i,11) to (\i,13) {};
	\draw[thick] (\i,11) to (\i-2,10) {};
	\draw[thick] (\i,11) to (\i+2,10) {};
}

\foreach \i in {2,6,10,14,18}
	\node[fill=white] (\i_1) at (\i,1) {};
\foreach \i in {4,8,12,16}
	\node[fill=white] (\i_4) at (\i,4) {};
\foreach \i in {6,10,14}
	\node[fill=white] (\i_7) at (\i,7) {};
\foreach \i in {8,12}
	\node[fill=white] (\i_10) at (\i,10) {};
\foreach \i in {10}
	\node[fill=white] (\i_13) at (\i,13) {};

\foreach \i in {4,8,12,16}
	\node[fill] (\i_2) at (\i,2) {};
\foreach \i in {6,10,14}
	\node[fill] (\i_5) at (\i,5) {};
\foreach \i in {8,12}
	\node[fill] (\i_8) at (\i,8) {};
\foreach \i in {10}
	\node[fill] (\i_11) at (\i,11) {};

\draw[thick, opacity=0, ->, shift={(10,15)}, rotate=240] (8,13.5) to (10,12.5) to (12,13.5);

\end{tikzpicture}
\caption{$PGL_5$-quiver in $D_3$ and the dual bi-colored graph.}
\label{network}
\end{figure}
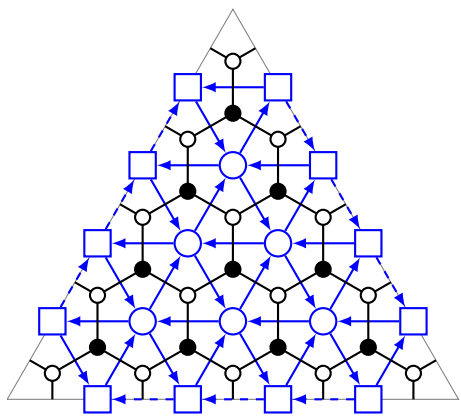

Figure~\ref{network} shows a bi-colored graph (in black) whose dual quiver (in blue) coincides with the one discussed in~\ref{locsys-intro}. Note, that the concatenation of such bi-colored graphs corresponds to the amalgamation of quivers. Let us choose a triangulation of $D_k$. We turn the corresponding bi-colored graph into a network $\Nc$ by orienting its edges. Let $\Qc$ be the dual quiver, which we assume is associated to a cluster seed $\Theta$. By collapsing each component of the boundary $\partial D_k$ to a point, we obtain a network $\overline\Nc$. The construction of $\Qc$ as the dual graph to $\mathcal{N}$ delivers an isomorphism $H_1(\overline{\mathcal{N}},\Z)\rightarrow \Lambda$. This isomorphism sends the (clockwise oriented) cycle $\gamma_k$ around the face $\mathcal{F}_k \in F(\overline\Nc)$ corresponding to a vertex $k\in V(\Qc)$, to the basis vector $e_k$ in the seed $\Theta$. 

Suppose that $(v,w)$ and  $(v_0,w_0)$ are two pairs vertices in $\Nc$ that both project to the same pair of vertices $(\bar v, \bar w)$ in the network $\overline\Nc$. Then for every pair of directed paths $\gamma \colon v \to w$ and $\gamma_0 \colon v_0 \to w_0$ in the network $\mathcal{N}$, the class $[\gamma\gamma_0^{-1}]$ defines an element of $H_1(\overline{\mathcal{N}},\Z) \simeq \Lambda$. Thus, we can define the weighted path count
\beq
\label{path-count}
Z_\mathcal{N}(v,w;v_0,w_0) = \sum_{\gamma, \gamma_0} X_{\hs{\gamma\gamma_0^{-1}}},
\eeq
where the sum is taken over all pairs of directed paths $\gamma \colon v \to w$ and $\gamma_0 \colon v_0 \to w_0$ in the network $\mathcal{N}$.
 
\begin{remark}
In practice, the vertices $v_0,w_0$ will usually be chosen in such a way that there is a unique directed path $\gamma_0 \colon v_0 \to w_0$ in the network $\mathcal{N}$; this path then serves as a ``reference path'' used to define the weights of all directed paths $\gamma \colon v\to w$. 
\end{remark}

In what follows, it will be useful to describe certain quiver mutations in the language of directed networks. Namely, consider a directed network $\Nc$ with dual quiver $\Qc$ and corresponding cluster seed $\Theta$. Let $k \in V(\Qc)$ be a 4-valent vertex with alternating incoming and outgoing arrows. Contracting or creating edges between vertices of the same color, as shown in Figure~\ref{net-id}, we can always make the corresponding face $\mathcal{F}_k \in F(\Nc)$ quadrilateral, with vertices of alternating color and of valency 3. Consider the seed $\Theta'=\mu_k(\Theta)$ obtained by mutating $\Theta$ in direction $k$. Then the network $\Nc'$ corresponding to the mutated cluster seed $\Theta'$ is obtained from $\Nc$ by a simple local rule, shown in Figure~\ref{net-mut}.

\begin{figure}[h]
\begin{tikzpicture}[every node/.style={inner sep=0, minimum size=0.2cm, circle, draw, thick}, x=0.25cm,y=0.25cm]

\node[fill=black] (a1) at (3,3) {};
\node[fill=white] (a2) at (7,3) {};
\node[fill=black] (a3) at (7,7) {};
\node[fill=white] (a4) at (3,7) {};

\node[draw=none, fill=white] (b1) at (0,0) {};
\node[draw=none, fill=white] (b2) at (10,0) {};
\node[draw=none, fill=white] (b3) at (10,10) {};
\node[draw=none, fill=white] (b4) at (0,10) {};

\draw [->, thick] (b1) to (a1);
\draw [<-, thick] (b2) to (a2);
\draw [<-, thick] (b3) to (a3);
\draw [->, thick] (b4) to (a4);

\draw [->, thick] (a1) to (a2);
\draw [->, thick] (a2) to (a3);
\draw [<-, thick] (a3) to (a4);
\draw [->, thick] (a4) to (a1);

\node[draw=none, text=blue] at (5,5) {$k$};
\node[draw=none, text=blue] at (5,1) {$a$};
\node[draw=none, text=blue] at (1,5) {$b$};
\node[draw=none, text=blue] at (5,9) {$c$};
\node[draw=none, text=blue] at (9,5) {$d$};

\draw [<->, very thick] (12,5) to (18,5);
\node[draw=none] at (15,6) {$\mu_k$};

\node[fill=white] (c1) at (3+20,3) {};
\node[fill=black] (c2) at (7+20,3) {};
\node[fill=white] (c3) at (7+20,7) {};
\node[fill=black] (c4) at (3+20,7) {};

\node[draw=none, fill=white] (d1) at (0+20,0) {};
\node[draw=none, fill=white] (d2) at (10+20,0) {};
\node[draw=none, fill=white] (d3) at (10+20,10) {};
\node[draw=none, fill=white] (d4) at (0+20,10) {};

\draw [->, thick] (d1) to (c1);
\draw [<-, thick] (d2) to (c2);
\draw [<-, thick] (d3) to (c3);
\draw [->, thick] (d4) to (c4);

\draw [->, thick] (c1) to (c2);
\draw [<-, thick] (c2) to (c3);
\draw [<-, thick] (c3) to (c4);
\draw [<-, thick] (c4) to (c1);

\node[draw=none, text=blue] at (25,5) {$k'$};
\node[draw=none, text=blue] at (25,1) {$a'$};
\node[draw=none, text=blue] at (21,5) {$b'$};
\node[draw=none, text=blue] at (25,9) {$c'$};
\node[draw=none, text=blue] at (29,5) {$d'$};

\end{tikzpicture}
\caption{Network mutation.}
\label{net-mut}
\end{figure}
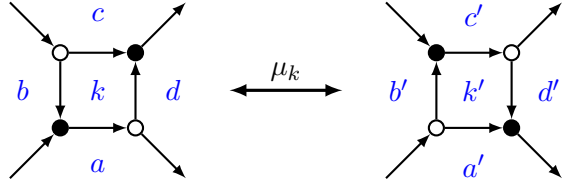

\begin{prop}
\label{measure-invariance}
Let $\Nc$ be a directed network, and let $\Nc'$ be the network obtained from $\Nc$ by mutating at the face $\Fc_k \in F(\Nc)$ as shown in Figure~\ref{net-mut}. Then for any pairs $(v,w)$ and $(v_0,w_0)$ that have the same image in $\overline\Nc$ we have
$$
\mu_k\hr{Z_{\Nc}(v,w;v_0,w_0)} = Z_{\Nc'}(v,w;v_0,w_0).
$$
\end{prop}

\begin{proof} 
The analogous statement for commutative cluster algebras is well-known, see for example~\cite[Lemma 12.2]{Pos06} or~\cite[Theorem 4.7]{GK11}. In formula~\eqref{path-count}, each class $[\gamma\gamma_0^{-1}]$ is a sum of simple closed curves, so we may assume that $[\gamma\gamma_0^{-1}]$ is a simple closed curve itself. Here we will consider only the following two cases; all the other cases can be treated similarly:
\begin{enumerate}
\item the cycle $[\gamma\gamma_0^{-1}]$ enters the graph on the left-hand-side of Figure~\ref{net-mut} through the bottom left source, leaves through the bottom right sink, and is otherwise disconnected from the faces of $\Nc$ shown in the Figure;
\item the cycle $[\gamma\gamma_0^{-1}]$ enters the graph on the left-hand-side of Figure~\ref{net-mut} through the top left source, leaves through the top right sink, and is otherwise disconnected from the faces of $\Nc$ shown in the Figure.
\end{enumerate}

Let us treat the first case. Then $X_{[\gamma\gamma_0^{-1}]} = X_{\lambda+e_a}$ for some $\lambda \in \Lambda$ satisfying $(\lambda,e_k)=0$. Every such path in the left network gives rise to two paths in the right one, with weights $X_{\lambda + e_{a'}}$ and $X_{\lambda + e_{a'}+e_{k'}}$ respectively. On the other hand, by \eqref{reg-mut}, we find
$$
\mu_k\hr{X_{\lambda+e_a}} = X_{\lambda + e_{a'}} + X_{\lambda + e_{a'}+e_{k'}}
$$
which proves the claim.

In the second case, there are two such paths in the left network, with weights $X_{\lambda+e_c}$ and $X_{\lambda + e_c + e_k}$, and only one such path in the right network, with weight $X_{\lambda + e_{c'}}$. By~\eqref{denom-mut} we have
$$
\mu_k(X_{\lambda+e_c}+ X_{\lambda + e_c + e_k})
= \mu_k\hr{X_{\lambda+e_c}(1+qX_{e_k})}
= X_{\lambda+ e_c}(1+q^{-1}X_{e_{k'}})^{-1}(1+qX_{e_k}).
$$
Using the equalities $e_{k'}=-e_k$ and $e_{c'}=e_c+e_k$, we obtain
$$
\mu_k(X_{\lambda+e_c}+ X_{\lambda + e_c + e_k})
= qX_{\lambda+e_c}X_{e_k}
= X_{\lambda+e_c+e_k}
= X_{\lambda+e_{c'}}.
$$
The Proposition is proved.

\end{proof}

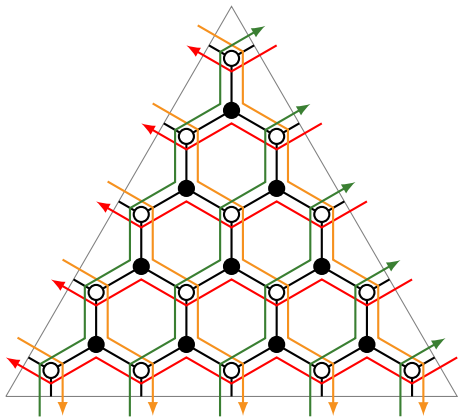
\subsection{Zig-zag paths}
The bi-colored graph in Figure~\ref{network} admits different network structures, which can be described as follows. Recall that a \emph{zig-zag path} in a bipartite graph is a path which turns maximally right at white vertices and maximally left at black ones. Figure~\ref{fig-path} shows all zig-zag paths on an the bipartite graph corresponding to the ideal triangle $D_3$ with $n=4$ . In an ideal triangle, each zig-zag path goes along the direction of one of the three sides. Thus, we can divide the zig-zag paths in the triangle into three types, which we label by different colors as shown in Figure~\ref{fig-path}. Paths of the same type do not intersect, whereas any two paths of different type intersect inside the triangle exactly once.

\begin{figure}[h]
\begin{minipage}{.5\textwidth}
\centering
\begin{tikzpicture}[every node/.style={inner sep=0, minimum size=0.2cm, draw, thick, circle}, x=0.3cm, y=0.346cm]

\draw[gray] (0,0) to (10,15) to (20,0) to (0,0);

\foreach \i in {2,6,10,14,18}
	\draw[thick] (\i,1) to (\i,0) {};
\foreach \i in {2,4,6,8,10}
	\draw[thick] (\i,3/2*\i-2) to (\i-1,3/2*\i-1.5) {};
\foreach \i in {10,12,14,16,18}
	\draw[thick] (\i,28-3/2*\i) to (\i+1,28.5-3/2*\i) {};
	
\foreach \i in {4,8,12,16}
{
	\draw[thick] (\i,2) to (\i,4) {};
	\draw[thick] (\i,2) to (\i-2,1) {};
	\draw[thick] (\i,2) to (\i+2,1) {};
}
\foreach \i in {6,10,14}
{
	\draw[thick] (\i,5) to (\i,7) {};
	\draw[thick] (\i,5) to (\i-2,4) {};
	\draw[thick] (\i,5) to (\i+2,4) {};
}
\foreach \i in {8,12}
{
	\draw[thick] (\i,8) to (\i,10) {};
	\draw[thick] (\i,8) to (\i-2,7) {};
	\draw[thick] (\i,8) to (\i+2,7) {};
}
\foreach \i in {10}
{
	\draw[thick] (\i,11) to (\i,13) {};
	\draw[thick] (\i,11) to (\i-2,10) {};
	\draw[thick] (\i,11) to (\i+2,10) {};
}

\draw[red, thick, <-] (0,1.5) to (2,0.5) to (4,1.5) to (6,0.5) to (8,1.5) to (10,0.5) to (12,1.5) to (14,0.5) to (16,1.5) to (18,0.5) to (20,1.5);
\draw[red, thick, <-] (2,4.5) to (4,3.5) to (6,4.5) to (8,3.5) to (10,4.5) to (12,3.5) to (14,4.5) to (16,3.5) to (18,4.5);
\draw[red, thick, <-] (4,7.5) to (6,6.5) to (8,7.5) to (10,6.5) to (12,7.5) to (14,6.5) to (16,7.5);
\draw[red, thick, <-] (6,10.5) to (8,9.5) to (10,10.5) to (12,9.5) to (14,10.5);
\draw[red, thick, <-] (8,13.5) to (10,12.5) to (12,13.5);

\draw[BurntOrange, thick, <-, shift={(20,0)}, rotate=120] (0,1.5) to (2,0.5) to (4,1.5) to (6,0.5) to (8,1.5) to (10,0.5) to (12,1.5) to (14,0.5) to (16,1.5) to (18,0.5) to (20,1.5);
\draw[BurntOrange, thick, <-, shift={(20,0)}, rotate=120] (2,4.5) to (4,3.5) to (6,4.5) to (8,3.5) to (10,4.5) to (12,3.5) to (14,4.5) to (16,3.5) to (18,4.5);
\draw[BurntOrange, thick, <-, shift={(20,0)}, rotate=120] (4,7.5) to (6,6.5) to (8,7.5) to (10,6.5) to (12,7.5) to (14,6.5) to (16,7.5);
\draw[BurntOrange, thick, <-, shift={(20,0)}, rotate=120] (6,10.5) to (8,9.5) to (10,10.5) to (12,9.5) to (14,10.5);
\draw[BurntOrange, thick, <-, shift={(20,0)}, rotate=120] (8,13.5) to (10,12.5) to (12,13.5);

\draw[OliveGreen, thick, <-, shift={(10,15)}, rotate=240] (0,1.5) to (2,0.5) to (4,1.5) to (6,0.5) to (8,1.5) to (10,0.5) to (12,1.5) to (14,0.5) to (16,1.5) to (18,0.5) to (20,1.5);
\draw[OliveGreen, thick, <-, shift={(10,15)}, rotate=240] (2,4.5) to (4,3.5) to (6,4.5) to (8,3.5) to (10,4.5) to (12,3.5) to (14,4.5) to (16,3.5) to (18,4.5);
\draw[OliveGreen, thick, <-, shift={(10,15)}, rotate=240] (4,7.5) to (6,6.5) to (8,7.5) to (10,6.5) to (12,7.5) to (14,6.5) to (16,7.5);
\draw[OliveGreen, thick, <-, shift={(10,15)}, rotate=240] (6,10.5) to (8,9.5) to (10,10.5) to (12,9.5) to (14,10.5);
\draw[OliveGreen, thick, <-, shift={(10,15)}, rotate=240] (8,13.5) to (10,12.5) to (12,13.5);

\foreach \i in {2,6,10,14,18}
	\node[fill=white] (\i_1) at (\i,1) {};
\foreach \i in {4,8,12,16}
	\node[fill=white] (\i_4) at (\i,4) {};
\foreach \i in {6,10,14}
	\node[fill=white] (\i_7) at (\i,7) {};
\foreach \i in {8,12}
	\node[fill=white] (\i_10) at (\i,10) {};
\foreach \i in {10}
	\node[fill=white] (\i_13) at (\i,13) {};

\foreach \i in {4,8,12,16}
	\node[fill] (\i_2) at (\i,2) {};
\foreach \i in {6,10,14}
	\node[fill] (\i_5) at (\i,5) {};
\foreach \i in {8,12}
	\node[fill] (\i_8) at (\i,8) {};
\foreach \i in {10}
	\node[fill] (\i_11) at (\i,11) {};

\end{tikzpicture}
\caption{Zig-zag paths.}
\label{fig-path}

\end{minipage}%
\begin{minipage}{.5\textwidth}
\centering

\begin{tikzpicture}[every node/.style={inner sep=0, minimum size=0.2cm, draw, thick, circle}, x=0.3cm, y=0.346cm]

\draw[gray] (0,0) to (10,15) to (20,0) to (0,0);

\foreach \i in {2,6,10,14,18}
	\node[fill=white] (\i_1) at (\i,1) {};
\foreach \i in {4,8,12,16}
	\node[fill=white] (\i_4) at (\i,4) {};
\foreach \i in {6,10,14}
	\node[fill=white] (\i_7) at (\i,7) {};
\foreach \i in {8,12}
	\node[fill=white] (\i_10) at (\i,10) {};
\foreach \i in {10}
	\node[fill=white] (\i_13) at (\i,13) {};

\foreach \i in {4,8,12,16}
	\node[fill] (\i_2) at (\i,2) {};
\foreach \i in {6,10,14}
	\node[fill] (\i_5) at (\i,5) {};
\foreach \i in {8,12}
	\node[fill] (\i_8) at (\i,8) {};
\foreach \i in {10}
	\node[fill] (\i_11) at (\i,11) {};

\draw[OliveGreen, thick, ->] (2,-1) to (2_1);
\draw[OliveGreen, thick, ->] (2_1) to (4_2);
\draw[OliveGreen, thick, ->] (4_2) to (4_4);
\draw[OliveGreen, thick, ->] (4_4) to (6_5);
\draw[OliveGreen, thick, ->] (6_5) to (6_7);
\draw[OliveGreen, thick, ->] (6_7) to (8_8);
\draw[OliveGreen, thick, ->] (8_8) to (8_10);
\draw[OliveGreen, thick, ->] (8_10) to (10_11);
\draw[OliveGreen, thick, ->] (10_11) to (10_13);
\draw[OliveGreen, thick, ->] (10_13) to (12,14);

\draw[OliveGreen, thick, ->] (6,-1) to (6_1);
\draw[OliveGreen, thick, ->] (6_1) to (8_2);
\draw[OliveGreen, thick, ->] (8_2) to (8_4);
\draw[OliveGreen, thick, ->] (8_4) to (10_5);
\draw[OliveGreen, thick, ->] (10_5) to (10_7);
\draw[OliveGreen, thick, ->] (10_7) to (12_8);
\draw[OliveGreen, thick, ->] (12_8) to (12_10);
\draw[OliveGreen, thick, ->] (12_10) to (14,11);

\draw[OliveGreen, thick, ->] (10,-1) to (10_1);
\draw[OliveGreen, thick, ->] (10_1) to (12_2);
\draw[OliveGreen, thick, ->] (12_2) to (12_4);
\draw[OliveGreen, thick, ->] (12_4) to (14_5);
\draw[OliveGreen, thick, ->] (14_5) to (14_7);
\draw[OliveGreen, thick, ->] (14_7) to (16,8);

\draw[OliveGreen, thick, ->] (14,-1) to (14_1);
\draw[OliveGreen, thick, ->] (14_1) to (16_2);
\draw[OliveGreen, thick, ->] (16_2) to (16_4);
\draw[OliveGreen, thick, ->] (16_4) to (18,5);

\draw[OliveGreen, thick, ->] (18,-1) to (18_1);
\draw[OliveGreen, thick, ->] (18_1) to (20,2);

\draw[BurntOrange, thick, ->] (0,2) to (2_1);
\draw[BurntOrange, thick, ->] (4_2) to (6_1);
\draw[BurntOrange, thick, ->] (8_2) to (10_1);
\draw[BurntOrange, thick, ->] (12_2) to (14_1);
\draw[BurntOrange, thick, ->] (16_2) to (18_1);

\draw[BurntOrange, thick, ->] (2,5) to (4_4);
\draw[BurntOrange, thick, ->] (6_5) to (8_4);
\draw[BurntOrange, thick, ->] (10_5) to (12_4);
\draw[BurntOrange, thick, ->] (14_5) to (16_4);

\draw[BurntOrange, thick, ->] (4,8) to (6_7);
\draw[BurntOrange, thick, ->] (8_8) to (10_7);
\draw[BurntOrange, thick, ->] (12_8) to (14_7);

\draw[BurntOrange, thick, ->] (6,11) to (8_10);
\draw[BurntOrange, thick, ->] (10_11) to (12_10);

\draw[BurntOrange, thick, ->] (8,14) to (10_13);

\end{tikzpicture}
\caption{Directed network.}
\label{fig-dir-net}
\end{minipage}
\end{figure}

Note that each edge of the bipartite graph has exactly two zig-zag paths passing through it, traversing the edge in opposite directions. Thus, to choose a direction on each edge in a triangle amounts to choosing a color for that edge. Now, suppose that we choose one side of the triangle and orient all the corresponding boundary edges pointing outward. Then, the network condition that each white vertex has exactly one outgoing edge while each black vertex has exactly one incoming completely determines the orientation of all the other edges. In particular, note that all other boundary edges are oriented inward, as illustrated in Figure~\ref{fig-dir-net}.

\subsection{Quivers on the punctured disk and the quantum universal cover}
\label{ann-net-sec}

In what follows, we will need to work not only with quivers drawn on a disk but also with those on a punctured disk. All quivers considered in this paper are of a very special kind -- they arise as (certain projections of) quivers dual to networks constructed in the following 4 steps:
\begin{enumerate}
\item choose an ideal triangulation of a marked surface $\wh S$;
\item in every ideal triangle draw a bi-colored graph of the form shown in Figure~\ref{network};
\item turn the resulting bi-colored graph into a network by orienting its edges;
\item mutate the resulting network in a (possibly infinite) sequence of faces of the form shown in Figure~\ref{net-mut}.
\end{enumerate}

Given a quiver $\Qc$ on the punctured disk one can draw a bi-colored graph $\Gamma$ such that $\Qc$ is dual to $\Gamma$. Let $\wdt\Gamma$ be the lift of $\Gamma$ to the universal cover of a punctured disk, and $\wdt \Qc$ be the planar quiver dual to $\wdt\Gamma$. We denote the cluster seeds associated to $\Qc$ and $\wdt \Qc$ by $\Theta$ and~$\wdt\Theta$ respectively. Part of the data of the seed $\Theta$ consists of the lattice~$\Lambda$ with basis $\{e_i\}$, $i\in V(\Qc)$. Let us choose a connected ``fundamental domain'' in $\wdt\Gamma$, so that the projection from the universal cover onto a punctured disk establishes a bijection between $\Gamma$ and the fundamental domain of~$\wdt\Gamma$. Then,  to the seed $\wdt\Theta$ is associated a lattice $\tLambda$ and a basis $\{e^n_i\}$ where $i\in V(\Qc)$ and $n\in\Z$. Here, vectors $\{e_i^0\}$ correspond to the faces of the fundamental domain, while vectors $\{e_i^n\}$, for any fixed $n$, correspond to the faces of $\wdt\Gamma$ on the $n$-th sheet of the universal cover. We denote by $\pi \colon \tLambda\rightarrow \Lambda$ the lattice projection defined by $\pi(e^n_i) = e_i$. By an abuse of notation, the corresponding projections $F(\wdt\Gamma) \to F(\Gamma)$ and $V(\wdt \Qc) \to V(\Qc)$ are denoted by the same symbol.

Let $k \in V(\Qc)$ be such that $(e_k^n, e_k^m)=0$ for all $n,m \in \Z$, so that the mutations at all vertices in the fiber $\pi^{-1}(k)$ commute. Let $\tilde\mu_k$ be the infinite sequence of mutations at all vertices in the fiber $\pi^{-1}(k)$. We denote the mutated cluster seed and quiver respectively by $\wdt\Theta'=\tilde\mu_k\big(\wdt\Theta\big)$ and $\wdt \Qc' = \tilde\mu_k(\wdt \Qc)$. By construction, there is an action of the group of deck transformations $\pi_1(S^1)=\Z$ on $\wdt \Qc'$: we have $\pi(\tilde\mu_k(e_i^n))=\pi(\tilde\mu_k(e_i^m))$ for all $n,m\in \Z$ and $i \in V(\Qc)$. One can therefore form a quiver $\wdt \Qc'/\Z$ as the quotient of $\wdt \Qc'$ by this action. We caution the reader that the quiver $\wdt \Qc'/\Z$ need not coincide with the quiver $\Qc' = \mu_k(\Qc)$. Let us note that the quivers $\tilde{\Qc}'/\Z$ were previously considered in~\cite{FG06a} in the context of \emph{special fatgraphs} in higher Teichm\"uller theory.

The following Lemma follows directly from the definition of quiver mutations and describes the relation between quivers $\wdt \Qc'/\Z$ and $\Qc'$ in the two cases of interest for this paper.

\begin{lemma}
\label{geom-mut}
\begin{enumerate}
\item Suppose $k \in V(\Qc)$ is a 4-valent vertex with alternating incoming and outgoing edges, such as for example vertices 2, 6, 12 in Figure~\ref{fig-D4}. Then $\wdt \Qc'/\Z = \Qc'$. 
\item Suppose $i,j,k,l \in V(\Qc)$ are as follows: for every $n \in \Z$ there are exactly 4 basis vectors $e_a^m \in \tLambda$ such that $(e_k^n, e_a^m) \ne 0$, namely
$$
(e_i^n,e_k^n)=(e_k^n,e_j^n)=(e_l^{n-1},e_k^n)=(e_k^n,e_l^n)=1.
$$
In particular, this implies that the vertex $k \in V(\Qc)$ is 2-valent with adjacent arrows $i \to k$ and $k \to j$. For example, in the notation of Figure~\ref{fig-D4} we can set $(i,j,k,l) = (19,21,20,28)$. Denote the bases of $\Lambda$ corresponding to the quivers $\Qc'$ and $\wdt \Qc'/\Z$ respectively by $\hc{e'_a}$ and $\hc{\hat e'_a}$, where $a \in V(\Qc)$. Then we have
$$
e'_l = e_l, \qquad \hat e'_l = e_l+e_k, \qquad\text{and}\qquad \hat e'_a = e'_a \qquad\text{if}\quad a \neq l.
$$
\end{enumerate}
\end{lemma}

Describing the relation between the mutation automorphisms of the quantum torus algebras $\mathcal{T}_{\tLambda}$ and $\mathcal{T}_{\Lambda}$ is a bit more subtle. 
The crucial point is that the projection map $\pi$ does not generally define a homomorphism of quantum torus algebras $\mathcal{T}_{\tLambda}\rightarrow \mathcal{T}_{\Lambda}$, for the simple reason that a pair of non-adjacent faces of $\wdt\Gamma$ may project under $\pi$ to a pair of adjacent faces of $\Gamma$. Moreover, the automorphism $\tilde{\mu}_k$ of $\mathcal{T}_{\tLambda}$ fails to descend to a well-defined automorphism of $ \mathcal{T}_{\Lambda}$ under the naive projection $\pi$. To get around this problem, we introduce a map of $\Z[q,q^{-1}]$-modules (although not a map of algebras)
\beq
\label{pr}
\pr \colon \Tc_{\tLambda} \to \Tc_\Lambda
\eeq
that can be regarded as the ``quantum correction'' of the covering map~$\pi$.

Let us define a symmetric bilinear form $\ha{\cdot, \cdot} \colon \tLambda \times \tLambda \to \C[q,q^{-1}]$ by the assignment
\beq
\label{formdef}
\ha{e^n_i, e^m_j} =
\begin{cases}
(e_i^n,e_j^m) - (e_i,e_j) & \text{if} \quad n>m, \\
(e_i^n,e_j^{n-1}-e_j^{n+1}) & \text{if} \quad n=m, \\
(e_i,e_j) - (e_i^n,e_j^m) & \text{if} \quad n<m.
\end{cases}
\eeq
The symmetry of the form $\ha{\cdot, \cdot}$ is made manifest by using the deck group action to write
$$
\ha{e_i^n, e_j^n} = (e_i^{n+1}, e_j^n) + (e_j^{n+1}, e_i^n).
$$
The definition of the form $\ha{\cdot,\cdot}$ is independent of our choice of fundamental domain for the deck group and is invariant under cluster mutations, see Lemmas~\ref{lem-index} and~\ref{lem-pr-mut}. Finally, we define
$$
\pr \colon \Tc_{\tLambda} \longra \Tc_\Lambda, \qquad X_\la \longmapsto q^{\frac{1}{2}\langle \la, \la \rangle} X_{\pi(\la)}
$$
for all $\la \in \tLambda$.

The following Proposition allows us to use directed networks on the universal cover in order to describe quiver mutations on a punctured disk. Its proof is given in the Appendix~\ref{appendix-1}.

\begin{prop}
\label{mut-descent}
Suppose that $k \in V(\Qc)$ is such that $(e_k^n, e_k^m)=0$ for all $n,m$. Then we have
$$
\pr\circ\tilde\mu_k  = \mu_k\circ\pr.
$$
\end{prop}

\subsection{Network for the standard triangulation}
\label{subsec-embed}

The embedding $\iota$ from Theorem~\ref{original-embed} can be described concisely using the language of directed networks on the punctured disk. The quiver $\Qstd$, depicted in Figure~\ref{fig-D4} for $n=4$, is obtained from the standard triangulation of the marked punctured disk $D_{2,1}$ shown in Figure~\ref{triangulations}. Let us consider the universal cover of the punctured disk, equipped with the corresponding lifted triangulation. It consists of an infinite sequence of triangles that share a common vertex, and such that any two adjacent triangles are glued by one side. The dual bi-colored graph can be made it into a directed network, by choosing an edge $e$ of the triangulation that belongs to exactly one triangle, and orienting outward all of the terminal edges of the bipartite graph that are adjacent to $e$. 

Two of the networks that arise this way are denoted $\Nc_n^{\mathrm{std}_\pm}$, and are shown in Figures~\ref{fig-std-net+} and~\ref{fig-std-net-}. The color-coding of the networks is consistent with the one in Figure~\ref{fig-path}, and the labelling of the faces is consistent with labelling of vertices in Figure~\ref{fig-D4}. In Figure~\ref{fig-std-net+} we label the sinks of the chosen boundary edge $e$ with $\hc{h_1, \dots, h_5}$ and the sources of an adjacent boundary edge by $\hc{t_1, \dots, t_5}$.  Figure~\ref{fig-std-net-} is drawn similarly, with the exception that the order of the sources and sinks is chosen differently differently. Note that the dual quiver to the underlying bi-colored graph of $\Nc_n^{\mathrm{std}_\pm}$ projects onto $\Qc^{\mathrm{std}}_n$ under the action of the deck group.

In addition to these, we have the network $\Nc_n^{\mathrm{std}}$ depicted in Figure~\ref{fig-std-net}. It can be thought of as a fundamental domain of $\Nc_n^{\mathrm{std}_\pm}$ with respect to the group of deck transformations. We use it to describe the Casimirs $\Omega_j$ of the quantum group $U_q(\sl_{n+1})$.

The following proposition gives a combinatorial reformulation of the embedding $\iota$ from Theorem~\ref{original-embed}.

\begin{prop}
\label{prop-measure}
Let
$$
M_n^{\mathrm{std}_\pm}(i,j;k,l) = \pr\hr{Z_{\Nc^{\mathrm{std}_\pm}_n}(t_i,h_j;t_k,h_l)}
$$
be the boundary measurements in the networks $\Nc_n^{\mathrm{std}_\pm}$. Then
\begin{align*}
&\iota(K_i) = {M_n^{\mathrm{std}_+}(i+1,i+1;i,i)}, &
&\iota(E_i) = {M_n^{\mathrm{std}_+}(i+1,i;i,i)}, \\
&\iota(K'_i) = {M_n^{\mathrm{std}_-}(i,i;i+1,i+1)}, &
&\iota(F_i) = {M_n^{\mathrm{std}_-}(i,i+1;i+1,i+1)}.
\end{align*}
\end{prop}

\begin{figure}[h]
\begin{tikzpicture}[every node/.style={inner sep=0, minimum size=0.2cm, draw, circle,thick}, x=0.35cm, y=0.35cm]

\node[draw=none, blue] at (0,2) {\scriptsize 25};
\node[draw=none, blue] at (0,4) {\scriptsize 26};
\node[draw=none, blue] at (0,6) {\scriptsize 27};
\node[draw=none, blue] at (0,8) {\scriptsize 28};

\node[draw=none, blue] at (4,0) {\scriptsize 16};
\node[draw=none, blue] at (8,0) {\scriptsize 9};
\node[draw=none, blue] at (12,0) {\scriptsize 4};
\node[draw=none, blue] at (16,0) {\scriptsize 1};
\node[draw=none, blue] at (6,2) {\scriptsize 17};
\node[draw=none, blue] at (10,2) {\scriptsize 10};
\node[draw=none, blue] at (14,2) {\scriptsize 5};
\node[draw=none, blue] at (8,4) {\scriptsize 18};
\node[draw=none, blue] at (12,4) {\scriptsize 11};
\node[draw=none, blue] at (10,6) {\scriptsize 19};

\node[draw=none, blue] at (20,2) {\scriptsize 2};
\node[draw=none, blue] at (20,4) {\scriptsize 6};
\node[draw=none, blue] at (20,6) {\scriptsize 12};
\node[draw=none, blue] at (20,8) {\scriptsize 20};

\node[draw=none, blue] at (24,0) {\scriptsize 3};
\node[draw=none, blue] at (28,0) {\scriptsize 8};
\node[draw=none, blue] at (32,0) {\scriptsize 15};
\node[draw=none, blue] at (36,0) {\scriptsize 24};
\node[draw=none, blue] at (26,2) {\scriptsize 7};
\node[draw=none, blue] at (30,2) {\scriptsize 14};
\node[draw=none, blue] at (34,2) {\scriptsize 23};
\node[draw=none, blue] at (28,4) {\scriptsize 13};
\node[draw=none, blue] at (32,4) {\scriptsize 22};
\node[draw=none, blue] at (30,6) {\scriptsize 21};

\node[draw=none, blue] at (40,2) {\scriptsize 25};
\node[draw=none, blue] at (40,4) {\scriptsize 26};
\node[draw=none, blue] at (40,6) {\scriptsize 27};
\node[draw=none, blue] at (40,8) {\scriptsize 28};

\node[minimum size=0.4cm] (t_5) at (2,-1) {\scriptsize $t_5$};
\node[minimum size=0.4cm] (t_4) at (6,-1) {\scriptsize $t_4$};
\node[minimum size=0.4cm] (t_3) at (10,-1) {\scriptsize $t_3$};
\node[minimum size=0.4cm] (t_2) at (14,-1) {\scriptsize $t_2$};
\node[minimum size=0.4cm] (t_1) at (18,-1) {\scriptsize $t_1$};

\node[minimum size=0.4cm] (h_1) at (22,-1) {\scriptsize $h_1$};
\node[minimum size=0.4cm] (h_2) at (26,-1) {\scriptsize $h_2$};
\node[minimum size=0.4cm] (h_3) at (30,-1) {\scriptsize $h_3$};
\node[minimum size=0.4cm] (h_4) at (34,-1) {\scriptsize $h_4$};
\node[minimum size=0.4cm] (h_5) at (38,-1) {\scriptsize $h_5$};

\foreach \i in {2,6,10,14,18,22,26,30,34,38}
	\node[fill=white] (\i_1) at (\i,1) {};
\foreach \i in {4,8,12,16,24,28,32,36}
	\node[fill=white] (\i_3) at (\i,3) {};
\foreach \i in {6,10,14,26,30,34}
	\node[fill=white] (\i_5) at (\i,5) {};
\foreach \i in {8,12,28,32}
	\node[fill=white] (\i_7) at (\i,7) {};
\foreach \i in {10,30}
	\node[fill=white] (\i_9) at (\i,9) {};

\foreach \i in {4,8,12,16,24,28,32,36}
	\node[fill=black] (\i_1) at (\i,1) {};
\foreach \i in {6,10,14,26,30,34}
	\node[fill=black] (\i_3) at (\i,3) {};
\foreach \i in {8,12,28,32}
	\node[fill=black] (\i_5) at (\i,5) {};
\foreach \i in {10,30}
	\node[fill=black] (\i_7) at (\i,7) {};

\draw[red, thick, ->] (t_5) to (2_1);
\draw[red, thick, ->] (2_1) to (4_1);
\draw[red, thick, ->] (4_1) to (4_3);
\draw[red, thick, ->] (4_3) to (6_3);
\draw[red, thick, ->] (6_3) to (6_5);
\draw[red, thick, ->] (6_5) to (8_5);
\draw[red, thick, ->] (8_5) to (8_7);
\draw[red, thick, ->] (8_7) to (10_7);
\draw[red, thick, ->] (10_7) to (10_9);
\draw[red, thick, ->] (10_9) to (30_9);
\draw[red, thick, ->] (30_9) to (30_7);
\draw[red, thick, ->] (30_7) to (32_7);
\draw[red, thick, ->] (32_7) to (32_5);
\draw[red, thick, ->] (32_5) to (34_5);
\draw[red, thick, ->] (34_5) to (34_3);
\draw[red, thick, ->] (34_3) to (36_3);
\draw[red, thick, ->] (36_3) to (36_1);
\draw[red, thick, ->] (36_1) to (38_1);
\draw[red, thick, ->] (38_1) to (h_5);

\draw[red, thick, ->] (t_4) to (6_1);
\draw[red, thick, ->] (6_1) to (8_1);
\draw[red, thick, ->] (8_1) to (8_3);
\draw[red, thick, ->] (8_3) to (10_3);
\draw[red, thick, ->] (10_3) to (10_5);
\draw[red, thick, ->] (10_5) to (12_5);
\draw[red, thick, ->] (12_5) to (12_7);
\draw[red, thick, ->] (12_7) to (28_7);
\draw[red, thick, ->] (28_7) to (28_5);
\draw[red, thick, ->] (28_5) to (30_5);
\draw[red, thick, ->] (30_5) to (30_3);
\draw[red, thick, ->] (30_3) to (32_3);
\draw[red, thick, ->] (32_3) to (32_1);
\draw[red, thick, ->] (32_1) to (34_1);
\draw[red, thick, ->] (34_1) to (h_4);

\draw[red, thick, ->] (t_3) to (10_1);
\draw[red, thick, ->] (10_1) to (12_1);
\draw[red, thick, ->] (12_1) to (12_3);
\draw[red, thick, ->] (12_3) to (14_3);
\draw[red, thick, ->] (14_3) to (14_5);
\draw[red, thick, ->] (14_5) to (26_5);
\draw[red, thick, ->] (26_5) to (26_3);
\draw[red, thick, ->] (26_3) to (28_3);
\draw[red, thick, ->] (28_3) to (28_1);
\draw[red, thick, ->] (28_1) to (30_1);
\draw[red, thick, ->] (30_1) to (h_3);

\draw[red, thick, ->] (t_2) to (14_1);
\draw[red, thick, ->] (14_1) to (16_1);
\draw[red, thick, ->] (16_1) to (16_3);
\draw[red, thick, ->] (16_3) to (24_3);
\draw[red, thick, ->] (24_3) to (24_1);
\draw[red, thick, ->] (24_1) to (26_1);
\draw[red, thick, ->] (26_1) to (h_2);

\draw[red, thick, ->] (t_1) to (18_1);
\draw[red, thick, ->] (18_1) to (22_1);
\draw[red, thick, ->] (22_1) to (h_1);

\draw[OliveGreen, thick, ->] (-1,1) to (2_1);
\draw[OliveGreen, thick, ->] (4_1) to (6_1);
\draw[OliveGreen, thick, ->] (8_1) to (10_1);
\draw[OliveGreen, thick, ->] (12_1) to (14_1);
\draw[OliveGreen, thick, ->] (16_1) to (18_1);

\draw[OliveGreen, thick, ->] (-1,3) to (4_3);
\draw[OliveGreen, thick, ->] (6_3) to (8_3);
\draw[OliveGreen, thick, ->] (10_3) to (12_3);
\draw[OliveGreen, thick, ->] (14_3) to (16_3);

\draw[OliveGreen, thick, ->] (-1,5) to (6_5);
\draw[OliveGreen, thick, ->] (8_5) to (10_5);
\draw[OliveGreen, thick, ->] (12_5) to (14_5);

\draw[OliveGreen, thick, ->] (-1,7) to (8_7);
\draw[OliveGreen, thick, ->] (10_7) to (12_7);

\draw[OliveGreen, thick, ->] (-1,9) to (10_9);

\draw[OliveGreen, thick, ->] (41,1) to (38_1);
\draw[OliveGreen, thick, ->] (36_1) to (34_1);
\draw[OliveGreen, thick, ->] (32_1) to (30_1);
\draw[OliveGreen, thick, ->] (28_1) to (26_1);
\draw[OliveGreen, thick, ->] (24_1) to (22_1);

\draw[OliveGreen, thick, ->] (41,3) to (36_3);
\draw[OliveGreen, thick, ->] (34_3) to (32_3);
\draw[OliveGreen, thick, ->] (30_3) to (28_3);
\draw[OliveGreen, thick, ->] (26_3) to (24_3);

\draw[OliveGreen, thick, ->] (41,5) to (34_5);
\draw[OliveGreen, thick, ->] (32_5) to (30_5);
\draw[OliveGreen, thick, ->] (28_5) to (26_5);

\draw[OliveGreen, thick, ->] (41,7) to (32_7);
\draw[OliveGreen, thick, ->] (30_7) to (28_7);

\draw[OliveGreen, thick, ->] (41,9) to (30_9);

\end{tikzpicture}
\caption{Network $\Nc_4^{\mathrm{std}_+}$.}
\label{fig-std-net+}
\end{figure}

\begin{figure}[h]
\begin{tikzpicture}[every node/.style={inner sep=0, minimum size=0.2cm, draw, circle,thick}, x=0.35cm, y=0.35cm]

\node[draw=none, blue] at (0,2) {\scriptsize 2};
\node[draw=none, blue] at (0,4) {\scriptsize 6};
\node[draw=none, blue] at (0,6) {\scriptsize 12};
\node[draw=none, blue] at (0,8) {\scriptsize 20};

\node[draw=none, blue] at (4,0) {\scriptsize 3};
\node[draw=none, blue] at (8,0) {\scriptsize 8};
\node[draw=none, blue] at (12,0) {\scriptsize 15};
\node[draw=none, blue] at (16,0) {\scriptsize 24};
\node[draw=none, blue] at (6,2) {\scriptsize 7};
\node[draw=none, blue] at (10,2) {\scriptsize 14};
\node[draw=none, blue] at (14,2) {\scriptsize 23};
\node[draw=none, blue] at (8,4) {\scriptsize 13};
\node[draw=none, blue] at (12,4) {\scriptsize 22};
\node[draw=none, blue] at (10,6) {\scriptsize 21};

\node[draw=none, blue] at (20,2) {\scriptsize 25};
\node[draw=none, blue] at (20,4) {\scriptsize 26};
\node[draw=none, blue] at (20,6) {\scriptsize 27};
\node[draw=none, blue] at (20,8) {\scriptsize 28};

\node[draw=none, blue] at (24,0) {\scriptsize 16};
\node[draw=none, blue] at (28,0) {\scriptsize 9};
\node[draw=none, blue] at (32,0) {\scriptsize 4};
\node[draw=none, blue] at (36,0) {\scriptsize 1};
\node[draw=none, blue] at (26,2) {\scriptsize 17};
\node[draw=none, blue] at (30,2) {\scriptsize 10};
\node[draw=none, blue] at (34,2) {\scriptsize 5};
\node[draw=none, blue] at (28,4) {\scriptsize 18};
\node[draw=none, blue] at (32,4) {\scriptsize 11};
\node[draw=none, blue] at (30,6) {\scriptsize 19};

\node[draw=none, blue] at (40,2) {\scriptsize 2};
\node[draw=none, blue] at (40,4) {\scriptsize 6};
\node[draw=none, blue] at (40,6) {\scriptsize 12};
\node[draw=none, blue] at (40,8) {\scriptsize 20};

\node[minimum size=0.4cm] (t_5) at (2,-1) {\scriptsize $t_1$};
\node[minimum size=0.4cm] (t_4) at (6,-1) {\scriptsize $t_2$};
\node[minimum size=0.4cm] (t_3) at (10,-1) {\scriptsize $t_3$};
\node[minimum size=0.4cm] (t_2) at (14,-1) {\scriptsize $t_4$};
\node[minimum size=0.4cm] (t_1) at (18,-1) {\scriptsize $t_5$};

\node[minimum size=0.4cm] (h_1) at (22,-1) {\scriptsize $h_5$};
\node[minimum size=0.4cm] (h_2) at (26,-1) {\scriptsize $h_4$};
\node[minimum size=0.4cm] (h_3) at (30,-1) {\scriptsize $h_3$};
\node[minimum size=0.4cm] (h_4) at (34,-1) {\scriptsize $h_2$};
\node[minimum size=0.4cm] (h_5) at (38,-1) {\scriptsize $h_1$};

\foreach \i in {2,6,10,14,18,22,26,30,34,38}
	\node[fill=white] (\i_1) at (\i,1) {};
\foreach \i in {4,8,12,16,24,28,32,36}
	\node[fill=white] (\i_3) at (\i,3) {};
\foreach \i in {6,10,14,26,30,34}
	\node[fill=white] (\i_5) at (\i,5) {};
\foreach \i in {8,12,28,32}
	\node[fill=white] (\i_7) at (\i,7) {};
\foreach \i in {10,30}
	\node[fill=white] (\i_9) at (\i,9) {};

\foreach \i in {4,8,12,16,24,28,32,36}
	\node[fill=black] (\i_1) at (\i,1) {};
\foreach \i in {6,10,14,26,30,34}
	\node[fill=black] (\i_3) at (\i,3) {};
\foreach \i in {8,12,28,32}
	\node[fill=black] (\i_5) at (\i,5) {};
\foreach \i in {10,30}
	\node[fill=black] (\i_7) at (\i,7) {};

\draw[OliveGreen, thick, ->] (t_5) to (2_1);
\draw[OliveGreen, thick, ->] (2_1) to (4_1);
\draw[OliveGreen, thick, ->] (4_1) to (4_3);
\draw[OliveGreen, thick, ->] (4_3) to (6_3);
\draw[OliveGreen, thick, ->] (6_3) to (6_5);
\draw[OliveGreen, thick, ->] (6_5) to (8_5);
\draw[OliveGreen, thick, ->] (8_5) to (8_7);
\draw[OliveGreen, thick, ->] (8_7) to (10_7);
\draw[OliveGreen, thick, ->] (10_7) to (10_9);
\draw[OliveGreen, thick, ->] (10_9) to (30_9);
\draw[OliveGreen, thick, ->] (30_9) to (30_7);
\draw[OliveGreen, thick, ->] (30_7) to (32_7);
\draw[OliveGreen, thick, ->] (32_7) to (32_5);
\draw[OliveGreen, thick, ->] (32_5) to (34_5);
\draw[OliveGreen, thick, ->] (34_5) to (34_3);
\draw[OliveGreen, thick, ->] (34_3) to (36_3);
\draw[OliveGreen, thick, ->] (36_3) to (36_1);
\draw[OliveGreen, thick, ->] (36_1) to (38_1);
\draw[OliveGreen, thick, ->] (38_1) to (h_5);

\draw[OliveGreen, thick, ->] (t_4) to (6_1);
\draw[OliveGreen, thick, ->] (6_1) to (8_1);
\draw[OliveGreen, thick, ->] (8_1) to (8_3);
\draw[OliveGreen, thick, ->] (8_3) to (10_3);
\draw[OliveGreen, thick, ->] (10_3) to (10_5);
\draw[OliveGreen, thick, ->] (10_5) to (12_5);
\draw[OliveGreen, thick, ->] (12_5) to (12_7);
\draw[OliveGreen, thick, ->] (12_7) to (28_7);
\draw[OliveGreen, thick, ->] (28_7) to (28_5);
\draw[OliveGreen, thick, ->] (28_5) to (30_5);
\draw[OliveGreen, thick, ->] (30_5) to (30_3);
\draw[OliveGreen, thick, ->] (30_3) to (32_3);
\draw[OliveGreen, thick, ->] (32_3) to (32_1);
\draw[OliveGreen, thick, ->] (32_1) to (34_1);
\draw[OliveGreen, thick, ->] (34_1) to (h_4);

\draw[OliveGreen, thick, ->] (t_3) to (10_1);
\draw[OliveGreen, thick, ->] (10_1) to (12_1);
\draw[OliveGreen, thick, ->] (12_1) to (12_3);
\draw[OliveGreen, thick, ->] (12_3) to (14_3);
\draw[OliveGreen, thick, ->] (14_3) to (14_5);
\draw[OliveGreen, thick, ->] (14_5) to (26_5);
\draw[OliveGreen, thick, ->] (26_5) to (26_3);
\draw[OliveGreen, thick, ->] (26_3) to (28_3);
\draw[OliveGreen, thick, ->] (28_3) to (28_1);
\draw[OliveGreen, thick, ->] (28_1) to (30_1);
\draw[OliveGreen, thick, ->] (30_1) to (h_3);

\draw[OliveGreen, thick, ->] (t_2) to (14_1);
\draw[OliveGreen, thick, ->] (14_1) to (16_1);
\draw[OliveGreen, thick, ->] (16_1) to (16_3);
\draw[OliveGreen, thick, ->] (16_3) to (24_3);
\draw[OliveGreen, thick, ->] (24_3) to (24_1);
\draw[OliveGreen, thick, ->] (24_1) to (26_1);
\draw[OliveGreen, thick, ->] (26_1) to (h_2);

\draw[OliveGreen, thick, ->] (t_1) to (18_1);
\draw[OliveGreen, thick, ->] (18_1) to (22_1);
\draw[OliveGreen, thick, ->] (22_1) to (h_1);

\draw[BurntOrange, thick, ->] (-1,1) to (2_1);
\draw[BurntOrange, thick, ->] (4_1) to (6_1);
\draw[BurntOrange, thick, ->] (8_1) to (10_1);
\draw[BurntOrange, thick, ->] (12_1) to (14_1);
\draw[BurntOrange, thick, ->] (16_1) to (18_1);

\draw[BurntOrange, thick, ->] (-1,3) to (4_3);
\draw[BurntOrange, thick, ->] (6_3) to (8_3);
\draw[BurntOrange, thick, ->] (10_3) to (12_3);
\draw[BurntOrange, thick, ->] (14_3) to (16_3);

\draw[BurntOrange, thick, ->] (-1,5) to (6_5);
\draw[BurntOrange, thick, ->] (8_5) to (10_5);
\draw[BurntOrange, thick, ->] (12_5) to (14_5);

\draw[BurntOrange, thick, ->] (-1,7) to (8_7);
\draw[BurntOrange, thick, ->] (10_7) to (12_7);

\draw[BurntOrange, thick, ->] (-1,9) to (10_9);

\draw[BurntOrange, thick, ->] (41,1) to (38_1);
\draw[BurntOrange, thick, ->] (36_1) to (34_1);
\draw[BurntOrange, thick, ->] (32_1) to (30_1);
\draw[BurntOrange, thick, ->] (28_1) to (26_1);
\draw[BurntOrange, thick, ->] (24_1) to (22_1);

\draw[BurntOrange, thick, ->] (41,3) to (36_3);
\draw[BurntOrange, thick, ->] (34_3) to (32_3);
\draw[BurntOrange, thick, ->] (30_3) to (28_3);
\draw[BurntOrange, thick, ->] (26_3) to (24_3);

\draw[BurntOrange, thick, ->] (41,5) to (34_5);
\draw[BurntOrange, thick, ->] (32_5) to (30_5);
\draw[BurntOrange, thick, ->] (28_5) to (26_5);

\draw[BurntOrange, thick, ->] (41,7) to (32_7);
\draw[BurntOrange, thick, ->] (30_7) to (28_7);

\draw[BurntOrange, thick, ->] (41,9) to (30_9);

\end{tikzpicture}
\caption{Network $\Nc_4^{\mathrm{std}_-}$.}
\label{fig-std-net-}
\end{figure}

\begin{remark}
\label{single-casimir-rmk}
The Casimir elements $\Omega_1,\ldots, \Omega_n$ defined in~\eqref{single-center-gens} that determine the central character of the positive representation $\mathcal{P}_\lambda$ also admit a simple interpretation in terms of the network $\mathcal{N}^{\mathrm{std}}_n$ shown in Figure~\ref{fig-std-net}. Namely, we have
\beq
\label{cas-single-std}
\Omega_j = \pr\hr{Z_{\Nc^{\mathrm{std}}_n}(t_{j+1},h_{j+1};t_j,h_j)}.
\eeq
Thus, the Casimirs $\Omega_j$ can be regarded as computing the monodromies around the rows of the network $\Nc^{\mathrm{std}}_n$. 
\end{remark}

\begin{figure}[h]
\begin{tikzpicture}[every node/.style={inner sep=0, minimum size=0.2cm, draw, circle,thick}, x=0.35cm, y=0.35cm]

\node[draw=none, blue] at (4,0) {\scriptsize 16};
\node[draw=none, blue] at (8,0) {\scriptsize 9};
\node[draw=none, blue] at (12,0) {\scriptsize 4};
\node[draw=none, blue] at (16,0) {\scriptsize 1};
\node[draw=none, blue] at (6,2) {\scriptsize 17};
\node[draw=none, blue] at (10,2) {\scriptsize 10};
\node[draw=none, blue] at (14,2) {\scriptsize 5};
\node[draw=none, blue] at (8,4) {\scriptsize 18};
\node[draw=none, blue] at (12,4) {\scriptsize 11};
\node[draw=none, blue] at (10,6) {\scriptsize 19};

\node[draw=none, blue] at (18,2) {\scriptsize 2};
\node[draw=none, blue] at (16,4) {\scriptsize 6};
\node[draw=none, blue] at (14,6) {\scriptsize 12};
\node[draw=none, blue] at (12,8) {\scriptsize 20};

\node[draw=none, blue] at (24,0) {\scriptsize 3};
\node[draw=none, blue] at (28,0) {\scriptsize 8};
\node[draw=none, blue] at (32,0) {\scriptsize 15};
\node[draw=none, blue] at (36,0) {\scriptsize 24};
\node[draw=none, blue] at (26,2) {\scriptsize 7};
\node[draw=none, blue] at (30,2) {\scriptsize 14};
\node[draw=none, blue] at (34,2) {\scriptsize 23};
\node[draw=none, blue] at (28,4) {\scriptsize 13};
\node[draw=none, blue] at (32,4) {\scriptsize 22};
\node[draw=none, blue] at (30,6) {\scriptsize 21};

\node[draw=none, blue] at (38,2) {\scriptsize 25};
\node[draw=none, blue] at (36,4) {\scriptsize 26};
\node[draw=none, blue] at (34,6) {\scriptsize 27};
\node[draw=none, blue] at (32,8) {\scriptsize 28};

\node[minimum size=0.4cm] (t_1) at (41,1) {\scriptsize $t_1$};
\node[minimum size=0.4cm] (t_2) at (41,3) {\scriptsize $t_2$};
\node[minimum size=0.4cm] (t_3) at (41,5) {\scriptsize $t_3$};
\node[minimum size=0.4cm] (t_4) at (41,7) {\scriptsize $t_4$};
\node[minimum size=0.4cm] (t_5) at (41,9) {\scriptsize $t_5$};

\node[minimum size=0.4cm] (h_1) at (-1,1) {\scriptsize $h_1$};
\node[minimum size=0.4cm] (h_2) at (-1,3) {\scriptsize $h_2$};
\node[minimum size=0.4cm] (h_3) at (-1,5) {\scriptsize $h_3$};
\node[minimum size=0.4cm] (h_4) at (-1,7) {\scriptsize $h_4$};
\node[minimum size=0.4cm] (h_5) at (-1,9) {\scriptsize $h_5$};

\foreach \i in {2,6,10,14,18,22,26,30,34,38}
	\node[fill=white] (\i_1) at (\i,1) {};
\foreach \i in {4,8,12,16,24,28,32,36}
	\node[fill=white] (\i_3) at (\i,3) {};
\foreach \i in {6,10,14,26,30,34}
	\node[fill=white] (\i_5) at (\i,5) {};
\foreach \i in {8,12,28,32}
	\node[fill=white] (\i_7) at (\i,7) {};
\foreach \i in {10,30}
	\node[fill=white] (\i_9) at (\i,9) {};

\foreach \i in {4,8,12,16,24,28,32,36}
	\node[fill=black] (\i_1) at (\i,1) {};
\foreach \i in {6,10,14,26,30,34}
	\node[fill=black] (\i_3) at (\i,3) {};
\foreach \i in {8,12,28,32}
	\node[fill=black] (\i_5) at (\i,5) {};
\foreach \i in {10,30}
	\node[fill=black] (\i_7) at (\i,7) {};

\draw[BurntOrange, thick, <-] (h_5) to (10_9);
\draw[BurntOrange, thick, <-] (10_9) to (30_9);
\draw[BurntOrange, thick, <-] (30_9) to (t_5);

\draw[BurntOrange, thick, <-] (h_4) to (8_7);
\draw[BurntOrange, thick, <-] (8_7) to (10_7);
\draw[BurntOrange, thick, <-] (10_7) to (12_7);
\draw[BurntOrange, thick, <-] (12_7) to (28_7);
\draw[BurntOrange, thick, <-] (28_7) to (30_7);
\draw[BurntOrange, thick, <-] (30_7) to (32_7);
\draw[BurntOrange, thick, <-] (32_7) to (t_4);

\draw[BurntOrange, thick, <-] (h_3) to (6_5);
\draw[BurntOrange, thick, <-] (6_5) to (8_5);
\draw[BurntOrange, thick, <-] (8_5) to (10_5);
\draw[BurntOrange, thick, <-] (10_5) to (12_5);
\draw[BurntOrange, thick, <-] (12_5) to (14_5);
\draw[BurntOrange, thick, <-] (14_5) to (26_5);
\draw[BurntOrange, thick, <-] (26_5) to (28_5);
\draw[BurntOrange, thick, <-] (28_5) to (30_5);
\draw[BurntOrange, thick, <-] (30_5) to (32_5);
\draw[BurntOrange, thick, <-] (32_5) to (34_5);
\draw[BurntOrange, thick, <-] (34_5) to (t_3);

\draw[BurntOrange, thick, <-] (h_2) to (4_3);
\draw[BurntOrange, thick, <-] (4_3) to (6_3);
\draw[BurntOrange, thick, <-] (6_3) to (8_3);
\draw[BurntOrange, thick, <-] (8_3) to (10_3);
\draw[BurntOrange, thick, <-] (10_3) to (12_3);
\draw[BurntOrange, thick, <-] (12_3) to (14_3);
\draw[BurntOrange, thick, <-] (14_3) to (16_3);
\draw[BurntOrange, thick, <-] (16_3) to (24_3);
\draw[BurntOrange, thick, <-] (24_3) to (26_3);
\draw[BurntOrange, thick, <-] (26_3) to (28_3);
\draw[BurntOrange, thick, <-] (28_3) to (30_3);
\draw[BurntOrange, thick, <-] (30_3) to (32_3);
\draw[BurntOrange, thick, <-] (32_3) to (34_3);
\draw[BurntOrange, thick, <-] (34_3) to (36_3);
\draw[BurntOrange, thick, <-] (36_3) to (t_2);

\draw[BurntOrange, thick, <-] (h_1) to (2_1);
\draw[BurntOrange, thick, <-] (2_1) to (4_1);
\draw[BurntOrange, thick, <-] (4_1) to (6_1);
\draw[BurntOrange, thick, <-] (6_1) to (8_1);
\draw[BurntOrange, thick, <-] (8_1) to (10_1);
\draw[BurntOrange, thick, <-] (10_1) to (12_1);
\draw[BurntOrange, thick, <-] (12_1) to (14_1);
\draw[BurntOrange, thick, <-] (14_1) to (16_1);
\draw[BurntOrange, thick, <-] (16_1) to (18_1);
\draw[BurntOrange, thick, <-] (18_1) to (22_1);
\draw[BurntOrange, thick, <-] (22_1) to (24_1);
\draw[BurntOrange, thick, <-] (24_1) to (26_1);
\draw[BurntOrange, thick, <-] (26_1) to (28_1);
\draw[BurntOrange, thick, <-] (28_1) to (30_1);
\draw[BurntOrange, thick, <-] (30_1) to (32_1);
\draw[BurntOrange, thick, <-] (32_1) to (34_1);
\draw[BurntOrange, thick, <-] (34_1) to (36_1);
\draw[BurntOrange, thick, <-] (36_1) to (38_1);
\draw[BurntOrange, thick, <-] (38_1) to (t_1);

\draw[red, thick, ->] (2,-1) to (2_1);
\draw[red, thick, ->] (4_1) to (4_3);
\draw[red, thick, ->] (6_3) to (6_5);
\draw[red, thick, ->] (8_5) to (8_7);
\draw[red, thick, ->] (10_7) to (10_9);

\draw[red, thick, ->] (6,-1) to (6_1);
\draw[red, thick, ->] (8_1) to (8_3);
\draw[red, thick, ->] (10_3) to (10_5);
\draw[red, thick, ->] (12_5) to (12_7);

\draw[red, thick, ->] (10,-1) to (10_1);
\draw[red, thick, ->] (12_1) to (12_3);
\draw[red, thick, ->] (14_3) to (14_5);

\draw[red, thick, ->] (14,-1) to (14_1);
\draw[red, thick, ->] (16_1) to (16_3);

\draw[red, thick, ->] (18,-1) to (18_1);

\draw[red, thick, ->] (22,-1) to (22_1);
\draw[red, thick, ->] (24_1) to (24_3);
\draw[red, thick, ->] (26_3) to (26_5);
\draw[red, thick, ->] (28_5) to (28_7);
\draw[red, thick, ->] (30_7) to (30_9);

\draw[red, thick, ->] (26,-1) to (26_1);
\draw[red, thick, ->] (28_1) to (28_3);
\draw[red, thick, ->] (30_3) to (30_5);
\draw[red, thick, ->] (32_5) to (32_7);

\draw[red, thick, ->] (30,-1) to (30_1);
\draw[red, thick, ->] (32_1) to (32_3);
\draw[red, thick, ->] (34_3) to (34_5);

\draw[red, thick, ->] (34,-1) to (34_1);
\draw[red, thick, ->] (36_1) to (36_3);

\draw[red, thick, ->] (38,-1) to (38_1);

\end{tikzpicture}
\caption{Network $\Nc_4^{\mathrm{std}}$.}
\label{fig-std-net}
\end{figure}

\subsection{Embedding $\iota$ and flips of triangulation}
\label{single-copy-muts}

We now explain how to rewrite the embedding $\iota$ from Theorem~\ref{original-embed} with respect to different quantum cluster seeds. Let us start by explaining how to produce a seed $\Theta_n^{\mathrm{sf}}$ corresponding to the self-folded triangulation of a punctured $D_{2,1}$ with 2 marked points shown on the right of Figure~\ref{triangulations}.

This triangulation is obtained from the standard one by performing a single flip, which is realized by a sequence of $\binom{n+2}{3}$ cluster mutations as explained in Section~\ref{locsys-intro}. Since some of these mutations occur at 2-valent vertices, some care must be taken in order to preserve the dictionary with the network formulation. Namely, for each of the $\binom{n+2}{3}$ vertices of a quiver~$\Qc$ at which we mutate, we perform mutations at each of the infinitely many 4-valent vertices in its covering fiber in the quiver $\wdt \Qc$. The result of this procedure is a seed $\wdt\Theta^{\mathrm{sf}}_n$ whose quiver is dual to the bi-colored graph of the networks $\Nc_n^{\mathrm{sf}_\pm}$, shown in Figures~\ref{net-sf+} and~\ref{net-sf-} for $n=4$. Finally, we arrive at the seed $\Theta^{\mathrm{sf}}_n=\wdt\Theta^{\mathrm{sf}}_n/\Z$ by taking the quotient of $\wdt\Theta^{\mathrm{sf}}_n$ by the action of the deck group. The quiver $\Qc_4^{\mathrm{sf}}$ corresponding to $\Theta^{\mathrm{sf}}_4$ is shown in Figure~\ref{D4-sf}. 

\begin{example} Let us spell out this procedure in our example when $n=4$. In this case, one mutates at each vertex in the fiber over the following 20 vertices of the quiver $\Qc_4^{\mathrm{std}}$:
$$
2, \, 6, \, 12, \, 20; \quad
5, \, 11, \, 19, \, 7, \, 13, \, 21; \quad
10, \, 18, \, 6, \, 12, \, 14, \, 22; \quad
17, \, 11, \, 13, \, 23.
$$
As mentioned in Section~\ref{locsys-intro}, these vertices are divided into $n=4$ blocks in such a way that the order of mutations within each block does not matter. 
\end{example}

\begin{figure}[h]
\begin{tikzpicture}[every node/.style={inner sep=0, minimum size=0.4cm, thick, draw, circle}, x=0.75cm, y=0.45cm]

\node (1) [rectangle] at (-1,0) {\scriptsize 1};
\node (2) at (0,2) {\scriptsize 2};
\node (3) [rectangle] at (1,0) {\scriptsize 3};
\node (4) [rectangle] at (-2,2) {\scriptsize 4};
\node (5) at (-1,4) {\scriptsize 5};
\node (6) at (0,6) {\scriptsize 6};
\node (7) at (1,4) {\scriptsize 7};
\node (8) [rectangle] at (2,2) {\scriptsize 8};
\node (9) [rectangle] at (-3,4) {\scriptsize 9};
\node (10) at (-2,6) {\scriptsize 10};
\node (11) at (-1,8) {\scriptsize 11};
\node (12) at (0,10) {\scriptsize 12};
\node (13) at (1,8) {\scriptsize 13};
\node (14) at (2,6) {\scriptsize 14};
\node (15) [rectangle] at (3,4) {\scriptsize 15};
\node (16) [rectangle] at (-4,6) {\scriptsize 16};
\node (17) at (-3,8) {\scriptsize 17};
\node (18) at (-2,10) {\scriptsize 18};
\node (19) at (-1,12) {\scriptsize 19};
\node (20) at (0,14) {\scriptsize 20};
\node (21) at (1,12) {\scriptsize 21};
\node (22) at (2,10) {\scriptsize 22};
\node (23) at (3,8) {\scriptsize 23};
\node (24) [rectangle] at (4,6) {\scriptsize 24};
\node (25) at (-4,10) {\scriptsize 25};
\node (26) at (-3,12) {\scriptsize 26};
\node (27) at (-2,14) {\scriptsize 27};
\node (28) at (-1,16) {\scriptsize 28};

\draw[BrickRed, dashed] (-4,8) to (4,8);

\draw [->,thick] (1) -- (3);

\draw [->,thick] (4) -- (2);
\draw [->,thick] (2) -- (8);

\draw [->,thick] (9) -- (5);
\draw [->,thick] (5) -- (7);
\draw [->,thick] (7) -- (15);

\draw [->,thick] (16) -- (10);
\draw [->,thick] (10) -- (6);
\draw [->,thick] (6) -- (14);
\draw [->,thick] (14) -- (24);

\draw [->,thick] (3) -- (2);
\draw [->,thick] (2) -- (5);
\draw [->,thick] (5) -- (10);
\draw [->,thick] (10) -- (17);
\draw [->,thick] (17) -- (18);
\draw [->,thick] (18) -- (19);

\draw [->,thick] (19) -- (20);
\draw [->,thick] (20) -- (21);


\draw [->,thick] (21) -- (22);
\draw [->,thick] (22) -- (23);
\draw [->,thick] (23) -- (14);
\draw [->,thick] (14) -- (7);
\draw [->,thick] (7) -- (2);
\draw [->,thick] (2) -- (1);

\draw [->,thick] (8) -- (7);
\draw [->,thick] (7) -- (6);
\draw [->,thick] (6) -- (11);
\draw [->,thick] (11) -- (12);
\draw [->,thick] (12) -- (21);
\draw [->,thick] (21) to [out = 145, in = -25] (27);
\draw [->,thick] (27) to (19);
\draw [->,thick] (19) -- (12);
\draw [->,thick] (12) -- (13);
\draw [->,thick] (13) -- (6);
\draw [->,thick] (6) -- (5);
\draw [->,thick] (5) -- (4);

\draw [->,thick] (15) -- (14);
\draw [->,thick] (14) -- (13);
\draw [->,thick] (13) -- (22);
\draw [->,thick] (22) to [out = 150, in =-20] (26);
\draw [->,thick] (26) to (18);
\draw [->,thick] (18) -- (11);
\draw [->,thick] (11) -- (10);
\draw [->,thick] (10) -- (9);

\draw [->,thick] (24) -- (23);
\draw [->,thick] (23) to [out = 155, in = -15] (25);
\draw [->,thick] (25) to (17);
\draw [->,thick] (17) -- (16);

\draw [->,thick] (21) -- (19);
\draw [->,thick] (22) -- (12);
\draw [->,thick] (12) -- (18);
\draw [->,thick] (25) to [bend left = 15] (22);
\draw [->,thick] (26) to [bend left = 20] (21);
\draw [->,thick] (19) to (26);
\draw [->,thick] (18) to (25);

\draw [->,thick,dashed] (24) -- (15);
\draw [->,thick,dashed] (15) -- (8);
\draw [->,thick,dashed] (8) -- (3);
\draw [->,thick,dashed] (1) -- (4);
\draw [->,thick,dashed] (4) -- (9);
\draw [->,thick,dashed] (9) -- (16);

\end{tikzpicture}
\caption{Quiver $\Qc_4^{\mathrm{sf}}$.}
\label{D4-sf}
\end{figure}

Now, as explained in Section~\ref{h-dilog-section}, the sequence of mutations passing from the seed $\Theta_n^{\mathrm{std}}$ to $\Theta_n^{\mathrm{sf}}$ yields a sequence of unitary operators on the positive representation $\mathcal{P}_\lambda$, each given by a non-compact quantum dilogarithm of a Heisenberg algebra generator. Let us denote by $\Phi^{\mathrm{sf}}$ the unitary operator on $\mathcal{P}_\lambda$ obtained as their composite. Then we have

\begin{cor}
\label{cor-folded}
Let
$$
M_n^{\mathrm{sf}_\pm}(i,j;k,l) = \mathrm{pr}\hr{Z_{\mathcal{N}^{\mathrm{sf}_\pm}_n}(t_i,h_j;t_k,h_l)}
$$
be the boundary measurements in the networks $\Nc_n^{\mathrm{sf}_\pm}$. Then as operators on $\mathcal{P}_\lambda$, we have
\begin{align*}
&\Ad_{\Phi^{\mathrm{sf}}}\cdot\iota(K_i) = {M_n^{\mathrm{sf}_+}(i+1,i+1;i,i)}, &
&\Ad_{\Phi^{\mathrm{sf}}}\cdot\iota(E_i) = {M_n^{\mathrm{sf}_+}(i+1,i;i,i)}, \\
&\Ad_{\Phi^{\mathrm{sf}}}\cdot\iota(K'_i) = {M_n^{\mathrm{sf}_-}(i,i;i+1,i+1)}, &
&\Ad_{\Phi^{\mathrm{sf}}}\cdot\iota(F_i) = {M_n^{\mathrm{sf}_-}(i,i+1;i+1,i+1)}.
\end{align*}
\end{cor}

\begin{figure}[h]
\begin{tikzpicture}[every node/.style={inner sep=0, minimum size=0.2cm, draw, circle, thick}, x=0.35cm, y=0.35cm]

\foreach \i in {1,5,9,13,17,23,27,31,35,39}
{
	\node[fill=white] (\i_1) at (\i,1) {};
	\node[fill=white] (\i_-1) at (\i,-1) {};
}
\foreach \i in {3,7,11,15,25,29,33,37}
{
	\node[fill=white] (\i_3) at (\i,3) {};
	\node[fill=white] (\i_-3) at (\i,-3) {};
	\node[fill=black] (\i_1) at (\i,1) {};
	\node[fill=black] (\i_-1) at (\i,-1) {};
}
\foreach \i in {5,9,13,27,31,35}
{
	\node[fill=white] (\i_5) at (\i,5) {};
	\node[fill=white] (\i_-5) at (\i,-5) {};
	\node[fill=black] (\i_3) at (\i,3) {};
	\node[fill=black] (\i_-3) at (\i,-3) {};
}

\foreach \i in {7,11,29,33}
{
	\node[fill=white] (\i_7) at (\i,7) {};
	\node[fill=white] (\i_-7) at (\i,-7) {};
	\node[fill=black] (\i_5) at (\i,5) {};
	\node[fill=black] (\i_-5) at (\i,-5) {};
}

\foreach \i in {9,31}
{
	\node[fill=white] (\i_9) at (\i,9) {};
	\node[fill=white] (\i_-9) at (\i,-9) {};
	\node[fill=black] (\i_7) at (\i,7) {};
	\node[fill=black] (\i_-7) at (\i,-7) {};
}

\node[minimum size=0.4cm] (t_5) at (10.5,-10.5) {\scriptsize $t_5$};
\node[minimum size=0.4cm] (t_4) at (12.5,-8.5) {\scriptsize $t_4$};
\node[minimum size=0.4cm] (t_3) at (14.5,-6.5) {\scriptsize $t_3$};
\node[minimum size=0.4cm] (t_2) at (16.5,-4.5) {\scriptsize $t_2$};
\node[minimum size=0.4cm] (t_1) at (18.5,-2.5) {\scriptsize $t_1$};

\node[minimum size=0.4cm] (h_1) at (-0.5,-2.5) {\scriptsize $h_1$};
\node[minimum size=0.4cm] (h_2) at (1.5,-4.5) {\scriptsize $h_2$};
\node[minimum size=0.4cm] (h_3) at (3.5,-6.5) {\scriptsize $h_3$};
\node[minimum size=0.4cm] (h_4) at (5.5,-8.5) {\scriptsize $h_4$};
\node[minimum size=0.4cm] (h_5) at (7.5,-10.5) {\scriptsize $h_5$};

\node[draw=none, blue] at (0,2) {\scriptsize 25};
\node[draw=none, blue] at (0,4) {\scriptsize 26};
\node[draw=none, blue] at (0,6) {\scriptsize 27};
\node[draw=none, blue] at (0,8) {\scriptsize 28};

\node[draw=none, blue] at (8,-8) {\scriptsize 1};
\node[draw=none, blue] at (10,-8) {\scriptsize 3};
\node[draw=none, blue] at (6,-6) {\scriptsize 4};
\node[draw=none, blue] at (9,-6) {\scriptsize 2};
\node[draw=none, blue] at (12,-6) {\scriptsize 8};
\node[draw=none, blue] at (4,-4) {\scriptsize 9};
\node[draw=none, blue] at (7,-4) {\scriptsize 5};
\node[draw=none, blue] at (11,-4) {\scriptsize 7};
\node[draw=none, blue] at (14,-4) {\scriptsize 15};
\node[draw=none, blue] at (2,-2) {\scriptsize 16};
\node[draw=none, blue] at (5,-2) {\scriptsize 10};
\node[draw=none, blue] at (9,-2) {\scriptsize 6};
\node[draw=none, blue] at (13,-2) {\scriptsize 14};
\node[draw=none, blue] at (16,-2) {\scriptsize 24};
\node[draw=none, blue] at (3,0) {\scriptsize 17};
\node[draw=none, blue] at (7,0) {\scriptsize 11};
\node[draw=none, blue] at (11,0) {\scriptsize 13};
\node[draw=none, blue] at (15,0) {\scriptsize 23};
\node[draw=none, blue] at (5,2) {\scriptsize 18};
\node[draw=none, blue] at (9,2) {\scriptsize 12};
\node[draw=none, blue] at (13,2) {\scriptsize 22};
\node[draw=none, blue] at (7,4) {\scriptsize 19};
\node[draw=none, blue] at (11,4) {\scriptsize 21};
\node[draw=none, blue] at (9,6) {\scriptsize 20};

\node[draw=none, blue] at (20,2) {\scriptsize 25};
\node[draw=none, blue] at (20,4) {\scriptsize 26};
\node[draw=none, blue] at (20,6) {\scriptsize 27};
\node[draw=none, blue] at (20,8) {\scriptsize 28};

\node[draw=none, blue] at (30,-8) {\scriptsize 1};
\node[draw=none, blue] at (32,-8) {\scriptsize 3};
\node[draw=none, blue] at (28,-6) {\scriptsize 4};
\node[draw=none, blue] at (31,-6) {\scriptsize 2};
\node[draw=none, blue] at (34,-6) {\scriptsize 8};
\node[draw=none, blue] at (26,-4) {\scriptsize 9};
\node[draw=none, blue] at (29,-4) {\scriptsize 5};
\node[draw=none, blue] at (33,-4) {\scriptsize 7};
\node[draw=none, blue] at (36,-4) {\scriptsize 15};
\node[draw=none, blue] at (24,-2) {\scriptsize 16};
\node[draw=none, blue] at (27,-2) {\scriptsize 10};
\node[draw=none, blue] at (31,-2) {\scriptsize 6};
\node[draw=none, blue] at (35,-2) {\scriptsize 14};
\node[draw=none, blue] at (38,-2) {\scriptsize 24};
\node[draw=none, blue] at (25,0) {\scriptsize 17};
\node[draw=none, blue] at (29,0) {\scriptsize 11};
\node[draw=none, blue] at (33,0) {\scriptsize 13};
\node[draw=none, blue] at (37,0) {\scriptsize 23};
\node[draw=none, blue] at (27,2) {\scriptsize 18};
\node[draw=none, blue] at (31,2) {\scriptsize 12};
\node[draw=none, blue] at (35,2) {\scriptsize 22};
\node[draw=none, blue] at (29,4) {\scriptsize 19};
\node[draw=none, blue] at (33,4) {\scriptsize 21};
\node[draw=none, blue] at (31,6) {\scriptsize 20};

\node[draw=none, blue] at (40,2) {\scriptsize 25};
\node[draw=none, blue] at (40,4) {\scriptsize 26};
\node[draw=none, blue] at (40,6) {\scriptsize 27};
\node[draw=none, blue] at (40,8) {\scriptsize 28};

\draw[BrickRed, dashed] (-0.5,0) to (40.5,0);

\draw[BurntOrange, thick, <-] (t_1) to (17_-1);
\draw[red, thick, <-] (17_-1) to (17_1);
\draw[OliveGreen, thick, <-] (17_1) to (23_1);
\draw[BurntOrange, thick, <-] (23_1) to (23_-1);
\draw[red, thick, <-] (23_-1) to (21.5,-2.5);

\draw[BurntOrange, thick, <-] (t_2) to (15_-3);
\draw[red, thick, <-] (15_-3) to (15_-1);
\draw[BurntOrange, thick, <-] (15_-1) to (13_-1);
\draw[red, thick, <-] (13_-1) to (13_1);
\draw[OliveGreen, thick, <-] (13_1) to (15_1);
\draw[red, thick, <-] (15_1) to (15_3);
\draw[OliveGreen, thick, <-] (15_3) to (25_3);
\draw[BurntOrange, thick, <-] (25_3) to (25_1);
\draw[OliveGreen, thick, <-] (25_1) to (27_1);
\draw[BurntOrange, thick, <-] (27_1) to (27_-1);
\draw[red, thick, <-] (27_-1) to (25_-1);
\draw[BurntOrange, thick, <-] (25_-1) to (25_-3);
\draw[red, thick, <-] (25_-3) to (23.5,-4.5);

\draw[BurntOrange, thick, <-] (t_3) to (13_-5);
\draw[red, thick, <-] (13_-5) to (13_-3);
\draw[BurntOrange, thick, <-] (13_-3) to (11_-3);
\draw[red, thick, <-] (11_-3) to (11_-1);
\draw[BurntOrange, thick, <-] (11_-1) to (9_-1);
\draw[red, thick, <-] (9_-1) to (9_1);
\draw[OliveGreen, thick, <-] (9_1) to (11_1);
\draw[red, thick, <-] (11_1) to (11_3);
\draw[OliveGreen, thick, <-] (11_3) to (13_3);
\draw[red, thick, <-] (13_3) to (13_5);
\draw[OliveGreen, thick, <-] (13_5) to (27_5);
\draw[BurntOrange, thick, <-] (27_5) to (27_3);
\draw[OliveGreen, thick, <-] (27_3) to (29_3);
\draw[BurntOrange, thick, <-] (29_3) to (29_1);
\draw[OliveGreen, thick, <-] (29_1) to (31_1);
\draw[BurntOrange, thick, <-] (31_1) to (31_-1);
\draw[red, thick, <-] (31_-1) to (29_-1);
\draw[BurntOrange, thick, <-] (29_-1) to (29_-3);
\draw[red, thick, <-] (29_-3) to (27_-3);
\draw[BurntOrange, thick, <-] (27_-3) to (27_-5);
\draw[red, thick, <-] (27_-5) to (25.5,-6.5);

\draw[BurntOrange, thick, <-] (t_4) to (11_-7);
\draw[red, thick, <-] (11_-7) to (11_-5);
\draw[BurntOrange, thick, <-] (11_-5) to (9_-5);
\draw[red, thick, <-] (9_-5) to (9_-3);
\draw[BurntOrange, thick, <-] (9_-3) to (7_-3);
\draw[red, thick, <-] (7_-3) to (7_-1);
\draw[BurntOrange, thick, <-] (7_-1) to (5_-1);
\draw[red, thick, <-] (5_-1) to (5_1);
\draw[OliveGreen, thick, <-] (5_1) to (7_1);
\draw[red, thick, <-] (7_1) to (7_3);
\draw[OliveGreen, thick, <-] (7_3) to (9_3);
\draw[red, thick, <-] (9_3) to (9_5);
\draw[OliveGreen, thick, <-] (9_5) to (11_5);
\draw[red, thick, <-] (11_5) to (11_7);
\draw[OliveGreen, thick, <-] (11_7) to (29_7);
\draw[BurntOrange, thick, <-] (29_7) to (29_5);
\draw[OliveGreen, thick, <-] (29_5) to (31_5);
\draw[BurntOrange, thick, <-] (31_5) to (31_3);
\draw[OliveGreen, thick, <-] (31_3) to (33_3);
\draw[BurntOrange, thick, <-] (33_3) to (33_1);
\draw[OliveGreen, thick, <-] (33_1) to (35_1);
\draw[BurntOrange, thick, <-] (35_1) to (35_-1);
\draw[red, thick, <-] (35_-1) to (33_-1);
\draw[BurntOrange, thick, <-] (33_-1) to (33_-3);
\draw[red, thick, <-] (33_-3) to (31_-3);
\draw[BurntOrange, thick, <-] (31_-3) to (31_-5);
\draw[red, thick, <-] (31_-5) to (29_-5);
\draw[BurntOrange, thick, <-] (29_-5) to (29_-7);
\draw[red, thick, <-] (29_-7) to (27.5,-8.5);

\draw[BurntOrange, thick, <-] (t_5) to (9_-9);
\draw[red, thick, <-] (9_-9) to (9_-7);
\draw[BurntOrange, thick, <-] (9_-7) to (7_-7);
\draw[red, thick, <-] (7_-7) to (7_-5);
\draw[BurntOrange, thick, <-] (7_-5) to (5_-5);
\draw[red, thick, <-] (5_-5) to (5_-3);
\draw[BurntOrange, thick, <-] (5_-3) to (3_-3);
\draw[red, thick, <-] (3_-3) to (3_-1);
\draw[BurntOrange, thick, <-] (3_-1) to (1_-1);
\draw[red, thick, <-] (1_-1) to (1_1);
\draw[OliveGreen, thick, <-] (1_1) to (3_1);
\draw[red, thick, <-] (3_1) to (3_3);
\draw[OliveGreen, thick, <-] (3_3) to (5_3);
\draw[red, thick, <-] (5_3) to (5_5);
\draw[OliveGreen, thick, <-] (5_5) to (7_5);
\draw[red, thick, <-] (7_5) to (7_7);
\draw[OliveGreen, thick, <-] (7_7) to (9_7);
\draw[red, thick, <-] (9_7) to (9_9);
\draw[OliveGreen, thick, <-] (9_9) to (31_9);
\draw[BurntOrange, thick, <-] (31_9) to (31_7);
\draw[OliveGreen, thick, <-] (31_7) to (33_7);
\draw[BurntOrange, thick, <-] (33_7) to (33_5);
\draw[OliveGreen, thick, <-] (33_5) to (35_5);
\draw[BurntOrange, thick, <-] (35_5) to (35_3);
\draw[OliveGreen, thick, <-] (35_3) to (37_3);
\draw[BurntOrange, thick, <-] (37_3) to (37_1);
\draw[OliveGreen, thick, <-] (37_1) to (39_1);
\draw[BurntOrange, thick, <-] (39_1) to (39_-1);
\draw[red, thick, <-] (39_-1) to (37_-1);
\draw[BurntOrange, thick, <-] (37_-1) to (37_-3);
\draw[red, thick, <-] (37_-3) to (35_-3);
\draw[BurntOrange, thick, <-] (35_-3) to (35_-5);
\draw[red, thick, <-] (35_-5) to (33_-5);
\draw[BurntOrange, thick, <-] (33_-5) to (33_-7);
\draw[red, thick, <-] (33_-7) to (31_-7);
\draw[BurntOrange, thick, <-] (31_-7) to (31_-9);
\draw[red, thick, <-] (31_-9) to (29.5,-10.5);

\draw[BurntOrange, thick, ->] (h_1) to (1_-1);
\draw[BurntOrange, thick, ->] (h_2) to (3_-3);
\draw[BurntOrange, thick, ->] (h_3) to (5_-5);
\draw[BurntOrange, thick, ->] (h_4) to (7_-7);
\draw[BurntOrange, thick, ->] (h_5) to (9_-9);

\draw[red, thick, ->] (-0.5,1) to (1_1);
\draw[red, thick, ->] (-0.5,3) to (3_3);
\draw[red, thick, ->] (-0.5,5) to (5_5);
\draw[red, thick, ->] (-0.5,7) to (7_7);
\draw[red, thick, ->] (-0.5,9) to (9_9);

\draw[BurntOrange, thick, ->] (32.5,-10.5) to (31_-9);
\draw[BurntOrange, thick, ->] (34.5,-8.5) to (33_-7);
\draw[BurntOrange, thick, ->] (36.5,-6.5) to (35_-5);
\draw[BurntOrange, thick, ->] (38.5,-4.5) to (37_-3);
\draw[BurntOrange, thick, ->] (40.5,-2.5) to (39_-1);

\draw[OliveGreen, thick, ->] (40.5,1) to (39_1);
\draw[OliveGreen, thick, ->] (40.5,3) to (37_3);
\draw[OliveGreen, thick, ->] (40.5,5) to (35_5);
\draw[OliveGreen, thick, ->] (40.5,7) to (33_7);
\draw[OliveGreen, thick, ->] (40.5,9) to (31_9);

\draw[red, thick, ->] (9_7) to (11_7);
\draw[OliveGreen, thick, ->] (31_7) to (29_7);

\draw[BurntOrange, thick, ->] (9_-7) to (11_-7);
\draw[BurntOrange, thick, ->] (31_-7) to (29_-7);

\draw[red, thick, ->] (7_5) to (9_5);
\draw[red, thick, ->] (11_5) to (13_5);
\draw[OliveGreen, thick, ->] (33_5) to (31_5);
\draw[OliveGreen, thick, ->] (29_5) to (27_5);

\draw[BurntOrange, thick, ->] (7_-5) to (9_-5);
\draw[BurntOrange, thick, ->] (11_-5) to (13_-5);
\draw[BurntOrange, thick, ->] (33_-5) to (31_-5);
\draw[BurntOrange, thick, ->] (29_-5) to (27_-5);

\draw[red, thick, ->] (5_3) to (7_3);
\draw[red, thick, ->] (9_3) to (11_3);
\draw[red, thick, ->] (13_3) to (15_3);
\draw[OliveGreen, thick, ->] (35_3) to (33_3);
\draw[OliveGreen, thick, ->] (31_3) to (29_3);
\draw[OliveGreen, thick, ->] (27_3) to (25_3);

\draw[BurntOrange, thick, ->] (5_-3) to (7_-3);
\draw[BurntOrange, thick, ->] (9_-3) to (11_-3);
\draw[BurntOrange, thick, ->] (13_-3) to (15_-3);
\draw[BurntOrange, thick, ->] (35_-3) to (33_-3);
\draw[BurntOrange, thick, ->] (31_-3) to (29_-3);
\draw[BurntOrange, thick, ->] (27_-3) to (25_-3);

\draw[red, thick, ->] (3_1) to (5_1);
\draw[red, thick, ->] (7_1) to (9_1);
\draw[red, thick, ->] (11_1) to (13_1);
\draw[red, thick, ->] (15_1) to (17_1);
\draw[OliveGreen, thick, ->] (37_1) to (35_1);
\draw[OliveGreen, thick, ->] (33_1) to (31_1);
\draw[OliveGreen, thick, ->] (29_1) to (27_1);
\draw[OliveGreen, thick, ->] (25_1) to (23_1);

\draw[BurntOrange, thick, ->] (3_-1) to (5_-1);
\draw[BurntOrange, thick, ->] (7_-1) to (9_-1);
\draw[BurntOrange, thick, ->] (11_-1) to (13_-1);
\draw[BurntOrange, thick, ->] (15_-1) to (17_-1);
\draw[BurntOrange, thick, ->] (37_-1) to (35_-1);
\draw[BurntOrange, thick, ->] (33_-1) to (31_-1);
\draw[BurntOrange, thick, ->] (29_-1) to (27_-1);
\draw[BurntOrange, thick, ->] (25_-1) to (23_-1);

\end{tikzpicture}
\caption{Network $\Nc_4^{\mathrm{sf}_+}$.}
\label{net-sf+}
\end{figure}

\begin{figure}[h]
\begin{tikzpicture}[every node/.style={inner sep=0, minimum size=0.2cm, draw, circle, thick}, x=0.35cm, y=0.35cm]

\foreach \i in {1,5,9,13,17,23,27,31,35,39}
{
	\node[fill=white] (\i_1) at (\i,1) {};
	\node[fill=white] (\i_-1) at (\i,-1) {};
}
\foreach \i in {3,7,11,15,25,29,33,37}
{
	\node[fill=white] (\i_3) at (\i,3) {};
	\node[fill=white] (\i_-3) at (\i,-3) {};
	\node[fill=black] (\i_1) at (\i,1) {};
	\node[fill=black] (\i_-1) at (\i,-1) {};
}
\foreach \i in {5,9,13,27,31,35}
{
	\node[fill=white] (\i_5) at (\i,5) {};
	\node[fill=white] (\i_-5) at (\i,-5) {};
	\node[fill=black] (\i_3) at (\i,3) {};
	\node[fill=black] (\i_-3) at (\i,-3) {};
}

\foreach \i in {7,11,29,33}
{
	\node[fill=white] (\i_7) at (\i,7) {};
	\node[fill=white] (\i_-7) at (\i,-7) {};
	\node[fill=black] (\i_5) at (\i,5) {};
	\node[fill=black] (\i_-5) at (\i,-5) {};
}

\foreach \i in {9,31}
{
	\node[fill=white] (\i_9) at (\i,9) {};
	\node[fill=white] (\i_-9) at (\i,-9) {};
	\node[fill=black] (\i_7) at (\i,7) {};
	\node[fill=black] (\i_-7) at (\i,-7) {};
}

\node[minimum size=0.4cm] (t_5) at (10.5,-10.5) {\scriptsize $t_1$};
\node[minimum size=0.4cm] (t_4) at (12.5,-8.5) {\scriptsize $t_2$};
\node[minimum size=0.4cm] (t_3) at (14.5,-6.5) {\scriptsize $t_3$};
\node[minimum size=0.4cm] (t_2) at (16.5,-4.5) {\scriptsize $t_4$};
\node[minimum size=0.4cm] (t_1) at (18.5,-2.5) {\scriptsize $t_5$};

\node[minimum size=0.4cm] (h_1) at (21.5,-2.5) {\scriptsize $h_5$};
\node[minimum size=0.4cm] (h_2) at (23.5,-4.5) {\scriptsize $h_4$};
\node[minimum size=0.4cm] (h_3) at (25.5,-6.5) {\scriptsize $h_3$};
\node[minimum size=0.4cm] (h_4) at (27.5,-8.5) {\scriptsize $h_2$};
\node[minimum size=0.4cm] (h_5) at (29.5,-10.5) {\scriptsize $h_1$};

\draw[BrickRed, dashed] (-0.5,0) to (40.5,0);

\node[draw=none, blue] at (0,2) {\scriptsize 25};
\node[draw=none, blue] at (0,4) {\scriptsize 26};
\node[draw=none, blue] at (0,6) {\scriptsize 27};
\node[draw=none, blue] at (0,8) {\scriptsize 28};

\node[draw=none, blue] at (8,-8) {\scriptsize 1};
\node[draw=none, blue] at (10,-8) {\scriptsize 3};
\node[draw=none, blue] at (6,-6) {\scriptsize 4};
\node[draw=none, blue] at (9,-6) {\scriptsize 2};
\node[draw=none, blue] at (12,-6) {\scriptsize 8};
\node[draw=none, blue] at (4,-4) {\scriptsize 9};
\node[draw=none, blue] at (7,-4) {\scriptsize 5};
\node[draw=none, blue] at (11,-4) {\scriptsize 7};
\node[draw=none, blue] at (14,-4) {\scriptsize 15};
\node[draw=none, blue] at (2,-2) {\scriptsize 16};
\node[draw=none, blue] at (5,-2) {\scriptsize 10};
\node[draw=none, blue] at (9,-2) {\scriptsize 6};
\node[draw=none, blue] at (13,-2) {\scriptsize 14};
\node[draw=none, blue] at (16,-2) {\scriptsize 24};
\node[draw=none, blue] at (3,0) {\scriptsize 17};
\node[draw=none, blue] at (7,0) {\scriptsize 11};
\node[draw=none, blue] at (11,0) {\scriptsize 13};
\node[draw=none, blue] at (15,0) {\scriptsize 23};
\node[draw=none, blue] at (5,2) {\scriptsize 18};
\node[draw=none, blue] at (9,2) {\scriptsize 12};
\node[draw=none, blue] at (13,2) {\scriptsize 22};
\node[draw=none, blue] at (7,4) {\scriptsize 19};
\node[draw=none, blue] at (11,4) {\scriptsize 21};
\node[draw=none, blue] at (9,6) {\scriptsize 20};

\node[draw=none, blue] at (20,2) {\scriptsize 25};
\node[draw=none, blue] at (20,4) {\scriptsize 26};
\node[draw=none, blue] at (20,6) {\scriptsize 27};
\node[draw=none, blue] at (20,8) {\scriptsize 28};

\node[draw=none, blue] at (30,-8) {\scriptsize 1};
\node[draw=none, blue] at (32,-8) {\scriptsize 3};
\node[draw=none, blue] at (28,-6) {\scriptsize 4};
\node[draw=none, blue] at (31,-6) {\scriptsize 2};
\node[draw=none, blue] at (34,-6) {\scriptsize 8};
\node[draw=none, blue] at (26,-4) {\scriptsize 9};
\node[draw=none, blue] at (29,-4) {\scriptsize 5};
\node[draw=none, blue] at (33,-4) {\scriptsize 7};
\node[draw=none, blue] at (36,-4) {\scriptsize 15};
\node[draw=none, blue] at (24,-2) {\scriptsize 16};
\node[draw=none, blue] at (27,-2) {\scriptsize 10};
\node[draw=none, blue] at (31,-2) {\scriptsize 6};
\node[draw=none, blue] at (35,-2) {\scriptsize 14};
\node[draw=none, blue] at (38,-2) {\scriptsize 24};
\node[draw=none, blue] at (25,0) {\scriptsize 17};
\node[draw=none, blue] at (29,0) {\scriptsize 11};
\node[draw=none, blue] at (33,0) {\scriptsize 13};
\node[draw=none, blue] at (37,0) {\scriptsize 23};
\node[draw=none, blue] at (27,2) {\scriptsize 18};
\node[draw=none, blue] at (31,2) {\scriptsize 12};
\node[draw=none, blue] at (35,2) {\scriptsize 22};
\node[draw=none, blue] at (29,4) {\scriptsize 19};
\node[draw=none, blue] at (33,4) {\scriptsize 21};
\node[draw=none, blue] at (31,6) {\scriptsize 20};

\node[draw=none, blue] at (40,2) {\scriptsize 25};
\node[draw=none, blue] at (40,4) {\scriptsize 26};
\node[draw=none, blue] at (40,6) {\scriptsize 27};
\node[draw=none, blue] at (40,8) {\scriptsize 28};

\draw[OliveGreen, thick, ->] (t_1) to (17_-1);
\draw[OliveGreen, thick, ->] (17_-1) to (17_1);
\draw[OliveGreen, thick, ->] (17_1) to (23_1);
\draw[OliveGreen, thick, ->] (23_1) to (23_-1);
\draw[OliveGreen, thick, ->] (23_-1) to (h_1);

\draw[OliveGreen, thick, ->] (t_2) to (15_-3);
\draw[OliveGreen, thick, ->] (15_-3) to (15_-1);
\draw[OliveGreen, thick, ->] (15_-1) to (13_-1);
\draw[OliveGreen, thick, ->] (13_-1) to (13_1);
\draw[OliveGreen, thick, ->] (13_1) to (15_1);
\draw[OliveGreen, thick, ->] (15_1) to (15_3);
\draw[OliveGreen, thick, ->] (15_3) to (25_3);
\draw[OliveGreen, thick, ->] (25_3) to (25_1);
\draw[OliveGreen, thick, ->] (25_1) to (27_1);
\draw[OliveGreen, thick, ->] (27_1) to (27_-1);
\draw[OliveGreen, thick, ->] (27_-1) to (25_-1);
\draw[OliveGreen, thick, ->] (25_-1) to (25_-3);
\draw[OliveGreen, thick, ->] (25_-3) to (h_2);

\draw[OliveGreen, thick, ->] (t_3) to (13_-5);
\draw[OliveGreen, thick, ->] (13_-5) to (13_-3);
\draw[OliveGreen, thick, ->] (13_-3) to (11_-3);
\draw[OliveGreen, thick, ->] (11_-3) to (11_-1);
\draw[OliveGreen, thick, ->] (11_-1) to (9_-1);
\draw[OliveGreen, thick, ->] (9_-1) to (9_1);
\draw[OliveGreen, thick, ->] (9_1) to (11_1);
\draw[OliveGreen, thick, ->] (11_1) to (11_3);
\draw[OliveGreen, thick, ->] (11_3) to (13_3);
\draw[OliveGreen, thick, ->] (13_3) to (13_5);
\draw[OliveGreen, thick, ->] (13_5) to (27_5);
\draw[OliveGreen, thick, ->] (27_5) to (27_3);
\draw[OliveGreen, thick, ->] (27_3) to (29_3);
\draw[OliveGreen, thick, ->] (29_3) to (29_1);
\draw[OliveGreen, thick, ->] (29_1) to (31_1);
\draw[OliveGreen, thick, ->] (31_1) to (31_-1);
\draw[OliveGreen, thick, ->] (31_-1) to (29_-1);
\draw[OliveGreen, thick, ->] (29_-1) to (29_-3);
\draw[OliveGreen, thick, ->] (29_-3) to (27_-3);
\draw[OliveGreen, thick, ->] (27_-3) to (27_-5);
\draw[OliveGreen, thick, ->] (27_-5) to (h_3);

\draw[OliveGreen, thick, ->] (t_4) to (11_-7);
\draw[OliveGreen, thick, ->] (11_-7) to (11_-5);
\draw[OliveGreen, thick, ->] (11_-5) to (9_-5);
\draw[OliveGreen, thick, ->] (9_-5) to (9_-3);
\draw[OliveGreen, thick, ->] (9_-3) to (7_-3);
\draw[OliveGreen, thick, ->] (7_-3) to (7_-1);
\draw[OliveGreen, thick, ->] (7_-1) to (5_-1);
\draw[OliveGreen, thick, ->] (5_-1) to (5_1);
\draw[OliveGreen, thick, ->] (5_1) to (7_1);
\draw[OliveGreen, thick, ->] (7_1) to (7_3);
\draw[OliveGreen, thick, ->] (7_3) to (9_3);
\draw[OliveGreen, thick, ->] (9_3) to (9_5);
\draw[OliveGreen, thick, ->] (9_5) to (11_5);
\draw[OliveGreen, thick, ->] (11_5) to (11_7);
\draw[OliveGreen, thick, ->] (11_7) to (29_7);
\draw[OliveGreen, thick, ->] (29_7) to (29_5);
\draw[OliveGreen, thick, ->] (29_5) to (31_5);
\draw[OliveGreen, thick, ->] (31_5) to (31_3);
\draw[OliveGreen, thick, ->] (31_3) to (33_3);
\draw[OliveGreen, thick, ->] (33_3) to (33_1);
\draw[OliveGreen, thick, ->] (33_1) to (35_1);
\draw[OliveGreen, thick, ->] (35_1) to (35_-1);
\draw[OliveGreen, thick, ->] (35_-1) to (33_-1);
\draw[OliveGreen, thick, ->] (33_-1) to (33_-3);
\draw[OliveGreen, thick, ->] (33_-3) to (31_-3);
\draw[OliveGreen, thick, ->] (31_-3) to (31_-5);
\draw[OliveGreen, thick, ->] (31_-5) to (29_-5);
\draw[OliveGreen, thick, ->] (29_-5) to (29_-7);
\draw[OliveGreen, thick, ->] (29_-7) to (h_4);

\draw[OliveGreen, thick, ->] (t_5) to (9_-9);
\draw[OliveGreen, thick, ->] (9_-9) to (9_-7);
\draw[OliveGreen, thick, ->] (9_-7) to (7_-7);
\draw[OliveGreen, thick, ->] (7_-7) to (7_-5);
\draw[OliveGreen, thick, ->] (7_-5) to (5_-5);
\draw[OliveGreen, thick, ->] (5_-5) to (5_-3);
\draw[OliveGreen, thick, ->] (5_-3) to (3_-3);
\draw[OliveGreen, thick, ->] (3_-3) to (3_-1);
\draw[OliveGreen, thick, ->] (3_-1) to (1_-1);
\draw[OliveGreen, thick, ->] (1_-1) to (1_1);
\draw[OliveGreen, thick, ->] (1_1) to (3_1);
\draw[OliveGreen, thick, ->] (3_1) to (3_3);
\draw[OliveGreen, thick, ->] (3_3) to (5_3);
\draw[OliveGreen, thick, ->] (5_3) to (5_5);
\draw[OliveGreen, thick, ->] (5_5) to (7_5);
\draw[OliveGreen, thick, ->] (7_5) to (7_7);
\draw[OliveGreen, thick, ->] (7_7) to (9_7);
\draw[OliveGreen, thick, ->] (9_7) to (9_9);
\draw[OliveGreen, thick, ->] (9_9) to (31_9);
\draw[OliveGreen, thick, ->] (31_9) to (31_7);
\draw[OliveGreen, thick, ->] (31_7) to (33_7);
\draw[OliveGreen, thick, ->] (33_7) to (33_5);
\draw[OliveGreen, thick, ->] (33_5) to (35_5);
\draw[OliveGreen, thick, ->] (35_5) to (35_3);
\draw[OliveGreen, thick, ->] (35_3) to (37_3);
\draw[OliveGreen, thick, ->] (37_3) to (37_1);
\draw[OliveGreen, thick, ->] (37_1) to (39_1);
\draw[OliveGreen, thick, ->] (39_1) to (39_-1);
\draw[OliveGreen, thick, ->] (39_-1) to (37_-1);
\draw[OliveGreen, thick, ->] (37_-1) to (37_-3);
\draw[OliveGreen, thick, ->] (37_-3) to (35_-3);
\draw[OliveGreen, thick, ->] (35_-3) to (35_-5);
\draw[OliveGreen, thick, ->] (35_-5) to (33_-5);
\draw[OliveGreen, thick, ->] (33_-5) to (33_-7);
\draw[OliveGreen, thick, ->] (33_-7) to (31_-7);
\draw[OliveGreen, thick, ->] (31_-7) to (31_-9);
\draw[OliveGreen, thick, ->] (31_-9) to (h_5);

\draw[BurntOrange, thick, ->] (7.5,-10.5) to (9_-9);
\draw[BurntOrange, thick, ->] (5.5,-8.5) to (7_-7);
\draw[BurntOrange, thick, ->] (3.5,-6.5) to (5_-5);
\draw[BurntOrange, thick, ->] (1.5,-4.5) to (3_-3);
\draw[BurntOrange, thick, ->] (-0.5,-2.5) to (1_-1);

\draw[BurntOrange, thick, ->] (-0.5,1) to (1_1);
\draw[BurntOrange, thick, ->] (-0.5,3) to (3_3);
\draw[BurntOrange, thick, ->] (-0.5,5) to (5_5);
\draw[BurntOrange, thick, ->] (-0.5,7) to (7_7);
\draw[BurntOrange, thick, ->] (-0.5,9) to (9_9);

\draw[BurntOrange, thick, ->] (32.5,-10.5) to (31_-9);
\draw[BurntOrange, thick, ->] (34.5,-8.5) to (33_-7);
\draw[BurntOrange, thick, ->] (36.5,-6.5) to (35_-5);
\draw[BurntOrange, thick, ->] (38.5,-4.5) to (37_-3);
\draw[BurntOrange, thick, ->] (40.5,-2.5) to (39_-1);

\draw[BurntOrange, thick, ->] (40.5,1) to (39_1);
\draw[BurntOrange, thick, ->] (40.5,3) to (37_3);
\draw[BurntOrange, thick, ->] (40.5,5) to (35_5);
\draw[BurntOrange, thick, ->] (40.5,7) to (33_7);
\draw[BurntOrange, thick, ->] (40.5,9) to (31_9);

\foreach \i in {7,-7}
{
\draw[BurntOrange, thick, ->] (9_\i) to (11_\i);
\draw[BurntOrange, thick, ->] (31_\i) to (29_\i);
}

\foreach \i in {5,-5}
{
\draw[BurntOrange, thick, ->] (7_\i) to (9_\i);
\draw[BurntOrange, thick, ->] (11_\i) to (13_\i);
\draw[BurntOrange, thick, ->] (33_\i) to (31_\i);
\draw[BurntOrange, thick, ->] (29_\i) to (27_\i);
}

\foreach \i in {3,-3}
{
\draw[BurntOrange, thick, ->] (5_\i) to (7_\i);
\draw[BurntOrange, thick, ->] (9_\i) to (11_\i);
\draw[BurntOrange, thick, ->] (13_\i) to (15_\i);
\draw[BurntOrange, thick, ->] (35_\i) to (33_\i);
\draw[BurntOrange, thick, ->] (31_\i) to (29_\i);
\draw[BurntOrange, thick, ->] (27_\i) to (25_\i);
}

\foreach \i in {1,-1}
{
\draw[BurntOrange, thick, ->] (3_\i) to (5_\i);
\draw[BurntOrange, thick, ->] (7_\i) to (9_\i);
\draw[BurntOrange, thick, ->] (11_\i) to (13_\i);
\draw[BurntOrange, thick, ->] (15_\i) to (17_\i);
\draw[BurntOrange, thick, ->] (37_\i) to (35_\i);
\draw[BurntOrange, thick, ->] (33_\i) to (31_\i);
\draw[BurntOrange, thick, ->] (29_\i) to (27_\i);
\draw[BurntOrange, thick, ->] (25_\i) to (23_\i);
}

\end{tikzpicture}
\caption{Network $\Nc_4^{\mathrm{sf}_-}$.}
\label{net-sf-}
\end{figure}

\subsection{Embedding $\tau$ and flips of triangulation}
\label{subsec-full-toda}

Our goal now is to show how the diagonal embedding $\tau \colon \Dgt_n \to \Dsq$, defined in~\eqref{tau}, can be rewritten in a cluster chart corresponding to a different triangulation of the twice punctured disk $D_{2,2}$. The seed $\Theta_n^{\mathrm{sq}}$ arises from the triangulation shown in the top left corner of Figure~\ref{top-muts}. The quiver $\Qfull$ for the seed $\Theta_n^{\mathrm{full}}$ corresponding to the triangulation in the bottom left corner of the same figure is shown in Figure~\ref{Q4-full}. As explained in Figure~\ref{Q4-full}, in order to obtain the seed $\Theta_n^{\mathrm{full}}$ from the initial seed~$\Theta_n^{\mathrm{sq}}$, one performs~6 flips of triangulation, which are realized by a total of $n(n+1)(n+2)$ cluster mutations.

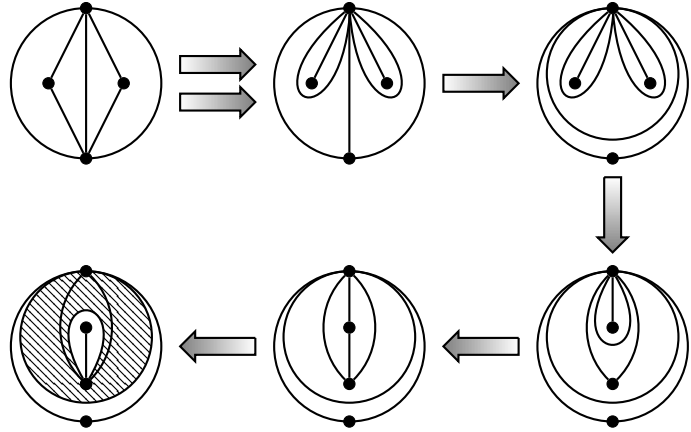
\begin{figure}[h]
\begin{tikzpicture}[every node/.style={inner sep=0, minimum size=0.15cm, circle, draw, fill=black}, x=0.25cm,y=0.25cm]

\node[shift={(0,14)}] (1) at (2,0) {};
\node[shift={(0,14)}] (2) at (0,4) {};
\node[shift={(0,14)}] (3) at (-2,0) {};
\node[shift={(0,14)}] (4) at (0,-4) {};

\draw[thick,shift={(0,14)}] (0,0) circle (4);
\draw[thick] (1) -- (2) -- (3) -- (4) -- (1);
\draw[thick] (2) --(4);

\node[shift={(14,14)}] (1) at (2,0) {};
\node[shift={(14,14)}] (2) at (0,4) {};
\node[shift={(14,14)}] (3) at (-2,0) {};
\node[shift={(14,14)}] (4) at (0,-4) {};

\draw[thick,shift={(14,14)}] (0,0) circle (4);
\draw[thick] (2) -- (4);
\draw[thick] (1) -- (2) -- (3);
\draw[thick] (2) to [out=-40,in=-85,min distance = 1.8cm] (2) to [out=-95,in=-140,min distance = 1.8cm] (2);

\node[shift={(28,14)}] (1) at (2,0) {};
\node[shift={(28,14)}] (2) at (0,4) {};
\node[shift={(28,14)}] (3) at (-2,0) {};
\node[shift={(28,14)}] (4) at (0,-4) {};

\draw[thick,shift={(28,14)}] (0,0) circle (4);
\draw[thick,shift={(28,14)}] (0,0.5) circle (3.5);
\draw[thick] (1) -- (2) -- (3);
\draw[thick] (2) to [out=-40,in=-85,min distance = 1.8cm] (2) to [out=-95,in=-140,min distance = 1.8cm] (2);

\node[shift={(28,0)}] (1) at (0,4) {};
\node[shift={(28,0)}] (2) at (0,1) {};
\node[shift={(28,0)}] (3) at (0,-2) {};
\node[shift={(28,0)}] (4) at (0,-4) {};

\draw[thick,shift={(28,0)}] (0,0) circle (4);
\draw[thick,shift={(28,0)}] (0,0.5) circle (3.5);
\draw[thick] (1) -- (2);
\draw[thick] (1) to [bend left = 45] (3);
\draw[thick] (1) to [bend right = 45] (3);
\draw[thick] (1) to [out=-60,in=-120,min distance = 1.4cm] (1);

\node[shift={(14,0)}] (1) at (0,4) {};
\node[shift={(14,0)}] (2) at (0,1) {};
\node[shift={(14,0)}] (3) at (0,-2) {};
\node[shift={(14,0)}] (4) at (0,-4) {};

\draw[thick,shift={(14,0)}] (0,0) circle (4);
\draw[thick,shift={{(14,0)}}] (0,0.5) circle (3.5);
\draw[thick] (1) -- (2) -- (3);
\draw[thick] (1) to [bend left = 45] (3);
\draw[thick] (1) to [bend right = 45] (3);

\node (1) at (0,4) {};
\node (2) at (0,1) {};
\node (3) at (0,-2) {};
\node (4) at (0,-4) {};

\draw[thick] (0,0) circle (4);
\draw[thick, pattern = north west lines] (0,0.5) circle (3.5);
\draw[thick, fill=white] (3) to [out=60,in=120,min distance = 1.4cm] (3);

\node (2) at (0,1) {};
\node (3) at (0,-2) {};

\draw[thick] (2) -- (3);
\draw[thick] (1) to [bend left = 45] (3);
\draw[thick] (1) to [bend right = 45] (3);

\draw[thick, shift={(7,15)}, shade, shading angle=-90] (-2,-0.45) to (-2,0.45) to (1.2,0.45) to (1.2,0.8) to (2,0) to (1.2,-0.8) to (1.2,-0.45) to (-2,-0.45) -- cycle;

\draw[thick, shift={(7,13)}, shade, shading angle=-90] (-2,-0.45) to (-2,0.45) to (1.2,0.45) to (1.2,0.8) to (2,0) to (1.2,-0.8) to (1.2,-0.45) to (-2,-0.45) -- cycle;

\draw[thick, shift={(21,14)}, shade, shading angle=-90] (-2,-0.45) to (-2,0.45) to (1.2,0.45) to (1.2,0.8) to (2,0) to (1.2,-0.8) to (1.2,-0.45) to (-2,-0.45) -- cycle;

\draw[thick, shift={(28,7)}, rotate=-90, shade, shading angle=180] (-2,-0.45) to (-2,0.45) to (1.2,0.45) to (1.2,0.8) to (2,0) to (1.2,-0.8) to (1.2,-0.45) to (-2,-0.45);

\draw[thick, shift={(21,0)}, rotate=180, shade, shading angle=90] (-2,-0.45) to (-2,0.45) to (1.2,0.45) to (1.2,0.8) to (2,0) to (1.2,-0.8) to (1.2,-0.45) to (-2,-0.45);

\draw[thick, shift={(7,0)}, rotate=180, shade, shading angle=90] (-2,-0.45) to (-2,0.45) to (1.2,0.45) to (1.2,0.8) to (2,0) to (1.2,-0.8) to (1.2,-0.45) to (-2,-0.45);

\end{tikzpicture}
\caption{Sequence of flips for a double quiver.}
\label{top-muts}
\end{figure}

\begin{example}
We spell out this sequence of mutations explicitly for $\Qc_4^{\mathrm{sq}}$. The first two flips, that lead to the middle triangulation in the top row of Figure~\ref{top-muts}, are as follows:
\begin{align*}
&2, \, 8, \, 18, \, 32; & &7, \, 17, \, 31, \, 9, \, 19, \, 33; & &16, \, 30, \, 8, \, 18, \, 20, \, 34; & &29, \, 17, \, 19, \, 35; \\
&4, \, 12, \, 24, \, 40; & &11, \, 23, \, 39, \, 13, \, 25, \, 41; & &22, \, 38, \, 12, \, 24, \, 26, \, 42; & &37, \, 23, \, 25, \, 43.
\end{align*}
The next flip is given by the sequence
\begin{align*}
&3, \, 10, \, 21, \, 36; & &2, \, 9, \, 20, \, 4, \, 11, \, 22; & &7, \, 8, \, 10, \, 21, \, 13, \, 12; & &16, \, 9, \, 11, \, 26.
\end{align*}
Now, we have to choose one of the two self-folded triangles and flip its arc. Choosing the right one, we get the mutation sequence
\begin{align*}
&37, \, 23, \, 25, \, 43; & &36, \, 22, \, 12, \, 38, \, 24, \, 42; & &20, \, 21, \, 23, \, 25, \, 39, \, 41; & &8, \, 22, \, 24, \, 40.
\end{align*}
Next, we flip the arc of the only self-folded triangle:
\begin{align*}
&29, \, 17, \, 19, \, 35; & &20, \, 36, \, 37, \, 30, \, 18, \, 34; & &23, \, 38, \, 17, \, 19, \, 31, \, 33; & &39, \, 36, \, 18, \, 32.
\end{align*}
Then, the final flip is realized by mutation at the vertices
\begin{align*}
&49, \, 50, \, 51, \, 52; & &29, \, 30, \, 31, \, 35, \, 34, \, 33; & &20, \, 17, \, 50, \, 51, \, 37, \, 19; & &23, \, 30, \, 34, \, 38.
\end{align*}
\end{example}

The quiver $\Qfull$ is divided into three parts: the top part corresponds to the innermost self-folded triangle in the bottom left disk in Figure~\ref{top-muts}; the middle one to the shaded annulus; and finally the bottom one to the outer triangle. Now, we observe that the image $\tau(\Dgt_n)$ depends only on cluster variables in the bottom two parts. This can be verified by inspecting the corresponding networks. We remind the reader that so far we have only dealt with mutations for quivers on a once-punctured disk, while here we have two punctures. To get around this problem it suffices to observe that since flips of triangulation are local operations, we may regard any such flip as taking place on an appropriate once punctured disk.

Now, we can regard the quiver obtained from the bottom two parts of $\Qfull$ as a quiver on an annulus; Figures~\ref{net-full+} and~\ref{net-full-} show two networks on the universal cover of an annulus whose dual quivers project to $\Qfull$. Let us denote by $\Phi^{\mathrm{top}}$ the unitary operator on $\mathcal{P}_\lambda\otimes\mathcal{P}_\mu$ obtained as the composite of the $n(n+1)(n+2)$ quantum mutations taking $\Theta_n^{\mathrm{sq}}$ to $\Theta_n^{\mathrm{full}}$. Then we have

\begin{cor}
\label{cor-full}
Let
$$
M_n^{\mathrm{full}_\pm}(i,j;k,l) = \mathrm{pr}\hr{Z_{\Nc_n^{\mathrm{full}_\pm}}(t_i,h_j;t_k,h_l)}
$$
be the boundary measurements in $\Nc_n^{\mathrm{full}_\pm}$. Then as operators on $\mathcal{P}_\lambda\otimes\mathcal{P}_\mu$, we have
\begin{align*}
&\Ad_{\Phi^{\mathrm{top}}}\cdot\iota(K_i) = {M_n^{\mathrm{full}_+}(i+1,i+1;i,i)}, &
&\Ad_{\Phi^{\mathrm{top}}}\cdot\iota(E_i) = {M_n^{\mathrm{full}_+}(i+1,i;i,i)}, \\
&\Ad_{\Phi^{\mathrm{top}}}\cdot\iota(K'_i) = {M_n^{\mathrm{full}_-}(i,i;i+1,i+1)}, &
&\Ad_{\Phi^{\mathrm{top}}}\cdot\iota(F_i) = {M_n^{\mathrm{full}_-}(i,i+1;i+1,i+1)}.
\end{align*}
\end{cor}

\vspace{-.1cm}
\begin{figure}[h]
\begin{tikzpicture}[every node/.style={inner sep=0, minimum size=0.2cm, draw, circle, thick}, x=0.35cm, y=0.35cm]

\foreach \i in {1,5,9,13,17,23,27,31,35,39}
{
	\node[fill=white] (\i_1) at (\i,1) {};
	\node[fill=white] (\i_-1) at (\i,-1) {};
}
\foreach \i in {3,7,11,15,25,29,33,37}
{
	\node[fill=white] (\i_3) at (\i,3) {};
	\node[fill=white] (\i_-3) at (\i,-3) {};
	\node[fill=black] (\i_1) at (\i,1) {};
	\node[fill=black] (\i_-1) at (\i,-1) {};
}
\foreach \i in {5,9,13,27,31,35}
{
	\node[fill=white] (\i_5) at (\i,5) {};
	\node[fill=white] (\i_-5) at (\i,-5) {};
	\node[fill=black] (\i_3) at (\i,3) {};
	\node[fill=black] (\i_-3) at (\i,-3) {};
}

\foreach \i in {7,11,29,33}
{
	\node[fill=white] (\i_7) at (\i,7) {};
	\node[fill=white] (\i_-7) at (\i,-7) {};
	\node[fill=black] (\i_5) at (\i,5) {};
	\node[fill=black] (\i_-5) at (\i,-5) {};
}

\foreach \i in {9,31}
{
	\node[fill=white] (\i_9) at (\i,9) {};
	\node[fill=white] (\i_-9) at (\i,-9) {};
	\node[fill=black] (\i_7) at (\i,7) {};
	\node[fill=black] (\i_-7) at (\i,-7) {};
}

\foreach \i in {7,29}
{
	\node[fill=white] (\i_9) at (\i,9) {};
}

\foreach \i in {5,27}
{
	\node[fill=black] (\i_9) at (\i,9) {};
	\node[fill=white] (\i_7) at (\i,7) {};
}

\foreach \i in {3,25}
{
	\node[fill=white] (\i_9) at (\i,9) {};
	\node[fill=black] (\i_7) at (\i,7) {};
	\node[fill=white] (\i_5) at (\i,5) {};
}

\foreach \i in {1,23}
{
	\node[fill=black] (\i_9) at (\i,9) {};
	\node[fill=white] (\i_7) at (\i,7) {};
	\node[fill=black] (\i_5) at (\i,5) {};
	\node[fill=white] (\i_3) at (\i,3) {};
}

	\node[fill=white] (21_9) at (21,9) {};
	\node[fill=black] (21_7) at (21,7) {};
	\node[fill=white] (21_5) at (21,5) {};
	\node[fill=black] (21_3) at (21,3) {};
	\node[fill=white] (21_1) at (21,1) {};

	\node[fill=black] (19_9) at (19,9) {};
	\node[fill=white] (19_7) at (19,7) {};
	\node[fill=black] (19_5) at (19,5) {};
	\node[fill=white] (19_3) at (19,3) {};

\foreach \i in {17,39}
{
	\node[fill=white] (\i_9) at (\i,9) {};
	\node[fill=black] (\i_7) at (\i,7) {};
	\node[fill=white] (\i_5) at (\i,5) {};
}

\foreach \i in {15,37}
{
	\node[fill=black] (\i_9) at (\i,9) {};
	\node[fill=white] (\i_7) at (\i,7) {};
}

\foreach \i in {13,35}
	\node[fill=white] (\i_9) at (\i,9) {};

\node[minimum size=0.4cm] (t_5) at (10.5,-10.5) {\scriptsize $t_5$};
\node[minimum size=0.4cm] (t_4) at (12.5,-8.5) {\scriptsize $t_4$};
\node[minimum size=0.4cm] (t_3) at (14.5,-6.5) {\scriptsize $t_3$};
\node[minimum size=0.4cm] (t_2) at (16.5,-4.5) {\scriptsize $t_2$};
\node[minimum size=0.4cm] (t_1) at (18.5,-2.5) {\scriptsize $t_1$};

\node[minimum size=0.4cm] (h_1) at (-0.5,-2.5) {\scriptsize $h_1$};
\node[minimum size=0.4cm] (h_2) at (1.5,-4.5) {\scriptsize $h_2$};
\node[minimum size=0.4cm] (h_3) at (3.5,-6.5) {\scriptsize $h_3$};
\node[minimum size=0.4cm] (h_4) at (5.5,-8.5) {\scriptsize $h_4$};
\node[minimum size=0.4cm] (h_5) at (7.5,-10.5) {\scriptsize $h_5$};

\node[draw=none, blue] at (1,2) {\scriptsize 8};
\node[draw=none, blue] at (5,2) {\scriptsize 21};
\node[draw=none, blue] at (9,2) {\scriptsize 12};
\node[draw=none, blue] at (13,2) {\scriptsize 43};
\node[draw=none, blue] at (17,2) {\scriptsize 45};
\node[draw=none, blue] at (3,4) {\scriptsize 22};
\node[draw=none, blue] at (7,4) {\scriptsize 25};
\node[draw=none, blue] at (11,4) {\scriptsize 42};
\node[draw=none, blue] at (15,4) {\scriptsize 46};
\node[draw=none, blue] at (1,6) {\scriptsize 29};
\node[draw=none, blue] at (5,6) {\scriptsize 24};
\node[draw=none, blue] at (9,6) {\scriptsize 41};
\node[draw=none, blue] at (13,6) {\scriptsize 47};
\node[draw=none, blue] at (3,8) {\scriptsize 20};
\node[draw=none, blue] at (7,8) {\scriptsize 40};
\node[draw=none, blue] at (11,8) {\scriptsize 48};
\node[draw=none, blue] at (1,10) {\scriptsize 30};
\node[draw=none, blue] at (5,10) {\scriptsize 23};

\node[draw=none, blue] at (8,-8) {\scriptsize 1};
\node[draw=none, blue] at (10,-8) {\scriptsize 5};
\node[draw=none, blue] at (6,-6) {\scriptsize 6};
\node[draw=none, blue] at (9,-6) {\scriptsize 3};
\node[draw=none, blue] at (12,-6) {\scriptsize 14};
\node[draw=none, blue] at (4,-4) {\scriptsize 15};
\node[draw=none, blue] at (7,-4) {\scriptsize 2};
\node[draw=none, blue] at (11,-4) {\scriptsize 4};
\node[draw=none, blue] at (14,-4) {\scriptsize 27};
\node[draw=none, blue] at (2,-2) {\scriptsize 28};
\node[draw=none, blue] at (5,-2) {\scriptsize 7};
\node[draw=none, blue] at (9,-2) {\scriptsize 10};
\node[draw=none, blue] at (13,-2) {\scriptsize 13};
\node[draw=none, blue] at (16,-2) {\scriptsize 44};
\node[draw=none, blue] at (3,0) {\scriptsize 16};
\node[draw=none, blue] at (7,0) {\scriptsize 9};
\node[draw=none, blue] at (11,0) {\scriptsize 11};
\node[draw=none, blue] at (15,0) {\scriptsize 26};

\node[draw=none, blue] at (23,2) {\scriptsize 8};
\node[draw=none, blue] at (27,2) {\scriptsize 21};
\node[draw=none, blue] at (31,2) {\scriptsize 12};
\node[draw=none, blue] at (35,2) {\scriptsize 43};
\node[draw=none, blue] at (39,2) {\scriptsize 45};
\node[draw=none, blue] at (21,4) {\scriptsize 49};
\node[draw=none, blue] at (25,4) {\scriptsize 22};
\node[draw=none, blue] at (29,4) {\scriptsize 25};
\node[draw=none, blue] at (33,4) {\scriptsize 42};
\node[draw=none, blue] at (37,4) {\scriptsize 46};
\node[draw=none, blue] at (19,6) {\scriptsize 35};
\node[draw=none, blue] at (23,6) {\scriptsize 29};
\node[draw=none, blue] at (27,6) {\scriptsize 24};
\node[draw=none, blue] at (31,6) {\scriptsize 41};
\node[draw=none, blue] at (35,6) {\scriptsize 47};
\node[draw=none, blue] at (17,8) {\scriptsize 37};
\node[draw=none, blue] at (21,8) {\scriptsize 50};
\node[draw=none, blue] at (25,8) {\scriptsize 20};
\node[draw=none, blue] at (29,8) {\scriptsize 40};
\node[draw=none, blue] at (33,8) {\scriptsize 48};
\node[draw=none, blue] at (15,10) {\scriptsize 38};
\node[draw=none, blue] at (19,10) {\scriptsize 34};
\node[draw=none, blue] at (23,10) {\scriptsize 30};
\node[draw=none, blue] at (27,10) {\scriptsize 23};

\node[draw=none, blue] at (30,-8) {\scriptsize 1};
\node[draw=none, blue] at (32,-8) {\scriptsize 5};
\node[draw=none, blue] at (28,-6) {\scriptsize 6};
\node[draw=none, blue] at (31,-6) {\scriptsize 3};
\node[draw=none, blue] at (34,-6) {\scriptsize 14};
\node[draw=none, blue] at (26,-4) {\scriptsize 15};
\node[draw=none, blue] at (29,-4) {\scriptsize 2};
\node[draw=none, blue] at (33,-4) {\scriptsize 4};
\node[draw=none, blue] at (36,-4) {\scriptsize 27};
\node[draw=none, blue] at (24,-2) {\scriptsize 28};
\node[draw=none, blue] at (27,-2) {\scriptsize 7};
\node[draw=none, blue] at (31,-2) {\scriptsize 10};
\node[draw=none, blue] at (35,-2) {\scriptsize 13};
\node[draw=none, blue] at (38,-2) {\scriptsize 44};
\node[draw=none, blue] at (25,0) {\scriptsize 16};
\node[draw=none, blue] at (29,0) {\scriptsize 9};
\node[draw=none, blue] at (33,0) {\scriptsize 11};
\node[draw=none, blue] at (37,0) {\scriptsize 26};

\node[draw=none, blue] at (39,8) {\scriptsize 37};
\node[draw=none, blue] at (37,10) {\scriptsize 38};

\draw[BrickRed, dashed] (-0.5,10) to (40.5,10);
\draw[BrickRed, dashed] (-0.5,0) to (40.5,0);

\draw[BurntOrange, thick, ->] (h_1) to (1_-1);
\draw[BurntOrange, thick, ->] (1_-1) to (3_-1);
\draw[BurntOrange, thick, ->] (3_-1) to (5_-1);
\draw[BurntOrange, thick, ->] (5_-1) to (7_-1);
\draw[BurntOrange, thick, ->] (7_-1) to (9_-1);
\draw[BurntOrange, thick, ->] (9_-1) to (11_-1);
\draw[BurntOrange, thick, ->] (11_-1) to (13_-1);
\draw[BurntOrange, thick, ->] (13_-1) to (15_-1);
\draw[BurntOrange, thick, ->] (15_-1) to (17_-1);
\draw[BurntOrange, thick, ->] (17_-1) to (t_1);

\draw[BurntOrange, thick, ->] (h_2) to (3_-3);
\draw[BurntOrange, thick, ->] (3_-3) to (5_-3);
\draw[BurntOrange, thick, ->] (5_-3) to (7_-3);
\draw[BurntOrange, thick, ->] (7_-3) to (9_-3);
\draw[BurntOrange, thick, ->] (9_-3) to (11_-3);
\draw[BurntOrange, thick, ->] (11_-3) to (13_-3);
\draw[BurntOrange, thick, ->] (13_-3) to (15_-3);
\draw[BurntOrange, thick, ->] (15_-3) to (t_2);

\draw[BurntOrange, thick, ->] (h_3) to (5_-5);
\draw[BurntOrange, thick, ->] (5_-5) to (7_-5);
\draw[BurntOrange, thick, ->] (7_-5) to (9_-5);
\draw[BurntOrange, thick, ->] (9_-5) to (11_-5);
\draw[BurntOrange, thick, ->] (11_-5) to (13_-5);
\draw[BurntOrange, thick, ->] (13_-5) to (t_3);

\draw[BurntOrange, thick, ->] (h_4) to (7_-7);
\draw[BurntOrange, thick, ->] (7_-7) to (9_-7);
\draw[BurntOrange, thick, ->] (9_-7) to (11_-7);
\draw[BurntOrange, thick, ->] (11_-7) to (t_4);

\draw[BurntOrange, thick, ->] (h_5) to (9_-9);
\draw[BurntOrange, thick, ->] (9_-9) to (t_5);

\draw[red, thick, ->] (9_-7) to (9_-9);
\draw[red, thick, ->] (7_-5) to (7_-7);
\draw[red, thick, ->] (11_-5) to (11_-7);
\draw[red, thick, ->] (5_-3) to (5_-5);
\draw[red, thick, ->] (9_-3) to (9_-5);
\draw[red, thick, ->] (13_-3) to (13_-5);
\draw[red, thick, ->] (3_-1) to (3_-3);
\draw[red, thick, ->] (7_-1) to (7_-3);
\draw[red, thick, ->] (11_-1) to (11_-3);
\draw[red, thick, ->] (15_-1) to (15_-3);

\draw[red, thick, ->] (-0.5,1) to (1_1);
\draw[red, thick, ->] (1_1) to (1_-1);

\draw[red, thick, ->] (-0.5,3) to (1_3);
\draw[red, thick, ->] (1_3) to (3_3);
\draw[red, thick, ->] (3_3) to (3_1);
\draw[red, thick, ->] (3_1) to (5_1);
\draw[red, thick, ->] (5_1) to (5_-1);

\draw[red, thick, ->] (-0.5,5) to (1_5);
\draw[red, thick, ->] (1_5) to (3_5);
\draw[red, thick, ->] (3_5) to (5_5);
\draw[red, thick, ->] (5_5) to (5_3);
\draw[red, thick, ->] (5_3) to (7_3);
\draw[red, thick, ->] (7_3) to (7_1);
\draw[red, thick, ->] (7_1) to (9_1);
\draw[red, thick, ->] (9_1) to (9_-1);

\draw[red, thick, ->] (-0.5,7) to (1_7);
\draw[red, thick, ->] (1_7) to (3_7);
\draw[red, thick, ->] (3_7) to (5_7);
\draw[red, thick, ->] (5_7) to (7_7);
\draw[red, thick, ->] (7_7) to (7_5);
\draw[red, thick, ->] (7_5) to (9_5);
\draw[red, thick, ->] (9_5) to (9_3);
\draw[red, thick, ->] (9_3) to (11_3);
\draw[red, thick, ->] (11_3) to (11_1);
\draw[red, thick, ->] (11_1) to (13_1);
\draw[red, thick, ->] (13_1) to (13_-1);

\draw[red, thick, ->] (-0.5,9) to (1_9);
\draw[red, thick, ->] (1_9) to (3_9);
\draw[red, thick, ->] (3_9) to (5_9);
\draw[red, thick, ->] (5_9) to (7_9);
\draw[red, thick, ->] (7_9) to (9_9);
\draw[red, thick, ->] (9_9) to (9_7);
\draw[red, thick, ->] (9_7) to (11_7);
\draw[red, thick, ->] (11_7) to (11_5);
\draw[red, thick, ->] (11_5) to (13_5);
\draw[red, thick, ->] (13_5) to (13_3);
\draw[red, thick, ->] (13_3) to (15_3);
\draw[red, thick, ->] (15_3) to (15_1);
\draw[red, thick, ->] (15_1) to (17_1);
\draw[red, thick, ->] (17_1) to (17_-1);

\draw[OliveGreen, thick, ->] (1_5) to (1_3);
\draw[OliveGreen, thick, ->] (3_7) to (3_5);
\draw[OliveGreen, thick, ->] (1_9) to (1_7);
\draw[OliveGreen, thick, ->] (5_9) to (5_7);
\draw[OliveGreen, thick, ->] (3,10.5) to (3_9);
\draw[OliveGreen, thick, ->] (7,10.5) to (7_9);

\draw[OliveGreen, thick, ->] (3_1) to (1_1);
\draw[OliveGreen, thick, ->] (7_1) to (5_1);
\draw[OliveGreen, thick, ->] (11_1) to (9_1);
\draw[OliveGreen, thick, ->] (15_1) to (13_1);
\draw[OliveGreen, thick, ->] (5_3) to (3_3);
\draw[OliveGreen, thick, ->] (9_3) to (7_3);
\draw[OliveGreen, thick, ->] (13_3) to (11_3);
\draw[OliveGreen, thick, ->] (7_5) to (5_5);
\draw[OliveGreen, thick, ->] (11_5) to (9_5);
\draw[OliveGreen, thick, ->] (9_7) to (7_7);

\draw[OliveGreen, thick, ->] (13,10.5) to (13_9);
\draw[OliveGreen, thick, ->] (13_9) to (9_9);

\draw[OliveGreen, thick, ->] (17,10.5) to (17_9);
\draw[OliveGreen, thick, ->] (17_9) to (15_9);
\draw[OliveGreen, thick, ->] (15_9) to (15_7);
\draw[OliveGreen, thick, ->] (15_7) to (11_7);

\draw[OliveGreen, thick, ->] (21,10.5) to (21_9);
\draw[OliveGreen, thick, ->] (21_9) to (19_9);
\draw[OliveGreen, thick, ->] (19_9) to (19_7);
\draw[OliveGreen, thick, ->] (19_7) to (17_7);
\draw[OliveGreen, thick, ->] (17_7) to (17_5);
\draw[OliveGreen, thick, ->] (17_5) to (13_5);

\draw[OliveGreen, thick, ->] (25,10.5) to (25_9);
\draw[OliveGreen, thick, ->] (25_9) to (23_9);
\draw[OliveGreen, thick, ->] (23_9) to (23_7);
\draw[OliveGreen, thick, ->] (23_7) to (21_7);
\draw[OliveGreen, thick, ->] (21_7) to (21_5);
\draw[OliveGreen, thick, ->] (21_5) to (19_5);
\draw[OliveGreen, thick, ->] (19_5) to (19_3);
\draw[OliveGreen, thick, ->] (19_3) to (15_3);

\draw[OliveGreen, thick, ->] (29,10.5) to (29_9);
\draw[OliveGreen, thick, ->] (29_9) to (27_9);
\draw[OliveGreen, thick, ->] (27_9) to (27_7);
\draw[OliveGreen, thick, ->] (27_7) to (25_7);
\draw[OliveGreen, thick, ->] (25_7) to (25_5);
\draw[OliveGreen, thick, ->] (25_5) to (23_5);
\draw[OliveGreen, thick, ->] (23_5) to (23_3);
\draw[OliveGreen, thick, ->] (23_3) to (21_3);
\draw[OliveGreen, thick, ->] (21_3) to (21_1);
\draw[OliveGreen, thick, ->] (21_1) to (17_1);

\draw[BurntOrange, thick, ->] (15_9) to (13_9);
\draw[BurntOrange, thick, ->] (19_9) to (17_9);
\draw[BurntOrange, thick, ->] (23_9) to (21_9);
\draw[BurntOrange, thick, ->] (27_9) to (25_9);
\draw[BurntOrange, thick, ->] (17_7) to (15_7);
\draw[BurntOrange, thick, ->] (21_7) to (19_7);
\draw[BurntOrange, thick, ->] (25_7) to (23_7);
\draw[BurntOrange, thick, ->] (19_5) to (17_5);
\draw[BurntOrange, thick, ->] (23_5) to (21_5);
\draw[BurntOrange, thick, ->] (21_3) to (19_3);

\draw[BurntOrange, thick, ->] (40.5,1) to (39_1);
\draw[BurntOrange, thick, ->] (39_1) to (37_1);
\draw[BurntOrange, thick, ->] (37_1) to (35_1);
\draw[BurntOrange, thick, ->] (35_1) to (33_1);
\draw[BurntOrange, thick, ->] (33_1) to (31_1);
\draw[BurntOrange, thick, ->] (31_1) to (29_1);
\draw[BurntOrange, thick, ->] (29_1) to (27_1);
\draw[BurntOrange, thick, ->] (27_1) to (25_1);
\draw[BurntOrange, thick, ->] (25_1) to (23_1);
\draw[BurntOrange, thick, ->] (23_1) to (21_1);

\draw[BurntOrange, thick, ->] (40.5,3) to (37_3);
\draw[BurntOrange, thick, ->] (37_3) to (35_3);
\draw[BurntOrange, thick, ->] (35_3) to (33_3);
\draw[BurntOrange, thick, ->] (33_3) to (31_3);
\draw[BurntOrange, thick, ->] (31_3) to (29_3);
\draw[BurntOrange, thick, ->] (29_3) to (27_3);
\draw[BurntOrange, thick, ->] (27_3) to (25_3);
\draw[BurntOrange, thick, ->] (25_3) to (23_3);

\draw[BurntOrange, thick, ->] (40.5,7) to (39_7);
\draw[BurntOrange, thick, ->] (39_7) to (39_5);
\draw[BurntOrange, thick, ->] (39_5) to (35_5);
\draw[BurntOrange, thick, ->] (35_5) to (33_5);
\draw[BurntOrange, thick, ->] (33_5) to (31_5);
\draw[BurntOrange, thick, ->] (31_5) to (29_5);
\draw[BurntOrange, thick, ->] (29_5) to (27_5);
\draw[BurntOrange, thick, ->] (27_5) to (25_5);

\draw[BurntOrange, thick, ->] (39,10.5) to (39_9);
\draw[BurntOrange, thick, ->] (39_9) to (37_9);
\draw[BurntOrange, thick, ->] (37_9) to (37_7);
\draw[BurntOrange, thick, ->] (37_7) to (33_7);
\draw[BurntOrange, thick, ->] (33_7) to (31_7);
\draw[BurntOrange, thick, ->] (31_7) to (29_7);
\draw[BurntOrange, thick, ->] (29_7) to (27_7);
\draw[BurntOrange, thick, ->] (27_7) to (25_7);

\draw[BurntOrange, thick, ->] (35,10.5) to (35_9);
\draw[BurntOrange, thick, ->] (35_9) to (31_9);
\draw[BurntOrange, thick, ->] (31_9) to (29_9);

\draw[red, thick, ->] (40.5,9) to (39_9);
\draw[red, thick, ->] (37_9) to (35_9);
\draw[red, thick, ->] (39_7) to (37_7);
\draw[red, thick, ->] (40.5,5) to (39_5);

\draw[red, thick, <-] (31_9) to (31_7);
\draw[red, thick, <-] (29_7) to (29_5);
\draw[red, thick, <-] (33_7) to (33_5);
\draw[red, thick, <-] (27_5) to (27_3);
\draw[red, thick, <-] (31_5) to (31_3);
\draw[red, thick, <-] (35_5) to (35_3);
\draw[red, thick, <-] (25_3) to (25_1);
\draw[red, thick, <-] (29_3) to (29_1);
\draw[red, thick, <-] (33_3) to (33_1);
\draw[red, thick, <-] (37_3) to (37_1);

\draw[red, thick, ->] (32.5,-10.5) to (31_-9);
\draw[red, thick, ->] (31_-9) to (31_-7);
\draw[red, thick, ->] (31_-7) to (29_-7);
\draw[red, thick, ->] (29_-7) to (29_-5);
\draw[red, thick, ->] (29_-5) to (27_-5);
\draw[red, thick, ->] (27_-5) to (27_-3);
\draw[red, thick, ->] (27_-3) to (25_-3);
\draw[red, thick, ->] (25_-3) to (25_-1);
\draw[red, thick, ->] (25_-1) to (23_-1);
\draw[red, thick, ->] (23_-1) to (23_1);

\draw[red, thick, ->] (34.5,-8.5) to (33_-7);
\draw[red, thick, ->] (33_-7) to (33_-5);
\draw[red, thick, ->] (33_-5) to (31_-5);
\draw[red, thick, ->] (31_-5) to (31_-3);
\draw[red, thick, ->] (31_-3) to (29_-3);
\draw[red, thick, ->] (29_-3) to (29_-1);
\draw[red, thick, ->] (29_-1) to (27_-1);
\draw[red, thick, ->] (27_-1) to (27_1);

\draw[red, thick, ->] (36.5,-6.5) to (35_-5);
\draw[red, thick, ->] (35_-5) to (35_-3);
\draw[red, thick, ->] (35_-3) to (33_-3);
\draw[red, thick, ->] (33_-3) to (33_-1);
\draw[red, thick, ->] (33_-1) to (31_-1);
\draw[red, thick, ->] (31_-1) to (31_1);

\draw[red, thick, ->] (38.5,-4.5) to (37_-3);
\draw[red, thick, ->] (37_-3) to (37_-1);
\draw[red, thick, ->] (37_-1) to (35_-1);
\draw[red, thick, ->] (35_-1) to (35_1);

\draw[red, thick, ->] (40.5,-2.5) to (39_-1);
\draw[red, thick, ->] (39_-1) to (39_1);

\draw[OliveGreen, thick, ->] (21.5,-2.5) to (23_-1);
\draw[OliveGreen, thick, ->] (25_-1) to (27_-1);
\draw[OliveGreen, thick, ->] (29_-1) to (31_-1);
\draw[OliveGreen, thick, ->] (33_-1) to (35_-1);
\draw[OliveGreen, thick, ->] (37_-1) to (39_-1);
\draw[OliveGreen, thick, ->] (23.5,-4.5) to (25_-3);
\draw[OliveGreen, thick, ->] (27_-3) to (29_-3);
\draw[OliveGreen, thick, ->] (31_-3) to (33_-3);
\draw[OliveGreen, thick, ->] (35_-3) to (37_-3);
\draw[OliveGreen, thick, ->] (25.5,-6.5) to (27_-5);
\draw[OliveGreen, thick, ->] (29_-5) to (31_-5);
\draw[OliveGreen, thick, ->] (33_-5) to (35_-5);
\draw[OliveGreen, thick, ->] (27.5,-8.5) to (29_-7);
\draw[OliveGreen, thick, ->] (31_-7) to (33_-7);
\draw[OliveGreen, thick, ->] (29.5,-10.5) to (31_-9);

\end{tikzpicture}
\caption{Network $\Nc_4^{\mathrm{full}_+}$.}
\label{net-full+}
\end{figure}

\vspace{-.1cm}
\begin{figure}[h]
\begin{tikzpicture}[every node/.style={inner sep=0, minimum size=0.2cm, draw, circle, thick}, x=0.35cm, y=0.35cm]

\foreach \i in {1,5,9,13,17,23,27,31,35,39}
{
	\node[fill=white] (\i_1) at (\i,1) {};
	\node[fill=white] (\i_-1) at (\i,-1) {};
}
\foreach \i in {3,7,11,15,25,29,33,37}
{
	\node[fill=white] (\i_3) at (\i,3) {};
	\node[fill=white] (\i_-3) at (\i,-3) {};
	\node[fill=black] (\i_1) at (\i,1) {};
	\node[fill=black] (\i_-1) at (\i,-1) {};
}
\foreach \i in {5,9,13,27,31,35}
{
	\node[fill=white] (\i_5) at (\i,5) {};
	\node[fill=white] (\i_-5) at (\i,-5) {};
	\node[fill=black] (\i_3) at (\i,3) {};
	\node[fill=black] (\i_-3) at (\i,-3) {};
}

\foreach \i in {7,11,29,33}
{
	\node[fill=white] (\i_7) at (\i,7) {};
	\node[fill=white] (\i_-7) at (\i,-7) {};
	\node[fill=black] (\i_5) at (\i,5) {};
	\node[fill=black] (\i_-5) at (\i,-5) {};
}

\foreach \i in {9,31}
{
	\node[fill=white] (\i_9) at (\i,9) {};
	\node[fill=white] (\i_-9) at (\i,-9) {};
	\node[fill=black] (\i_7) at (\i,7) {};
	\node[fill=black] (\i_-7) at (\i,-7) {};
}

\foreach \i in {7,29}
{
	\node[fill=white] (\i_9) at (\i,9) {};
}

\foreach \i in {5,27}
{
	\node[fill=black] (\i_9) at (\i,9) {};
	\node[fill=white] (\i_7) at (\i,7) {};
}

\foreach \i in {3,25}
{
	\node[fill=white] (\i_9) at (\i,9) {};
	\node[fill=black] (\i_7) at (\i,7) {};
	\node[fill=white] (\i_5) at (\i,5) {};
}

\foreach \i in {1,23}
{
	\node[fill=black] (\i_9) at (\i,9) {};
	\node[fill=white] (\i_7) at (\i,7) {};
	\node[fill=black] (\i_5) at (\i,5) {};
	\node[fill=white] (\i_3) at (\i,3) {};
}

	\node[fill=white] (21_9) at (21,9) {};
	\node[fill=black] (21_7) at (21,7) {};
	\node[fill=white] (21_5) at (21,5) {};
	\node[fill=black] (21_3) at (21,3) {};
	\node[fill=white] (21_1) at (21,1) {};

	\node[fill=black] (19_9) at (19,9) {};
	\node[fill=white] (19_7) at (19,7) {};
	\node[fill=black] (19_5) at (19,5) {};
	\node[fill=white] (19_3) at (19,3) {};

\foreach \i in {17,39}
{
	\node[fill=white] (\i_9) at (\i,9) {};
	\node[fill=black] (\i_7) at (\i,7) {};
	\node[fill=white] (\i_5) at (\i,5) {};
}

\foreach \i in {15,37}
{
	\node[fill=black] (\i_9) at (\i,9) {};
	\node[fill=white] (\i_7) at (\i,7) {};
}

\foreach \i in {13,35}
	\node[fill=white] (\i_9) at (\i,9) {};

\node[minimum size=0.4cm] (h_1) at (10.5,-10.5) {\scriptsize $h_1$};
\node[minimum size=0.4cm] (h_2) at (12.5,-8.5) {\scriptsize $h_2$};
\node[minimum size=0.4cm] (h_3) at (14.5,-6.5) {\scriptsize $h_3$};
\node[minimum size=0.4cm] (h_4) at (16.5,-4.5) {\scriptsize $h_4$};
\node[minimum size=0.4cm] (h_5) at (18.5,-2.5) {\scriptsize $h_5$};

\node[minimum size=0.4cm] (t_1) at (29.5,-10.5) {\scriptsize $t_1$};
\node[minimum size=0.4cm] (t_2) at (27.5,-8.5) {\scriptsize $t_2$};
\node[minimum size=0.4cm] (t_3) at (25.5,-6.5) {\scriptsize $t_3$};
\node[minimum size=0.4cm] (t_4) at (23.5,-4.5) {\scriptsize $t_4$};
\node[minimum size=0.4cm] (t_5) at (21.5,-2.5) {\scriptsize $t_5$};

\node[draw=none, blue] at (1,2) {\scriptsize 8};
\node[draw=none, blue] at (5,2) {\scriptsize 21};
\node[draw=none, blue] at (9,2) {\scriptsize 12};
\node[draw=none, blue] at (13,2) {\scriptsize 43};
\node[draw=none, blue] at (17,2) {\scriptsize 45};
\node[draw=none, blue] at (3,4) {\scriptsize 22};
\node[draw=none, blue] at (7,4) {\scriptsize 25};
\node[draw=none, blue] at (11,4) {\scriptsize 42};
\node[draw=none, blue] at (15,4) {\scriptsize 46};
\node[draw=none, blue] at (1,6) {\scriptsize 29};
\node[draw=none, blue] at (5,6) {\scriptsize 24};
\node[draw=none, blue] at (9,6) {\scriptsize 41};
\node[draw=none, blue] at (13,6) {\scriptsize 47};
\node[draw=none, blue] at (3,8) {\scriptsize 20};
\node[draw=none, blue] at (7,8) {\scriptsize 40};
\node[draw=none, blue] at (11,8) {\scriptsize 48};
\node[draw=none, blue] at (1,10) {\scriptsize 30};
\node[draw=none, blue] at (5,10) {\scriptsize 23};

\node[draw=none, blue] at (8,-8) {\scriptsize 1};
\node[draw=none, blue] at (10,-8) {\scriptsize 5};
\node[draw=none, blue] at (6,-6) {\scriptsize 6};
\node[draw=none, blue] at (9,-6) {\scriptsize 3};
\node[draw=none, blue] at (12,-6) {\scriptsize 14};
\node[draw=none, blue] at (4,-4) {\scriptsize 15};
\node[draw=none, blue] at (7,-4) {\scriptsize 2};
\node[draw=none, blue] at (11,-4) {\scriptsize 4};
\node[draw=none, blue] at (14,-4) {\scriptsize 27};
\node[draw=none, blue] at (2,-2) {\scriptsize 28};
\node[draw=none, blue] at (5,-2) {\scriptsize 7};
\node[draw=none, blue] at (9,-2) {\scriptsize 10};
\node[draw=none, blue] at (13,-2) {\scriptsize 13};
\node[draw=none, blue] at (16,-2) {\scriptsize 44};
\node[draw=none, blue] at (3,0) {\scriptsize 16};
\node[draw=none, blue] at (7,0) {\scriptsize 9};
\node[draw=none, blue] at (11,0) {\scriptsize 11};
\node[draw=none, blue] at (15,0) {\scriptsize 26};

\node[draw=none, blue] at (23,2) {\scriptsize 8};
\node[draw=none, blue] at (27,2) {\scriptsize 21};
\node[draw=none, blue] at (31,2) {\scriptsize 12};
\node[draw=none, blue] at (35,2) {\scriptsize 43};
\node[draw=none, blue] at (39,2) {\scriptsize 45};
\node[draw=none, blue] at (21,4) {\scriptsize 49};
\node[draw=none, blue] at (25,4) {\scriptsize 22};
\node[draw=none, blue] at (29,4) {\scriptsize 25};
\node[draw=none, blue] at (33,4) {\scriptsize 42};
\node[draw=none, blue] at (37,4) {\scriptsize 46};
\node[draw=none, blue] at (19,6) {\scriptsize 35};
\node[draw=none, blue] at (23,6) {\scriptsize 29};
\node[draw=none, blue] at (27,6) {\scriptsize 24};
\node[draw=none, blue] at (31,6) {\scriptsize 41};
\node[draw=none, blue] at (35,6) {\scriptsize 47};
\node[draw=none, blue] at (17,8) {\scriptsize 37};
\node[draw=none, blue] at (21,8) {\scriptsize 50};
\node[draw=none, blue] at (25,8) {\scriptsize 20};
\node[draw=none, blue] at (29,8) {\scriptsize 40};
\node[draw=none, blue] at (33,8) {\scriptsize 48};
\node[draw=none, blue] at (15,10) {\scriptsize 38};
\node[draw=none, blue] at (19,10) {\scriptsize 34};
\node[draw=none, blue] at (23,10) {\scriptsize 30};
\node[draw=none, blue] at (27,10) {\scriptsize 23};

\node[draw=none, blue] at (30,-8) {\scriptsize 1};
\node[draw=none, blue] at (32,-8) {\scriptsize 5};
\node[draw=none, blue] at (28,-6) {\scriptsize 6};
\node[draw=none, blue] at (31,-6) {\scriptsize 3};
\node[draw=none, blue] at (34,-6) {\scriptsize 14};
\node[draw=none, blue] at (26,-4) {\scriptsize 15};
\node[draw=none, blue] at (29,-4) {\scriptsize 2};
\node[draw=none, blue] at (33,-4) {\scriptsize 4};
\node[draw=none, blue] at (36,-4) {\scriptsize 27};
\node[draw=none, blue] at (24,-2) {\scriptsize 28};
\node[draw=none, blue] at (27,-2) {\scriptsize 7};
\node[draw=none, blue] at (31,-2) {\scriptsize 10};
\node[draw=none, blue] at (35,-2) {\scriptsize 13};
\node[draw=none, blue] at (38,-2) {\scriptsize 44};
\node[draw=none, blue] at (25,0) {\scriptsize 16};
\node[draw=none, blue] at (29,0) {\scriptsize 9};
\node[draw=none, blue] at (33,0) {\scriptsize 11};
\node[draw=none, blue] at (37,0) {\scriptsize 26};

\node[draw=none, blue] at (39,8) {\scriptsize 37};
\node[draw=none, blue] at (37,10) {\scriptsize 38};

\draw[BrickRed, dashed] (-0.5,10) to (40.5,10);
\draw[BrickRed, dashed] (-0.5,0) to (40.5,0);

\draw[OliveGreen, thick, ->] (h_1) to (9_-9);
\draw[OliveGreen, thick, ->] (9_-9) to (9_-7);
\draw[OliveGreen, thick, ->] (9_-7) to (7_-7);
\draw[OliveGreen, thick, ->] (7_-7) to (7_-5);
\draw[OliveGreen, thick, ->] (7_-5) to (5_-5);
\draw[OliveGreen, thick, ->] (5_-5) to (5_-3);
\draw[OliveGreen, thick, ->] (5_-3) to (3_-3);
\draw[OliveGreen, thick, ->] (3_-3) to (3_-1);
\draw[OliveGreen, thick, ->] (3_-1) to (1_-1);
\draw[OliveGreen, thick, ->] (1_-1) to (1_1);
\draw[OliveGreen, thick, ->] (1_1) to (3_1);
\draw[OliveGreen, thick, ->] (3_1) to (3_3);
\draw[OliveGreen, thick, ->] (3_3) to (5_3);
\draw[OliveGreen, thick, ->] (5_3) to (5_5);
\draw[OliveGreen, thick, ->] (5_5) to (7_5);
\draw[OliveGreen, thick, ->] (7_5) to (7_7);
\draw[OliveGreen, thick, ->] (7_7) to (9_7);
\draw[OliveGreen, thick, ->] (9_7) to (9_9);
\draw[OliveGreen, thick, ->] (9_9) to (13_9);
\draw[OliveGreen, thick, ->] (13_9) to (15_9);
\draw[OliveGreen, thick, ->] (15_9) to (17_9);
\draw[OliveGreen, thick, ->] (17_9) to (19_9);
\draw[OliveGreen, thick, ->] (19_9) to (21_9);
\draw[OliveGreen, thick, ->] (21_9) to (23_9);
\draw[OliveGreen, thick, ->] (23_9) to (25_9);
\draw[OliveGreen, thick, ->] (25_9) to (27_9);
\draw[OliveGreen, thick, ->] (27_9) to (29_9);
\draw[OliveGreen, thick, ->] (29_9) to (31_9);
\draw[OliveGreen, thick, ->] (31_9) to (31_7);
\draw[OliveGreen, thick, ->] (31_7) to (33_7);
\draw[OliveGreen, thick, ->] (33_7) to (33_5);
\draw[OliveGreen, thick, ->] (33_5) to (35_5);
\draw[OliveGreen, thick, ->] (35_5) to (35_3);
\draw[OliveGreen, thick, ->] (35_3) to (37_3);
\draw[OliveGreen, thick, ->] (37_3) to (37_1);
\draw[OliveGreen, thick, ->] (37_1) to (39_1);
\draw[OliveGreen, thick, ->] (39_1) to (39_-1);
\draw[OliveGreen, thick, ->] (39_-1) to (37_-1);
\draw[OliveGreen, thick, ->] (37_-1) to (37_-3);
\draw[OliveGreen, thick, ->] (37_-3) to (35_-3);
\draw[OliveGreen, thick, ->] (35_-3) to (35_-5);
\draw[OliveGreen, thick, ->] (35_-5) to (33_-5);
\draw[OliveGreen, thick, ->] (33_-5) to (33_-7);
\draw[OliveGreen, thick, ->] (33_-7) to (31_-7);
\draw[OliveGreen, thick, ->] (31_-7) to (31_-9);
\draw[OliveGreen, thick, ->] (31_-9) to (t_1);

\draw[OliveGreen, thick, ->] (h_2) to (11_-7);
\draw[OliveGreen, thick, ->] (11_-7) to (11_-5);
\draw[OliveGreen, thick, ->] (11_-5) to (9_-5);
\draw[OliveGreen, thick, ->] (9_-5) to (9_-3);
\draw[OliveGreen, thick, ->] (9_-3) to (7_-3);
\draw[OliveGreen, thick, ->] (7_-3) to (7_-1);
\draw[OliveGreen, thick, ->] (7_-1) to (5_-1);
\draw[OliveGreen, thick, ->] (5_-1) to (5_1);
\draw[OliveGreen, thick, ->] (5_1) to (7_1);
\draw[OliveGreen, thick, ->] (7_1) to (7_3);
\draw[OliveGreen, thick, ->] (7_3) to (9_3);
\draw[OliveGreen, thick, ->] (9_3) to (9_5);
\draw[OliveGreen, thick, ->] (9_5) to (11_5);
\draw[OliveGreen, thick, ->] (11_5) to (11_7);
\draw[OliveGreen, thick, ->] (11_7) to (15_7);
\draw[OliveGreen, thick, ->] (15_7) to (17_7);
\draw[OliveGreen, thick, ->] (17_7) to (19_7);
\draw[OliveGreen, thick, ->] (19_7) to (21_7);
\draw[OliveGreen, thick, ->] (21_7) to (23_7);
\draw[OliveGreen, thick, ->] (23_7) to (25_7);
\draw[OliveGreen, thick, ->] (25_7) to (27_7);
\draw[OliveGreen, thick, ->] (27_7) to (29_7);
\draw[OliveGreen, thick, ->] (29_7) to (29_5);
\draw[OliveGreen, thick, ->] (29_5) to (31_5);
\draw[OliveGreen, thick, ->] (31_5) to (31_3);
\draw[OliveGreen, thick, ->] (31_3) to (33_3);
\draw[OliveGreen, thick, ->] (33_3) to (33_1);
\draw[OliveGreen, thick, ->] (33_1) to (35_1);
\draw[OliveGreen, thick, ->] (35_1) to (35_-1);
\draw[OliveGreen, thick, ->] (35_-1) to (33_-1);
\draw[OliveGreen, thick, ->] (33_-1) to (33_-3);
\draw[OliveGreen, thick, ->] (33_-3) to (31_-3);
\draw[OliveGreen, thick, ->] (31_-3) to (31_-5);
\draw[OliveGreen, thick, ->] (31_-5) to (29_-5);
\draw[OliveGreen, thick, ->] (29_-5) to (29_-7);
\draw[OliveGreen, thick, ->] (29_-7) to (t_2);

\draw[OliveGreen, thick, ->] (h_3) to (13_-5);
\draw[OliveGreen, thick, ->] (13_-5) to (13_-3);
\draw[OliveGreen, thick, ->] (13_-3) to (11_-3);
\draw[OliveGreen, thick, ->] (11_-3) to (11_-1);
\draw[OliveGreen, thick, ->] (11_-1) to (9_-1);
\draw[OliveGreen, thick, ->] (9_-1) to (9_1);
\draw[OliveGreen, thick, ->] (9_1) to (11_1);
\draw[OliveGreen, thick, ->] (11_1) to (11_3);
\draw[OliveGreen, thick, ->] (11_3) to (13_3);
\draw[OliveGreen, thick, ->] (13_3) to (13_5);
\draw[OliveGreen, thick, ->] (13_5) to (17_5);
\draw[OliveGreen, thick, ->] (17_5) to (19_5);
\draw[OliveGreen, thick, ->] (19_5) to (21_5);
\draw[OliveGreen, thick, ->] (21_5) to (23_5);
\draw[OliveGreen, thick, ->] (23_5) to (25_5);
\draw[OliveGreen, thick, ->] (25_5) to (27_5);
\draw[OliveGreen, thick, ->] (27_5) to (27_3);
\draw[OliveGreen, thick, ->] (27_3) to (29_3);
\draw[OliveGreen, thick, ->] (29_3) to (29_1);
\draw[OliveGreen, thick, ->] (29_1) to (31_1);
\draw[OliveGreen, thick, ->] (31_1) to (31_-1);
\draw[OliveGreen, thick, ->] (31_-1) to (29_-1);
\draw[OliveGreen, thick, ->] (29_-1) to (29_-3);
\draw[OliveGreen, thick, ->] (29_-3) to (27_-3);
\draw[OliveGreen, thick, ->] (27_-3) to (27_-5);
\draw[OliveGreen, thick, ->] (27_-5) to (t_3);

\draw[OliveGreen, thick, ->] (h_4) to (15_-3);
\draw[OliveGreen, thick, ->] (15_-3) to (15_-1);
\draw[OliveGreen, thick, ->] (15_-1) to (13_-1);
\draw[OliveGreen, thick, ->] (13_-1) to (13_1);
\draw[OliveGreen, thick, ->] (13_1) to (15_1);
\draw[OliveGreen, thick, ->] (15_1) to (15_3);
\draw[OliveGreen, thick, ->] (15_3) to (19_3);
\draw[OliveGreen, thick, ->] (19_3) to (21_3);
\draw[OliveGreen, thick, ->] (21_3) to (23_3);
\draw[OliveGreen, thick, ->] (23_3) to (25_3);
\draw[OliveGreen, thick, ->] (25_3) to (25_1);
\draw[OliveGreen, thick, ->] (25_1) to (27_1);
\draw[OliveGreen, thick, ->] (27_1) to (27_-1);
\draw[OliveGreen, thick, ->] (27_-1) to (25_-1);
\draw[OliveGreen, thick, ->] (25_-1) to (25_-3);
\draw[OliveGreen, thick, ->] (25_-3) to (t_4);

\draw[OliveGreen, thick, ->] (h_5) to (17_-1);
\draw[OliveGreen, thick, ->] (17_-1) to (17_1);
\draw[OliveGreen, thick, ->] (17_1) to (21_1);
\draw[OliveGreen, thick, ->] (21_1) to (23_1);
\draw[OliveGreen, thick, ->] (23_1) to (23_-1);
\draw[OliveGreen, thick, ->] (23_-1) to (t_5);

\draw[BurntOrange, thick, ->] (-0.5,-2.5) to (1_-1);
\draw[BurntOrange, thick, ->] (3_-1) to (5_-1);
\draw[BurntOrange, thick, ->] (7_-1) to (9_-1);
\draw[BurntOrange, thick, ->] (11_-1) to (13_-1);
\draw[BurntOrange, thick, ->] (15_-1) to (17_-1);
\draw[BurntOrange, thick, ->] (1.5,-4.5) to (3_-3);
\draw[BurntOrange, thick, ->] (5_-3) to (7_-3);
\draw[BurntOrange, thick, ->] (9_-3) to (11_-3);
\draw[BurntOrange, thick, ->] (13_-3) to (15_-3);
\draw[BurntOrange, thick, ->] (3.5,-6.5) to (5_-5);
\draw[BurntOrange, thick, ->] (7_-5) to (9_-5);
\draw[BurntOrange, thick, ->] (11_-5) to (13_-5);
\draw[BurntOrange, thick, ->] (5.5,-8.5) to (7_-7);
\draw[BurntOrange, thick, ->] (9_-7) to (11_-7);
\draw[BurntOrange, thick, ->] (7.5,-10.5) to (9_-9);

\draw[BurntOrange, thick, ->] (-0.5,1) to (1_1);

\draw[BurntOrange, thick, ->] (-0.5,3) to (1_3);
\draw[BurntOrange, thick, ->] (1_3) to (3_3);
\draw[BurntOrange, thick, ->] (3_1) to (5_1);

\draw[BurntOrange, thick, ->] (-0.5,5) to (1_5);
\draw[BurntOrange, thick, ->] (1_5) to (3_5);
\draw[BurntOrange, thick, ->] (3_5) to (5_5);
\draw[BurntOrange, thick, ->] (5_3) to (7_3);
\draw[BurntOrange, thick, ->] (7_1) to (9_1);

\draw[BurntOrange, thick, ->] (-0.5,7) to (1_7);
\draw[BurntOrange, thick, ->] (1_7) to (3_7);
\draw[BurntOrange, thick, ->] (3_7) to (5_7);
\draw[BurntOrange, thick, ->] (5_7) to (7_7);
\draw[BurntOrange, thick, ->] (7_5) to (9_5);
\draw[BurntOrange, thick, ->] (9_3) to (11_3);
\draw[BurntOrange, thick, ->] (11_1) to (13_1);

\draw[BurntOrange, thick, ->] (-0.5,9) to (1_9);
\draw[BurntOrange, thick, ->] (1_9) to (3_9);
\draw[BurntOrange, thick, ->] (3_9) to (5_9);
\draw[BurntOrange, thick, ->] (5_9) to (7_9);
\draw[BurntOrange, thick, ->] (7_9) to (9_9);
\draw[BurntOrange, thick, ->] (9_7) to (11_7);
\draw[BurntOrange, thick, ->] (11_5) to (13_5);
\draw[BurntOrange, thick, ->] (13_3) to (15_3);
\draw[BurntOrange, thick, ->] (15_1) to (17_1);

\draw[red, thick, ->] (1_5) to (1_3);
\draw[red, thick, ->] (3_7) to (3_5);
\draw[red, thick, ->] (1_9) to (1_7);
\draw[red, thick, ->] (5_9) to (5_7);
\draw[red, thick, ->] (3,10.5) to (3_9);
\draw[red, thick, ->] (7,10.5) to (7_9);

\draw[BurntOrange, thick, ->] (13,10.5) to (13_9);

\draw[BurntOrange, thick, ->] (17,10.5) to (17_9);
\draw[BurntOrange, thick, ->] (15_9) to (15_7);

\draw[BurntOrange, thick, ->] (21,10.5) to (21_9);
\draw[BurntOrange, thick, ->] (19_9) to (19_7);
\draw[BurntOrange, thick, ->] (17_7) to (17_5);

\draw[BurntOrange, thick, ->] (25,10.5) to (25_9);
\draw[BurntOrange, thick, ->] (23_9) to (23_7);
\draw[BurntOrange, thick, ->] (21_7) to (21_5);
\draw[BurntOrange, thick, ->] (19_5) to (19_3);

\draw[BurntOrange, thick, ->] (29,10.5) to (29_9);
\draw[BurntOrange, thick, ->] (27_9) to (27_7);
\draw[BurntOrange, thick, ->] (25_7) to (25_5);
\draw[BurntOrange, thick, ->] (23_5) to (23_3);
\draw[BurntOrange, thick, ->] (21_3) to (21_1);

\draw[BurntOrange, thick, ->] (40.5,1) to (39_1);
\draw[BurntOrange, thick, ->] (37_1) to (35_1);
\draw[BurntOrange, thick, ->] (33_1) to (31_1);
\draw[BurntOrange, thick, ->] (29_1) to (27_1);
\draw[BurntOrange, thick, ->] (25_1) to (23_1);

\draw[BurntOrange, thick, ->] (40.5,3) to (37_3);
\draw[BurntOrange, thick, ->] (35_3) to (33_3);
\draw[BurntOrange, thick, ->] (31_3) to (29_3);
\draw[BurntOrange, thick, ->] (27_3) to (25_3);

\draw[BurntOrange, thick, ->] (40.5,7) to (39_7);
\draw[BurntOrange, thick, ->] (39_7) to (39_5);
\draw[BurntOrange, thick, ->] (39_5) to (35_5);
\draw[BurntOrange, thick, ->] (33_5) to (31_5);
\draw[BurntOrange, thick, ->] (29_5) to (27_5);

\draw[BurntOrange, thick, ->] (39,10.5) to (39_9);
\draw[BurntOrange, thick, ->] (39_9) to (37_9);
\draw[BurntOrange, thick, ->] (37_9) to (37_7);
\draw[BurntOrange, thick, ->] (37_7) to (33_7);
\draw[BurntOrange, thick, ->] (31_7) to (29_7);

\draw[BurntOrange, thick, ->] (35,10.5) to (35_9);
\draw[BurntOrange, thick, ->] (35_9) to (31_9);

\draw[red, thick, ->] (40.5,9) to (39_9);
\draw[red, thick, ->] (37_9) to (35_9);
\draw[red, thick, ->] (39_7) to (37_7);
\draw[red, thick, ->] (40.5,5) to (39_5);

\draw[BurntOrange, thick, ->] (32.5,-10.5) to (31_-9);
\draw[BurntOrange, thick, ->] (31_-7) to (29_-7);
\draw[BurntOrange, thick, ->] (29_-5) to (27_-5);
\draw[BurntOrange, thick, ->] (27_-3) to (25_-3);
\draw[BurntOrange, thick, ->] (25_-1) to (23_-1);

\draw[BurntOrange, thick, ->] (34.5,-8.5) to (33_-7);
\draw[BurntOrange, thick, ->] (33_-5) to (31_-5);
\draw[BurntOrange, thick, ->] (31_-3) to (29_-3);
\draw[BurntOrange, thick, ->] (29_-1) to (27_-1);

\draw[BurntOrange, thick, ->] (36.5,-6.5) to (35_-5);
\draw[BurntOrange, thick, ->] (35_-3) to (33_-3);
\draw[BurntOrange, thick, ->] (33_-1) to (31_-1);

\draw[BurntOrange, thick, ->] (38.5,-4.5) to (37_-3);
\draw[BurntOrange, thick, ->] (37_-1) to (35_-1);

\draw[BurntOrange, thick, ->] (40.5,-2.5) to (39_-1);

\end{tikzpicture}
\caption{Network $\Nc_4^{\mathrm{full}_-}$.}
\label{net-full-}
\end{figure}

\begin{figure}[h]
\begin{tikzpicture}[every node/.style={inner sep=0, minimum size=0.4cm, thick, draw, circle}, x=0.75cm, y=0.45cm]

\node [rectangle] (1) at (-1,0) {\scriptsize 1};
\node (3) at (0,2) {\scriptsize 3};
\node [rectangle] (5) at (1,0) {\scriptsize 5};
\node [rectangle] (6) at (-2,2) {\scriptsize 6};
\node (2) at (-1,4) {\scriptsize 2};
\node (10) at (0,6) {\scriptsize 10};
\node (4) at (1,4) {\scriptsize 4};
\node [rectangle] (14) at (2,2) {\scriptsize 14};
\node [rectangle] (15) at (-3,4) {\scriptsize 15};
\node (7) at (-2,6) {\scriptsize 7};
\node (9) at (-1,8) {\scriptsize 9};
\node (11) at (1,8) {\scriptsize 11};
\node (13) at (2,6) {\scriptsize 13};
\node [rectangle] (27) at (3,4) {\scriptsize 27};
\node [rectangle] (28) at (-4,6) {\scriptsize 28};
\node (16) at (-3,8) {\scriptsize 16};
\node (26) at (3,8) {\scriptsize 26};
\node [rectangle] (44) at (4,6) {\scriptsize 44};

\node (8) at (-4,10) {\scriptsize 8};
\node (21) at (-2,10) {\scriptsize 21};
\node (12) at (0,10) {\scriptsize 12};
\node (43) at (2,10) {\scriptsize 43};
\node (45) at (4,10) {\scriptsize 45};

\node (49) at (-5,12) {\scriptsize 49};
\node (22) at (-3,12) {\scriptsize 22};
\node (25) at (-1,12) {\scriptsize 25};
\node (42) at (1,12) {\scriptsize 42};
\node (46) at (3,12) {\scriptsize 46};

\node (35) at (-6,14) {\scriptsize 35};
\node (29) at (-4,14) {\scriptsize 29};
\node (24) at (-2,14) {\scriptsize 24};
\node (41) at (0,14) {\scriptsize 41};
\node (47) at (2,14) {\scriptsize 47};

\node (37) at (-7,16) {\scriptsize 37};
\node (50) at (-5,16) {\scriptsize 50};
\node (20) at (-3,16) {\scriptsize 20};
\node (40) at (-1,16) {\scriptsize 40};
\node (48) at (1,16) {\scriptsize 48};

\node (38) at (-8,18) {\scriptsize 38};
\node (34) at (-6,18) {\scriptsize 34};
\node (30) at (-4,18) {\scriptsize 30};
\node (23) at (-2,18) {\scriptsize 23};

\node (51) at (-5,20) {\scriptsize 51};
\node (19) at (-7,20) {\scriptsize 19};
\node (33) at (-6,22) {\scriptsize 33};
\node (32) at (-5,24) {\scriptsize 32};
\node (31) at (-4,22) {\scriptsize 31};
\node (17) at (-3,20) {\scriptsize 17};
\node (39) at (-9,20) {\scriptsize 39};
\node (36) at (-8,22) {\scriptsize 36};
\node (18) at (-7,24) {\scriptsize 18};
\node (52) at (-6,26) {\scriptsize 52};

\draw [->,thick] (19) -- (33);

\draw [->,thick] (33) -- (32);
\draw [->,thick] (32) -- (31);

\draw [->,thick] (31) -- (17);

\draw [->,thick] (51) -- (31);
\draw [->,thick] (31) to [out = 140, in = -30] (18);
\draw [->,thick] (18) to (33);
\draw [->,thick] (33) -- (51);

\draw [->,thick] (17) to [out = 150, in =-20] (36);
\draw [->,thick] (36) to (19);

\draw [->,thick] (31) -- (33);
\draw [->,thick] (17) -- (51);
\draw [->,thick] (51) -- (19);
\draw [->,thick] (39) to [bend left = 15] (17);
\draw [->,thick] (36) to [bend left = 20] (31);
\draw [->,thick] (33) to (36);
\draw [->,thick] (19) to (39);

\draw[BrickRed, dashed] (-9,18) to (-1,18);

\draw[->, thick] (23) to [out=160, in=-20] (39);
\draw[->, thick] (39) to (38);
\draw[->, thick] (38) to (19);
\draw[->, thick] (19) to (34);
\draw[->, thick] (34) to (51);
\draw[->, thick] (51) to (30);
\draw[->, thick] (30) to (17);
\draw[->, thick] (17) to (23);

\draw[->, thick] (40) to (23);
\draw[->, thick] (23) to (20);
\draw[->, thick] (20) to (30);
\draw[->, thick] (30) to (50);
\draw[->, thick] (50) to (34);
\draw[->, thick] (34) to (37);
\draw[->, thick] (37) to (38);
\draw[->, thick] (38) to [out=-15, in=160] (48);

\draw [<-, thick] (8) to (21);
\draw [<-, thick] (21) to (12);
\draw [<-, thick] (12) to (43);
\draw [<-, thick] (43) to (45);
\draw [->, thick] (45) to [bend right = 11] (8);
\draw [<-, thick] (22) to (25);
\draw [<-, thick] (25) to (42);
\draw [<-, thick] (42) to (46);
\draw [->, thick] (46) to [bend right = 11] (49);
\draw [->, thick] (49) to (22);
\draw [<-, thick] (24) to (41);
\draw [<-, thick] (41) to (47);
\draw [->, thick] (47) to [bend right = 11] (35);
\draw [->, thick] (35) to (29);
\draw [->, thick] (29) to (24);
\draw [<-, thick] (40) to (48);
\draw [->, thick] (48) to [bend right = 11] (37);
\draw [->, thick] (37) to (50);
\draw [->, thick] (50) to (20);
\draw [->, thick] (20) to (40);

\draw [->, thick] (21) to (25);
\draw [->, thick] (25) to (41);
\draw [->, thick] (41) to (48);
\draw [->, thick] (12) to (42);
\draw [->, thick] (42) to (47);
\draw [->, thick] (37) to [out=-15, in=160] (47);
\draw [->, thick] (43) to (46);
\draw [->, thick] (35) to [out=-15, in=160] (46);
\draw [<-, thick] (35) to (50);
\draw [->, thick] (49) to [out=-15, in=160] (45);
\draw [<-, thick] (49) to (29);
\draw [<-, thick] (29) to (20);

\draw [->, thick] (22) to (21);
\draw [->, thick] (24) to (25);
\draw [->, thick] (25) to (12);
\draw [->, thick] (40) to (41);
\draw [->, thick] (41) to (42);
\draw [->, thick] (42) to (43);
\draw [->, thick] (49) to (35);
\draw [->, thick] (35) to (37);
\draw [->, thick] (29) to (50);

\draw [->, thick] (8) to (49);
\draw [->, thick] (22) to (29);
\draw [->, thick] (24) to (20);

\draw [->, thick] (8) to (16);
\draw [->, thick] (16) to (21);
\draw [->, thick] (21) to (9);
\draw [->, thick] (9) to (12);
\draw [->, thick] (12) to (11);
\draw [->, thick] (11) to (43);
\draw [->, thick] (43) to (26);
\draw [->, thick] (26) to (45);

\draw[BrickRed, dashed] (-4,8) to (4,8);

\draw [->,thick] (1) -- (5);

\draw [->,thick] (6) -- (3);
\draw [->,thick] (3) -- (14);

\draw [->,thick] (15) -- (2);
\draw [->,thick] (2) -- (4);
\draw [->,thick] (4) -- (27);

\draw [->,thick] (28) -- (7);
\draw [->,thick] (7) -- (10);
\draw [->,thick] (10) -- (13);
\draw [->,thick] (13) -- (44);

\draw [->,thick] (5) -- (3);
\draw [->,thick] (3) -- (2);
\draw [->,thick] (2) -- (7);
\draw [->,thick] (7) -- (16);

\draw [->,thick] (26) -- (13);
\draw [->,thick] (13) -- (4);
\draw [->,thick] (4) -- (3);
\draw [->,thick] (3) -- (1);

\draw [->,thick] (14) -- (4);
\draw [->,thick] (4) -- (10);
\draw [->,thick] (10) -- (9);
\draw [->,thick] (11) -- (10);
\draw [->,thick] (10) -- (2);
\draw [->,thick] (2) -- (6);

\draw [->,thick] (27) -- (13);
\draw [->,thick] (13) -- (11);
\draw [->,thick] (9) -- (7);
\draw [->,thick] (7) -- (15);

\draw [->,thick] (44) -- (26);
\draw [->,thick] (16) -- (28);

\draw [->,thick,dashed] (44) -- (27);
\draw [->,thick,dashed] (27) -- (14);
\draw [->,thick,dashed] (14) -- (5);
\draw [->,thick,dashed] (1) -- (6);
\draw [->,thick,dashed] (6) -- (15);
\draw [->,thick,dashed] (15) -- (28);

\end{tikzpicture}
\caption{Quiver $\Qc_4^{\mathrm{full}}$.}
\label{Q4-full}
\end{figure}

\section{Combinatorics and Poisson geometry of quiver mutations}
\label{sec-muts}

\subsection{Factorization coordinates on double Bruhat cells}
\label{subsec-double-Bruhat}

Let $G$ be a semi-simple complex Lie group endowed with the standard Poisson-Lie structure. For each element $w \in W$ of the Weyl group, denote by $\dot w$ its representative in the coset $N_G(H)/H$, where $N_G(H)$ is the normalizer of the torus $H$ in $G$. Then the double Bruhat cells
$$
G^{u,v} = B_+ \dot u B_+ \cap B_- \dot v B_-
$$
are unions of symplectic leaves of $G$ fibered over a certain torus, see~\cite{HKKR00}. It was shown in~\cite{FZ99, BFZ05} that a double Bruhat cell in a simply-connected Lie group admits the structure of a cluster variety. In~\cite{FG06b}, a closely related construction was carried out for groups of adjoint type. We shall now recall some results from the latter reference. For the rest of this section we fix $G = PGL_{n+1}$.

Let us consider a pair of alphabets $\mathfrak A_\pm = \hc{\pm 1, \dots, \pm n}$. We introduce the convenient notation
\beq
\label{bars}
\overline j =
\begin{cases}
j - (n + 1) & \text{if} \;\; j \in \mathfrak A_+, \\
j + (n + 1) & \text{if} \;\; j \in \mathfrak A_-,
\end{cases}
\qquad\text{and}\qquad
\underline j =
\begin{cases}
-j + (n + 1) \quad \text{if} \;\; j \in \mathfrak A_+, \\
-j - (n + 1) \quad \text{if} \;\; j \in \mathfrak A_-.
\end{cases}
\eeq
Note that $\underline{\mathfrak A}_\pm = \mathfrak A_\pm$ and $\overline{\mathfrak A}_\pm = \mathfrak A_\mp$. The Hecke semigroup of type $A_{n}$ is the quotient of the free monoid generated by $\mathfrak A_+ \sqcup \mathfrak A_-$ by the following relations. For any letters $\alpha$ and $\beta$ from the same alphabet, we impose the relations
\beq
\label{ab}
\begin{aligned}
\alpha\beta = \beta\alpha \quad &\text{if} \quad \hm{\alpha - \beta} > 1, \\
\overline\alpha\beta = \beta\overline\alpha \quad &\text{if} \quad \hm{\alpha + \beta} \ne n+1.
\end{aligned}
\eeq
In addition to these, we have
\begin{align}
\label{braid}
&\text{\emph{braid move}} & &\alpha\beta\alpha \longleftrightarrow \beta\alpha\beta & &\hm{\alpha - \beta} = 1 \quad\text{and}\quad \alpha,\beta \in \mathfrak A_\pm, \\
\label{shf}
&\text{\emph{shuffling}}& &\alpha\beta \longleftrightarrow \beta\alpha & &\alpha + \beta = 0 \quad\text{and}\quad \alpha,\beta \in \mathfrak A_\pm, \\
\label{21}
&\text{\emph{2-1 move}}& &\alpha\alpha \longrightarrow \alpha& &\alpha \in \mathfrak A_\pm.
\end{align}
The \emph{length} $l(\i_\pm)$ of a word $\i_\pm$ in the alphabet $\mathfrak A_\pm$ is the number of its letters. We say that a word is irreducible if its length cannot be reduced by the relations~\eqref{ab} and moves~\eqref{braid} -- \eqref{21}.

A double word $\i$ is a word in the alphabet $\mathfrak A = \mathfrak A_+ \sqcup \mathfrak A_-$. Clearly, for each double word $\i$, there exist a pair of words $\i_\pm$ such that $\i$ is a shuffle of the letters of $\i_+$ through those of $\i_-$. We define the length of the double word by $l(\i) = l(\i_+) + l(\i_-)$ and call $\i$ irreducible if both $\i_+$ and $\i_-$ are such.

Let $\i = (i_1, \dots, i_p)$ be a reduced double word, $\i_+ = (j_1, \dots, j_a)$, and $\i_- = (k_1, \dots, k_b)$ be its positive and negative parts. Then following~\cite{FG06b}, one can define a birational mapping
$$
\chi_\i \colon \hr{\C^\times}^{n+p} \longra G^{u,v}
$$
where $u = s_{k_1} \dots s_{k_b}$ and $v = s_{j_1} \dots s_{j_a}$ are reduced decompositions of the elements $u,v \in W$ and $s_i \in W$ is the reflection corresponding to a simple root $\alpha_i$. In order to describe the map $\chi_\i$ explicitly we shall need a few more notations.

Let $\hr{e_i, f_i, h_i}$, $i = 1, \dots, n$ be the Chavalley generators of the Lie algebra $\sl_{n+1}$. We consider a different set of Cartan elements $h^i$, with the property that
$$
[h^i,e_j] = \delta^i_je_j \qquad\text{and}\qquad [h^i,f_j] = -\delta^i_jf_j.
$$
The two sets of Cartan generators are related by $h_i = \sum_j c_{ij} h^j$ where $C = (c_{ij})$ is the Cartan matrix. Set
$$
E_i = \exp(e_i), \qquad F_i = \exp(f_i), \qquad H_i(x) = \exp(\log(x) h^i).
$$
Now, the map $\chi_\i$ can be described as follows: given a point $(x_1, \dots, x_{n+p}) \in \hr{\C^\times}^{n+p}$, we form the product $H_1(x_1) \dots H_n(x_n)$. Next, we read the word $\i$ left to right and multiply the aforementioned product by $E_{i_t}H_{i_t}(x_t)$ if $i_t>0$ and by $F_{i_t}H_{i_t}(x_t)$ if $i_t<0$ as $t$ runs from $n+1$ to $n+p$. We refer to variables $x_j$, $j=1, \dots, n+p$ as \emph{factorization coordinates} on the double Bruhat cell $G^{u,v}$.

\begin{example}
If $n=2$ and $\i = (1,-2,-1,2)$ we get
$$
\chi_\i(x_1,x_2,x_3,x_4,x_5,x_6) = H_1(x_1)H_2(x_2)E_1H_1(x_3)F_2H_2(x_4)F_1H_1(x_5)E_2H_2(x_6)
$$
where
\begin{align*}
& H_1(x) = \begin{pmatrix} x & 0 & 0 \\ 0 & 1 & 0 \\ 0 & 0 & 1 \end{pmatrix} &
& E_1 = \begin{pmatrix} 1 & 1 & 0 \\ 0 & 1 & 0 \\ 0 & 0 & 1 \end{pmatrix} &
& F_1 = \begin{pmatrix} 1 & 0 & 0 \\ 1 & 1 & 0 \\ 0 & 0 & 1 \end{pmatrix} \\
& H_2(x) = \begin{pmatrix} x & 0 & 0 \\ 0 & x & 0 \\ 0 & 0 & 1 \end{pmatrix} &
& E_2 = \begin{pmatrix} 1 & 0 & 0 \\ 0 & 1 & 1 \\ 0 & 0 & 1 \end{pmatrix} &
& F_2 = \begin{pmatrix} 1 & 0 & 0 \\ 0 & 1 & 0 \\ 0 & 1 & 1 \end{pmatrix}
\end{align*}
\end{example}

\begin{remark}
It is clear that one can define maps $\chi_\i$ for arbitrary double words $\i$, not necessarily reduced, in a similar fashion. In general, these maps will have nontrivial fibers. 
\end{remark}

\subsection{Cluster mutations}
\label{subsec-cluster-mut}

For each double word $\i$, the mapping $\chi_\i$ defines a chart on a double Bruhat cell $G^{u,v}$, where elements of the Weyl group $u$ and $v$ are determined by the words $\i_\pm$. Thus, transformations of the double words give rise to the gluing data between the corresponding charts. Following~\cite{FG06b}, we describe each of the relations~\eqref{ab} --~\eqref{21} in terms of generators of the group $G$.

The equalities~\eqref{ab} correspond to formulas
\beq
\begin{aligned}
&[E_i, E_j] = [F_i, F_j]=0 & &\text{if} \quad \hm{i - j} > 1, \\
&[E_i, F_j] = 0 & &\text{if} \quad i \ne j,
\end{aligned}
\eeq
The braid moves are described by
\begin{align}
\label{EEE}
&E_i H_i(x) E_j E_i = H_i(1 + x) H_j(1+x^{-1})^{-1} E_j H_j(x)^{-1} E_i H_i(1+x^{-1})^{-1} E_j H_j(1 + x), \\
\label{FFF}
&F_i H_i(x) F_j F_i = H_i(1+x^{-1})^{-1} H_j(1 + x) F_j H_j(x)^{-1} F_i H_i(1 + x) F_j H_j(1+x^{-1})^{-1}
\end{align}
whenever $\hm{i-j} = 1$. The shufflings of letters $i$ and $-i$ are realized as follows
\begin{align}
\label{EF}
&E_i H_i(x)F_i = H_{i-1}(1+x^{-1})^{-1} H_i(1 + x) F_i H_i(x)^{-1} E_i H_i(1 + x) H_{i+1}(1+x^{-1})^{-1}, \\
\label{FE}
&F_i H_i(x)E_i = H_{i-1}(1 + x) H_i(1+x^{-1})^{-1} E_i H_i(x)^{-1} F_i H_i(1+x^{-1})^{-1} H_{i+1}(1 + x).
\end{align}
Finally, the 2-1 moves correspond to 
\begin{align}
\label{EE}
&E_i H_i(x) E_i = H_i(1 + x) E_i H_i(1+x^{-1})^{-1}, \\
\label{FF}
&F_i H_i(x) F_i = H_i(1+x^{-1})^{-1} F_i H_i(1 + x).
\end{align}


\subsection{Poisson brackets and quivers}

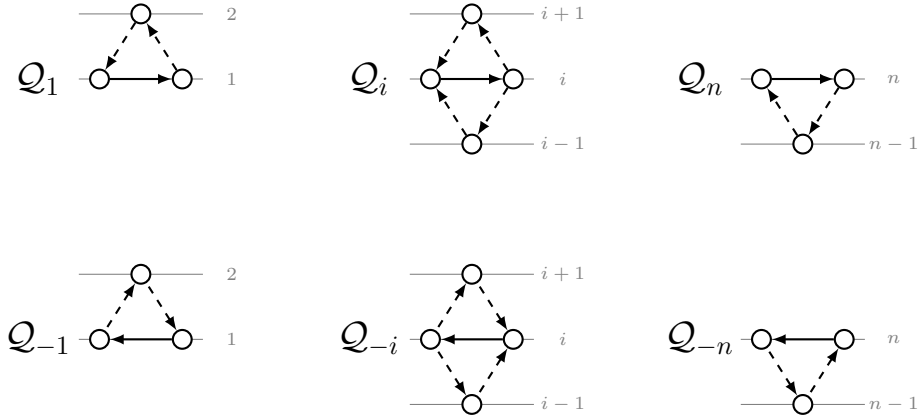
\begin{figure}[b]
\begin{tikzpicture}[every node/.style={inner sep=0, minimum size=0.25cm, circle, draw, thick, fill=white}, x=0.55cm, y=0.866cm]

\node [draw=none] (L1) at (-0.5,2) {\Large $\Qc_1$};

\node [draw=none, text=gray] at (4.2,2) {\tiny 1};
\node [draw=none, text=gray] at (4.2,3) {\tiny 2};

\draw [gray] (0.5,2) -- (3.5,2);
\draw [gray] (0.5,3) -- (3.5,3);

\node (A) at (1,2) {};
\node (C) at (3,2) {};
\node (D) at (2,3) {};

\draw [->, thick] (A) -- (C);
\draw [->, thick, dashed] (C) -- (D);
\draw [->, thick, dashed] (D) -- (A);

\node [draw=none] (L1) at (7.5,2) {\Large $\Qc_i$};

\node [draw=none, text=gray] at (12.2,1) {\tiny $i-1$};
\node [draw=none, text=gray] at (12.2,2) {\tiny $i$};
\node [draw=none, text=gray] at (12.2,3) {\tiny $i+1$};

\draw [gray] (8.5,1) -- (11.5,1);
\draw [gray] (8.5,2) -- (11.5,2);
\draw [gray] (8.5,3) -- (11.5,3);

\node (A) at (9,2) {};
\node (B) at (10,3) {};
\node (C) at (11,2) {};
\node (D) at (10,1) {};

\draw [->, thick] (A) -- (C);
\draw [->, thick, dashed] (C) -- (B);
\draw [->, thick, dashed] (B) -- (A);
\draw [->, thick, dashed] (C) -- (D);
\draw [->, thick, dashed] (D) -- (A);

\node [draw=none] (L1) at (15.5,2) {\Large $\Qc_n$};

\node [draw=none, text=gray] at (20.2,1) {\tiny $n-1$};
\node [draw=none, text=gray] at (20.2,2) {\tiny $n$};

\draw [gray] (16.5,1) -- (19.5,1);
\draw [gray] (16.5,2) -- (19.5,2);

\node (A) at (17,2) {};
\node (C) at (19,2) {};
\node (D) at (18,1) {};

\draw [->, thick] (A) -- (C);
\draw [->, thick, dashed] (C) -- (D);
\draw [->, thick, dashed] (D) -- (A);

\node [draw=none] (L1) at (-0.5,-2) {\Large $\Qc_{-1}$};

\node [draw=none, text=gray] at (4.2,-2) {\tiny 1};
\node [draw=none, text=gray] at (4.2,-1) {\tiny 2};

\draw [gray] (0.5,-2) -- (3.5,-2);
\draw [gray] (0.5,-1) -- (3.5,-1);

\node (A) at (1,-2) {};
\node (B) at (2,-1) {};
\node (C) at (3,-2) {};

\draw [<-, thick] (A) -- (C);
\draw [<-, thick, dashed] (C) -- (B);
\draw [<-, thick, dashed] (B) -- (A);

\node [draw=none] (L1) at (7.5,-2) {\Large $\Qc_{-i}$};

\node [draw=none, text=gray] at (12.2,-3) {\tiny $i-1$};
\node [draw=none, text=gray] at (12.2,-2) {\tiny $i$};
\node [draw=none, text=gray] at (12.2,-1) {\tiny $i+1$};

\draw [gray] (8.5,-3) -- (11.5,-3);
\draw [gray] (8.5,-2) -- (11.5,-2);
\draw [gray] (8.5,-1) -- (11.5,-1);

\node (A) at (9,-2) {};
\node (B) at (10,-1) {};
\node (C) at (11,-2) {};
\node (D) at (10,-3) {};

\draw [<-, thick] (A) -- (C);
\draw [<-, thick, dashed] (C) -- (B);
\draw [<-, thick, dashed] (B) -- (A);
\draw [<-, thick, dashed] (C) -- (D);
\draw [<-, thick, dashed] (D) -- (A);

\node [draw=none] (L1) at (15.5,-2) {\Large $\Qc_{-n}$};

\node [draw=none, text=gray] at (20.2,-3) {\tiny $n-1$};
\node [draw=none, text=gray] at (20.2,-2) {\tiny $n$};

\draw [gray] (16.5,-3) -- (19.5,-3);
\draw [gray] (16.5,-2) -- (19.5,-2);

\node (A) at (17,-2) {};
\node (C) at (19,-2) {};
\node (D) at (18,-3) {};

\draw [<-, thick] (A) -- (C);
\draw [<-, thick, dashed] (C) -- (D);
\draw [<-, thick, dashed] (D) -- (A);

\end{tikzpicture}
\caption{Quivers corresponding to single letters.}
\label{fig-letters}
\end{figure}

The torus $\hr{\C^\times}^{n+p}$ can be equipped with a nontrivial log-canonical Poisson structure in such a way that $\chi_\i$ becomes a morphism of Poisson varieties. Following~\cite{FM16}, we describe the Poisson brackets between the factorization coordinates $\hc{x_1,\ldots, x_{n+p}}$ with the help of quivers. In order to do so, we shall associate a quiver $\Qc_\i$ to any double word $\i$.

We start by drawing $n$ horizontal lines numbered 1 to $n$, and defining elementary quivers $\Qc_i$ and $\Qc_{-i}$ for each letter of the alphabet $\mathfrak A$ as shown in Figure~\ref{fig-letters}. Note, that these quivers have vertices at lines $i-1$, $i$, and $i+1$ if these lines exist. To associate a quiver to a word $\i$, we read it left to right and draw the quiver $\Qc_\alpha$ whenever we encounter letter $\alpha \in \mathfrak A$. Then, we amalgamate (or glue together) adjacent vertices in each of the rows if they belong to different quivers. For example, the full subquiver lying strictly below the dashed line in Figure~\ref{D4-sf} corresponds to the word~$(4,3,2,1,4,3,2,4,3,4)$.

Note that the product $\chi_\i(x_1, \dots, x_{n+p})$ has the following properties:
\begin{itemize}
\item the sequence of $E$'s and $F$'s reproduce the sequence of letters in the word $\i$;
\item every $H$ depends on its own variable;
\item there is at least one $E_i$ or $F_i$ between any pair of $H_i$'s.
\item for any $i \in \mathfrak A$ the number of $H_i$'s is 1 greater than the total number of $E_i$'s and $F_i$'s.
\end{itemize}
By construction, the number of right (respectively, left) arrows in the $i$-th row of the quiver $\Qc_\i$ coincides with the number of $E_i$'s (respectively, $F_i$'s) appearing in the factorization $\chi_\i(x_1, \dots, x_{n+p})$. Similarly, the number of vertices in the $i$-th row coincides with the number of $H_i$'s. Now, if we associate vertices of the quiver $\Qc_\i$ with the factorization parameters on the corresponding double Bruhat cell, the following theorem holds.

\begin{theorem} \cite{FG06b}
Poisson brackets between the factorization parameters are determined by the quiver $\Qc_\i$. Namely, $\hc{x_i,x_j} = \eps_{ij}x_ix_j$ where $\eps$ is the adjacency matrix of the quiver, and each dashed arrow is counted as one half of a solid one.
\end{theorem}

\begin{figure}[h]
\begin{tikzpicture}[every node/.style={inner sep=0, minimum size=0.25cm, circle, draw, thick, fill=white}, x=0.5cm, y=0.866cm]

\node [draw=none, text=gray] at (6.2,1) {\tiny $i-1$};
\node [draw=none, text=gray] at (6.2,2) {\tiny $i$};
\node [draw=none, text=gray] at (6.2,3) {\tiny $i+1$};
\node [draw=none, text=gray] at (6.2,4) {\tiny $i+2$};

\draw [gray] (0.5,1) -- (5.5,1);
\draw [gray] (0.5,2) -- (5.5,2);
\draw [gray] (0.5,3) -- (5.5,3);
\draw [gray] (0.5,4) -- (5.5,4);

\node (A) at (3,1) {};
\node (B) at (1,2) {};
\node [fill=BurntOrange] (C) at (3,2) {};
\node (D) at (5,2) {};
\node (E) at (2,3) {};
\node (F) at (4,3) {};
\node (G) at (3,4) {};

\draw [->, thick] (B) -- (C);
\draw [->, thick] (C) -- (E);
\draw [->, thick] (E) -- (F);
\draw [->, thick] (F) -- (C);
\draw [->, thick] (C) -- (D);
\draw [->, thick, dashed] (D) -- (F);
\draw [->, thick, dashed] (F) -- (G);
\draw [->, thick, dashed] (G) -- (E);
\draw [->, thick, dashed] (E) -- (B);
\draw [<-, thick, dashed] (B) -- (A);
\draw [<-, thick, dashed] (A) -- (D);

\draw [<->, ultra thick] (7,2.5) to (9,2.5);

\node [draw=none, text=gray] at (15.2,1) {\tiny $i-1$};
\node [draw=none, text=gray] at (15.2,2) {\tiny $i$};
\node [draw=none, text=gray] at (15.2,3) {\tiny $i+1$};
\node [draw=none, text=gray] at (15.2,4) {\tiny $i+2$};

\draw [gray] (9.5,1) -- (14.5,1);
\draw [gray] (9.5,2) -- (14.5,2);
\draw [gray] (9.5,3) -- (14.5,3);
\draw [gray] (9.5,4) -- (14.5,4);

\node (A) at (12,1) {};
\node (B) at (11,2) {};
\node (C) at (13,2) {};
\node (D) at (10,3) {};
\node [fill=BurntOrange] (E) at (12,3) {};
\node (F) at (14,3) {};
\node (G) at (12,4) {};

\draw [->, thick] (D) -- (E);
\draw [->, thick] (E) -- (B);
\draw [->, thick] (B) -- (C);
\draw [->, thick] (C) -- (E);
\draw [->, thick] (E) -- (F);
\draw [->, thick, dashed] (F) -- (G);
\draw [->, thick, dashed] (G) -- (D);
\draw [<-, thick, dashed] (D) -- (B);
\draw [<-, thick, dashed] (B) -- (A);
\draw [<-, thick, dashed] (A) -- (C);
\draw [<-, thick, dashed] (C) -- (F);

\end{tikzpicture}
\caption{Braid move: $(i,i+1,i) \leftrightarrow (i+1,i,i+1)$.}
\label{fig-braid}
\end{figure}
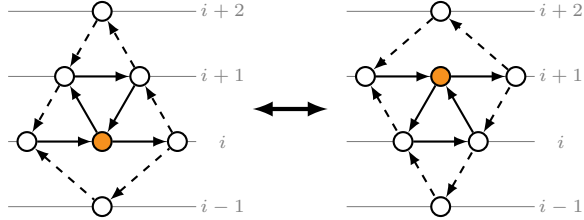

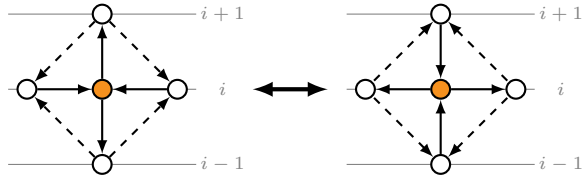
\begin{figure}[b]
\begin{tikzpicture}[every node/.style={inner sep=0, minimum size=0.25cm, circle, draw, thick, fill=white}, x=0.5cm, y=1cm]

\node [draw=none, text=gray] at (6.2,1) {\tiny $i-1$};
\node [draw=none, text=gray] at (6.2,2) {\tiny $i$};
\node [draw=none, text=gray] at (6.2,3) {\tiny $i+1$};

\draw [gray] (0.5,1) -- (5.5,1);
\draw [gray] (0.5,2) -- (5.5,2);
\draw [gray] (0.5,3) -- (5.5,3);

\node (A) at (3,1) {};
\node (B) at (1,2) {};
\node [fill=BurntOrange] (C) at (3,2) {};
\node (D) at (5,2) {};
\node (E) at (3,3) {};

\draw [->, thick] (B) -- (C);
\draw [->, thick] (D) -- (C);
\draw [->, thick] (C) -- (E);
\draw [->, thick] (C) -- (A);
\draw [->, thick, dashed] (A) -- (B);
\draw [->, thick, dashed] (A) -- (D);
\draw [->, thick, dashed] (E) -- (B);
\draw [->, thick, dashed] (E) -- (D);

\draw [<->, ultra thick] (7,2) to (9,2);

\node [draw=none, text=gray] at (15.2,1) {\tiny $i-1$};
\node [draw=none, text=gray] at (15.2,2) {\tiny $i$};
\node [draw=none, text=gray] at (15.2,3) {\tiny $i+1$};

\draw [gray] (9.5,1) -- (14.5,1);
\draw [gray] (9.5,2) -- (14.5,2);
\draw [gray] (9.5,3) -- (14.5,3);

\node (A) at (12,1) {};
\node (B) at (10,2) {};
\node [fill=BurntOrange] (C) at (12,2) {};
\node (D) at (14,2) {};
\node (E) at (12,3) {};

\draw [<-, thick] (B) -- (C);
\draw [<-, thick] (D) -- (C);
\draw [<-, thick] (C) -- (E);
\draw [<-, thick] (C) -- (A);
\draw [<-, thick, dashed] (A) -- (B);
\draw [<-, thick, dashed] (A) -- (D);
\draw [<-, thick, dashed] (E) -- (B);
\draw [<-, thick, dashed] (E) -- (D);

\end{tikzpicture}
\caption{Plus-minus move: $(i,-i) \leftrightarrow (-i,i)$.}
\label{fig-pm}
\end{figure}

Moves~\eqref{braid}, \eqref{shf}, and~\eqref{21} also admit an interpretation in terms of quivers. Namely, each move can be realized via a quiver mutation. Indeed, Figure~\ref{fig-braid} shows a mutation corresponding to the braid move~\eqref{braid} expressed by the formula~\eqref{EEE} with $j=i+1$. Here and in the following two examples we always mutate at an orange vertex. Similarly, Figure~\ref{fig-pm} shows a mutation corresponding to the move~\eqref{shf}, shuffling a letter $i$ through a letter $-i$. In terms of generators of the group $G$, this mutation is described by equation~\eqref{EF}. One can easily check that unless $\alpha + \beta=0$, shuffling of the letters $\alpha$ and $\beta$ give identical quivers. As for the move~\eqref{21}, in what follows we shall only apply it to the letters $\pm1$ or $\pm n$. Figure~\ref{fig-21} shows the cluster mutation for the case $11 \to 1$. Note that the orange vertex in the right quiver in Figure~\ref{fig-21} does not belong to any row. If we now erase that vertex, then the composition of mutation and erasing corresponds to the formula~\eqref{EE} with $i=1$.

\begin{figure}[h]
\begin{tikzpicture}[every node/.style={inner sep=0, minimum size=0.25cm, circle, draw, thick, fill=white}, x=0.5cm, y=0.866cm]

\node [draw=none, text=gray] at (6.2,1) {\tiny $1$};
\node [draw=none, text=gray] at (6.2,2) {\tiny $2$};

\draw [gray] (0.5,1) -- (5.5,1);
\draw [gray] (0.5,2) -- (5.5,2);

\node (A) at (1,1) {};
\node (B) at (3,2) {};
\node [fill=BurntOrange] (C) at (3,1) {};
\node (D) at (5,1) {};

\draw [->, thick] (A) -- (C);
\draw [->, thick] (C) -- (D);
\draw [->, thick, dashed] (D) -- (B);
\draw [->, thick, dashed] (B) -- (A);

\draw [<->, ultra thick] (7,1) to (9,1);

\node [draw=none, text=gray] at (15.2,1) {\tiny $1$};
\node [draw=none, text=gray] at (15.2,2) {\tiny $2$};

\draw [gray] (9.5,1) -- (14.5,1);
\draw [gray] (9.5,2) -- (14.5,2);

\node (A) at (11,1) {};
\node (B) at (12,2) {};
\node [fill=BurntOrange] (C) at (12,0) {};
\node (D) at (13,1) {};

\draw [->, thick] (A) -- (D);
\draw [->, thick] (D) -- (C);
\draw [->, thick] (C) -- (A);
\draw [->, thick, dashed] (D) -- (B);
\draw [->, thick, dashed] (B) -- (A);

\end{tikzpicture}
\caption{2-1 move: $(1,1) \to (1)$.}
\label{fig-21}
\end{figure}

To summarize, each double Bruhat cell $G^{u,v}$ is a cluster Poisson variety, with a cluster chart $\chi_\i$ for any double reduced word $\i$, and cluster mutations correspond to the moves~\eqref{braid} and~\eqref{shf} with $\alpha + \beta=0$.

\subsection{Reduced double Bruhat cells}

Consider the space $G^{u,v}/\Ad H$ of $H$-orbits on the double Bruhat cell $G^{u,v}$ under the conjugation action. We refer to this set as a \emph{reduced double Bruhat cell}. We observe that in each chart $\chi_\i$ on $G^{u,v}$, the ring of $\Ad_H$-invariant functions is generated by the products $x_lx_r$, where $x_l$ is the argument of the left-most $H_i$ and $x_r$ is the argument of the right-most $H_i$, together with the factorization coordinates $x_t$ arising as arguments of $H_i$'s which are neither the left-most nor the right-most. Thus the reduced double Bruhat cell $G^{u,v}/\Ad H$ is also a cluster variety, and its cluster structure is an amalgamation of that of $G^{u,v}$. We denote by $\overline\chi_\i$ the cluster chart on the reduced double Bruhat cell $G^{u,v}/\Ad H$ obtained by amalgamating the cluster chart $\chi_\i$ on $G^{u,v}$.

The quiver $\overline\Qc_\i$ that defines the Poisson brackets between the factorization coordinates on the chart $\overline\chi_\i$ is obtained from the quiver $\Qc_\i$ by amalgamating its left-most and right-most vertices in each row. For example, quiver in figure~\ref{fig-Qbox} coincides with $\overline\Qc_\i$ where $\i = (\i_-, \i_+)$ and
\begin{align}
\label{iw0}
\i_+ &= (4,3,2,1,4,3,2,4,3,4), \\
\label{iw0bar}
\i_- &= (\overline 4, \overline 3, \overline 2, \overline 1, \overline 4, \overline 3, \overline 2, \overline 4, \overline 3, \overline 4),
\end{align}
with $\overline j$ defined in~\eqref{bars}. Similarly, after erasing all of its frozen variables, the $\Dc_4$-quiver in Figure~\ref{fig-D4} coincides with the quiver~$\overline\Qc_\i$ with $\i = (\i_-, \i_-)$ and $\i_-$ as in~\eqref{iw0bar}.

\subsection{Combinatorics of word transformations}
\label{subsec-Weyl-comb}

In the next subsection we describe some quiver mutation sequences that we will use to analyze the cluster realization of $U_q(\mathfrak{sl}_{n+1})$. These mutations arise from certain transformations of double words that were explained to us by M.~Gekhtman. Consider the cyclic shift operators that take a word $\i = \i_1 = (i_1, \dots, i_p)$ into its $k$-th cyclic shift $\i_k = (i_k, \dots, i_p, i_1, \dots, i_{k-1})$. If there exists a cyclic shift $\i_k$ that is a reducible word, we say that the original word $\i$ is \emph{cyclically reducible.}

Recall that a \emph{Coxeter element} $c$ of the Weyl group $W$ is one is given by a reduced word in which each simple reflection appears exactly once.  We shall now describe a particular sequence of cyclic shifts and moves of types~\eqref{braid}~--~\eqref{21} that reduces a word $\ifull$ describing a cluster chart on the big double Bruhat cell $G^{w_0,w_0}$ to the word $\icox$ describing a cluster chart on the Coxeter cell $G^{c,c}$.

Let us introduce several words that we will use frequently in the sequel. First, for any words $\i = (i_1, \dots, i_p)$ in $\mathfrak{A}_+$ and $\mathbf{j} = (j_1, \dots, j_p)$ in $\mathfrak{A}_-$, we set
$$
\overline \i = (\overline i_1, \dots, \overline i_p) \quad\text{and}\quad \underline{\mathbf{j}} = (\underline{j}_1, \dots, \underline{j}_p)
$$
where $\overline i $ and $\underline i$ are defined in~\eqref{bars}. We also define
$$
\i_{[i,j]} =
\begin{cases}
(i,i+1,\dots,j) &\text{if} \quad i \le j, \\
(i,i-1,\dots,j) &\text{if} \quad i \ge j,
\end{cases}
$$
so that, for example, $\overline\i_{[i,j]} = \i_{[\overline j, \overline i]}$. Now, let us consider words
$$
\i_c = \i_{[1,n]}, \qquad \i_{w_0} = \hr{\i_{[n,1]},\i_{[n,2]},\dots,\i_{[n,n]}}
$$
and their conjugates $\overline\i_c$ and $\overline \i_{w_0}$. For $n=4$, the words $\i_{w_0}$ and $\overline \i_{w_0}$ are given by formulas~\eqref{iw0} and~\eqref{iw0bar} respectively.

We now describe word transformations
\begin{align}
\label{phi1}
& \phi_1 \colon \ifull \longmapsto \ihalf, \\
\label{phi2}
& \phi_2 \colon \ihalf \longmapsto \icox, \\
\label{phi3}
& \phi_3 \colon \isf \longmapsto \isym.
\end{align}
where
$$
\ifull = \hr{\overline \i_{w_0}, \i_{w_0}}, \qquad
\ihalf = \hr{\overline \i_{w_0}, \i_c}, \qquad
\icox = \hr{\overline \i_c, \i_c},
$$
and
$$
\isf = \hr{\underline\i_c,\i_c,\overline\i_{w_0}}, \qquad
\isym = \hr{\underline\i_c,\overline\i_{w_0}}.
$$

\subsubsection{The transformation $\phi_1$}

We factor the transformation $\phi_1$ into a composite of $n-1$ ``waves'', where the $j$-th wave consists of $n-j$ steps. After performing the $k$-th step of the $(j+1)$-st wave of $\phi_1$, we shall obtain the word
$$
\ifull^{(j,k)} = \hr{\overline \i_{w_0},1, \dots, j,\i_{[n-k,j+1]},\i_{[n,j+2]},\dots,\i_{[n,n]} }.
$$
Note that our conventions imply
$$
\ifull^{(j-1,n-j)} = \ifull^{(j,0)} \qquad\text{and}\qquad \ifull^{(n-1,0)} = \ihalf.
$$
Now, the $(k+1)$-st step of the $j$-th wave is given as follows. Note that
$$
\ifull^{(j,k)} = \hr{\overline \i_{w_0},1, \dots, j,n-k,\i_{[n-k-1,j+1]},\i_{[n,j+2]},\dots,\i_{[n,n]} }.
$$
First, we commute the letter $n-k$ all the way to the left through the word $\overline \i_{w_0}$. In particular, we apply exactly $k+1$ shufflings of the form $(k-n,n-k) \mapsto (n-k,k-n)$. Then, we perform the cyclic shift that erases the letter $n-k$ on the very left and writes it on the far right, to obtain
\beq
\label{temp1}
\hr{\overline \i_{w_0},1, \dots, j,\i_{[n-k-1,j+1]},\i_{[n,j+2]},\dots,\i_{[n,n]},n-k}.
\eeq
The part of the latter word which reads $\hr{ \i_{[n,j+2]},\dots,\i_{[n,n]},n-k }$ can be rewritten as follows:
\begin{multline}
\hr{\i_{[n,j+2]},\dots,\i_{[n,n+1-k]},\i_{[n,n+2-k]},\dots,\i_{[n,n]},n-k} = \\
\label{temp2}
\hr{\i_{[n,j+2]},\dots,\i_{[n,n+1-k]},n-k,\i_{[n,n+2-k]},\dots,\i_{[n,n]}}
\end{multline}
We shall now restrict our attention to the subword
\beq
\label{temp3}
\hr{\i_{[n,n-k]},\i_{[n,n+1-k]},n-k}.
\eeq
We rewrite it as
$$
\hr{\i_{[n,n+1-k]},\i_{[n,n+2-k]},n-k,n+1-k,n-k},
$$
and then apply a braid move to get
$$
\hr{\i_{[n,n+1-k]},\i_{[n,n+2-k]},n+1-k,n-k,n+1-k}.
$$
The latter is a particular case of the word
$$
\hr{\i_{[n,s]},\i_{[n,s+1]},s,\i_{[s-1,t-1]},\i_{[s,t]}}
$$
for $s=t=n+1-k$. We can once again rewrite it as
$$
\hr{\i_{[n,s+1]},\i_{[n,s+2]},s,s+1,s,\i_{[s-1,t-1]},\i_{[s,t]}}
$$
and use a braid move to get
$$
\hr{\i_{[n,s+1]},\i_{[n,s+2]},s+1,s,s+1,\i_{[s-1,t-1]},\i_{[s,t]}} = \hr{\i_{[n,s+1]},\i_{[n,s+2]},s+1,\i_{[s,t-1]},\i_{[s+1,t]}}.
$$
Proceeding by induction on $s$ we arrive at
$$
\hr{n,n-1,n,n-1,\i_{[n-2,t-1]},\i_{[n-1,t]}},
$$
use a braid move once again, to get
$$
\hr{n,n,n-1,n,\i_{[n-2,t-1]},\i_{[n-1,t]}} = \hr{n,n,\i_{[n-1,t-1]},\i_{[n,t]}},
$$
and finally, apply 2-1-move $(n,n) \mapsto n$ to obtain
$$
\hr{n,\i_{[n-1,t-1]},\i_{[n,t]}} = \hr{\i_{[n,t-1]},\i_{[n,t]}}.
$$
Recall that $t = n+1-k$ and thus the word~\eqref{temp3} has been reduced to $\hr{\i_{[n,n-k]},\i_{[n,n+1-k]}}$. Re-inserting it into the word~\eqref{temp2} and then~\eqref{temp1} we arrive at
$$
\ifull^{(j,k+1)} = \hr{\overline \i_{w_0},1, \dots, j,\i_{[n-k-1,j+1]},\i_{[n,j+2]},\dots,\i_{[n,n]} }.
$$
To summarize, the $(k+1)$-st step in each wave consists of
\begin{itemize}
\item $k+1$ shufflings $(k-n,n-k) \mapsto (n-k,k-n)$;
\item a cyclic shift $(n-k,\i) \mapsto (\i,n-k)$;
\item $k$ braid moves: $(n-s,n+1-s,n-s) \mapsto (n+1-s,n-s,n+1-s)$, $s=k,\dots,2,1$;
\item a 2-1-move $(n,n) \mapsto n$.
\end{itemize}

\begin{example}
Let us spell out $\phi_1 \colon \ifull \longmapsto \ihalf$ in case $n=4$. Since the $(k+1)$-st step of each wave starts with shuffling the letter $n-k$ through $\overline \i_{w_0}$ we will indicate the number of moves $(k-n,n-k) \mapsto (n-k,k-n)$ by $\stackrel{k+1}\longmapsto$. Now, the first wave reads
\begin{align*}
\text{Step 1:} &&
(\overline \i_{w_0},4,3,2,1,4,3,2,4,3,4) &\stackrel{1}\longmapsto (4,\overline \i_{w_0},3,2,1,4,3,2,4,3,4) \\
&& \longmapsto (\overline \i_{w_0},3,2,1,4,3,2,4,3,4,4) &\longmapsto (\overline \i_{w_0},3,2,1,4,3,2,4,3,4) \\
\text{Step 2:} &&
(\overline \i_{w_0},3,2,1,4,3,2,4,3,4) &\stackrel{2}\longmapsto (3,\overline \i_{w_0},2,1,4,3,2,4,3,4) \\
&& \longmapsto (\overline \i_{w_0},2,1,4,3,2,4,3,4,3) &\longmapsto (\overline \i_{w_0},2,1,4,3,2,4,4,3,4) \\
&& &\longmapsto (\overline \i_{w_0},2,1,4,3,2,4,3,4) \\
\text{Step 3:} &&
(\overline \i_{w_0},2,1,4,3,2,4,3,4) &\stackrel{3}\longmapsto (2,\overline \i_{w_0},1,4,3,2,4,3,4) \\
&& \longmapsto (\overline \i_{w_0},1,4,3,2,4,3,4,2) &\longmapsto (\overline \i_{w_0},1,4,3,4,3,2,3,4) \\
&& \longmapsto (\overline \i_{w_0},1,4,4,3,4,2,3,4) &\longmapsto (\overline \i_{w_0},1,4,3,4,2,3,4)
\end{align*}
The second wave becomes
\begin{align*}
\text{Step 1:} &&
(\overline \i_{w_0},1,4,3,4,2,3,4) &\stackrel{1}\longmapsto (4,\overline \i_{w_0},1,3,4,2,3,4) \\
&& \longmapsto (\overline \i_{w_0},1,3,4,2,3,4,4) &\longmapsto (\overline \i_{w_0},1,3,4,2,3,4) \\
\text{Step 2:} &&
(\overline \i_{w_0},1,3,4,2,3,4) &\stackrel{2}\longmapsto (3,\overline \i_{w_0},1,4,2,3,4) \\
&& \longmapsto (\overline \i_{w_0},1,4,2,3,4,3) &\longmapsto (\overline \i_{w_0},1,4,2,4,3,4) \\
&& &\longmapsto (\overline \i_{w_0},1,4,2,3,4)
\end{align*}
And finally, the third wave is
\begin{align*}
\text{Step 1:} &&
(\overline \i_{w_0},1,4,2,3,4) &\stackrel{1}\longmapsto (4,\overline \i_{w_0},1,2,3,4) \\
&&\longmapsto (\overline \i_{w_0},1,2,3,4,4) &\longmapsto (\overline \i_{w_0},1,2,3,4).
\end{align*}
\end{example}

\subsubsection{The transformation $\phi_2$}

Similarly to $\phi_1$, we break the transformation $\phi_2$ into $n-1$ waves, such that the $j$-th wave consists of $n-j$ steps, although in $\phi_2$ we also include a preliminary and a concluding step. The preliminary one is a cyclic shift that changes $\ihalf = (\overline\i_{w_0},\i_c)$ into $\ihalf^{(0,0)} = (\i_c,\overline\i_{w_0})$. Now, after the $k$-th step of the $(j+1)$-st wave of the transformation $\phi_2$, we shall obtain the word
$$
\ihalf^{(j,k)} = \hr{\i_c, \overline 1, \dots, \overline j, \overline \i_{[n-k,j+1]}, \overline \i_{[n,j+2]},\dots,\overline \i_{[n,n]} }.
$$
The steps of $\phi_2$ are similar to the ones of $\phi_1$ with the following minor alterations: letters $i \in \mathfrak U_+$ are replaced with $\overline i \in \mathfrak A_-$, and $\overline\i_{w_0}$ is replaced with $\i_c$. The latter also implies that it takes just one shuffling to move the letter $\overline i$ through $\i_c$ for any $i = 1, \dots, n$. Thus, the $(k+1)$-st step in each wave of $\phi_2$ consists of
\begin{itemize}
\item a shuffling $(k,-k) \mapsto (-k,k)$;
\item a cyclic shift $(-k,\i) \mapsto (\i,-k)$;
\item $k$ braid moves: $(-s-1,-s,-s-1) \mapsto (-s,-s-1,-s)$, $s=k,\dots,2,1$;
\item a 2-1-move $(-1,-1) \mapsto -1$.
\end{itemize}
Finally, we arrive at the word $\ihalf^{(n-1,0)} = (\i_c,\overline\i_c)$. The concluding step is simply one more cyclic shift after which we obtain $\icox$.

\subsubsection{The transformation $\phi_3$} We break the transformation $\phi_3$ into a sequence of $n$ waves where the $j$-th wave consists of $j$ steps. After the $j$-th wave we shall turn the word $\isf^{(0,0)} = \isf$ into
$$
\isf^{(j,0)} = \hr{\underline\i_c,1,2,\dots,n-j,\overline\i_{w_0}}.
$$
Similarly, after the $k$-th step of the $j$-th wave, here $1 \le k \le j$, we obtain
$$
\isf^{(j,k)} = \hr{\underline\i_c,1,2,\dots,n-j-1,n-j+k,\overline\i_{w_0}}
$$
Note that our notations imply
$$
\isf^{(n,0)} = \isym.
$$
The $(k+1)$-st step of the $j$-th wave looks as follows. Setting $s = n-j+k$, we first apply $n+1-s$ shufflings to move the letter $s$ all the way to the right through $\overline\i_{w_0}$. Next, we use a cyclic shift to put $s$ on the far left. Then, if~$k \ne j$ we apply the braid move $(s,s+1,s) \mapsto (s+1,s,s+1)$ and arrive at $\isf^{(j,k+1)}$, whereas if $k = j$ we apply the 2-1-move $(n,n) \mapsto n$ and arrive at $\isf^{(j+1,0)}$.

\begin{example}
Let us spell out the transformation $\phi_3$ for $n=2$. We have
$$
\isf = (2,1,1,2,\overline\i_{w_0}) \qquad\text{where}\qquad \overline\i_{w_0} = (-1,-2,-1)
$$
and the first wave becomes
\begin{align*}
\text{Step 1:} &&
(2,1,1,2,\overline\i_{w_0}) &\stackrel{1}\longmapsto (2,1,1,\overline\i_{w_0},2) \\
&& \longmapsto (2,2,1,1,\overline\i_{w_0}) &\longmapsto (2,1,1,\overline\i_{w_0}).
\end{align*}
The second wave reads as
\begin{align*}
\text{Step 1:} &&
(2,1,1,\overline\i_{w_0}) &\stackrel{2}\longmapsto (2,1,\overline\i_{w_0},1) \\
&& \longmapsto (1,2,1,\overline\i_{w_0}) &\longmapsto (2,1,2,\overline\i_{w_0}) \\
\text{Step 2:} &&
(2,1,2,\overline\i_{w_0}) &\stackrel{1}\longmapsto (2,1,\overline\i_{w_0},2) \\
&&\longmapsto (2,2,1,\overline\i_{w_0}) &\longmapsto (2,1,\overline\i_{w_0}).
\end{align*}
\end{example}

\subsection{Quiver mutations}
\label{subsec-quiver-mut}

Now we interpret the transformations~\eqref{phi1}~--~\eqref{phi2} in terms of quivers. By an abuse of notation, we denote transformations of quivers by the same symbols as the transformations of the double words. We have already seen how braid moves, shufflings, and 2-1-moves correspond to quiver mutations. Moreover, the cyclic quiver $\overline\Qc_\i$ evidently does not change under cyclic shifts of the word $\i$.

\vspace{-1pt}

\begin{figure}[h]
\begin{tikzpicture}[every node/.style={inner sep=0, minimum size=0.5cm, thick, draw, circle}, x=0.75cm, y=0.75cm]

\node (B1) at (0,2) {\tiny$X_1^1$};
\node (B2) at (2,2) {\tiny$X_2^1$};
\node (B3) at (4,2) {\tiny$X_3^1$};
\node (B4) at (6,2) {\tiny$X_4^1$};
\node (B5) at (8,2) {\tiny$X_5^1$};

\node (C5) at (0,4) {\tiny$X_1^2$};
\node (C1) at (2,4) {\tiny$X_2^2$};
\node (C2) at (4,4) {\tiny$X_3^2$};
\node (C3) at (6,4) {\tiny$X_4^2$};
\node (C4) at (8,4) {\tiny$X_5^2$};

\node (D4) at (0,6) {\tiny$X_1^3$};
\node (D5) at (2,6) {\tiny$X_2^3$};
\node (D1) at (4,6) {\tiny$X_3^3$};
\node (D2) at (6,6) {\tiny$X_4^3$};
\node (D3) at (8,6) {\tiny$X_5^3$};

\node (E3) at (0,8) {\tiny$X_1^4$};
\node (E4) at (2,8) {\tiny$X_2^4$};
\node (E5) at (4,8) {\tiny$X_3^4$};
\node (E1) at (6,8) {\tiny$X_4^4$};
\node (E2) at (8,8) {\tiny$X_5^4$};

\draw [<-, thick] (B1) to (B2);
\draw [<-, thick] (B2) to (B3);
\draw [<-, thick] (B3) to (B4);
\draw [<-, thick] (B4) to (B5);
\draw [->, thick] (B5) to [bend right = 15] (B1);
\draw [<-, thick] (C1) to (C2);
\draw [<-, thick] (C2) to (C3);
\draw [<-, thick] (C3) to (C4);
\draw [->, thick] (C4) to [bend right = 15] (C5);
\draw [->, thick] (C5) to (C1);
\draw [<-, thick] (D1) to (D2);
\draw [<-, thick] (D2) to (D3);
\draw [->, thick] (D3) to [bend right = 15] (D4);
\draw [->, thick] (D4) to (D5);
\draw [->, thick] (D5) to (D1);
\draw [<-, thick] (E1) to (E2);
\draw [->, thick] (E2) to [bend right = 15] (E3);
\draw [->, thick] (E3) to (E4);
\draw [->, thick] (E4) to (E5);
\draw [->, thick] (E5) to (E1);

\draw [->, thick] (B2) to (C2);
\draw [->, thick] (C2) to (D2);
\draw [->, thick] (D2) to (E2);
\draw [->, thick] (B3) to (C3);
\draw [->, thick] (C3) to (D3);
\draw [->, thick] (E3) to [out=-30, in=150] (D3);
\draw [->, thick] (B4) to (C4);
\draw [->, thick] (D4) to [out=-30, in=150] (C4);
\draw [<-, thick] (D4) to (E4);
\draw [->, thick] (C5) to [out=-30, in=150] (B5);
\draw [<-, thick] (C5) to (D5);
\draw [<-, thick] (D5) to (E5);

\draw [->, thick] (C1) to (B2);
\draw [->, thick] (D1) to (C2);
\draw [->, thick] (C2) to (B3);
\draw [->, thick] (E1) to (D2);
\draw [->, thick] (D2) to (C3);
\draw [->, thick] (C3) to (B4);
\draw [->, thick] (C5) to (D4);
\draw [->, thick] (D4) to (E3);
\draw [->, thick] (D5) to (E4);

\draw [->, thick] (B1) to (C5);
\draw [->, thick] (C1) to (D5);
\draw [->, thick] (D1) to (E5);

\end{tikzpicture}
\caption{Quiver $\Qc_4^{\mathrm{box}}$.}
\label{fig-Qbox}
\end{figure}

\vspace{-1pt}

\subsubsection{Quiver mutations for $\phi_1$}
\label{subsec-muts-double1}

We shall denote the cyclic quiver for the word $\ifull$ by $\Qc_n^{\mathrm{box}}$; the quiver $\Qc_4^{\mathrm{box}}$ is shown in Figure~\ref{fig-Qbox}. Note, that the nodes of $\Qc_n^{\mathrm{box}}$ are arranged into a square grid with $n$ rows and $n+1$ columns. After each step of $\phi_1$ we shall rearrange the nodes and rename them so that the node in position $(c,r)$ is denoted by $X^r_c$. Now, the $(k+1)$-st step of the $j$-th wave of $\phi_1$ looks as follows:
\begin{itemize}
\item mutate consecutively at vertices $X_{n+1}^{n-k}$, $X_n^{n-k}$, \dots, $X_{n+1-k}^{n-k}$;
\item mutate consecutively at vertices $X_{n-k}^{n-k}$, $X_{n-k}^{n+1-k}$, \dots, $X_{n-k}^n$;
\item shift the vertex $X_{n-k}^n$ to the position $(n-k,2n-j)$;
\item shift vertices $X_{n-k}^r$ to the position $(n-k,r+1)$ for $n-k \le r \le n-1$;
\item shift vertices $X_c^{n-k}$ to the position $(c-1,n-k)$ for $n+1-k \le c \le n+1$;
\item shift the vertex $X_j^{n-k}$ to the position $(n+1,n-k)$;
\item rename vertices according to their new position.
\end{itemize}

\begin{figure}[h]
\begin{tikzpicture}[every node/.style={inner sep=0, minimum size=0.5cm, thick, draw, circle}, x=0.75cm, y=0.5cm]

\node (1) at (-4,2) {\tiny{$X_1^1$}};
\node (2) at (-2,2) {\tiny{$X_2^1$}};
\node (3) at (0,2) {\tiny{$X_3^1$}};
\node (4) at (2,2) {\tiny{$X_4^1$}};
\node (5) at (4,2) {\tiny{$X_5^1$}};
\node (6) at (-3,4) {\tiny{$X_2^2$}};
\node (7) at (-1,4) {\tiny{$X_3^2$}};
\node (8) at (1,4) {\tiny{$X_4^2$}};
\node (9) at (3,4) {\tiny{$X_5^2$}};
\node (10) at (-2,6) {\tiny{$X_3^3$}};
\node (11) at (0,6) {\tiny{$X_4^3$}};
\node (12) at (2,6) {\tiny{$X_5^3$}};
\node (13) at (-1,8) {\tiny{$X_4^4$}};
\node (14) at (1,8) {\tiny{$X_5^4$}};
\node (15) at (0,10) {\tiny{$X_4^5$}};
\node (16) at (-1,12) {\tiny{$X_3^6$}};
\node (17) at (1,12) {\tiny{$X_4^6$}};
\node (18) at (-2,14) {\tiny{$X_2^7$}};
\node (19) at (0,14) {\tiny{$X_3^7$}};
\node (20) at (2,14) {\tiny{$X_4^7$}};

\draw [->, thick] (13) -- (15);
\draw [->, thick] (15) to [out = -60, in=120] (14);
\draw [->, thick] (14) -- (13);
\draw [->, thick] (13) -- (11);
\draw [->, thick] (11) -- (14);
\draw [->, thick] (14) -- (12);
\draw [->, thick] (12) -- (11);
\draw [->, thick] (11) -- (10);
\draw [->, thick] (10) -- (7);
\draw [->, thick] (7) -- (11);
\draw [->, thick] (11) -- (8);
\draw [->, thick] (8) -- (12);
\draw [->, thick] (12) -- (9);
\draw [->, thick] (9) -- (8);
\draw [->, thick] (8) -- (7);
\draw [->, thick] (7) -- (6);
\draw [->, thick] (6) -- (2);
\draw [->, thick] (2) -- (7);
\draw [->, thick] (7) -- (3);
\draw [->, thick] (3) -- (8);
\draw [->, thick] (8) -- (4);
\draw [->, thick] (4) -- (9);
\draw [->, thick] (9) -- (5);
\draw [->, thick] (5) -- (4);
\draw [->, thick] (4) -- (3);
\draw [->, thick] (3) -- (2);
\draw [->, thick] (2) -- (1);

\draw [->, thick] (5) to [bend right = 15] (1);
\draw [->, thick] (1) to [out = 30, in=-160] (9);
\draw [->, thick] (9) to [bend right = 18] (6);
\draw [->, thick] (6) to [out = 35, in=-155] (12);
\draw [->, thick] (12) to [bend right = 21] (10);
\draw [->, thick] (10) to [out = 40, in=-150] (14);
\draw [->, thick] (14) to [bend right = 24] (13);

\draw [->, thick] (18) -- (19);
\draw [->, thick] (19) -- (20);
\draw [->, thick] (20) -- (17);
\draw [->, thick] (17) -- (19);
\draw [->, thick] (19) -- (16);
\draw [->, thick] (16) -- (17);
\draw [->, thick] (17) -- (15);

\draw [->, thick] (15) to [out = 120, in=-60] (16);
\draw [->, thick] (16) to [bend left = 24] (17);
\draw [->, thick] (17) to [out = 140, in = -40] (18);
\draw [->, thick] (18) to [bend left = 21] (20);

\end{tikzpicture}
\caption{Quiver $\Qc_4^{\mathrm{cone}}$.}
\label{fig-Qcone}
\end{figure}

After applying all $n-1$ waves, we arrive at the quiver $\Qc_n^{\mathrm{cone}}$. We show the quiver $\Qc_4^{\mathrm{cone}}$ in Figure~\ref{fig-Qcone}, in which we have shifted all vertices in the $r$-th row to the left by $\frac{r-1}{2}$ if $1 \le r \le n$ and by $\frac{2n-1-r}{2}$ if $n \le r \le 2n-1$. 

Let us denote by $\Qc_n^{\overline w_0, c}$ the cyclic quiver corresponding to the word $\ihalf$. Note that the full subquiver formed by the bottom $n$ rows of $\Qc_n^{\mathrm{cone}}$ coincides with the cyclic quiver $\Qc_n^{\overline w_0, c}$. At the same time, the full subquiver formed by the top $n-2$ rows of $\Qhalf$ is the cyclic quiver corresponding to the word $\hr{\underline{\i}_{\mathrm{cox}}, \overline\i_{w_0}}$ where both $\icox$ and $\i_{w_0}$ are written in an alphabet with $n-2$ letters. We denote this quiver by $\Qc_{n-2}^{\underline c,\overline w_0}$. 

\begin{figure}[h]
\begin{tikzpicture}[every node/.style={inner sep=0, minimum size=0.5cm, thick, draw, circle}, x=0.75cm, y=0.75cm]

\node (1) at (2,2) {\tiny{$Y_2^1$}};
\node (2) at (4,2) {\tiny{$Y_3^1$}};
\node (3) at (6,2) {\tiny{$Y_4^1$}};
\node (4) at (8,2) {\tiny{$Y_5^1$}};
\node (5) at (0,2) {\tiny{$Y_1^1$}};
\node (6) at (4,4) {\tiny{$Y_3^2$}};
\node (7) at (6,4) {\tiny{$Y_4^2$}};
\node (8) at (8,4) {\tiny{$Y_5^2$}};
\node (9) at (2,4) {\tiny{$Y_2^2$}};
\node (10) at (6,6) {\tiny{$Y_4^3$}};
\node (11) at (8,6) {\tiny{$Y_5^3$}};
\node (12) at (4,6) {\tiny{$Y_3^3$}};
\node (13) at (8,8) {\tiny{$Y_5^4$}};
\node (14) at (6,8) {\tiny{$Y_4^4$}};

\draw [->, thick] (14) to [bend left = 24] (13);
\draw [->, thick] (13) -- (11);
\draw [->, thick] (11) to [out=120, in=-45] (14);
\draw [->, thick] (14) -- (12);
\draw [->, thick] (12) to [bend left = 21] (11);
\draw [->, thick] (11) -- (10);
\draw [->, thick] (10) -- (7);
\draw [->, thick] (7) -- (11);
\draw [->, thick] (11) -- (8);
\draw [->, thick] (8) to [out=125, in=-45] (12);
\draw [->, thick] (12) -- (9);
\draw [->, thick] (9) to [bend left = 18] (8);
\draw [->, thick] (8) -- (7);
\draw [->, thick] (7) -- (6);
\draw [->, thick] (6) -- (2);
\draw [->, thick] (2) -- (7);
\draw [->, thick] (7) -- (3);
\draw [->, thick] (3) -- (8);
\draw [->, thick] (8) -- (4);
\draw [->, thick] (4) to [out=130, in=-45] (9);
\draw [->, thick] (9) -- (5);
\draw [->, thick] (5) to [bend left = 15] (4);
\draw [->, thick] (4) -- (3);
\draw [->, thick] (3) -- (2);
\draw [->, thick] (2) -- (1);

\draw [->, thick] (5) -- (1);
\draw [->, thick] (1) -- (9);
\draw [->, thick] (9) -- (6);
\draw [->, thick] (6) -- (12);
\draw [->, thick] (12) -- (10);
\draw [->, thick] (10) -- (14);
\draw [->, thick] (14) -- (13);

\end{tikzpicture}
\caption{Quiver $\Qc_4^{\overline w_0,c}$.}
\label{fig-Q0c}
\end{figure}

\subsubsection{Quiver mutations for $\phi_2$}
\label{subsec-muts-double2}

We now apply $\phi_2$ to the quiver $\Qc_n^{\overline w_0,c}$, or equivalently to the first $n$ rows of~$\Qc_n^{\mathrm{cone}}$. Let us arrange the vertices in the bottom $n$ rows of $\Qc_n^{\mathrm{cone}}$ as in Figure~\ref{fig-Q0c}, where $Y_i^i = X_{n+1}^i$ for $1 \le i \le n$, and $Y_{j+1}^i = X_j^i$ for $1 \le i \le j \le n$.
 After each step of $\phi_2$ we shall rearrange the vertices and rename them, so that the vertex in position $(c,r)$ is denoted by $Y^r_c$. The $k$-th step of the $j$-th wave now reads as follows:
\begin{itemize}
\item mutate at vertex $Y_k^k$;
\item mutate consecutively at vertices $Y_{k+1}^k$, $Y_{k+1}^{k-1}$, \dots, $Y_{k+1}^1$;
\item shift the vertex $Y_{k+1}^1$ to the position $(k+1,j+1-n)$;
\item shift vertices $Y_{k+1}^r$ to the position $(k+1,r-1)$ for $2 \le r \le k$;
\item shift the vertex $Y_k^k$ to the position $(k+1,k)$;
\item shift the vertex $Y_{n+2-j}^k$ to the position $(k,k)$;
\item rename vertices according to their new position.
\end{itemize}

\begin{figure}[h]
\begin{tikzpicture}[every node/.style={inner sep=0, minimum size=0.5cm, thick, draw, circle}, x=0.75cm, y=0.5cm]

\node (1) at (-2,0) {\tiny{$Y_1^1$}};
\node (2) at (0,0) {\tiny{$Y_2^1$}};
\node (3) at (2,0) {\tiny{$Y_3^1$}};
\node (4) at (-1,2) {\tiny{$Y_2^2$}};
\node (5) at (1,2) {\tiny{$Y_3^2$}};
\node (6) at (0,4) {\tiny{$Y_3^3$}};
\node (7) at (-1,6) {\tiny{$C_1^1$}};
\node (8) at (1,6) {\tiny{$C_2^1$}};
\node (9) at (-1,8) {\tiny{$C_1^2$}};
\node (10) at (1,8) {\tiny{$C_2^2$}};
\node (11) at (-1,10) {\tiny{$C_1^3$}};
\node (12) at (1,10) {\tiny{$C_2^3$}};
\node (13) at (-1,12) {\tiny{$C_1^4$}};
\node (14) at (1,12) {\tiny{$C_2^4$}};
\node (15) at (0,14) {\tiny{$X_1^1$}};
\node (16) at (-1,16) {\tiny{$X_1^2$}};
\node (17) at (1,16) {\tiny{$X_2^2$}};
\node (18) at (-2,18) {\tiny{$X_1^3$}};
\node (19) at (0,18) {\tiny{$X_2^3$}};
\node (20) at (2,18) {\tiny{$X_3^3$}};

\draw [->, thick] (3) -- (2);
\draw [->, thick] (2) -- (1);
\draw [->, thick] (1) -- (4);
\draw [->, thick] (4) -- (2);
\draw [->, thick] (2) -- (5);
\draw [->, thick] (5) -- (4);
\draw [->, thick] (4) -- (6);

\draw [->, thick] (6) to [out = -60, in=120] (5);
\draw [->, thick] (5) to [bend right = 21] (4);
\draw [->, thick] (4) to [out = -30, in=135] (3);
\draw [->, thick] (3) to [bend right = 21] (1);

\draw [->, thick] (8) -- (6);
\draw [->, thick] (6) to [out = 120, in=-60] (7);
\draw [->, thick] (7) -- (8);
\draw [->, thick] (8) -- (9);
\draw [->, thick] (9) -- (10);
\draw [->, thick] (10) -- (11);
\draw [->, thick] (11) -- (12);
\draw [->, thick] (12) -- (13);
\draw [->, thick] (13) -- (14);
\draw [->, thick] (14) -- (15);
\draw [->, thick] (15) -- (13);
\draw [->, thick] (13) to [bend left = 21] (14);
\draw [->, thick] (14) to [out = -135, in=45] (11);
\draw [->, thick] (11) to [bend left = 21] (12);
\draw [->, thick] (12) to [out = -135, in=45] (9);
\draw [->, thick] (9) to [bend left = 21] (10);
\draw [->, thick] (10) to [out = -135, in=45] (7);
\draw [->, thick] (7) to [bend left = 21] (8);

\draw [->, thick] (18) -- (19);
\draw [->, thick] (19) -- (20);
\draw [->, thick] (20) -- (17);
\draw [->, thick] (17) -- (19);
\draw [->, thick] (19) -- (16);
\draw [->, thick] (16) -- (17);
\draw [->, thick] (17) -- (15);

\draw [->, thick] (15) to [out = 120, in=-60] (16);
\draw [->, thick] (16) to [bend left = 21] (17);
\draw [->, thick] (17) to [out = 140, in=-40] (18);
\draw [->, thick] (18) to [bend left = 21] (20);

\end{tikzpicture}
\caption{Quiver $\Qc_4^{\mathrm{candy}}$.}
\label{fig-Qcandy}
\end{figure}

Finally, we arrive at the quiver $\Qc_n^{\mathrm{candy}}$, which is shown in Figure~\ref{fig-Qcandy} for $n=4$. As before, we applied horizontal shifts to each row of $\Qc_n^{\mathrm{candy}}$ and relabelled its vertices: the top and bottom $n-1$ rows consist of vertices $X_i^j$ and $Y_j^i$ with $1 \le i \le j \le n-1$ respectively, vertices in the middle $n$ rows are labelled $C_1^i$ and $C_2^i$, $1 \le i \le n$. The top and bottom $n-2$ rows of the quiver $\Qc_n^{\mathrm{candy}}$ form identical quivers $\Qc_{n-2}^{\underline c,\overline w_0}$, the identification is given by $X_i^j \to Y_{n-i}^{n-j}$. Note that the middle $n$ rows form a cyclic quiver corresponding to the word $\icox$.

\subsubsection{Quiver mutations for $\phi_3$}
\label{subsec-muts-single}

Let us consider factorization coordinates defined by the word $\i = \overline\i_{w_0}$ on the reduced double Bruhat cell $G^{w_0,e}/\Ad H$, and let $\overline\Qc_\i$ be the corresponding cyclic quiver. Recall that the quiver $\overline\Qc_\i$ is obtained from the non-cyclic quiver $\Qc_\i$ by amalgamating the first and the last vertex in each row. Therefore, $\overline\Qc_\i$ has a detached vertex, the one that correspond to the product of the arguments $x_n x_{s_n}$ of the two $H_n$'s in the decomposition
$$
H_1(x_1) \dots H_n(x_n)\prod_{i=1}^n\hr{F_iH_i(x_{s_i+1}) \dots F_1H_1(x_{s_i+i})} \qquad\text{where}\qquad s_i = n+i(i-1)/2.
$$
However, we shall use slight modifications of the corresponding cluster chart and the quiver. Namely, we shall replace the product $x_nx_{s_n}$ by $x_nx_{s_{n-1}}x_{s_n}$, that is multiply it additionally by the argument of the second $H_{n-1}$ in the above product. We shall denote the resulting quiver by $\Qc_n^{\overline w_0}$; it differs from $\overline\Qc_\i$ only in that the vertex under discussion is no longer detached but has exactly the same arrows as the one corresponding to the factorization coordinate~$x_{s_{n-1}}$. The quiver $\Qc_4^{\overline w_0}$ is shown in Figure~\ref{fig-Q0}, and the vertex of interest is labelled $X_4^4$ there.

\begin{figure}[h]
\begin{tikzpicture}[every node/.style={inner sep=0, minimum size=0.5cm, thick, draw, circle}, x=1.5cm, y=1.5cm]

\node (1) at (1,1) {\tiny $X_1^1$};
\node (2) at (2,1) {\tiny $X_2^1$};
\node (3) at (3,1) {\tiny $X_3^1$};
\node (4) at (4,1) {\tiny $X_4^1$};
\node (5) at (2,2) {\tiny $X_2^2$};
\node (6) at (3,2) {\tiny $X_3^2$};
\node (7) at (4,2) {\tiny $X_4^2$};
\node (8) at (3,3) {\tiny $X_3^3$};
\node (9) at (4,3) {\tiny $X_4^3$};
\node (10) at (4,4) {\tiny $X_4^4$};

\draw [->,thick] (4) -- (3);
\draw [->,thick] (3) -- (2);
\draw [->,thick] (2) -- (1);
\draw [->,thick] (7) -- (6);
\draw [->,thick] (6) to (5);

\draw [->,thick] (5) -- (2);
\draw [->,thick] (8) -- (6);
\draw [->,thick] (6) -- (3);
\draw [->,thick] (9) -- (7);
\draw [->,thick] (7) -- (4);

\draw [->,thick] (2) -- (6);
\draw [->,thick] (6) -- (9);
\draw [->,thick] (3) -- (7);

\draw [->,thick] (1) to [bend left = 25] (4);
\draw [->,thick] (4) to [out = 120, in =-45] (5);
\draw [->,thick] (5) to [bend left = 30] (7);
\draw [->,thick] (7) to [out = 120, in = -60] (8);

\draw [->,thick] (6) to [out = 70, in = -110] (10);
\draw [->,thick] (10) to [bend left = 30] (7);

\end{tikzpicture}
\caption{Quiver $\Qc_4^{\overline w_0}$.}
\label{fig-Q0}
\end{figure}

We are now ready to describe the quiver transformation $\phi_3$. We break it into a preliminary step followed by $n-2$ waves, where $j$-th wave consists of $j$ steps. In turn, there are $k+1$ mutations involved in the $k$-th step of the $j$-th wave. At the end of each step we as usual, shift the vertices and relabel them according to their position, so that vertex $X_c^r$ has coordinates $(c,r)$. The preliminary step reads as follows:
\begin{itemize}
\item mutate at vertices $X_i^i$ for $i = 1, \dots, n-1$;
\item move vertex $X_{n-1}^{n-1}$ to the position $(n,n+1)$ and relabel it accordingly.
\end{itemize}
\begin{remark}
Note that after applying this preliminary step, the full subquiver formed by the bottom $n-2$ rows is the cyclic quiver corresponding to the double word $\isf = \hr{\underline\i_c,\i_c,\overline\i_{w_0}}$, where $\i_{w_0}$, $\i_c$, and $\underline\i_c$ are all written in the smaller alphabet $\{\pm1,\ldots,\pm(n-2)\}$. 
This observation suggests an existence of a non-trivial Poisson morphism between the reduced double Bruhat cells $PGL_{n+1}^{w_0,e}/\Ad H$ and $PGL_{n-1}^{w_0,c}/\Ad_H$. We do not investigate this subject here, leaving it for a separate publication.
\end{remark}

Now the remaining part of the transformation $\phi_3$ consists of $n-2$ waves, and the $k$-th step of the $j$-th wave is as follows:
\begin{itemize}
\item mutate consecutively at vertices $X_{n-1-j+k}^{n-2-j+k}, \dots, X_n^{n-2-j+k}$;
\item move the vertex $X_n^{n-2-j+k}$ to the position $(n-1-j+k,n-1-j+k)$ if $k \ne j$ and to the position $(n,n+1+j)$ if $k=j$;
\item move vertex $X_c^{n-2-j+k}$ to the position $(c+1,n-2-j+k)$ for $n-2-j+k \le c \le n-1$;
\item relabel vertices according to their new position.
\end{itemize}

\begin{figure}[h]
\begin{tikzpicture}[every node/.style={inner sep=0, minimum size=0.5cm, thick, draw, circle}, x=0.75cm, y=0.5cm]

\node (1) at (-2,0) {\tiny{$X_1^1$}};
\node (2) at (0,0) {\tiny{$X_2^1$}};
\node (3) at (2,0) {\tiny{$X_3^1$}};
\node (4) at (-1,2) {\tiny{$X_2^2$}};
\node (5) at (1,2) {\tiny{$X_3^2$}};
\node (6) at (0,4) {\tiny{$X_3^3$}};
\node (7) at (0,5.5) {\tiny $X_3^4$};
\node (8) at (0,7) {\tiny $X_3^5$};
\node (9) at (0,8.5) {\tiny $X_3^6$};
\node (10) at (0,10) {\tiny $X_3^7$};

\draw [->, thick] (3) -- (2);
\draw [->, thick] (2) -- (1);
\draw [->, thick] (1) -- (4);
\draw [->, thick] (4) -- (2);
\draw [->, thick] (2) -- (5);
\draw [->, thick] (5) -- (4);
\draw [->, thick] (4) -- (6);

\draw [->, thick] (6) to [out = -60, in=120] (5);
\draw [->, thick] (5) to [bend right = 21] (4);
\draw [->, thick] (4) to [out = -30, in=135] (3);
\draw [->, thick] (3) to [bend right = 21] (1);

\draw[->, thick] (4) to [out = 70, in=-130] (7);
\draw[->, thick] (7) to [out = -50, in=110] (5);
\draw[->, thick] (4) to [out = 85, in=-130] (8);
\draw[->, thick] (8) to [out = -50, in=95] (5);
\draw[->, thick] (4) to [out = 100, in=-130] (9);
\draw[->, thick] (9) to [out = -50, in=80] (5);
\draw[->, thick] (4) to [out = 115, in=-130] (10);
\draw[->, thick] (10) to [out = -50, in=65] (5);

\end{tikzpicture}
\caption{Quiver $\Qc_4^{\mathrm{mitre}}$.}
\label{fig-Qmitre}
\end{figure}

The final result is a quiver that we denote by $\Qc_n^{\mathrm{mitre}}$. Figure~\ref{fig-Qmitre} shows the quiver $\Qc_4^{\mathrm{mitre}}$, where we again applied horizontal shifts in each row and relabelled the vertices to make the quiver more readable. Note that the bottom $n-2$ rows of $\Qc_n^{\mathrm{mitre}}$ form the quiver $\Qc_{n-2}^{\underline c,\overline w_0}$, while the vertices $X_{n-1}^c$ with $n-1 \le c \le 2n-1$ are attached to the rest of the quiver in an identical way.

\subsection{Symmetric cluster realization of $U_q(\mathfrak{sl}_{n+1})$}
\label{single-copy-muts}

Let us now observe that the part of the quiver $\Qc_n^{\mathrm{sf}}$ that sits above the dashed line in Figure~\ref{D4-sf} is identical to the quiver $\Qc_n^{\overline w_0}$, except that the vertex 28 is disconnected instead of being $2$-valent as in Figure~\ref{fig-Q0}. More precisely, the two quivers are related by the simple basis change discussed in part (2) of Lemma~\ref{geom-mut}. We can thus perform the sequence of mutations described in Section~\ref{subsec-muts-single} on the universal cover of a punctured disk, thereby obtaining a new seed $\wdt\Theta^{\mathrm{sym}}_n$ and networks $\Nc_n^{\mathrm{sym}_\pm}$. Note that the bottom parts of $\Nc_n^{\mathrm{sym}_\pm}$ coincide with those of $\Nc_n^{\mathrm{sf}_\pm}$. For $n=4$, the top part of the network $\Nc_n^{\mathrm{sym}_-}$ is shown in Figure~\ref{net-sym}, and the dashed line shows where it is glued to the bottom part of $\Nc_n^{\mathrm{sym}_-}$. Now, we can take the quotient of the seed $\wdt\Theta^{\mathrm{sym}}_n$ by the deck group and arrive at a seed $\Theta^{\mathrm{sym}}_{n}$, whose quiver $\Qsym$ is shown in Figure~\ref{Q4-sym} for $n=4$.

\begin{figure}[h]
\begin{tikzpicture}[every node/.style={inner sep=0, minimum size=0.4cm, thick, draw, circle}, x=0.75cm, y=0.45cm]

\node (1) [rectangle] at (-1,0) {\scriptsize 1};
\node (2) at (0,2) {\scriptsize 2};
\node (3) [rectangle] at (1,0) {\scriptsize 3};
\node (4) [rectangle] at (-2,2) {\scriptsize 4};
\node (5) at (-1,4) {\scriptsize 5};
\node (6) at (0,6) {\scriptsize 6};
\node (7) at (1,4) {\scriptsize 7};
\node (8) [rectangle] at (2,2) {\scriptsize 8};
\node (9) [rectangle] at (-3,4) {\scriptsize 9};
\node (10) at (-2,6) {\scriptsize 10};
\node (11) at (-1,8) {\scriptsize 11};
\node (12) at (2,10) {\scriptsize 12};
\node (13) at (1,8) {\scriptsize 13};
\node (14) at (2,6) {\scriptsize 14};
\node (15) [rectangle] at (3,4) {\scriptsize 15};
\node (16) [rectangle] at (-4,6) {\scriptsize 16};
\node (17) at (-3,8) {\scriptsize 17};
\node (18) at (0,10) {\scriptsize 18};
\node (19) at (0,14) {\scriptsize 19};
\node (20) at (0,17) {\scriptsize 20};
\node (21) at (0,15) {\scriptsize 21};
\node (22) at (-1,12) {\scriptsize 22};
\node (23) at (3,8) {\scriptsize 23};
\node (24) [rectangle] at (4,6) {\scriptsize 24};
\node (25) at (-2,10) {\scriptsize 25};
\node (26) at (1,12) {\scriptsize 26};
\node (27) at (0,16) {\scriptsize 27};
\node (28) at (0,18) {\scriptsize 28};

\draw [->, thick] (17) -- (25);
\draw [->, thick] (25) -- (11);
\draw [->, thick] (11) -- (18);
\draw [->, thick] (18) -- (13);
\draw [->, thick] (13) -- (12);
\draw [->, thick] (12) -- (18);
\draw [->, thick] (18) -- (25);
\draw [->, thick] (25) -- (22);
\draw [->, thick] (22) -- (18);
\draw [->, thick] (18) -- (26);
\draw [->, thick] (26) -- (22);
\draw [->, thick] (22) -- (19);

\draw [->, thick] (19) to [out = -45, in=120] (26);
\draw [->, thick] (26) to [bend right = 21] (22);
\draw [->, thick] (22) to [out = -30, in=135] (12);
\draw [->, thick] (12) to [bend right = 18] (25);
\draw [->, thick] (25) to [out = -25, in=150] (23);
\draw [->, thick] (23) to [bend right = 15] (17);


\draw[BrickRed, dashed] (-4,8) to (4,8);

\draw [->,thick] (1) -- (3);

\draw [->,thick] (4) -- (2);
\draw [->,thick] (2) -- (8);

\draw [->,thick] (9) -- (5);
\draw [->,thick] (5) -- (7);
\draw [->,thick] (7) -- (15);

\draw [->,thick] (16) -- (10);
\draw [->,thick] (10) -- (6);
\draw [->,thick] (6) -- (14);
\draw [->,thick] (14) -- (24);

\draw [->,thick] (3) -- (2);
\draw [->,thick] (2) -- (5);
\draw [->,thick] (5) -- (10);
\draw [->,thick] (10) -- (17);

\draw [->,thick] (23) -- (14);
\draw [->,thick] (14) -- (7);
\draw [->,thick] (7) -- (2);
\draw [->,thick] (2) -- (1);

\draw [->,thick] (8) -- (7);
\draw [->,thick] (7) -- (6);
\draw [->,thick] (6) -- (11);
\draw [->,thick] (13) -- (6);
\draw [->,thick] (6) -- (5);
\draw [->,thick] (5) -- (4);

\draw [->,thick] (15) -- (14);
\draw [->,thick] (14) -- (13);
\draw [->,thick] (11) -- (10);
\draw [->,thick] (10) -- (9);

\draw [->,thick] (24) -- (23);
\draw [->,thick] (17) -- (16);

\draw [->,thick,dashed] (24) -- (15);
\draw [->,thick,dashed] (15) -- (8);
\draw [->,thick,dashed] (8) -- (3);
\draw [->,thick,dashed] (1) -- (4);
\draw [->,thick,dashed] (4) -- (9);
\draw [->,thick,dashed] (9) -- (16);

\end{tikzpicture}
\caption{Quiver $\Qc_4^{\mathrm{sym}}$.}
\label{Q4-sym}
\end{figure}

\begin{figure}[h]
\begin{tikzpicture}[every node/.style={inner sep=0, minimum size=0.15cm, circle, draw}, x=0.45cm,y=0.45cm]

\foreach \i in {3,9,...,30}
{
	\node[draw=none, text=blue] at (\i-2,-0.5) {\tiny 13};
	\node[draw=none, text=blue] at (\i-0.5,-0.5) {\tiny 23};
	\node[draw=none, text=blue] at (\i+0.5,-0.5) {\tiny 17};
	\node[draw=none, text=blue] at (\i+2,-0.5) {\tiny 11};

	\node[draw=none, text=blue] at (\i-3,0.5) {\tiny 18};
	\node[draw=none, text=blue] at (\i-1.5,0.5) {\tiny 12};
	\node[draw=none, text=blue] at (\i+0.5,0.5) {\tiny 25};

	\node[draw=none, text=blue] at (\i-2.5,1.5) {\tiny 26};
	\node[draw=none, text=blue] at (\i+0.5,1.5) {\tiny 22};

	\node[draw=none, text=blue] at (\i,2.5) {\tiny 19};
	\node[draw=none, text=blue] at (\i-3,3.5) {\tiny 21};
	\node[draw=none, text=blue] at (\i,4.5) {\tiny 27};
	\node[draw=none, text=blue] at (\i-3,5.5) {\tiny 20};
	\node[draw=none, text=blue] at (\i,6.5) {\tiny 28};
}

\node[draw=none, text=blue] at (30,0.5) {\tiny 18};
\node[draw=none, text=blue] at (30,3.5) {\tiny 21};
\node[draw=none, text=blue] at (30,5.5) {\tiny 20};

\foreach \i in {3,9,...,30}
	\node[fill=white] (a\i) at (\i,-1) {};

\foreach \i in {0,2,...,30}
 \node[fill=white] (b\i) at (\i,0) {};
\foreach \i in {1,3,...,30}
 \node[fill=black] (b\i) at (\i,0) {};

\foreach \i in {0,6,...,30}
	\node[fill=black] (c\i) at (\i,1) {};
\foreach \i in {1,7,...,30}
	\node[fill=white] (c\i) at (\i,1) {};
\foreach \i in {2,8,...,30}
	\node[fill=black] (c\i) at (\i,1) {};
\foreach \i in {5,11,...,30}
	\node[fill=white] (c\i) at (\i,1) {};

\foreach \i in {0,6,...,30}
	\node[fill=white] (d\i) at (\i,2) {};
\foreach \i in {1,7,...,30}
	\node[fill=black] (d\i) at (\i,2) {};

\foreach \i in {0,6,...,30}
	\node[fill=white] (e\i) at (\i,3) {};
\foreach \i in {3,9,...,30}
	\node[fill=black] (e\i) at (\i,3) {};

\foreach \i in {0,6,...,30}
	\node[fill=black] (f\i) at (\i,4) {};
\foreach \i in {3,9,...,30}
	\node[fill=white] (f\i) at (\i,4) {};

\foreach \i in {0,6,...,30}
	\node[fill=white] (g\i) at (\i,5) {};
\foreach \i in {3,9,...,30}
	\node[fill=black] (g\i) at (\i,5) {};

\foreach \i in {0,6,...,30}
	\node[fill=black] (h\i) at (\i,6) {};
\foreach \i in {3,9,...,30}
	\node[fill=white] (h\i) at (\i,6) {};

\foreach \i in {0,6,...,30}
	\node[fill=white] (i\i) at (\i,7) {};

\draw[dashed, BrickRed, thick] (3,-2) to (15,-2);

\draw [->] (-1,0) to (b0);
\draw [->] (b0) to (b1);
\draw [->] (b1) to (b2);
\draw [<-, thick, OliveGreen] (b2) to (b3);
\draw [->] (b3) to (b4);
\draw [->, thick, OliveGreen] (b4) to (b5);
\draw [->] (b5) to (b6);
\draw [->, thick, OliveGreen] (b6) to (b7);
\draw [->] (b7) to (b8);
\draw [->, thick, OliveGreen] (b8) to (b9);
\draw [->, thick, OliveGreen] (b9) to (b10);
\draw [<-] (b10) to (b11);
\draw [->, thick, OliveGreen] (b11) to (b12);
\draw [<-] (b12) to (b13);
\draw [->, thick, OliveGreen] (b13) to (b14);
\draw [<-] (b14) to (b15);
\draw [<-, thick, OliveGreen] (b15) to (b16);
\draw [<-, thick, OliveGreen] (b16) to (b17);
\draw [<-, thick, OliveGreen] (b17) to (b18);
\draw [<-, thick, OliveGreen] (b18) to (b19);
\draw [<-, thick, OliveGreen] (b19) to (b20);
\draw [<-] (b20) to (b21);
\draw [->] (b21) to (b22);
\draw [->] (b22) to (b23);
\draw [->] (b23) to (b24);
\draw [->] (b24) to (b25);
\draw [->] (b25) to (b26);
\draw [<-] (b26) to (b27);
\draw [->] (b27) to (b28);
\draw [->] (b28) to (b29);
\draw [->] (b29) to (b30);
\draw [->] (b30) to (31,0);

\draw [->] (-1,1) to (c0);
\draw [->] (c0) to (c1);
\draw [<-, thick, OliveGreen] (c1) to (c2);
\draw [->] (c2) to (c5);
\draw [->, thick, OliveGreen] (c5) to (c6);
\draw [->] (c6) to (c7);
\draw [->, thick, OliveGreen] (c7) to (c8);
\draw [->, thick, OliveGreen] (c8) to (c11);
\draw [<-] (c11) to (c12);
\draw [->, thick, OliveGreen] (c12) to (c13);
\draw [<-] (c13) to (c14);
\draw [<-] (c14) to (c17);
\draw [<-] (c17) to (c18);
\draw [<-] (c18) to (c19);
\draw [<-] (c19) to (c20);
\draw [<-, thick, OliveGreen] (c20) to (c23);
\draw [<-, thick, OliveGreen] (c23) to (c24);
\draw [<-, thick, OliveGreen] (c24) to (c25);
\draw [<-] (c25) to (c26);
\draw [<-] (c26) to (c29);
\draw [<-] (c29) to (c30);
\draw [<-] (c30) to (31,1);

\draw [->] (-1,2) to (d0);
\draw [<-, thick, OliveGreen] (d0) to (d1);
\draw [->] (d1) to (d6);
\draw [->, thick, OliveGreen] (d6) to (d7);
\draw [->, thick, OliveGreen] (d7) to (d12);
\draw [<-] (d12) to (d13);
\draw [<-] (d13) to (d18);
\draw [<-] (d18) to (d19);
\draw [<-] (d19) to (d24);
\draw [<-] (d24) to (d25);
\draw [<-, thick, OliveGreen] (d25) to (d30);
\draw [<-] (d30) to (31,2);

\draw [->] (-1,3) to (e0);
\draw [->, thick, OliveGreen] (e0) to (e3);
\draw [->] (e3) to (e6);
\draw [<-] (e6) to (e9);
\draw [->] (e9) to (e12);
\draw [<-] (e12) to (e15);
\draw [->] (e15) to (e18);
\draw [<-] (e18) to (e21);
\draw [->] (e21) to (e24);
\draw [<-] (e24) to (e27);
\draw [->, thick, OliveGreen] (e27) to (e30);
\draw [<-] (e30) to (31,3);

\draw [->] (-1,4) to (f0);
\draw [->] (f0) to (f3);
\draw [->, thick, OliveGreen] (f3) to (f6);
\draw [->] (f6) to (f9);
\draw [<-] (f9) to (f12);
\draw [->] (f12) to (f15);
\draw [<-] (f15) to (f18);
\draw [->] (f18) to (f21);
\draw [<-] (f21) to (f24);
\draw [->, thick, OliveGreen] (f24) to (f27);
\draw [<-] (f27) to (f30);
\draw [<-] (f30) to (31,4);

\draw [->] (-1,5) to (g0);
\draw [->] (g0) to (g3);
\draw [->] (g3) to (g6);
\draw [->, thick, OliveGreen] (g6) to (g9);
\draw [->] (g9) to (g12);
\draw [<-] (g12) to (g15);
\draw [->] (g15) to (g18);
\draw [<-] (g18) to (g21);
\draw [->, thick, OliveGreen] (g21) to (g24);
\draw [<-] (g24) to (g27);
\draw [<-] (g27) to (g30);
\draw [<-] (g30) to (31,5);

\draw [->] (-1,6) to (h0);
\draw [->] (h0) to (h3);
\draw [->] (h3) to (h6);
\draw [->] (h6) to (h9);
\draw [->, thick, OliveGreen] (h9) to (h12);
\draw [->] (h12) to (h15);
\draw [<-] (h15) to (h18);
\draw [->, thick, OliveGreen] (h18) to (h21);
\draw [<-] (h21) to (h24);
\draw [<-] (h24) to (h27);
\draw [<-] (h27) to (h30);
\draw [<-] (h30) to (31,6);

\draw [->] (-1,7) to (i0);
\draw [->] (i0) to (i6);
\draw [->] (i6) to (i12);
\draw [->, thick, OliveGreen] (i12) to (i18);
\draw [<-] (i18) to (i24);
\draw [<-] (i24) to (i30);
\draw [<-] (i30) to (31,7);

\draw [->] (2,-2) to [out=90, in=-150] (a3);
\draw [->, thick, OliveGreen] (8,-2) to [out=90, in=-150] (a9);
\draw [<-, thick, OliveGreen] (14,-2) to [out=90, in=-150] (a15);
\draw [->] (20,-2) to [out=90, in=-150] (a21);
\draw [->] (26,-2) to [out=90, in=-150] (a27);

\draw [->, thick, OliveGreen] (4,-2) to [out=90, in=-30] (a3);
\draw [<-, thick, OliveGreen] (10,-2) to [out=90, in=-30] (a9);
\draw [->] (16,-2) to [out=90, in=-30] (a15);
\draw [->] (22,-2) to [out=90, in=-30] (a21);
\draw [->] (28,-2) to [out=90, in=-30] (a27);

\draw [->] (0,-2) to (b0);
\draw [->] (1,-2) to [out=90,in=-90] (b2);
\draw [->, thick, OliveGreen] (5,-2) to [out=90,in=-90] (b4);
\draw [->, thick, OliveGreen] (6,-2) to (b6);
\draw [->, thick, OliveGreen] (7,-2) to [out=90,in=-90] (b8);
\draw [<-, thick, OliveGreen] (11,-2) to [out=90,in=-90] (b10);
\draw [<-, thick, OliveGreen] (12,-2) to (b12);
\draw [<-, thick, OliveGreen] (13,-2) to [out=90,in=-90] (b14);
\draw [->] (17,-2) to [out=90,in=-90] (b16);
\draw [->] (18,-2) to (b18);
\draw [->] (19,-2) to [out=90,in=-90] (b20);
\draw [->] (23,-2) to [out=90,in=-90] (b22);
\draw [->] (24,-2) to (b24);
\draw [->] (25,-2) to [out=90,in=-90] (b26);
\draw [->] (29,-2) to [out=90,in=-90] (b28);
\draw [->] (30,-2) to (b30);

\draw [->, thick, OliveGreen] (a3) to (b3);
\draw [<-] (a9) to (b9);
\draw [<-, thick, OliveGreen] (a15) to (b15);
\draw [->] (a21) to (b21);
\draw [->] (a27) to (b27);

\draw [->] (b1) to (c1);
\draw [->, thick, OliveGreen] (b2) to (c2);
\draw [->, thick, OliveGreen] (b5) to (c5);
\draw [->, thick, OliveGreen] (b7) to (c7);
\draw [<-] (b8) to (c8);
\draw [<-, thick, OliveGreen] (b11) to (c11);
\draw [<-, thick, OliveGreen] (b13) to (c13);
\draw [<-] (b14) to (c14);
\draw [->] (b17) to (c17);
\draw [->] (b19) to (c19);
\draw [<-, thick, OliveGreen] (b20) to (c20);
\draw [->] (b23) to (c23);
\draw [->] (b25) to (c25);
\draw [->] (b26) to (c26);
\draw [->] (b29) to (c29);

\draw [->] (c0) to (d0);
\draw [->, thick, OliveGreen] (c1) to (d1);
\draw [->, thick, OliveGreen] (c6) to (d6);
\draw [<-] (c7) to (d7);
\draw [<-, thick, OliveGreen] (c12) to (d12);
\draw [<-] (c13) to (d13);
\draw [->] (c18) to (d18);
\draw [<-] (c19) to (d19);
\draw [->] (c24) to (d24);
\draw [<-, thick, OliveGreen] (c25) to (d25);
\draw [->] (c30) to (d30);

\draw [->, thick, OliveGreen] (d0) to (e0);
\draw [<-] (d6) to (e6);
\draw [<-] (d12) to (e12);
\draw [<-] (d18) to (e18);
\draw [<-] (d24) to (e24);
\draw [<-, thick, OliveGreen] (d30) to (e30);

\draw [->, thick, OliveGreen] (e3) to (f3);
\draw [<-] (e9) to (f9);
\draw [<-] (e15) to (f15);
\draw [<-] (e21) to (f21);
\draw [<-, thick, OliveGreen] (e27) to (f27);

\draw [->] (f0) to (g0);
\draw [->, thick, OliveGreen] (f6) to (g6);
\draw [<-] (f12) to (g12);
\draw [<-] (f18) to (g18);
\draw [<-, thick, OliveGreen] (f24) to (g24);
\draw [->] (f30) to (g30);

\draw [->] (g3) to (h3);
\draw [->, thick, OliveGreen] (g9) to (h9);
\draw [<-] (g15) to (h15);
\draw [<-, thick, OliveGreen] (g21) to (h21);
\draw [->] (g27) to (h27);

\draw [->] (h0) to (i0);
\draw [->] (h6) to (i6);
\draw [->, thick, OliveGreen] (h12) to (i12);
\draw [<-, thick, OliveGreen] (h18) to (i18);
\draw [->] (h24) to (i24);
\draw [->] (h30) to (i30);

\end{tikzpicture}
\caption{Top part of the network $\Nc_4^{\mathrm{sym}_-}$.}
\label{net-sym}
\end{figure}

Denote by $\Dsym$ the based quantum torus algebra corresponding to the quiver $\Qc_n^{\mathrm{sym}}$. Let us also denote by $\Phi^{\mathrm{sym}}$ the unitary operator on $\mathcal{P}_{\lambda}$ given by the composite of all quantum mutations taking us from the seed $\Theta_n^{\mathrm{sf}}$ to $\Theta_n^{\mathrm{sym}}$. Then the combinatorial description of $\iota$ given in Corollary~\ref{cor-folded}, together with Propositions~\ref{measure-invariance}~and~\ref{mut-descent}, implies

\begin{cor}
\label{gsv-measure}
Let
$$
M_n^{\mathrm{sym}_\pm}(i,j;k,l) = \mathrm{pr}\hr{Z_{\mathcal{N}^{\mathrm{sym}_\pm}_n}(t_i,h_j;t_k,h_l)}
$$
be the boundary measurements in the networks $\Nc_n^{\mathrm{sym}_\pm}$, and set
$$
\iota_{\mathrm{sym}} = \Ad_{\Phi^{\mathrm{sym}}\circ\Phi^{\mathrm{sf}}}\cdot\iota.
$$
Then we have
\begin{align*}
&\iota_{\mathrm{sym}}(K_i) = {M_n^{\mathrm{sym}_+}(i+1,i+1;i,i)}, &
&\iota_{\mathrm{sym}}(E_i) = {M_n^{\mathrm{sym}_+}(i+1,i;i,i)}, \\
&\iota_{\mathrm{sym}}(K'_i) = {M_n^{\mathrm{sym}_-}(i,i;i+1,i+1)}, &
&\iota_{\mathrm{sym}}(F_i) = {M_n^{\mathrm{sym}_-}(i,i+1;i+1,i+1)}.
\end{align*}
\end{cor}

Let us order the disconnected nodes of the quiver $\Qsym$ in the same way as the faces of networks $\Nc_n^{\mathrm{sym}_\pm}$, c.f. Figures~\ref{Q4-sym} and~\ref{net-sym}. Now, we denote by $\omega_0,\dots,\omega_n$ the quantum torus algebra elements corresponding to the top $n+1$ nodes of the quiver $\Qsym$ counting bottom to top. For example, in Figure~\ref{Q4-sym} we get
$$
\omega_0 = X_{19}, \qquad
\omega_1 = X_{21}, \qquad
\omega_2 = X_{27}, \qquad
\omega_3 = X_{20}, \qquad
\omega_4 = X_{28}.
$$
Note that $\omega_0$ comes from the only 2-valent vertex in the quiver $\Qsym$, while $\omega_1, \dots, \omega_n$ come from the disconnected ones. Then the central elements $\Omega_j\in\Dstd$ defined in~\eqref{single-center-gens} take the form 
$$ 
\Omega_j = \omega_j \qquad\text{for}\qquad j = 1, \dots, n.
$$
This is most easily seen using the network description of $\Omega_j$ given in Remark~\ref{single-casimir-rmk}; one checks that each of the ``rows'' of $\Nc_n^{\mathrm{std}_\pm}$ described in the Remark becomes a single face of the mutated network $\Nc_n^{\mathrm{sym}_\pm}$. 

Consider the subquiver $\Qc^{\mathrm{sym}\dag}_{n}$ of $\Qc^{\mathrm{sym}}_{n}$ obtained by forgetting all disconnected vertices. Corollary~\ref{gsv-measure} implies that the center of the corresponding quantum torus algebra $\Dc_n^{\mathrm{sym}\dag}$ is the Laurent polynomial ring in the $n$ generators
$$
\iota\hr{K_1K'_1}\prod_{j=1}^{n}\omega_j^{j-n-1}
\qquad\text{and}\qquad
\iota\hr{K_iK'_i} \qquad\text{for}\qquad i = 2, \dots, n.
$$
We denote by $\Vc_n$ the positive representation of $\Dc_n^{\mathrm{sym}\dag}$ in which these generators act by the identity. Choosing a polarization, we can identify $\Vc_n$ with $L^2\hr{\R^{\frac{n(n+1)}{2}}}$.

We also consider a subquiver $\Qc_n^{\mathrm{com}}$ of $\Qc_n^{\mathrm{sym}}$ that consists of the only 2-valent vertex and $n$ disconnected vertices, so that there are $n+1$ vertices in $\Qc_n^{\mathrm{com}}$ and no arrows. The corresponding subalgebra $\Dc_n^{\mathrm{com}} \subset \Dsym$ is commutative and is generated by $\omega_0, \omega_1,\dots, \omega_n$. For any real numbers $\lambda_0, \lambda_1,\dots, \lambda_n$ satisfying $\sum_{k=0}^n\lambda_k=0$, we denote by $\C_\la$ the 1-dimensional positive representation of $\Dc_n^{\mathrm{com}}$ in which 
\beq
\label{spec}
\omega_0 \longmapsto e^{2\pi\hbar\lambda_0}
\qquad\text{and}\qquad
\omega_k\mapsto e^{2\pi\hbar\hr{\lambda_k-\lambda_{k-1}}}
\qquad\text{for}\qquad k=1, \dots, n.
\eeq
Given a quiver $\Qc$ we abuse notation and write $\rk(\Qc)$ for the rank of its adjacency matrix. Then
$$
\rk\hr{\Qc_n^{\mathrm{sym}}} = \rk\hr{\Qc_n^{\mathrm{sym}\dag}}+\rk\hr{\Qc_n^{\mathrm{com}}},
$$
we can apply Lemma~\ref{rep-concatenation} and obtain
\beq
\label{single-concat}
\Pc_\la \simeq \Vc_n \otimes \C_\la.
\eeq
One readily observes that the expressions for $E_1,\ldots, E_n$ and $F_2,\ldots, F_n$ given in Corollary~\ref{gsv-measure} do not depend on the cluster variables in the commutative subalgebra $\Dc_n^{\mathrm{com}} \subset \Dsym$, and thus act by the identity on the second tensor factor in~\eqref{single-concat}. The formula for $F_1$, however, does involve variables from $\Dc_n^{\mathrm{com}}$, and encodes all information about the central character of the positive representation.

\begin{figure}[h]
\begin{tikzpicture}[every node/.style={inner sep=0, minimum size=0.15cm, circle, draw}, x=0.45cm,y=0.45cm]

\foreach \i in {3,9,...,30}
{
	\node[draw=none, text=blue] at (\i,2.5) {\tiny 19};
	\node[draw=none, text=blue] at (\i-3,3.5) {\tiny 21};
	\node[draw=none, text=blue] at (\i,4.5) {\tiny 27};
	\node[draw=none, text=blue] at (\i-3,5.5) {\tiny 20};
	\node[draw=none, text=blue] at (\i,6.5) {\tiny 28};
}

\node[draw=none, text=blue] at (30,3.5) {\tiny 21};
\node[draw=none, text=blue] at (30,5.5) {\tiny 20};

\foreach \i in {0,6,...,30}
	\node[fill=white] (e\i) at (\i,3) {};
\foreach \i in {3,9,...,30}
	\node[fill=black] (e\i) at (\i,3) {};

\foreach \i in {0,6,...,30}
	\node[fill=black] (f\i) at (\i,4) {};
\foreach \i in {3,9,...,30}
	\node[fill=white] (f\i) at (\i,4) {};

\foreach \i in {0,6,...,30}
	\node[fill=white] (g\i) at (\i,5) {};
\foreach \i in {3,9,...,30}
	\node[fill=black] (g\i) at (\i,5) {};

\foreach \i in {0,6,...,30}
	\node[fill=black] (h\i) at (\i,6) {};
\foreach \i in {3,9,...,30}
	\node[fill=white] (h\i) at (\i,6) {};

\foreach \i in {0,6,...,30}
	\node[fill=white] (i\i) at (\i,7) {};

\node[minimum size=0.4cm, thick] (d0) at (0,1.7) {\scriptsize $v_0$};
\node[minimum size=0.4cm, thick] (d6) at (6,1.7) {\scriptsize $v_1$};
\node[minimum size=0.4cm, thick] (d12) at (12,1.7) {\scriptsize $v_2$};
\node[minimum size=0.4cm, thick] (d18) at (18,1.7) {\scriptsize $v_3$};
\node[minimum size=0.4cm, thick] (d24) at (24,1.7) {\scriptsize $v_4$};
\node[minimum size=0.4cm, thick] (d30) at (30,1.7) {\scriptsize $v_5$};

\draw [->] (-1,3) to (e0);
\draw [->, thick, OliveGreen] (e0) to (e3);
\draw [->] (e3) to (e6);
\draw [<-] (e6) to (e9);
\draw [->] (e9) to (e12);
\draw [<-] (e12) to (e15);
\draw [->] (e15) to (e18);
\draw [<-] (e18) to (e21);
\draw [->] (e21) to (e24);
\draw [<-] (e24) to (e27);
\draw [->, thick, OliveGreen] (e27) to (e30);
\draw [<-] (e30) to (31,3);

\draw [->] (-1,4) to (f0);
\draw [->] (f0) to (f3);
\draw [->, thick, OliveGreen] (f3) to (f6);
\draw [->] (f6) to (f9);
\draw [<-] (f9) to (f12);
\draw [->] (f12) to (f15);
\draw [<-] (f15) to (f18);
\draw [->] (f18) to (f21);
\draw [<-] (f21) to (f24);
\draw [->, thick, OliveGreen] (f24) to (f27);
\draw [<-] (f27) to (f30);
\draw [<-] (f30) to (31,4);

\draw [->] (-1,5) to (g0);
\draw [->] (g0) to (g3);
\draw [->] (g3) to (g6);
\draw [->, thick, OliveGreen] (g6) to (g9);
\draw [->] (g9) to (g12);
\draw [<-] (g12) to (g15);
\draw [->] (g15) to (g18);
\draw [<-] (g18) to (g21);
\draw [->, thick, OliveGreen] (g21) to (g24);
\draw [<-] (g24) to (g27);
\draw [<-] (g27) to (g30);
\draw [<-] (g30) to (31,5);

\draw [->] (-1,6) to (h0);
\draw [->] (h0) to (h3);
\draw [->] (h3) to (h6);
\draw [->] (h6) to (h9);
\draw [->, thick, OliveGreen] (h9) to (h12);
\draw [->] (h12) to (h15);
\draw [<-] (h15) to (h18);
\draw [->, thick, OliveGreen] (h18) to (h21);
\draw [<-] (h21) to (h24);
\draw [<-] (h24) to (h27);
\draw [<-] (h27) to (h30);
\draw [<-] (h30) to (31,6);

\draw [->] (-1,7) to (i0);
\draw [->] (i0) to (i6);
\draw [->] (i6) to (i12);
\draw [->, thick, OliveGreen] (i12) to (i18);
\draw [<-] (i18) to (i24);
\draw [<-] (i24) to (i30);
\draw [<-] (i30) to (31,7);

\draw [->, thick, OliveGreen] (d0) to (e0);
\draw [<-] (d6) to (e6);
\draw [<-] (d12) to (e12);
\draw [<-] (d18) to (e18);
\draw [<-] (d24) to (e24);
\draw [<-, thick, OliveGreen] (d30) to (e30);

\draw [->, thick, OliveGreen] (e3) to (f3);
\draw [<-] (e9) to (f9);
\draw [<-] (e15) to (f15);
\draw [<-] (e21) to (f21);
\draw [<-, thick, OliveGreen] (e27) to (f27);

\draw [->] (f0) to (g0);
\draw [->, thick, OliveGreen] (f6) to (g6);
\draw [<-] (f12) to (g12);
\draw [<-] (f18) to (g18);
\draw [<-, thick, OliveGreen] (f24) to (g24);
\draw [->] (f30) to (g30);

\draw [->] (g3) to (h3);
\draw [->, thick, OliveGreen] (g9) to (h9);
\draw [<-] (g15) to (h15);
\draw [<-, thick, OliveGreen] (g21) to (h21);
\draw [->] (g27) to (h27);

\draw [->] (h0) to (i0);
\draw [->] (h6) to (i6);
\draw [->, thick, OliveGreen] (h12) to (i12);
\draw [<-, thick, OliveGreen] (h18) to (i18);
\draw [->] (h24) to (i24);
\draw [->] (h30) to (i30);

\end{tikzpicture}
\caption{Network $\Bc_4$.}
\label{net-B}
\end{figure}

\begin{figure}[h]
\begin{tikzpicture}[every node/.style={inner sep=0, minimum size=0.15cm, circle, draw}, x=0.45cm,y=0.45cm]

\foreach \i in {3,9,...,30}
{
	\node[draw=none, text=blue] at (\i-2,-0.5) {\tiny 13};
	\node[draw=none, text=blue] at (\i-0.5,-0.5) {\tiny 23};
	\node[draw=none, text=blue] at (\i+0.5,-0.5) {\tiny 17};
	\node[draw=none, text=blue] at (\i+2,-0.5) {\tiny 11};

	\node[draw=none, text=blue] at (\i-3,0.5) {\tiny 18};
	\node[draw=none, text=blue] at (\i-1.5,0.5) {\tiny 12};
	\node[draw=none, text=blue] at (\i+0.5,0.5) {\tiny 25};

	\node[draw=none, text=blue] at (\i-2.5,1.5) {\tiny 26};
	\node[draw=none, text=blue] at (\i+0.5,1.5) {\tiny 22};
}

\node[draw=none, text=blue] at (30,0.5) {\tiny 18};

\node[minimum size=0.4cm, thick] (e0) at (0,3.3) {\scriptsize $v_0$};
\node[minimum size=0.4cm, thick] (e6) at (6,3.3) {\scriptsize $v_1$};
\node[minimum size=0.4cm, thick] (e12) at (12,3.3) {\scriptsize $v_2$};
\node[minimum size=0.4cm, thick] (e18) at (18,3.3) {\scriptsize $v_3$};
\node[minimum size=0.4cm, thick] (e24) at (24,3.3) {\scriptsize $v_4$};
\node[minimum size=0.4cm, thick] (e30) at (30,3.3) {\scriptsize $v_5$};

\foreach \i in {3,9,...,30}
	\node[fill=white] (a\i) at (\i,-1) {};

\foreach \i in {0,2,...,30}
 \node[fill=white] (b\i) at (\i,0) {};
\foreach \i in {1,3,...,30}
 \node[fill=black] (b\i) at (\i,0) {};

\foreach \i in {0,6,...,30}
	\node[fill=black] (c\i) at (\i,1) {};
\foreach \i in {1,7,...,30}
	\node[fill=white] (c\i) at (\i,1) {};
\foreach \i in {2,8,...,30}
	\node[fill=black] (c\i) at (\i,1) {};
\foreach \i in {5,11,...,30}
	\node[fill=white] (c\i) at (\i,1) {};

\foreach \i in {0,6,...,30}
	\node[fill=white] (d\i) at (\i,2) {};
\foreach \i in {1,7,...,30}
	\node[fill=black] (d\i) at (\i,2) {};

\draw[dashed, BrickRed, thick] (3,-2) to (15,-2);

\draw [->] (-1,0) to (b0);
\draw [->] (b0) to (b1);
\draw [->] (b1) to (b2);
\draw [<-, thick, OliveGreen] (b2) to (b3);
\draw [->] (b3) to (b4);
\draw [->, thick, OliveGreen] (b4) to (b5);
\draw [->] (b5) to (b6);
\draw [->, thick, OliveGreen] (b6) to (b7);
\draw [->] (b7) to (b8);
\draw [->, thick, OliveGreen] (b8) to (b9);
\draw [->, thick, OliveGreen] (b9) to (b10);
\draw [<-] (b10) to (b11);
\draw [->, thick, OliveGreen] (b11) to (b12);
\draw [<-] (b12) to (b13);
\draw [->, thick, OliveGreen] (b13) to (b14);
\draw [<-] (b14) to (b15);
\draw [<-, thick, OliveGreen] (b15) to (b16);
\draw [<-, thick, OliveGreen] (b16) to (b17);
\draw [<-, thick, OliveGreen] (b17) to (b18);
\draw [<-, thick, OliveGreen] (b18) to (b19);
\draw [<-, thick, OliveGreen] (b19) to (b20);
\draw [<-] (b20) to (b21);
\draw [->] (b21) to (b22);
\draw [->] (b22) to (b23);
\draw [->] (b23) to (b24);
\draw [->] (b24) to (b25);
\draw [->] (b25) to (b26);
\draw [<-] (b26) to (b27);
\draw [->] (b27) to (b28);
\draw [->] (b28) to (b29);
\draw [->] (b29) to (b30);
\draw [->] (b30) to (31,0);

\draw [->] (-1,1) to (c0);
\draw [->] (c0) to (c1);
\draw [<-, thick, OliveGreen] (c1) to (c2);
\draw [->] (c2) to (c5);
\draw [->, thick, OliveGreen] (c5) to (c6);
\draw [->] (c6) to (c7);
\draw [->, thick, OliveGreen] (c7) to (c8);
\draw [->, thick, OliveGreen] (c8) to (c11);
\draw [<-] (c11) to (c12);
\draw [->, thick, OliveGreen] (c12) to (c13);
\draw [<-] (c13) to (c14);
\draw [<-] (c14) to (c17);
\draw [<-] (c17) to (c18);
\draw [<-] (c18) to (c19);
\draw [<-] (c19) to (c20);
\draw [<-, thick, OliveGreen] (c20) to (c23);
\draw [<-, thick, OliveGreen] (c23) to (c24);
\draw [<-, thick, OliveGreen] (c24) to (c25);
\draw [<-] (c25) to (c26);
\draw [<-] (c26) to (c29);
\draw [<-] (c29) to (c30);
\draw [<-] (c30) to (31,1);

\draw [->] (-1,2) to (d0);
\draw [<-, thick, OliveGreen] (d0) to (d1);
\draw [->] (d1) to (d6);
\draw [->, thick, OliveGreen] (d6) to (d7);
\draw [->, thick, OliveGreen] (d7) to (d12);
\draw [<-] (d12) to (d13);
\draw [<-] (d13) to (d18);
\draw [<-] (d18) to (d19);
\draw [<-] (d19) to (d24);
\draw [<-] (d24) to (d25);
\draw [<-, thick, OliveGreen] (d25) to (d30);
\draw [<-] (d30) to (31,2);

\draw [->] (2,-2) to [out=90, in=-150] (a3);
\draw [->, thick, OliveGreen] (8,-2) to [out=90, in=-150] (a9);
\draw [<-, thick, OliveGreen] (14,-2) to [out=90, in=-150] (a15);
\draw [->] (20,-2) to [out=90, in=-150] (a21);
\draw [->] (26,-2) to [out=90, in=-150] (a27);

\draw [->, thick, OliveGreen] (4,-2) to [out=90, in=-30] (a3);
\draw [<-, thick, OliveGreen] (10,-2) to [out=90, in=-30] (a9);
\draw [->] (16,-2) to [out=90, in=-30] (a15);
\draw [->] (22,-2) to [out=90, in=-30] (a21);
\draw [->] (28,-2) to [out=90, in=-30] (a27);

\draw [->] (0,-2) to (b0);
\draw [->] (1,-2) to [out=90,in=-90] (b2);
\draw [->, thick, OliveGreen] (5,-2) to [out=90,in=-90] (b4);
\draw [->, thick, OliveGreen] (6,-2) to (b6);
\draw [->, thick, OliveGreen] (7,-2) to [out=90,in=-90] (b8);
\draw [<-, thick, OliveGreen] (11,-2) to [out=90,in=-90] (b10);
\draw [<-, thick, OliveGreen] (12,-2) to (b12);
\draw [<-, thick, OliveGreen] (13,-2) to [out=90,in=-90] (b14);
\draw [->] (17,-2) to [out=90,in=-90] (b16);
\draw [->] (18,-2) to (b18);
\draw [->] (19,-2) to [out=90,in=-90] (b20);
\draw [->] (23,-2) to [out=90,in=-90] (b22);
\draw [->] (24,-2) to (b24);
\draw [->] (25,-2) to [out=90,in=-90] (b26);
\draw [->] (29,-2) to [out=90,in=-90] (b28);
\draw [->] (30,-2) to (b30);

\draw [->, thick, OliveGreen] (a3) to (b3);
\draw [<-] (a9) to (b9);
\draw [<-, thick, OliveGreen] (a15) to (b15);
\draw [->] (a21) to (b21);
\draw [->] (a27) to (b27);

\draw [->] (b1) to (c1);
\draw [->, thick, OliveGreen] (b2) to (c2);
\draw [->, thick, OliveGreen] (b5) to (c5);
\draw [->, thick, OliveGreen] (b7) to (c7);
\draw [<-] (b8) to (c8);
\draw [<-, thick, OliveGreen] (b11) to (c11);
\draw [<-, thick, OliveGreen] (b13) to (c13);
\draw [<-] (b14) to (c14);
\draw [->] (b17) to (c17);
\draw [->] (b19) to (c19);
\draw [<-, thick, OliveGreen] (b20) to (c20);
\draw [->] (b23) to (c23);
\draw [->] (b25) to (c25);
\draw [->] (b26) to (c26);
\draw [->] (b29) to (c29);

\draw [->] (c0) to (d0);
\draw [->, thick, OliveGreen] (c1) to (d1);
\draw [->, thick, OliveGreen] (c6) to (d6);
\draw [<-] (c7) to (d7);
\draw [<-, thick, OliveGreen] (c12) to (d12);
\draw [<-] (c13) to (d13);
\draw [->] (c18) to (d18);
\draw [<-] (c19) to (d19);
\draw [->] (c24) to (d24);
\draw [<-, thick, OliveGreen] (c25) to (d25);
\draw [->] (c30) to (d30);

\draw [->, thick, OliveGreen] (d0) to (e0);
\draw [<-] (d6) to (e6);
\draw [<-] (d12) to (e12);
\draw [<-] (d18) to (e18);
\draw [<-] (d24) to (e24);
\draw [<-, thick, OliveGreen] (d30) to (e30);

\end{tikzpicture}
\caption{Network $\Ac_4$.}
\label{net-A}
\end{figure}

Let us break the network $\Nc_n^{\mathrm{sym}_-}$ into a pair of networks $\Bc_n$ and $\Ac_n$ shown in Figures~\ref{net-B} and~\ref{net-A} respectively; there $\Ac_n$ is glued to the bottom part of $\Nc_n^{\mathrm{sf}_-}$ along the dashed line. For all $1 \le k \le n+1$, let $\mathscr B_k$ be the set of paths $\gamma \colon v_0 \to v_k$ in $\Bc_n$, and $\mathscr A_k$ be the set of paths $\gamma \colon v_k \to h_2$ in $\Ac_n$. We also denote by $\mathscr A_0$ the set of all paths $\gamma \colon t_1 \to h_2$ contained entirely in $\Ac_n$. Finally, we write $\gamma_1$ and $\gamma_2$ for the unique paths in $\Ac_n$ connecting $t_1$ to $v_0$ and $t_2$ to $h_2$ respectively.

Let us define
\beq
\label{alpha-def}
A_k^{(n+1)} = \pr\hr{\sum_{\gamma\in\mathscr A_k}X_{\hs{\gamma_1\gamma\gamma_2^{-1}}}}
\qquad\text{and}\qquad
B_k^{(n+1)} = \pr\hr{\sum_{\gamma\in\mathscr B_k}X_{[\gamma]}}
\eeq
where for a path $\gamma$, $[\gamma]$ denotes the homology class of the projection of that path in the network obtained by collapsing the boundary components of the punctured disk to points, and collapsing all vertices $v_0, \dots, v_{n+1}$ to a single vertex $v$. Set $B_0^{(n+1)}=1$, then it follows from Corollary~\ref{gsv-measure} that
$$
\iota_{\mathrm{sym}}(F_1) = \sum_{k=0}^{n+1} A_k^{(n+1)} \otimes B_k^{(n+1)},
$$
where both sides of the equation are regarded as operators on $\Pc_\la \simeq \Vc_n \otimes \C_\la$.

\begin{lemma}
The elements $B_k^{(n+1)}$ satisfy, and are determined by, the recurrence relation
$$
B_k^{(n+1)} = \omega_0^k \hr{q^{\delta_{k,1}-1} B_{k-1}^{(n)} + B_k^{(n)}} ,
$$
where $\delta_{i,j}$ is the Kronecker delta, and the initial conditions
$$
B_0^{(1)} = 1, \qquad B_1^{(1)} = \omega_n, \qquad B_k^{(1)} = 0 \qquad\text{if}\qquad k\ne0,1.
$$
Here $B_i^{(j)}$ is regarded as a function of the variables $\omega_{n-j}, \dots, \omega_n$.
\end{lemma}

\begin{proof}
It is evident that the initial conditions are satisfied. To obtain the recurrence, we write the set of paths $\mathscr B_k$ as a disjoint union of two subsets: $\mathscr B_k^\shortleftarrow$ and $\mathscr B_k^\shortrightarrow$. These subsets are defined as follows. Note that the last edge of any path $\gamma \in \mathscr B_k$ points down to $v_k$, while the one before last points either to the left, in which case we declare $\gamma \in \mathscr B_k^\shortleftarrow$, or to the right, then $\gamma \in \mathscr B_k^\shortrightarrow$. Now, one can check that
$$
\pr\hr{\sum_{\gamma\in\mathscr B_k^\shortleftarrow} X_{[\gamma]}} = \omega_0^k B_k^{(n)}
\qquad\text{and}\qquad
\pr\hr{\sum_{\gamma\in\mathscr B_k^\shortrightarrow}X_{[\gamma]}} = q^{\delta_{k,1}-1}\omega_0^k B_{k-1}^{(n)},
$$
and the statement of the Lemma follows.
\end{proof}

\begin{cor}
Set  $\varpi_k = \prod_{0\le j\le k}\omega_j$. Then we have
\beq
\label{sym-fn-calc}
B_k^{(n+1)} = q^{1-k}e_k(\varpi_0, \ldots, \varpi_n) \qquad\text{for}\qquad 1\le k\le n+1,
\eeq
where $e_k(\varpi_0, \ldots, \varpi_n)$ is the $k$-th elementary symmetric polynomial in the variables $\varpi_j$. 
\end{cor}

Note that under the specialization~\eqref{spec}, we have 
$$
\varpi_k = e^{2\pi\hbar\lambda_k} \qquad\text{for all}\qquad 0\le k\le n.
$$ 
Therefore combining~\eqref{sym-fn-calc} with Proposition~\ref{gsv-measure}, we obtain
\begin{prop}
As an operator on $\Pc_\la \simeq  \Vc_n \otimes \C_\la$, we have
\beq
\label{single-formula}
\iota_{\mathrm{sym}}(F_1) = A_0^{(n+1)} \otimes 1 + \sum_{k=1}^{n+1}q^{1-k} A_k^{(n+1)} \otimes e_k\hr{e^{2\pi\hbar\lambda_0},\ldots,e^{2\pi\hbar\lambda_n}}.
\eeq
\end{prop}

\subsection{Coxeter embedding of $\Delta(U_q(\sl_{n+1}))$}
\label{diag-embed-top}

The middle part of quiver $\Qfull$ in Figure~\ref{Q4-full} coincides with the quiver shown in Figure~\ref{fig-Qbox}, which describes factorization coordinates on $G^{w_0,w_0}/\Ad_H$. Thus we may perform the sequence of mutations exaplined in~\ref{subsec-muts-double1} and~\ref{subsec-muts-double2} on $\Qfull$. We denote the resulting quiver $\Qcox$, see Figure~\ref{Q4-cox}, the corresponding seed $\Theta_n^{\mathrm{cox}}$, and the quantum torus algebra $\Dcox$. We also denote by $\Nc_n^{\mathrm{cox}_\pm}$ the networks obtained from $\Nc_n^{\mathrm{full}_\pm}$ under these mutations. The top part of the network $\Nc_4^{\mathrm{cox}_-}$ is shown in Figure~\ref{net-cox}.

Let us denote by $\Phi^{\mathrm{cox}}$ the unitary operator on $\mathcal{P}_{\lambda}\otimes\mathcal{P}_\mu$ given by the composite of all quantum mutations transforming the seed $\Theta_n^{\mathrm{full}}$ into $\Theta_n^{\mathrm{cox}}$. 

\begin{cor}
\label{cox-embed}
Let
$$
M_n^{\mathrm{cox}_\pm}(i,j;k,l) = \mathrm{pr}\hr{Z_{\Nc_n^{\mathrm{cox}_\pm}}(t_i,h_j;t_k,h_l)}
$$
be the boundary measurements in the networks $\Nc_n^{\mathrm{cox}_\pm}$, and set
$$
\tau_{\mathrm{cox}} = \Ad_{\Phi^{\mathrm{cox}}\circ\Phi^{\mathrm{top}}}\cdot\tau.
$$
Then we have
\begin{align*}
&\tau_{\mathrm{cox}}(K_i) = {M_n^{\mathrm{cox}_+}(i+1,i+1;i,i)}, &
&\tau_{\mathrm{cox}}(E_i) = {M_n^{\mathrm{cox}_+}(i+1,i;i,i)}, \\
&\tau_{\mathrm{cox}}(K'_i) = {M_n^{\mathrm{cox}_-}(i,i;i+1,i+1)}, &
&\tau_{\mathrm{cox}}(F_i) = {M_n^{\mathrm{cox}_-}(i,i+1;i+1,i+1)}.
\end{align*}

\end{cor}

\begin{figure}[p]
\begin{tikzpicture}[every node/.style={inner sep=0, minimum size=0.4cm, thick, draw, circle}, x=0.75cm, y=0.5cm]

\node [rectangle] (1) at (-1,0) {\scriptsize 1};
\node (3) at (0,2) {\scriptsize 3};
\node [rectangle] (5) at (1,0) {\scriptsize 5};
\node [rectangle] (6) at (-2,2) {\scriptsize 6};
\node (2) at (-1,4) {\scriptsize 2};
\node (10) at (0,6) {\scriptsize 10};
\node (4) at (1,4) {\scriptsize 4};
\node [rectangle] (14) at (2,2) {\scriptsize 14};
\node [rectangle] (15) at (-3,4) {\scriptsize 15};
\node (7) at (-2,6) {\scriptsize 7};
\node (9) at (-1,8) {\scriptsize 9};
\node (11) at (1,8) {\scriptsize 11};
\node (13) at (2,6) {\scriptsize 13};
\node [rectangle] (27) at (3,4) {\scriptsize 27};
\node [rectangle] (28) at (-4,6) {\scriptsize 28};
\node (16) at (-3,8) {\scriptsize 16};
\node (26) at (3,8) {\scriptsize 26};
\node [rectangle] (44) at (4,6) {\scriptsize 44};

\node (8) at (-2,10) {\scriptsize 8};
\node (21) at (0,10) {\scriptsize 21};
\node (12) at (2,10) {\scriptsize 12};
\node (45) at (-1,12) {\scriptsize 45};
\node (25) at (1,12) {\scriptsize 25};
\node (43) at (0,14) {\scriptsize 43};
\node (49) at (-1,16) {\scriptsize 49};
\node (42) at (1,16) {\scriptsize 42};
\node (47) at (-1,18) {\scriptsize 47};
\node (46) at (1,18) {\scriptsize 46};
\node (35) at (-1,20) {\scriptsize 35};
\node (22) at (1,20) {\scriptsize 22};
\node (41) at (-1,22) {\scriptsize 41};
\node (29) at (1,22) {\scriptsize 29};
\node (37) at (0,24) {\scriptsize 37};
\node (24) at (-1,26) {\scriptsize 24};
\node (48) at (1,26) {\scriptsize 48};
\node (50) at (-2,28) {\scriptsize 50};
\node (20) at (0,28) {\scriptsize 20};
\node (40) at (2,28) {\scriptsize 40};

\node (38) at (-3,30) {\scriptsize 38};
\node (34) at (-1,30) {\scriptsize 34};
\node (30) at (1,30) {\scriptsize 30};
\node (23) at (3,30) {\scriptsize 23};

\node (51) at (0,32) {\scriptsize 51};
\node (19) at (-2,32) {\scriptsize 19};
\node (33) at (-1,34) {\scriptsize 33};
\node (32) at (0,36) {\scriptsize 32};
\node (31) at (1,34) {\scriptsize 31};
\node (17) at (2,32) {\scriptsize 17};
\node (39) at (-4,32) {\scriptsize 39};
\node (36) at (-3,34) {\scriptsize 36};
\node (18) at (-2,36) {\scriptsize 18};
\node (52) at (-1,38) {\scriptsize 52};

\draw [->,thick] (19) -- (33);

\draw [->,thick] (33) -- (32);
\draw [->,thick] (32) -- (31);

\draw [->,thick] (31) -- (17);

\draw [->,thick] (51) -- (31);
\draw [->,thick] (31) to [out = 140, in = -30] (18);
\draw [->,thick] (18) to (33);
\draw [->,thick] (33) -- (51);

\draw [->,thick] (17) to [out = 145, in =-30] (36);
\draw [->,thick] (36) to (19);

\draw [->,thick] (31) -- (33);
\draw [->,thick] (17) -- (51);
\draw [->,thick] (51) -- (19);
\draw [->,thick] (39) to [bend left = 15] (17);
\draw [->,thick] (36) to [bend left = 20] (31);
\draw [->,thick] (33) to (36);
\draw [->,thick] (19) to (39);

\draw[->, thick] (23) to [out=145, in=-20] (39);
\draw[->, thick] (39) to (38);
\draw[->, thick] (38) to (19);
\draw[->, thick] (19) to (34);
\draw[->, thick] (34) to (51);
\draw[->, thick] (51) to (30);
\draw[->, thick] (30) to (17);
\draw[->, thick] (17) to (23);
\draw[->, thick] (23) to (40);
\draw[->, thick] (40) to (30);
\draw[->, thick] (30) to (20);
\draw[->, thick] (20) to (34);
\draw[->, thick] (34) to (50);

\draw[BrickRed, dashed] (-4,24) to (4,24);

\draw [->, thick] (16) -- (8);
\draw [->, thick] (8) -- (9);
\draw [->, thick] (9) -- (21);
\draw [->, thick] (21) -- (11);
\draw [->, thick] (11) -- (12);
\draw [->, thick] (12) -- (21);
\draw [->, thick] (21) -- (8);
\draw [->, thick] (8) -- (45);
\draw [->, thick] (45) -- (21);
\draw [->, thick] (21) -- (25);
\draw [->, thick] (25) -- (45);
\draw [->, thick] (45) -- (43);

\draw [->, thick] (43) to [out = -60, in=120] (25);
\draw [->, thick] (25) to [bend right = 21] (45);
\draw [->, thick] (45) to [out = -30, in=135] (12);
\draw [->, thick] (12) to [bend right = 18] (8);
\draw [->, thick] (8) to [out = -20, in=150] (26);
\draw [->, thick] (26) to [bend right = 15] (16);

\draw [->, thick] (42) -- (43);
\draw [->, thick] (43) to [out = 120, in=-60] (49);
\draw [->, thick] (49) -- (42);
\draw [->, thick] (42) -- (47);
\draw [->, thick] (47) -- (46);
\draw [->, thick] (46) -- (35);
\draw [->, thick] (35) -- (22);
\draw [->, thick] (22) -- (41);
\draw [->, thick] (41) -- (29);
\draw [->, thick] (29) -- (37);
\draw [->, thick] (37) -- (41);
\draw [->, thick] (41) to [bend left = 21] (29);
\draw [->, thick] (29) to [out = -135, in=45] (35);
\draw [->, thick] (35) to [bend left = 21] (22);
\draw [->, thick] (22) to [out = -135, in=45] (47);
\draw [->, thick] (47) to [bend left = 21] (46);
\draw [->, thick] (46) to [out = -135, in=45] (49);
\draw [->, thick] (49) to [bend left = 21] (42);

\draw [->, thick] (50) -- (20);
\draw [->, thick] (20) -- (40);
\draw [->, thick] (40) -- (48);
\draw [->, thick] (48) -- (20);
\draw [->, thick] (20) -- (24);
\draw [->, thick] (24) -- (48);
\draw [->, thick] (48) -- (37);

\draw [->, thick] (37) to [out = 120, in=-60] (24);
\draw [->, thick] (24) to [bend left = 21] (48);
\draw [->, thick] (48) to [out = 140, in=-40] (50);
\draw [->, thick] (50) to [bend left = 18] (40);
\draw [->, thick] (40) to [out = 140, in=-20] (38);
\draw [->, thick] (38) to [bend left = 15] (23);

\draw[BrickRed, dashed] (-4,14) to (4,14);

\draw [->,thick] (1) -- (5);

\draw [->,thick] (6) -- (3);
\draw [->,thick] (3) -- (14);

\draw [->,thick] (15) -- (2);
\draw [->,thick] (2) -- (4);
\draw [->,thick] (4) -- (27);

\draw [->,thick] (28) -- (7);
\draw [->,thick] (7) -- (10);
\draw [->,thick] (10) -- (13);
\draw [->,thick] (13) -- (44);

\draw [->,thick] (5) -- (3);
\draw [->,thick] (3) -- (2);
\draw [->,thick] (2) -- (7);
\draw [->,thick] (7) -- (16);

\draw [->,thick] (26) -- (13);
\draw [->,thick] (13) -- (4);
\draw [->,thick] (4) -- (3);
\draw [->,thick] (3) -- (1);

\draw [->,thick] (14) -- (4);
\draw [->,thick] (4) -- (10);
\draw [->,thick] (10) -- (9);
\draw [->,thick] (11) -- (10);
\draw [->,thick] (10) -- (2);
\draw [->,thick] (2) -- (6);

\draw [->,thick] (27) -- (13);
\draw [->,thick] (13) -- (11);
\draw [->,thick] (9) -- (7);
\draw [->,thick] (7) -- (15);

\draw [->,thick] (44) -- (26);
\draw [->,thick] (16) -- (28);

\draw [->,thick,dashed] (44) -- (27);
\draw [->,thick,dashed] (27) -- (14);
\draw [->,thick,dashed] (14) -- (5);
\draw [->,thick,dashed] (1) -- (6);
\draw [->,thick,dashed] (6) -- (15);
\draw [->,thick,dashed] (15) -- (28);

\end{tikzpicture}
\caption{Quiver $\Qc_4^{\mathrm{cox}}$.}
\label{Q4-cox}
\end{figure}

\begin{figure}[h]
\begin{tikzpicture}[every node/.style={inner sep=0, minimum size=0.15cm, circle, draw}, x=0.45cm,y=0.45cm]

\foreach \i in {3,9,...,30}
{
	\node[draw=none, text=blue] at (\i-2,-0.5) {\tiny 11};
	\node[draw=none, text=blue] at (\i-0.5,-0.5) {\tiny 26};
	\node[draw=none, text=blue] at (\i+0.5,-0.5) {\tiny 16};
	\node[draw=none, text=blue] at (\i+2,-0.5) {\tiny 9};

	\node[draw=none, text=blue] at (\i-3,0.5) {\tiny 21};
	\node[draw=none, text=blue] at (\i-1.5,0.5) {\tiny 12};
	\node[draw=none, text=blue] at (\i+0.5,0.5) {\tiny 8};

	\node[draw=none, text=blue] at (\i-2.5,1.5) {\tiny 25};
	\node[draw=none, text=blue] at (\i+0.5,1.5) {\tiny 45};

	\node[draw=none, text=blue] at (\i,2.5) {\tiny 43};

	\node[draw=none, text=blue] at (\i-1.5,3.5) {\tiny 42};
	\node[draw=none, text=blue] at (\i+1.5, 3.5) {\tiny 49};
	\node[draw=none, text=blue] at (\i-1.5,4.5) {\tiny 47};
	\node[draw=none, text=blue] at (\i+1.5,4.5) {\tiny 46};
	\node[draw=none, text=blue] at (\i-1.5,5.5) {\tiny 22};
	\node[draw=none, text=blue] at (\i+1.5,5.5) {\tiny 35};
	\node[draw=none, text=blue] at (\i-1.5,6.5) {\tiny 41};
	\node[draw=none, text=blue] at (\i+1.5,6.5) {\tiny 29};

	\node[draw=none, text=blue] at (\i,7.5) {\tiny 37};

	\node[draw=none, text=blue] at (\i+2.5,8.5) {\tiny 24};
	\node[draw=none, text=blue] at (\i-0.5,8.5) {\tiny 48};

	\node[draw=none, text=blue] at (\i-3,9.5) {\tiny 20};
	\node[draw=none, text=blue] at (\i-0.5,9.5) {\tiny 40};
	\node[draw=none, text=blue] at (\i+1.5,9.5) {\tiny 50};

	\node[draw=none, text=blue] at (\i-2,10.5) {\tiny 30};
	\node[draw=none, text=blue] at (\i-0.5,10.5) {\tiny 23};
	\node[draw=none, text=blue] at (\i+0.5,10.5) {\tiny 38};
	\node[draw=none, text=blue] at (\i+2,10.5) {\tiny 34};
}

\node[draw=none, text=blue] at (30,0.5) {\tiny 18};
\node[draw=none, text=blue] at (30,9.5) {\tiny 20};

\foreach \i in {0,2,...,30}
{
	\node[fill=white] (b\i) at (\i,0) {};
	\node[fill=white] (l\i) at (\i,10) {};
}

\foreach \i in {1,3,...,30}
{
	\node[fill=black] (b\i) at (\i,0) {};
	\node[fill=black] (l\i) at (\i,10) {};
}

\foreach \i in {0,6,...,30}
{
	\node[fill=black] (c\i) at (\i,1) {};
	\node[fill=white] (d\i) at (\i,2) {};
	\node[fill=white] (j\i) at (\i,8) {};
	\node[fill=white] (e\i) at (\i,3) {};
	\node[fill=black] (f\i) at (\i,4) {};
	\node[fill=white] (g\i) at (\i,5) {};
	\node[fill=black] (h\i) at (\i,6) {};
	\node[fill=white] (i\i) at (\i,7) {};
	\node[fill=black] (k\i) at (\i,9) {};
}

\foreach \i in {3,9,...,30}
{
	\node[fill=white] (a\i) at (\i,-1) {};
	\node[fill=black] (e\i) at (\i,3) {};
	\node[fill=white] (f\i) at (\i,4) {};
	\node[fill=black] (g\i) at (\i,5) {};
	\node[fill=white] (h\i) at (\i,6) {};
	\node[fill=black] (i\i) at (\i,7) {};
	\node[fill=white] (m\i) at (\i,11) {};
}

\foreach \i in {1,7,...,30}
{
	\node[fill=white] (c\i) at (\i,1) {};
	\node[fill=black] (d\i) at (\i,2) {};
	\node[fill=white] (k\i) at (\i,9) {};
}

\foreach \i in {2,8,...,30}
	\node[fill=black] (c\i) at (\i,1) {};

\foreach \i in {4,10,...,30}
	\node[fill=black] (k\i) at (\i,9) {};

\foreach \i in {5,11,...,30}
{
	\node[fill=white] (c\i) at (\i,1) {};
	\node[fill=black] (j\i) at (\i,8) {};
	\node[fill=white] (k\i) at (\i,9) {};
}
\draw[dashed, BrickRed, thick] (3,-2) to (15,-2);

\draw [->] (-1,0) to (b0);
\draw [->] (b0) to (b1);
\draw [->] (b1) to (b2);
\draw [<-, OliveGreen] (b2) to (b3);
\draw [->] (b3) to (b4);
\draw [->, OliveGreen] (b4) to (b5);
\draw [->] (b5) to (b6);
\draw [->, OliveGreen] (b6) to (b7);
\draw [->] (b7) to (b8);
\draw [->, OliveGreen] (b8) to (b9);
\draw [->, OliveGreen] (b9) to (b10);
\draw [<-] (b10) to (b11);
\draw [->, OliveGreen] (b11) to (b12);
\draw [<-] (b12) to (b13);
\draw [->, OliveGreen] (b13) to (b14);
\draw [<-] (b14) to (b15);
\draw [<-, OliveGreen] (b15) to (b16);
\draw [<-, OliveGreen] (b16) to (b17);
\draw [<-, OliveGreen] (b17) to (b18);
\draw [<-, OliveGreen] (b18) to (b19);
\draw [<-, OliveGreen] (b19) to (b20);
\draw [<-] (b20) to (b21);
\draw [->] (b21) to (b22);
\draw [->] (b22) to (b23);
\draw [->] (b23) to (b24);
\draw [->] (b24) to (b25);
\draw [->] (b25) to (b26);
\draw [<-] (b26) to (b27);
\draw [->] (b27) to (b28);
\draw [->] (b28) to (b29);
\draw [->] (b29) to (b30);
\draw [->] (b30) to (31,0);

\draw [->] (-1,1) to (c0);
\draw [->] (c0) to (c1);
\draw [<-, OliveGreen] (c1) to (c2);
\draw [->] (c2) to (c5);
\draw [->, OliveGreen] (c5) to (c6);
\draw [->] (c6) to (c7);
\draw [->, OliveGreen] (c7) to (c8);
\draw [->, OliveGreen] (c8) to (c11);
\draw [<-] (c11) to (c12);
\draw [->, OliveGreen] (c12) to (c13);
\draw [<-] (c13) to (c14);
\draw [<-] (c14) to (c17);
\draw [<-] (c17) to (c18);
\draw [<-] (c18) to (c19);
\draw [<-] (c19) to (c20);
\draw [<-, OliveGreen] (c20) to (c23);
\draw [<-, OliveGreen] (c23) to (c24);
\draw [<-, OliveGreen] (c24) to (c25);
\draw [<-] (c25) to (c26);
\draw [<-] (c26) to (c29);
\draw [<-] (c29) to (c30);
\draw [<-] (c30) to (31,1);

\draw [->] (-1,2) to (d0);
\draw [<-, OliveGreen] (d0) to (d1);
\draw [->] (d1) to (d6);
\draw [->, OliveGreen] (d6) to (d7);
\draw [->, OliveGreen] (d7) to (d12);
\draw [<-] (d12) to (d13);
\draw [<-] (d13) to (d18);
\draw [<-] (d18) to (d19);
\draw [<-] (d19) to (d24);
\draw [<-] (d24) to (d25);
\draw [<-, OliveGreen] (d25) to (d30);
\draw [<-] (d30) to (31,2);

\draw [->] (-1,3) to (e0);
\draw [->, OliveGreen] (e0) to (e3);
\draw [->] (e3) to (e6);
\draw [<-] (e6) to (e9);
\draw [->] (e9) to (e12);
\draw [<-] (e12) to (e15);
\draw [->] (e15) to (e18);
\draw [<-] (e18) to (e21);
\draw [->] (e21) to (e24);
\draw [<-] (e24) to (e27);
\draw [->, OliveGreen] (e27) to (e30);
\draw [<-] (e30) to (31,3);

\draw [->] (-1,4) to (f0);
\draw [->] (f0) to (f3);
\draw [->, OliveGreen] (f3) to (f6);
\draw [->] (f6) to (f9);
\draw [<-] (f9) to (f12);
\draw [->] (f12) to (f15);
\draw [<-] (f15) to (f18);
\draw [->] (f18) to (f21);
\draw [<-] (f21) to (f24);
\draw [->, OliveGreen] (f24) to (f27);
\draw [<-] (f27) to (f30);
\draw [<-] (f30) to (31,4);

\draw [->] (-1,5) to (g0);
\draw [->] (g0) to (g3);
\draw [->] (g3) to (g6);
\draw [->, OliveGreen] (g6) to (g9);
\draw [->] (g9) to (g12);
\draw [<-] (g12) to (g15);
\draw [->] (g15) to (g18);
\draw [<-] (g18) to (g21);
\draw [->, OliveGreen] (g21) to (g24);
\draw [<-] (g24) to (g27);
\draw [<-] (g27) to (g30);
\draw [<-] (g30) to (31,5);

\draw [->] (-1,6) to (h0);
\draw [->] (h0) to (h3);
\draw [->] (h3) to (h6);
\draw [->] (h6) to (h9);
\draw [->, OliveGreen] (h9) to (h12);
\draw [->] (h12) to (h15);
\draw [<-] (h15) to (h18);
\draw [->, OliveGreen] (h18) to (h21);
\draw [<-] (h21) to (h24);
\draw [<-] (h24) to (h27);
\draw [<-] (h27) to (h30);
\draw [<-] (h30) to (31,6);

\draw [->] (-1,7) to (i0);
\draw [->] (i0) to (i6);
\draw [->] (i6) to (i12);
\draw [->, OliveGreen] (i12) to (i15);
\draw [->, OliveGreen] (i15) to (i18);
\draw [<-] (i18) to (i24);
\draw [<-] (i24) to (i30);
\draw [<-] (i30) to (31,7);

\draw [->] (2,-2) to [out=90, in=-150] (a3);
\draw [->, OliveGreen] (8,-2) to [out=90, in=-150] (a9);
\draw [<-, OliveGreen] (14,-2) to [out=90, in=-150] (a15);
\draw [->] (20,-2) to [out=90, in=-150] (a21);
\draw [->] (26,-2) to [out=90, in=-150] (a27);

\draw [->, OliveGreen] (4,-2) to [out=90, in=-30] (a3);
\draw [<-, OliveGreen] (10,-2) to [out=90, in=-30] (a9);
\draw [->] (16,-2) to [out=90, in=-30] (a15);
\draw [->] (22,-2) to [out=90, in=-30] (a21);
\draw [->] (28,-2) to [out=90, in=-30] (a27);

\draw [->] (0,-2) to (b0);
\draw [->] (1,-2) to [out=90,in=-90] (b2);
\draw [->, OliveGreen] (5,-2) to [out=90,in=-90] (b4);
\draw [->, OliveGreen] (6,-2) to (b6);
\draw [->, OliveGreen] (7,-2) to [out=90,in=-90] (b8);
\draw [<-, OliveGreen] (11,-2) to [out=90,in=-90] (b10);
\draw [<-, OliveGreen] (12,-2) to (b12);
\draw [<-, OliveGreen] (13,-2) to [out=90,in=-90] (b14);
\draw [->] (17,-2) to [out=90,in=-90] (b16);
\draw [->] (18,-2) to (b18);
\draw [->] (19,-2) to [out=90,in=-90] (b20);
\draw [->] (23,-2) to [out=90,in=-90] (b22);
\draw [->] (24,-2) to (b24);
\draw [->] (25,-2) to [out=90,in=-90] (b26);
\draw [->] (29,-2) to [out=90,in=-90] (b28);
\draw [->] (30,-2) to (b30);

\draw [->, OliveGreen] (a3) to (b3);
\draw [<-] (a9) to (b9);
\draw [<-, OliveGreen] (a15) to (b15);
\draw [->] (a21) to (b21);
\draw [->] (a27) to (b27);

\draw [->] (b1) to (c1);
\draw [->, OliveGreen] (b2) to (c2);
\draw [->, OliveGreen] (b5) to (c5);
\draw [->, OliveGreen] (b7) to (c7);
\draw [<-] (b8) to (c8);
\draw [<-, OliveGreen] (b11) to (c11);
\draw [<-, OliveGreen] (b13) to (c13);
\draw [<-] (b14) to (c14);
\draw [->] (b17) to (c17);
\draw [->] (b19) to (c19);
\draw [<-, OliveGreen] (b20) to (c20);
\draw [->] (b23) to (c23);
\draw [->] (b25) to (c25);
\draw [->] (b26) to (c26);
\draw [->] (b29) to (c29);

\draw [->] (c0) to (d0);
\draw [->, OliveGreen] (c1) to (d1);
\draw [->, OliveGreen] (c6) to (d6);
\draw [<-] (c7) to (d7);
\draw [<-, OliveGreen] (c12) to (d12);
\draw [<-] (c13) to (d13);
\draw [->] (c18) to (d18);
\draw [<-] (c19) to (d19);
\draw [->] (c24) to (d24);
\draw [<-, OliveGreen] (c25) to (d25);
\draw [->] (c30) to (d30);

\draw [->, OliveGreen] (d0) to (e0);
\draw [<-] (d6) to (e6);
\draw [<-] (d12) to (e12);
\draw [<-] (d18) to (e18);
\draw [<-] (d24) to (e24);
\draw [<-, OliveGreen] (d30) to (e30);

\draw [<-] (e0) to (f0);
\draw [->, OliveGreen] (e3) to (f3);
\draw [<-] (e6) to (f6);
\draw [<-] (e9) to (f9);
\draw [<-] (e12) to (f12);
\draw [<-] (e15) to (f15);
\draw [<-] (e18) to (f18);
\draw [<-] (e21) to (f21);
\draw [<-] (e24) to (f24);
\draw [<-, OliveGreen] (e27) to (f27);
\draw [<-] (e30) to (f30);

\draw [->] (f0) to (g0);
\draw [<-] (f3) to (g3);
\draw [->, OliveGreen] (f6) to (g6);
\draw [<-] (f9) to (g9);
\draw [<-] (f12) to (g12);
\draw [<-] (f15) to (g15);
\draw [<-] (f18) to (g18);
\draw [<-] (f21) to (g21);
\draw [<-, OliveGreen] (f24) to (g24);
\draw [<-] (f27) to (g27);
\draw [->] (f30) to (g30);

\draw [<-] (g0) to (h0);
\draw [->] (g3) to (h3);
\draw [<-] (g6) to (h6);
\draw [->, OliveGreen] (g9) to (h9);
\draw [<-] (g12) to (h12);
\draw [<-] (g15) to (h15);
\draw [<-] (g18) to (h18);
\draw [<-, OliveGreen] (g21) to (h21);
\draw [<-] (g24) to (h24);
\draw [->] (g27) to (h27);
\draw [<-] (g30) to (h30);

\draw [->] (h0) to (i0);
\draw [<-] (h3) to (i3);
\draw [->] (h6) to (i6);
\draw [<-] (h9) to (i9);
\draw [->, OliveGreen] (h12) to (i12);
\draw [<-] (h15) to (i15);
\draw [<-, OliveGreen] (h18) to (i18);
\draw [<-] (h21) to (i21);
\draw [->] (h24) to (i24);
\draw [<-] (h27) to (i27);
\draw [->] (h30) to (i30);

\foreach \i in {0,6,...,30}
{
	\draw [->] (j\i) to (i\i);
	\draw [->] (k\i) to (j\i);
}

\foreach \i in {1,7,...,30}
	\draw [->] (l\i) to (k\i);

\foreach \i in {3,9,...,30}
	\draw [->] (m\i) to (l\i);

\foreach \i in {4,10,...,30}
	\draw [->] (l\i) to (k\i);

\foreach \i in {5,11,...,30}
{
	\draw [->] (k\i) to (j\i);
	\draw [->] (l\i) to (k\i);
}

\draw [->] (2,12) to [out=-90, in=150] (m3);
\draw [->] (8,12) to [out=-90, in=150] (m9);
\draw [->] (14,12) to [out=-90, in=150] (m15);
\draw [->] (20,12) to [out=-90, in=150] (m21);
\draw [->] (26,12) to [out=-90, in=150] (m27);

\draw [->] (4,12) to [out=-90, in=30] (m3);
\draw [->] (10,12) to [out=-90, in=30] (m9);
\draw [->] (16,12) to [out=-90, in=30] (m15);
\draw [->] (22,12) to [out=-90, in=30] (m21);
\draw [->] (28,12) to [out=-90, in=30] (m27);

\draw [->] (0,12) to (l0);
\draw [->] (1,12) to [out=-90,in=90] (l2);
\draw [->] (5,12) to [out=-90,in=90] (l4);
\draw [->] (6,12) to (l6);
\draw [->] (7,12) to [out=-90,in=90] (l8);
\draw [->] (11,12) to [out=-90,in=90] (l10);
\draw [->] (12,12) to (l12);
\draw [->] (13,12) to [out=-90,in=90] (l14);
\draw [->] (17,12) to [out=-90,in=90] (l16);
\draw [->] (18,12) to (l18);
\draw [->] (19,12) to [out=-90,in=90] (l20);
\draw [->] (23,12) to [out=-90,in=90] (l22);
\draw [->] (24,12) to (l24);
\draw [->] (25,12) to [out=-90,in=90] (l26);
\draw [->] (29,12) to [out=-90,in=90] (l28);
\draw [->] (30,12) to (l30);

\draw [->] (-1,8) to (j0);
\draw [<-] (j0) to (j5);
\draw [->] (j5) to (j6);
\draw [<-] (j6) to (j11);
\draw [->] (j11) to (j12);
\draw [<-] (j12) to (j17);
\draw [->] (j17) to (j18);
\draw [<-] (j18) to (j23);
\draw [->] (j23) to (j24);
\draw [<-] (j24) to (j29);
\draw [->] (j29) to (j30);
\draw [<-] (j30) to (31,8);

\draw [<-] (-1,9) to (k0);
\draw [<-] (k0) to (k1);
\draw [<-] (k1) to (k4);
\draw [->] (k4) to (k5);
\draw [<-] (k5) to (k6);
\draw [<-] (k6) to (k7);
\draw [<-] (k7) to (k10);
\draw [->] (k10) to (k11);
\draw [<-] (k11) to (k12);
\draw [<-] (k12) to (k13);
\draw [<-] (k13) to (k16);
\draw [->] (k16) to (k17);
\draw [<-] (k17) to (k18);
\draw [<-] (k18) to (k19);
\draw [<-] (k19) to (k22);
\draw [->] (k22) to (k23);
\draw [<-] (k23) to (k24);
\draw [<-] (k24) to (k25);
\draw [<-] (k25) to (k28);
\draw [->] (k28) to (k29);
\draw [<-] (k29) to (k30);
\draw [<-] (k30) to (31,9);

\draw [<-] (-1,10) to (l0);
\draw [<-] (l0) to (l1);
\draw [<-] (l1) to (l2);
\draw [<-] (l2) to (l3);
\draw [->] (l3) to (l4);
\draw [<-] (l4) to (l5);
\draw [<-] (l5) to (l6);
\draw [<-] (l6) to (l7);
\draw [<-] (l7) to (l8);
\draw [<-] (l8) to (l9);
\draw [->] (l9) to (l10);
\draw [<-] (l10) to (l11);
\draw [<-] (l11) to (l12);
\draw [<-] (l12) to (l13);
\draw [<-] (l13) to (l14);
\draw [<-] (l14) to (l15);
\draw [->] (l15) to (l16);
\draw [<-] (l16) to (l17);
\draw [<-] (l17) to (l18);
\draw [<-] (l18) to (l19);
\draw [<-] (l19) to (l20);
\draw [<-] (l20) to (l21);
\draw [->] (l21) to (l22);
\draw [<-] (l22) to (l23);
\draw [<-] (l23) to (l24);
\draw [<-] (l24) to (l25);
\draw [<-] (l25) to (l26);
\draw [<-] (l26) to (l27);
\draw [->] (l27) to (l28);
\draw [<-] (l28) to (l29);
\draw [<-] (l29) to (l30);
\draw [<-] (l30) to (31,10);

\end{tikzpicture}
\caption{Middle part of the network $\Nc_4^{\mathrm{cox}_-}$.}
\label{net-cox}
\end{figure}

\subsection{$\mathcal{P}_\lambda\otimes\mathcal{P}_\mu$ as a concatenation of positive representations}

By the results of the previous section, the positive representation $\mathcal{P}_\lambda\otimes\mathcal{P}_\mu$ of $U_q\hr{\mathfrak{sl}_{n+1}}$ is unitary equivalent to one defined by the cluster seed $\Theta_n^{\mathrm{cox}}$ in Corollary~\ref{cox-embed}. Let us now divide the quiver $\Qcox$ as indicated in Figure~\ref{Q4-cox}, and analyze the corresponding subquivers.

Recall that the quiver $\Qc_n^{\mathrm{std}}$ arises from the standard triangulation of a disk $D_{2,1}$. The quiver $\Qc_n^{\mathrm{sym}}$ is mutation equivalent to $\Qc_n^{\mathrm{std}}$, and $\Qc_n^{\mathrm{sym}\dag}$ is obtained from $\Qc_n^{\mathrm{sym}}$ by forgetting the disconnected nodes. Now, let us glue the two open boundary components of $D_{2,1}$ to get a thrice punctured sphere $S_3$. The corresponding quiver is obtained from $\Qc_n^{\mathrm{std}}$ by amalgamating pairs of frozen variables in the same row. Since we never mutate at frozen variables, we can apply to the resulting quiver the same sequence of mutations that turns $\Qc_n^{\mathrm{std}}$ into $\Qc_n^{\mathrm{sym}}$. Naturally, the resulting quiver is obtained from $\Qc_n^{\mathrm{sym}}$ by amalgamating pairs of frozen variables in the same row; we denote it by $\Qsph$. Finally, we denote by $\Qc_n^{\mathrm{sph}\dag}$ the quiver obtained from $\Qsph$ by forgetting the disconnected nodes. 

Having set up this notation, we see that the quiver $\Qc_n^{\mathrm{cox}}$ is divided in the following three parts:
\begin{enumerate}
\item the quiver $\Qc_n^{\mathrm{sph}\dag}$ formed by the nodes on or above the top dashed line;
\item the quiver $\Qc_n^{\mathrm{chain}}$, formed by the nodes on or between the two dashed lines;
\item the quiver $\Qc_n^{\mathrm{sym}\dag}$ formed by the nodes on or below the bottom dashed line.
\end{enumerate}

Denote by $\Dsph$ and $\Dc_n^{\mathrm{sph}\dag}$ the based quantum torus algebras corresponding respectively to the quivers $\Qsph$ and $\Qc_n^{\mathrm{sph}\dag}$. One can see that the elements $\iota_{\mathrm{sym}}(K_j)$ and $\iota_{\mathrm{sym}}(K'_j)$ lie in the subalgebra $\Dc_n^{\mathrm{sph}\dag} \subset \Dsym$. The following Lemma follows from calculating the ranks of the aforementioned quivers.

\begin{lemma}
The center of the quantum torus algebra $\Dc_n^{\mathrm{sph}\dag}$ is the Laurent polynomial ring generated by the elements
$$
\iota_{\mathrm{sym}}(K_j) \qquad\text{and}\qquad \iota_{\mathrm{sym}}(K'_j)
$$
for all $1 \le j \le n$.
\end{lemma}

Therefore, the quantum torus $\Dc_n^{\mathrm{sph}\dag}$ has positive representations $\Mc_{\lambda,\mu}$ labelled by pairs $(\lambda,\mu)$ of points in $\hgt_n$, in which for all $1\le j\le n$
$$
\tau_{\mathrm{cox}}(K_j) \longmapsto e^{2\pi\hbar (\lambda_j-\lambda_{j-1})}
\qquad\text{and}\qquad
\tau_{\mathrm{cox}}(K_j') \longmapsto e^{2\pi\hbar (\mu_j-\mu_{j-1})}.
$$
The quantum torus $\Dsph$ is generated by its subalgebra $\Dc_n^{\mathrm{sph}\dag}$ and central elements $\omega_1,\ldots, \omega_n$ corresponding to the disconnected nodes of the quiver $\Qsph$. Thus, it has positive representations $\Mc_{\lambda,\mu}^\nu$ labelled by triples $(\la, \mu, \nu)$ of points in $\hgt_n$ where the central elements $\omega_1,\ldots, \omega_n$ act via multiplication by the scalars 
$$
\omega_j \mapsto e^{2\pi\hbar\hr{\nu_j-\nu_{j-1}}} \qquad\text{for}\qquad j = 1, \dots,  n.
$$
Note that we have isomorphisms of Hilbert spaces
$$
\Mc_{\lambda,\mu} \simeq \Mc_{\lambda,\mu}^\nu \simeq L^2\hr{\R^{\frac{n(n-1)}{2}}}.
$$

Finally, let us consider the quiver $\Qc_n^{\mathrm{chain}}$. We divide its nodes into rows numbered top to bottom, so that the top dashed line crosses the 0-th row, and the bottom one crosses the $(n+1)$-st row. We denote the  quantum torus generators corresponding to the vertices in the 0-th and $(n+1)$-st rows by $\alpha_0$ and $\alpha_{n+1}$ respectively, and those corresponding to vertices in the $j$-th row by $\alpha_j^1$ and $\alpha_j^2$, so that there is a double arrow from the vertex labeled $(j,1)$ to the one labeled $(j,2)$. For example, in the notations of the Figure~\ref{Q4-cox} we get
\begin{align*}
&\alpha_0 = X_{37}, &
&\alpha_1^1 = X_{41}, &
&\alpha_2^1 = X_{35}, &
&\alpha_3^1 = X_{47}, &
&\alpha_4^1 = X_{49}, \\
&\alpha_5 = X_{43}, &
&\alpha_1^2 = X_{29}, &
&\alpha_2^2 = X_{22}, &
&\alpha_3^2 = X_{46}, &
&\alpha_4^2 = X_{42}.
\end{align*}
The quantum torus algebra associated to the quiver $\Qc_n^{\mathrm{chain}}$ has a positive representation modeled on the space $L^2(\R^n)$ of square-integrable functions in the variables $x_j-x_{j+1}$, where $j =1, \dots, n$, defined by
\beq
\label{alpha-def}
\begin{aligned}
&\alpha_0 \longmapsto e^{-2\pi\hbar p_1}, &\qquad\qquad
&\alpha_j^1 \longmapsto e^{2\pi\hbar(x_{j+1}-x_j)}, \\
&\alpha_{n+1} \longmapsto e^{2\pi\hbar p_{n+1}}, &\qquad\qquad
&\alpha_j^2 \longmapsto e^{2\pi\hbar(p_j-p_{j+1} +x_j-x_{j+1})}.
\end{aligned}
\eeq

\begin{prop}
\label{prop-concat}
We have an isomorphism of Hilbert spaces
\beq
\label{double-concat}
\Pc_\la \otimes \Pc_\mu \simeq \Vc_n \otimes L^2\hr{\R^n} \otimes \Mc_{\lambda,\mu}.
\eeq
\end{prop}

\begin{proof}
One can verify that the sequence of mutations taking the seed $\Theta_n^{\mathrm{sq}}$ to $\Theta_n^{\mathrm{cox}}$ transforms the central elements as follows
$$
\Omega_j^1 \longmapsto \iota_{\mathrm{sym}}(K_j)
\qquad\text{and}\qquad
\Omega_j^2 \longmapsto \iota_{\mathrm{sym}}(K_j'),
$$
therefore the weights in the two sides of the equation coincide. Now, we calculate the ranks of the quivers involved and obtain
$$
\rk\hr{\Qsq} = 2n(n+1),
$$
while
$$
\rk\hr{\Qc_n^{\mathrm{sym}\dag}} = n(n+1), \qquad
\rk\hr{\Qc_n^{\mathrm{chain}}} = 2n, \qquad
\rk\hr{\Qc_n^{\mathrm{sph}\dag}} = n(n-1).
$$
Therefore
$$
\rk\hr{\Qsq} =
\rk\hr{\Qc_n^{\mathrm{sym}\dag}} +
\rk\hr{\Qc_n^{\mathrm{chain}}} +
\rk\hr{\Qc_n^{\mathrm{sph}\dag}},
$$
so Lemma~\ref{rep-concatenation} applies and yields the desired isomorphism of Hilbert spaces.
\end{proof}

Inspecting the network $\Nc_n^{\mathrm{cox}_-}$, see Figure~\ref{net-cox}, one can verify that the expressions $\tau_{\mathrm{cox}}(E_i)$, $\tau_{\mathrm{cox}}(K_i)$, and $\tau_{\mathrm{cox}}(F_j)$ for $1 \le i \le n$ and $2 \le j \le n$ depend only on the cluster variables from the subalgebra $\Dc_n^{\mathrm{sym}\dag} \subset \Dcox$, and thus act by the identity on the second and the third tensor factor in~\eqref{double-concat}. Moreover, since the relevant parts of the networks $\Nc_n^{\mathrm{cox}_-}$ and $\Nc_n^{\mathrm{sym}_-}$ are identical, we have the following equality of operators on $\Pc_\la \otimes \Pc_\mu \simeq \Vc_n \otimes L^2\hr{\R^n} \otimes \Mc_{\la,\mu}$
\beq
\label{easy-gens}
\begin{aligned}
\tau_{\mathrm{cox}}(E_j) &= \iota_{\mathrm{sym}}(E_j) \otimes 1 \otimes 1, & &1\le j\le n,\\
\tau_{\mathrm{cox}}(K_j) &= \iota_{\mathrm{sym}}(K_j) \otimes 1 \otimes 1, & &1\le j\le n,\\
\tau_{\mathrm{cox}}(F_j) &= \iota_{\mathrm{sym}}(F_j) \otimes 1 \otimes 1, & &2\le j\le n.
\end{aligned}
\eeq
Once again the action of $F_1$ must be treated separately, since the relevant boundary measurement $M_n^{\mathrm{cox}_-}(1,2;2,2)$ involves cluster variables from the quiver $\Qc_n^{\mathrm{chain}}$.

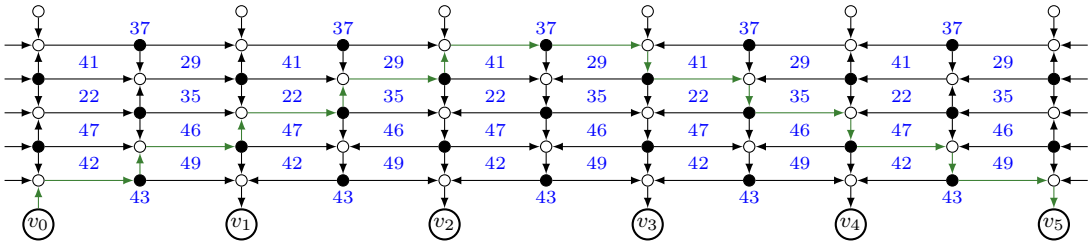
\begin{figure}[h]
\begin{tikzpicture}[every node/.style={inner sep=0, minimum size=0.15cm, circle, draw}, x=0.45cm,y=0.45cm]

\foreach \i in {3,9,...,30}
{
	\node[draw=none, text=blue] at (\i,2.5) {\tiny 43};

	\node[draw=none, text=blue] at (\i-1.5,3.5) {\tiny 42};
	\node[draw=none, text=blue] at (\i+1.5, 3.5) {\tiny 49};
	\node[draw=none, text=blue] at (\i-1.5,4.5) {\tiny 47};
	\node[draw=none, text=blue] at (\i+1.5,4.5) {\tiny 46};
	\node[draw=none, text=blue] at (\i-1.5,5.5) {\tiny 22};
	\node[draw=none, text=blue] at (\i+1.5,5.5) {\tiny 35};
	\node[draw=none, text=blue] at (\i-1.5,6.5) {\tiny 41};
	\node[draw=none, text=blue] at (\i+1.5,6.5) {\tiny 29};

	\node[draw=none, text=blue] at (\i,7.5) {\tiny 37};
}

\node[minimum size=0.4cm, thick] (d0) at (0,1.7) {\scriptsize $v_0$};
\node[minimum size=0.4cm, thick] (d6) at (6,1.7) {\scriptsize $v_1$};
\node[minimum size=0.4cm, thick] (d12) at (12,1.7) {\scriptsize $v_2$};
\node[minimum size=0.4cm, thick] (d18) at (18,1.7) {\scriptsize $v_3$};
\node[minimum size=0.4cm, thick] (d24) at (24,1.7) {\scriptsize $v_4$};
\node[minimum size=0.4cm, thick] (d30) at (30,1.7) {\scriptsize $v_5$};

\foreach \i in {0,6,...,30}
{
	\node[fill=white] (e\i) at (\i,3) {};
	\node[fill=black] (f\i) at (\i,4) {};
	\node[fill=white] (g\i) at (\i,5) {};
	\node[fill=black] (h\i) at (\i,6) {};
	\node[fill=white] (i\i) at (\i,7) {};
	\node[fill=white] (j\i) at (\i,8) {};
}

\foreach \i in {3,9,...,30}
{
	\node[fill=black] (e\i) at (\i,3) {};
	\node[fill=white] (f\i) at (\i,4) {};
	\node[fill=black] (g\i) at (\i,5) {};
	\node[fill=white] (h\i) at (\i,6) {};
	\node[fill=black] (i\i) at (\i,7) {};
}

\draw [->] (-1,3) to (e0);
\draw [->, OliveGreen] (e0) to (e3);
\draw [->] (e3) to (e6);
\draw [<-] (e6) to (e9);
\draw [->] (e9) to (e12);
\draw [<-] (e12) to (e15);
\draw [->] (e15) to (e18);
\draw [<-] (e18) to (e21);
\draw [->] (e21) to (e24);
\draw [<-] (e24) to (e27);
\draw [->, OliveGreen] (e27) to (e30);
\draw [<-] (e30) to (31,3);

\draw [->] (-1,4) to (f0);
\draw [->] (f0) to (f3);
\draw [->, OliveGreen] (f3) to (f6);
\draw [->] (f6) to (f9);
\draw [<-] (f9) to (f12);
\draw [->] (f12) to (f15);
\draw [<-] (f15) to (f18);
\draw [->] (f18) to (f21);
\draw [<-] (f21) to (f24);
\draw [->, OliveGreen] (f24) to (f27);
\draw [<-] (f27) to (f30);
\draw [<-] (f30) to (31,4);

\draw [->] (-1,5) to (g0);
\draw [->] (g0) to (g3);
\draw [->] (g3) to (g6);
\draw [->, OliveGreen] (g6) to (g9);
\draw [->] (g9) to (g12);
\draw [<-] (g12) to (g15);
\draw [->] (g15) to (g18);
\draw [<-] (g18) to (g21);
\draw [->, OliveGreen] (g21) to (g24);
\draw [<-] (g24) to (g27);
\draw [<-] (g27) to (g30);
\draw [<-] (g30) to (31,5);

\draw [->] (-1,6) to (h0);
\draw [->] (h0) to (h3);
\draw [->] (h3) to (h6);
\draw [->] (h6) to (h9);
\draw [->, OliveGreen] (h9) to (h12);
\draw [->] (h12) to (h15);
\draw [<-] (h15) to (h18);
\draw [->, OliveGreen] (h18) to (h21);
\draw [<-] (h21) to (h24);
\draw [<-] (h24) to (h27);
\draw [<-] (h27) to (h30);
\draw [<-] (h30) to (31,6);

\draw [->] (-1,7) to (i0);
\draw [->] (i0) to (i6);
\draw [->] (i6) to (i12);
\draw [->, OliveGreen] (i12) to (i15);
\draw [->, OliveGreen] (i15) to (i18);
\draw [<-] (i18) to (i24);
\draw [<-] (i24) to (i30);
\draw [<-] (i30) to (31,7);

\draw [->, OliveGreen] (d0) to (e0);
\draw [<-] (d6) to (e6);
\draw [<-] (d12) to (e12);
\draw [<-] (d18) to (e18);
\draw [<-] (d24) to (e24);
\draw [<-, OliveGreen] (d30) to (e30);

\draw [<-] (e0) to (f0);
\draw [->, OliveGreen] (e3) to (f3);
\draw [<-] (e6) to (f6);
\draw [<-] (e9) to (f9);
\draw [<-] (e12) to (f12);
\draw [<-] (e15) to (f15);
\draw [<-] (e18) to (f18);
\draw [<-] (e21) to (f21);
\draw [<-] (e24) to (f24);
\draw [<-, OliveGreen] (e27) to (f27);
\draw [<-] (e30) to (f30);

\draw [->] (f0) to (g0);
\draw [<-] (f3) to (g3);
\draw [->, OliveGreen] (f6) to (g6);
\draw [<-] (f9) to (g9);
\draw [<-] (f12) to (g12);
\draw [<-] (f15) to (g15);
\draw [<-] (f18) to (g18);
\draw [<-] (f21) to (g21);
\draw [<-, OliveGreen] (f24) to (g24);
\draw [<-] (f27) to (g27);
\draw [->] (f30) to (g30);

\draw [<-] (g0) to (h0);
\draw [->] (g3) to (h3);
\draw [<-] (g6) to (h6);
\draw [->, OliveGreen] (g9) to (h9);
\draw [<-] (g12) to (h12);
\draw [<-] (g15) to (h15);
\draw [<-] (g18) to (h18);
\draw [<-, OliveGreen] (g21) to (h21);
\draw [<-] (g24) to (h24);
\draw [->] (g27) to (h27);
\draw [<-] (g30) to (h30);

\draw [->] (h0) to (i0);
\draw [<-] (h3) to (i3);
\draw [->] (h6) to (i6);
\draw [<-] (h9) to (i9);
\draw [->, OliveGreen] (h12) to (i12);
\draw [<-] (h15) to (i15);
\draw [<-, OliveGreen] (h18) to (i18);
\draw [<-] (h21) to (i21);
\draw [->] (h24) to (i24);
\draw [<-] (h27) to (i27);
\draw [->] (h30) to (i30);

\foreach \i in {0,6,...,30}
	\draw [->] (j\i) to (i\i);

\end{tikzpicture}
\caption{Network $\mathcal C_4$.}
\label{net-C}
\end{figure}

Figure~\ref{net-C} shows, for $n=4$, a part of the network $\Nc_n^{\mathrm{cox}_-}$ relevant to the quiver $\Qc_n^{\mathrm{chain}}$. As was mentioned above, the part of $\Nc_n^{\mathrm{cox}_-}$ lying below what is shown in Figure~\ref{net-C} coincides with the bottom part of $\Nc_n^{\mathrm{sf}_-}$ glued by the dashed line to $\Ac_n$, see Figures~\ref{net-sf-} and~\ref{net-A}. For all $1 \le k \le n+1$, let $\mathscr C_k$ be the set of paths $\gamma \colon v_0 \to v_k$ in $\Cc_n$. We then define
\beq
\label{c-def}
C_k^{(n+1)} = \pr\hr{\sum_{\gamma\in \mathscr C_k}X_{[\gamma]}}
\qquad\text{for}\qquad
1\le k\le n+1,
\eeq
and declare $C_0^{(n+1)}=1$. Let us also recall the elements $A_k^{(n+1)}$ defined by~\eqref{alpha-def}. Then Corollary~\ref{cox-embed} implies

\begin{cor}
\label{double-cor}
As operators on $\Pc_\la \otimes \Pc_\mu \simeq \Vc_n \otimes L^2\hr{\R^n} \otimes \Mc_{\la,\mu}$, we have
$$
\tau_{\mathrm{cox}}\hr{F_1} = \sum_{k=0}^{n+1} A_k^{(n+1)} \otimes C_k^{(n+1)} \otimes 1.
$$
\end{cor}

\section{$q$-deformed open Toda lattice}
\label{sec-Toda}

We now recall, following \cite{KLS02}, the definition of the quantum integrable system that plays a central role in our story. Let $p$ and $x$ be generators of the Heisenberg algebra with commutator $[p,x]=(2\pi i)^{-1}$, and let $z$ be an indeterminate. Then the Lax operator of the $q$-deformed Toda lattice is the operator-valued matrix

$$
L(z,p,x) =
\begin{pmatrix}
z-z^{-1}e^{2\pi \hbar p} & e^{-2\pi \hbar x} \\
-e^{2\pi \hbar(p+x)} & 0
\end{pmatrix}.
$$
Consider the trigonometric $R$-matrix
\begin{align*}
R(z,w) &= \frac{qz^2-q^{-1}w^2}{z^2-w^2}\hr{e_{11}\otimes e_{11}+e_{22}\otimes e_{22}} + \hr{e_{11}\otimes e_{22}+e_{22}\otimes e_{11}} \\
& + \hr{q-q^{-1}}\frac{zw}{z^2-w^2}\hr{e_{12}\otimes e_{21} + e_{21}\otimes e_{12}},
\end{align*}
where $e_{jk}$ is the $2\times 2$-matrix whose only nonzero entry is a 1 in position $(j,k)$. Then we have the fundamental commutation relation
$$
R_{12}(z,w)L_1(z)L_2(w) = L_2(w) L_1(z)R_{12}(z,w),
$$
where $L_1(z) = L(z) \otimes 1$, $L_2(z) = 1 \otimes L(z)$, and we abbreviate $L(z) = L(z,p,x)$. To obtain the $q$-deformed Toda lattice associated to Lie algebra $\mathfrak{gl}_{n+1}$, one considers the Heisenberg algebra with canonical generators $\hc{p_i, x_i}$ where $ i=1, \dots, n+1$, and forms the monodromy matrix
\beq
\label{monodromy-matrix}
T^{(n+1)}(z) =
\begin{pmatrix} 
A_{n+1}(z) & B_{n+1}(z)\\
C_{n+1}(z) & D_{n+1}(z)
\end{pmatrix}
= L(z,p_1,x_1)L(z,p_2,x_2)\cdots L(z,p_{n+1},x_{n+1})
\eeq
Then the monodromy matrix also satisfies the fundamental commutation relation
\beq
\label{monod-fcr}
R_{12}(z,w)T^{(n+1)}_1(z)T^{(n+1)}_2(w) = T^{(n+1)}_2(w) T^{(n+1)}_1(z)R_{12}(z,w).
\eeq
\begin{defn}
The $q$-deformed $\mathfrak{gl}_{n+1}$ open Toda Hamiltonians $H^{(n+1)}_k$, $1\le k\le n+1$, are the Heisenberg algebra elements defined by
$$
A_{n+1}(z) = \sum_{k=0}^{n+1}\hr{-q}^{k-2}z^{n+1-2k}H^{(n+1)}_k.
$$
By convention, we declare $H_0^{(n+1)}=1.$  
\end{defn}
The following Proposition easily follows from the relation \eqref{monod-fcr}.
\begin{prop}
\label{toda-int}
The $q$-deformed $\mathfrak{gl}_{n+1}$ open Toda Hamiltonians commute: one has
$$
\hs{H^{(n+1)}_j,H^{(n+1)}_k}=0 \qquad\text{for any}\qquad 0\le j,k \le n+1.
$$
\end{prop}

The $\mathfrak{gl}_{n+1}$ open Toda Hamiltonians act naturally on the Hilbert space $L^2(\R^{n+1})$, where
$$
p_j \longmapsto \frac{1}{2\pi i}\partial_j \qquad\text{and}\qquad x_j \longmapsto x_j.
$$
Note that $H_{n+1}^{(n+1)} = e^{2\pi\hbar(p_1+ \ldots + p_{n+1})}$, and consider the invariant subspace of functions $f$ satisfying
$$
H_{n+1}f = f.
$$
We refer to the operators $H_k^{(n+1)}$, $1 \le k \le n+1$ acting on this  subspace as the $q$-deformed $\sl_{n+1}$ open Toda Hamiltonians. 

Using the Lax operator representation \eqref{monodromy-matrix}, it is straightforward to obtain the following recursive definition of the quantum Toda Hamiltonians.

\begin{lemma}
\label{lem-Toda}
The open Toda Hamiltonians $H_k^{(n+1)}$ are determined by the recurrence relation
\beq
\label{toda-recurrence}
H_k^{(n+1)} = H_{k}^{(n)} + q^{\delta_{1,k}-1}\hr{e^{2\pi\hbar p_{n+1}}H_{k-1}^{(n)} + e^{2\pi\hbar(p_{n+1}+x_{n+1}-x_{n})}H_{k-1}^{(n-1)}}
\eeq
and initial conditions
$$
H_0^{(0)} = H_0^{(1)} = 1, \qquad H_1^{(1)} = e^{2\pi\hbar p_{n+1}}, \qquad H_k^{(j)} = 0 \quad\text{unless}\quad 0 \le k \le j.
$$
\end{lemma}

\begin{example}
Consider the case $n=1$, where the only non-trivial $\mathfrak{gl}_2$ Toda Hamiltonian is 
$$
H_1^{(2)} = e^{2\pi\hbar p_1} + e^{2\pi\hbar p_2} + e^{2\pi\hbar(p_2+x_2-x_1)}.
$$
The corresponding $\mathfrak{sl}_2$ Hamiltonian $H_1$ is obtained by restricting $H_1^{(2)}$ to the space of functions of the variable $x=x_2-x_1$. We thus have
$$
H_1 = e^{2\pi\hbar p}+e^{-2\pi\hbar p}+e^{2\pi\hbar (p+x)},
$$
where $p = \frac{1}{2\pi i }\frac{d}{dx}$.  Conjugating by the unitary transformation under which $p\mapsto p$, $x\mapsto x+p$ brings $H_1$ to the operator
$$
H = e^{2\pi \hbar x} + e^{2\pi \hbar p}+e^{-2\pi \hbar p},
$$
which coincides with Kashaev's geodesic length operator~\eqref{kashaev-operator}. 
\end{example}

In the paper \cite{KLS02}, a set of common eigenfunctions of the $q$-deformed $\mathfrak{sl}_{n+1}$ Toda Hamiltonians was determined. In particular, the following is true. 

\begin{theorem}\cite[Theorem 3.1]{KLS02}
For each vector $\vec\nu = (\nu_0, \dots, \nu_n) \in \hgt_n$, there exists a $q$-Whittaker function $\Psi_{\vec\nu}\hr{x_1-x_2,\ldots, x_n-x_{n+1}}$ which is a common eigenfunction of the $q$-deformed $\sl_{n+1}$ open Toda Hamiltonians:
for $1\le k\le n+1$, one has
\beq
\label{eigen}
H^{(n+1)}_k \cdot \Psi_{\vec\nu}\hr{\vec{x}} = q^{1-k}e_k\big(e^{2\pi \hbar\nu_0}, \ldots, e^{2\pi \hbar\nu_n}\big)\Psi_{\vec\nu}\hr{\vec{x}},
\eeq
where $e_k\big(e^{2\pi \hbar\nu_0}, \ldots, e^{2\pi \hbar\nu_n}\big)$ is the $k$-th elementary symmetric polynomial in variables~$e^{2\pi \hbar\nu_j}$.
\end{theorem}

Explicit formulas for the $q$-Whittaker function $\Psi_{\vec\nu}\hr{\vec{x}}$ can be obtained via integrals of Mellin-Barnes type, see \cite[Theorem 3.2]{KLS02}. 

\begin{conjecture}
\label{plancherel-conjecture}
The $q$--Whittaker functions are orthogonal and complete in $L^2\hr{\R^n}$: one has
$$
\int_{\R^n}\Psi_{\vec{\lambda}}\hr{\vec{x}}\overline{\Psi_{\vec{\mu}}\hr{\vec{x}}}dx = \delta\hr{\vec{\lambda}-\vec{\mu}},
$$
as well as the Plancherel inversion formula
\beq
\label{planch-formula}
\int_{\Cc^+}\Psi_{\vec{\nu}}\hr{\vec{x}}\overline{\Psi_{\vec{\nu}}\hr{\vec{y}}}dm(\vec{\nu}) = \delta\hr{\vec{x}-\vec{y}},
\eeq
where $dm(\vec{\nu})$ is the Sklyanin measure
\beq
\label{sklyanin-mes}
dm(\vec{\nu}) = \prod_{j<k}4\sinh\big(\pi\hbar(\nu_j-\nu_k)\big)\sinh\big(\pi\hbar^{-1}(\nu_j-\nu_k)\big),
\eeq
and $\Cc^+$ is the positive Weyl chamber 
\beq
\label{weyl-def}
\Cc^+= \hc{\vec\nu\in\hgt_n \,\Big|\, \nu_k-\nu_{k-1}\ge0, \quad 1\le k \le n+1}.
\eeq
\end{conjecture}

For $n=1$, Conjecture \ref{plancherel-conjecture} was proven in~\cite{Kas01}. Its proof for $n>1$ will be presented in the forthcoming paper~\cite{SS17}. This conjecture implies that the integral transform
$$
\mathfrak{L} \colon f\hr{\vec{x}} \longmapsto \hat{f}(\vec\nu) = \int_{\R^{n}}\Psi_{\vec{\nu}}(\vec{x})f(x)dx
$$
is an $L^2$-isometry, that is
$$
\mathfrak{L} \colon L^2\hr{\R^n} \simeq L^2\hr{\Cc^+, dm(\vec\nu)}.
$$

Let us now explain the relevance of the open Toda lattice to the problem of decomposing the tensor product $\mathcal{P}_\lambda\otimes\mathcal{P}_\mu$. 
Recall from Section 5.9 the quantum torus algebra $\Dcox$, and its set of generators $\alpha_0$, $\alpha_{n+1}$, and $\alpha_i^1$, $\alpha_i^2$ for $i=1, \dots, n$. Note that there is an algebra embedding
\beq
\label{embed-system}
\Dc_{n-1}^{\mathrm{cox}}\hooklongra \Dcox  ,
\eeq
such that
$$
\alpha_0 \longmapsto \alpha_0, \qquad 
\alpha_j^a \longmapsto \alpha_j^a, \qquad\text{and}\qquad
\alpha_n \longmapsto q^2\alpha_n^1\alpha_n^2\alpha_{n+1}
$$
where $1 \le j \le n-1$ and $a=1,2$. By means of this embedding, we may regard the quantum torus algebras $\mathcal{D}_k^{\mathrm{cox}}$  with $k\le n$ as subalgebras in $\mathcal{D}_{n}^{\mathrm{cox}}$. Finally, let us recall the quantum torus algebra elements $C_k^{(n+1)} \in \Dcox$ defined by formula~\eqref{c-def}, which appear in the formula for $\tau_{\mathrm{cox}}\hr{F_1}$ given in Corollary~\ref{double-cor}. Then we have

\begin{lemma}
\label{toda-calc}
The algebra elements $C_k^{(n+1)}$ are the unique solution of the recurrence relation
$$
C_k^{(n+1)} = C_k^{(n)}
+ q^{\delta_{k,1}-1}\hr{ \alpha_{n+1}C_{k-1}^{(n)}
+ q \alpha_{n+1}\alpha_n^1C_{k-1}^{(n-1)}}
$$
satisfying the initial conditions
$$
C_0^{(0)} = C_0^{(1)} = 1, \qquad
C_1^{(1)}(\alpha_1) = \alpha_1, \qquad
C_k^{(j)}=0 \quad\text{unless}\quad 0 \le k \le j.
$$
Here we understand $C_k^{(n)}$, $C_{k-1}^{(n)}$, and $C_{k-1}^{(n-1)}$ as elements of $\Dcox$ by means of the system of embeddings~\eqref{embed-system}.
\end{lemma}

\begin{proof}
Let us write the set of paths $\mathscr C_k$ as a disjoint union of three subsets: $\mathscr C_k^\shortleftarrow$, $\mathscr C_k^\shortdownarrow$, and $\mathscr C_k^\shortrightarrow$. Note that the last edge of any path $\gamma \in \mathscr C_k$ points down to $v_k$, while the one before last may be directed to the left, in which case we declare $\gamma \in \mathscr C_k^\shortleftarrow$; downward, in which case we declare $\gamma \in \mathscr C_k^\shortdownarrow$; or to the right, in which case we declare $\gamma \in \mathscr C_k^\shortrightarrow$. Now, one can check that
\begin{align*}
&\pr\hr{\sum_{\gamma\in C_k^\shortleftarrow}X_{[{\gamma}]}} = C_k^{(n)},
\\
&\pr\hr{\sum_{\gamma\in C_k^\shortdownarrow}X_{[{\gamma}]}}= q^{\delta_{k,1}} \alpha_{n+1}\alpha_n^1C_{k-1}^{(n-1)},
\\
&\pr\hr{\sum_{\gamma\in C_k^\shortrightarrow}X_{[{\gamma}]}} = q^{\delta_{k,1}-1} \alpha_{n+1}C_{k-1}^{(n)},
\end{align*}
and the recurrence \eqref{toda-recurrence} follows. That the initial conditions are satisfied is evident. 
\end{proof}

Now consider the action of $\Dcox$ act on $L^2(\R^n)$ given by the assignments~\eqref{alpha-def}. Then Lemmas~\ref{lem-Toda} and~\ref{toda-calc} imply
\begin{cor}
As operators on $L^2(\R^n)$, we have
$$
C^{(n+1)}_k =H^{(n+1)}_k \qquad\text{for}\qquad n \ge 1 \quad\text{and}\quad 0\le k\le n+1. 
$$
\end{cor}
\begin{proof}
This follows directly from Lemmas~\ref{lem-Toda} and~\ref{toda-calc}, after observing that in the representation~\eqref{alpha-def} we have
\beq
\alpha_{n+1}\mapsto e^{2\pi\hbar p_{n+1}} \qquad\text{and}\qquad q \alpha_{n+1}\alpha_n^1 \mapsto e^{2\pi\hbar(p_{n+1}+x_{n+1}-x_n)}.
\eeq
\end{proof}

Combining this with Corollary~\ref{double-cor} yields
\begin{prop}
\label{double-prop}
As operators on $\Pc_{\lambda}\otimes\Pc_\mu \simeq \Vc_n\otimes L^2\hr{\R^n}\otimes \Mc_{\lambda,\mu}$, we have
\beq
\label{double-formula}
\tau_{\mathrm{cox}}\hr{F_1} = \sum_{k=0}^{n+1} A^{(n+1)}_k \otimes H^{(n+1)}_k\otimes 1.
\eeq
\end{prop}

\section{Construction of the intertwiner}
\label{sec-int}

Let us now explain how the ingredients presented in the previous sections can be combined to obtain the decomposition of the tensor product of positive representations of $U_q(\mathfrak{sl}_{n+1})$.

\subsection{A representation of $U_q(\mathfrak{sl}_{n+1})$}

Let us fix $\lambda,\mu\in \mathfrak{h}_n$ and recall the following isomorphisms of Hilbert spaces
$$
\Pc_\nu \simeq \Vc_n \simeq L^2\hr{\R^{\frac{n(n+1)}{2}}}
\qquad\text{and}\qquad
\Mc_{\la,\mu}^\nu \simeq \Mc_{\la,\mu} \simeq L^2\hr{\R^{\frac{n(n-1)}{2}}}.
$$
Consider the space 
$$
\Vc_n \otimes L^2\hr{\Cc^+, dm(\nu)} \otimes \Mc_{\la,\mu} \simeq
L^2\hr{\R^{\frac{n(n+1)}{2}}} \otimes L^2 \hr{\Cc^+, dm(\nu)} \otimes L^2\hr{\R^{\frac{n(n-1)}{2}}}
$$
where the first and the third tensor factors are equipped with the standard Lebesgue measure, the second tensor factor carries the Sklyanin measure $dm(\nu)$ defined in~\eqref{sklyanin-mes}, and $\Cc^+$ is the positive Weyl chamber~\eqref{weyl-def}. This Hilbert space carries an action of the quantum torus algebra $\Dstd \otimes \Dsph$, such that the slice obtained by fixing $\nu\in\Cc^+$ yields the positive representation $\Pc_\nu \otimes M_{\la,\mu}^{\nu}$ of~$\Dstd \otimes \Dsph$. By definition, this representation decomposes as a direct integral
\beq
\label{int}
\Vc_n\otimes L^2\hr{\Cc^+,dm(\nu)} \otimes \Mc_{\la,\mu} \simeq
\int^{\oplus}_{\Cc^+} \Pc_\nu \otimes \Mc_{\la,\mu}^\nu dm(\nu).
\eeq
Now, we can restrict the representation~\eqref{int} to the subalgebra $U_q\hr{\sl_{n+1}} \otimes 1 \subset \Dstd \otimes \Dsph$ and obtain a representation of $U_q\hr{\sl_{n+1}}$, where the action on the second tensor factor of each slice is trivial.

\subsection{The intertwiner}
Recall the statement of the Proposition~\ref{prop-concat}:
$$
\Pc_\la \otimes \Pc_\mu \simeq \Vc_n \otimes L^2\hr{\R^n} \otimes \Mc_{\la,\mu}.
$$
Thus $\Pc_\la \otimes \Pc_\mu$ can be modeled as the space of functions
$$
f(x,y,z) \qquad\text{with}\qquad
x \in \R^{\frac{n(n+1)}{2}}, \quad y \in \R^{\frac{n(n-1)}{2}}, \quad z \in \R^n,
$$
that are square-integrable with respect to the product of standard Lebesgue measures on all three spaces. Similarly, the Hilbert space~\eqref{int} consists of functions 
$$
\hat{f}(x,y,\nu) \qquad\text{with}\qquad
x \in \R^{\frac{n(n+1)}{2}}, \quad y \in \R^{\frac{n(n-1)}{2}}, \quad \nu \in \Cc^+,
$$
that are square-integrable with respect to the Sklyanin measure on $\Cc^+$ and the Lebesgue measures on $\R^{\frac{n(n+1)}{2}}$ and $\R^{\frac{n(n-1)}{2}}$.

\begin{theorem}
\label{main-thm}
Suppose that Conjecture~\ref{plancherel-conjecture} holds. Then the integral transform 
\begin{align}
\label{integral-transform}
\mathfrak{L} \colon \Pc_\la \otimes \Pc_\mu &\longra
\int^{\oplus}_{\Cc^+} \Pc_\nu \otimes \Mc_{\la,\mu}^{\nu}~dm(\nu), \\
\nonumber f(x,y,z) &\longmapsto \hat{f}(x,y,\nu) = \int_{\R^n} f(x,y,z) \overline{\Psi_\nu(z)}dz,
\end{align}
is a Hilbert space isometry, and the map
\beq
\label{intertwiner}
\mathfrak{I} = \hr{\Phi^{\mathrm{sym}}\circ\Phi^{\mathrm{sf}}\otimes 1}^{-1}\circ\mathfrak{L}\circ\Phi^{\mathrm{cox}}\circ\Phi^{\mathrm{top}} 
\eeq 
is an isomorphism of $U_q(\mathfrak{sl}_{n+1})$-modules
\beq
\label{intertwiner-iso}
\mathfrak{I} \colon \Pc_\la \otimes \Pc_\mu \simeq \int^{\oplus}_{\Cc^+} \Pc_\nu \otimes M_{\la,\mu}^\nu~dm(\nu).
\eeq
In particular, the category of positive representations of $U_q(\mathfrak{sl}_{n+1})$ is closed under tensor product in the continuous sense. 
\end{theorem}
\begin{proof}
The algebraic crux of the proof is to show that the map $\mathfrak{I}$ intertwines the actions of $U_q(\mathfrak{sl}_{n+1})$ on each side. 
As observed in formula~\eqref{easy-gens}, the actions of the Chevalley generators $E_i$, $K_i$, and $F_j$ for $1 \le i \le n$ and $2 \le j \le n$ already match on both sides. Thus, it remains to show that the actions of the generator $F_1$ are intertwined as well. By comparing formulas~\eqref{single-formula} and~\eqref{double-formula} we observe that the relevant boundary measurements $M_n^{\mathrm{sym}_-}(1,2;2,2)$ and $M_n^{\mathrm{cox}_-}(1,2;2,2)$ become identical if the $k$-th Coxeter-Toda Hamiltonian $H_k^{(n+1)}$ is replaced by its eigenfunction, see~\eqref{eigen}. By the self-adjointness of the Toda Hamiltonians, we have
$$
\int_{\R^n} \hr{H_k^{(n+1)} \cdot f(x,y,z)} \overline{\Psi_\nu(z)} dz = \int_{\R^n} f(x,y,z)\overline{\hr{H_k^{(n+1)} \cdot \Psi_\nu(z)}} dz,
$$
thus
$$
\mathfrak{L} \cdot f(x,y,z) = q^{1-k} e_k \big(e^{2\pi \hbar \nu_0}, \ldots, e^{2\pi \hbar \nu_n} \big) \hat{f}(x,y,\nu),
$$
and the fact that $\mathfrak{I}$ is an intertwiner for the $U_q(\sl_{n+1})$ actions follows. The transformations $\Phi^{\mathrm{top}}$, $\Phi^{\mathrm{cox}}$, $\Phi^{\mathrm{sf}}$, and $\Phi^{\mathrm{sym}}$ are unitary since they are products of noncompact quantum dilogarithms of self-adjoint Heisenberg algebra generators; the unitarity of $\mathfrak{L}$ is the content of Conjecture~\ref{plancherel-conjecture}. 
\end{proof}

\section{Comparison with the results of Ponsot and Teschner for $n=1$}
\label{sec-comp}

We conclude by comparing our results with those obtained in~\cite{PT01} for the $U_q(\sl_2)$ Clebsch-Gordan maps 
$$
C_{s_1,s_2} \colon \Pc_{s_1} \otimes \Pc_{s_2} \longra \int_{\R_{\ge0}}^\oplus \Pc_{s_3}~dm(s_3).
$$
Note that when $n=1$, our formula~\eqref{intertwiner} for the intertwiner $\mathfrak{I}$ simplifies. Indeed, in the rank 1 case we have $\Phi^{\mathrm{cox}} = \Phi^{\mathrm{sym}} = \mathrm{id}$, since $\Theta_1^{\mathrm{sf}} = \Theta_1^{\mathrm{sym}}$ and $\Theta_1^{\mathrm{full}} = \Theta_1^{\mathrm{cox}}$, and thus there is no need to perform any additional Lie-theoretic mutations following those corresponding to flips of triangulation. Moreover, the quantum torus algebra $\Dc_1^{\mathrm{sph}}$ associated to the thrice-punctured sphere is abelian, so that all representations $\Mc_{\la,\mu}^{\nu}$ appearing in~\eqref{intertwiner-iso} are 1-dimensional.

Ponsot and Teschner's Clebsch-Gordan maps were reconsidered in \cite{NT13} in the context of quantum Teichm\"uller theory. Let us first remark that realization of the positive representation $\mathcal{P}_s\simeq L^2(\R)$ referred to in \cite{NT13} as the ``Whittaker model'' is nothing but the representation of $U_q(\mathfrak{sl}_2)$ defined by the quantum torus embedding $\iota_{\mathrm{sf}} = \Ad_{\Phi^{\mathrm{sf}}}\cdot\iota$, see Corollary~\ref{cor-folded}. This realization is the one associated to the self-folded triangulation of the punctured disk, which in the case $n=1$ is obtained from the standard one by applying a single cluster mutation. Now, in \cite{NT13} the Clebsch-Gordan intertwiner is factored into a product 
$$
C_{s_1,s_2} = S_1\circ C_1\circ T_{12}^{-1}
$$ of three unitary transformations $S_1,C_1$, and $T_{12}^{-1}$. In our terms, these transformations can be interpreted as follows. The action of $U_q(\mathfrak{sl}_2)$ on $\mathcal{P}_{s_1}\otimes \mathcal{P}_{s_2}$ considered in \cite{NT13} is the one associated to the triangulation of the twice punctured disk shown in the middle of the top row of Figure~\ref{top-muts}. Let us denote the corresponding seed by $\Theta_1^{\mathrm{fold}}$. Then we require 4 flips of triangulation to pass from $\Theta_1^{\mathrm{fold}}$ to the seed $\Theta_1^{\mathrm{full}}$ corresponding to the bottom left of Figure~\ref{top-muts}, which translates to a total of $4$ mutations since $n=1$. Now, the transformation~$T_{12}^{-1}$ is nothing but the quantum mutation operator corresponding to the first of these flips. After conjugating by $T_{12}^{-1}$, the action of $\Delta(E)$ on $\mathcal{P}_{s_1}\otimes\mathcal{P}_{s_2}$ is given by $\iota_{\mathrm{sf}}(E)\otimes 1$. 

 The composite of the remaining three quantum mutation operators coincides with the transformation $C_1$. After conjugating by $C_1\circ T_{12}^{-1}$, the action of the quadratic Casimir of $U_q(\mathfrak{sl}_2)$ on $\mathcal{P}_{s_1}\otimes\mathcal{P}_{s_2}$ is transformed to that of $1\otimes H$, where $H$ is Kashaev's geodesic length operator
\beq
H = e^{2\pi \hbar x} + e^{2\pi \hbar p}+e^{-2\pi \hbar p}.
\eeq
Finally, the operator $S_1$ is simply the $n=1$ case of the integral transform~\eqref{integral-transform}, which diagonalizes the operator $H$. 

\appendix

\section{Proof of the Proposition~\ref{mut-descent}}
\label{appendix-1}

Let $\Gamma$ be a bi-colored graph on a punctured disk, and $\wdt\Nc$ be a network on the universal cover that projects onto $\Gamma$. We choose a connected fundamental domain in $\Nc$, and think of it as lying on the $0$-th sheet of the universal cover. For any fixed $n \in \Z$ we denote faces on the $n$-th sheet of the cover by $\Fc_k^n$. Let $\wdt\Theta = (\wdt\La, (\cdot,\cdot), \hc{e_i^n}, \wdt{I_0})$ be the cluster seed corresponding to the network $\wdt\Nc$ where a vector $e_k^n$ corresponds to the face $\Fc_k^n$. In this setup, it is evident that
\beq
\label{dist}
(e_i^n,e_j^m) = 0 \qquad\text{for}\qquad \hm{m-n}>1 \quad\text{and any}\quad i,j \in V(\Qc).
\eeq
In turn, the above property implies
\beq
\label{app-distr}
(e_i, e_j) = (e_i^m, e_j^{m-1}+e_j^m + e_j^{m+1})
\eeq
for any $m \in \Z$.

\begin{defn}
We say that a network face $\Fc_i \in F(\wdt\Nc)$ is \emph{left interior} if $(e_i^n,e_j^{n-1})=0$ for all~$j$. Similarly, the face $\Fc_i$ is said to be \emph{right interior} if $(e_i^n,e_j^{n+1})=0$ for all $j$.
\end{defn}

A face may be both left and right interior -- these are the network faces that do not touch the boundary of the fundamental domain. Note also that if $\Fc_k$ is either left or right interior, then $(e_k^{(n)},e_k^{(m)})=0$ for all $n,m$. 

\begin{lemma}
\label{lem-index}
In all quivers under consideration, the definition of the form $\ha{\cdot,\cdot}$ is independent of the choice of fundamental domain. Namely, let $e_i^n$ satisfy $(e_i^n,e_j^{n\pm1})=0$ for all $j \ne i$, and set
$$
\bar e_j^n =
\begin{cases}
e_i^{n\pm1}, & \text{if}\quad j=i, \\
e_j^n, & \text{if}\quad j \ne  i.
\end{cases}
$$
Denote by $\ha{\ha{\cdot, \cdot}}$ the bilinear form defined by~\eqref{formdef} with respect to the basis $\{\bar e_i^n\}$. Then we have
$$
\ha{\ha{\bar e_i^n, \bar e_j^m}} = \ha{e_i^{n\pm1}, e_j^m}.
$$
\end{lemma}

\begin{proof}
We consider the case in which $e_i^n$ satisfies $(e_i^n,e_j^{n+1})=0$ for all $j \ne i$, the other is treated similarly. By the symmetry of the form, it's enough to show that for all $i,j \in V(\Qc)$ and $m \ge n$ we have
\beq
\label{shift-goal}
\ha{\ha{\bar e_i^n, \bar e_j^m}} = \ha{e_i^{n+1}, e_j^m}.
\eeq
Note, that the property~\eqref{dist} still holds for the basis $\hc{\bar e_i^n}$, thus for $m>n+1$ the equality~\eqref{shift-goal} is immediate. Now, if $m=n+1$, we have
$$
\langle\langle \bar e_i^n, \bar e_j^m \rangle\rangle
= (\bar e_i^n, \bar e_j^{n-1} + \bar e_j^n)
= (e_i^{n+1}, e_j^{n-1} + e_j^n)
= (e_i^{n+1}, e_j^n),
$$
where we used the equality~\eqref{app-distr}. On the other hand, we have
$$
\ha{e_i^{n+1}, e_j^m}
= \{e_i^{n+1}, e_j^{n+1}\}
= (e_i^{n+1}, e_j^n - e_j^{n+2})
= (e_i^{n+1}, e_j^n).
$$
Finally, for $m=n$, we get
$$
\langle\langle \bar e_i^n, \bar e_j^m \rangle\rangle
= (\bar e_i^n, \bar e_j^{n-1} - \bar e_j^{n+1})
= (e_i^{n+1}, \bar e_j^{n-1} - \bar e_j^{n+1})
= - (e_i^{n+1}, e_j^{n+1}),
$$
while
$$
\langle e_i^{n+1}, e_j^n \rangle
= (e_i^{n+1}, -e_j^{n+1} - e_j^{n+2})
= -(e_i^{n+1}, e_j^{n+1}).
$$
\end{proof}

In this article we only consider mutations at faces of the network which are either left or right interior. The following Lemma shows that the definition of the form $\ha{\cdot, \cdot}$ is invariant under such mutations.

\begin{lemma}
\label{lem-pr-mut}
Let $\{f_i^n\}$ be the basis of $\tLambda$ obtained by applying $\tilde\mu_k$ to the basis $\{e_i^n\}$, where $\Fc_k^n$ is either left or right interior. Denote by $\ha{\cdot, \cdot}_k$ the bilinear form defined by~\eqref{formdef} with respect to the basis $\{f_i^n\}$. Then for all $\lambda, \mu\in\tLambda$ we have
$$
\ha{ \lambda, \mu}_k = \ha{ \lambda, \mu}.
$$

\end{lemma}

\begin{proof}
The proof is a straightforward verification. By Lemma~\ref{lem-index}, we can shift the fundamental domain to ensure that the faces $\Fc_k^n$ are right interior, so that $(e_k^n,e_j^{n+1})=0$ for all $j$. Then for $j\neq k$, we have
$$
f_j^n = e_j^n + [\eps_{jk}^0]_+ e_k^n + [\eps_{jk}^+]_+ e_k^{n+1},
$$
with
$$
\eps_{jk}^0 = (e^n_j, e^n_k), \qquad\text{and}\qquad \eps_{jk}^+ = (e^n_j, e^{n+1}_k).
$$
We also have
$$
f_k^n = -e_k^n.
$$
By the symmetry of the form $\ha{\cdot, \cdot}$, it suffices to show that for $m \ge n$ and all $i,j \in V(\Qc)$ we have
\beq
\label{toshow}
\langle f_i^n, f_j^m \rangle_k = \langle f_i^n, f_j^m \rangle.
\eeq
The equality \eqref{toshow} is a tautology for $m>n+1$, thus it is left to consider the cases $m=n+1$ and $m=n$. Suppose that $m=n+1$ and $i,j\neq k$. 
Then
\begin{align*}
\ha{f_i^n, f_j^m}_k = \big(f_i^n,f_j^n+f_j^{n-1}\big)
= &\big( e_i^n + [\eps_{ik}^0]_+ e_k^n + [\eps_{ik}^+]_+ e_k^{n+1},
e_j^n + [\eps_{jk}^0]_+ e_k^n + [\eps_{jk}^+]_+ e_k^{n+1} \big) \\
+ &\big( e_i^n + [\eps_{ik}^0]_+ e_k^n + [\eps_{ik}^+]_+ e_k^{n+1},
e_j^{n-1} + [\eps_{jk}^0]_+ e_k^{n-1} + [\eps_{jk}^+]_+ e_k^n] \big).
\end{align*}
The latter can be simplified to
\beq
\label{A1}
\big( e_i^n, e_j^{n-1} + e_j^n + [\eps_{jk}^0]_+ e_k^n + [\eps_{jk}^+]_+ (e_k^n+e_k^{n+1}) \big)
+ [\eps_{ik}^0]_+ \big( e_k^n, e_j^{n-1}+e_j^n \big)
+ [\eps_{ik}^+]_+ \big( e_k^{n+1},e_j^n \big).
\eeq
On the other hand, we have
$$
\ha{f_i^n, f_j^m} = \big\langle e_i^n + [\eps_{ik}^0]_+ e_k^n + [\eps_{ik}^+]_+ e_k^{n+1},
e_j^{n+1} + [\eps_{jk}^0]_+ e_k^{n+1} + [\eps_{jk}^+]_+ e_k^{n+2} \big\rangle.
$$
Unwrapping the right hand side we arrive at~\eqref{A1} once again. In case $j=k$ we get
$$
\ha{ f_i^n, f_k^{n+1} }_k = \big( f_i^n, f_k^{n-1} + f_k^n \big) = \big( f_i^n, f_k^n \big) = -\big( e^n_i, e^n_k \big)
$$
and
$$
\ha{ f_i^n, f_k^{n+1} }
= \big\langle e_i^n + [\eps_{ik}^0]_+ e_k^n + [\eps_{ik}^+]_+ e_k^{n+1}, -e_k^{n+1} \big\rangle
= -\big( e^n_i,e^{n-1}_k + e^n_k \big)
= -\big( e^n_i, e^n_k \big).
$$
In case $i=k$ we obtain
$$
\big\langle f_k^n, f_j^{n+1} \big\rangle_k = \big( f_k^n, f_j^{n-1} + f_j^n \big) = -\big( e_k^n, e_j^{n-1} + e_j^n \big)
$$
and
$$
\big\langle f_k^n, f_j^{n+1} \big\rangle = - \big\langle e_k^n, e_j^{n+1} + [\eps_{jk}^0]_+ e_k^{n+1} + [\eps_{jk}^+]_+ e_k^{n+2} \big\rangle = -\big( e_k^n, e_j^{n-1} + e_j^n \big).
$$

The remaining case $n=m$ is treated in a similarly tedious fashion. We have
$$
\big\langle f_i^n, f_k^n \big\rangle_k = \big( f_i^n, f_k^{n-1} - f_k^{n+1} \big) = \big( e^n_i, e^{n+1}_k \big),
$$
while
$$
\big\langle f_i^n, f_k^n \big\rangle = \big( e^n_i, e_k^{n+1} - e_k^{n-1} \big) = \big( e^n_i, e^{n+1}_k \big).
$$
Finally, when $i,j\neq k$ we get
$$
\big\langle f_i^n, f_j^n \big\rangle_k = \big( f_i^n, f_j^{n-1} - f_j^{n+1} \big)
$$
which simplifies to
\beq
\label{A2}
\big( e_i^n, e_j^{n-1} - e_j^{n+1} - [\eps_{jk}^0]_+ e_k^{n+1} + [\eps_{jk}^+]_+ e_k^n \big) + [\eps_{ik}^0]_+ \big( e_k^n, e_j^{n-1} \big) - [\eps_{ik}^+]_+ \big( e_k^{n+1}, e_j^{n+1} \big).
\eeq
On the other hand,
$$
\big\langle f_i^n, f_j^n \big\rangle = \big\langle e_i^n + [\eps_{ik}^0]_+ e_k^n + [\eps_{ik}^+]_+ e_k^{n+1}, e_j^n + [\eps_{jk}^0]_+ e_k^n + [\eps_{jk}^+]_+ e_k^{n+1} \big\rangle
$$
and the right hand side can be transformed to~\eqref{A2}.
\end{proof}

Before we move to the proof of the Proposition~\ref{mut-descent}, let us establish a few preparatory statements. Note that any $\la \in \tLambda$ can be written as
$$
\la = \sum_{i\in V(\Qc)}\sum_{n\in \Z} c^i_n e_i^n
$$
where $c^i_n \in \Z$. We  set
$$
\la^{< m} = \sum_{i\in V(\Qc)}\sum_{n< m} c^i_n e_i^n
\qquad\text{and}\qquad
\la^{m} = \sum_{i\in V(\Qc)} c^i_m e_i^m.
$$
Let us define
$$
\beta_{k,m}(\la) = \big( \la^{< m+1}, e_k^m \big) + \big( \la^m, e_k^{m+1} \big) - \big( \pi(\la^{< m+1}), e_k \big).
$$

Let us emphasize once again that the map $\pr$ is not a homomorphism of algebras. The following two lemma follow directly from the definitions of $\beta_{k,m}$ and the map $\pr$.

\begin{lemma}
\label{mult-lemma}
We have
$$
\pr\big( X_\la \cdot X_{-e_k^m} \big) = q^{2\beta_{k,m}(\la)} \cdot \pr(X_{\la}) \cdot X_{-e_k}.
$$
\end{lemma}

\begin{lemma}
\label{tel-lemma}
The quantity $\beta_{k,m}(\la)$ satisfies
$$
\beta_{k,m+1}(\la) = \beta_{k,m}(\la) + (e_k^{m},\la).
$$
\end{lemma}

Fianlly, we are ready to prove Proposition~\ref{mut-descent}.

\begin{proof}[Proof of Proposition~\ref{mut-descent}]
Choose $\lambda \in \tLambda$ and denote by $r$ and $s$ respectively the smallest and largest integers $n$ for which $(\la, e^n_k)\ne0$. Then we have
$$
\tilde\mu_k (X_\la) = X_\la \cdot \frac
{\Psi^q \hr{q^{2(e_k^r,\la)} X_{-e_k^r}} \dots \Psi^q \hr{q^{2(e_k^s,\la)} X_{-e_k^s}}}
{\Psi^q \hr{X_{-e_k^r}} \dots \Psi^q \hr{X_{-e_k^s}}}.
$$
Note that the factors within the fraction commute, since $(e_k^n, e_k^m)=0$ for all $m,n \in \Z$.  Applying Lemma~\ref{mult-lemma}, we get
\beq
\label{telescopic-product}
\pr\hr{\tilde\mu_k(X_{\la})} = X_{\pi(\la)} \cdot \frac
{\Psi^q \hr{q^{2(e_k^r,\la)+2\beta_{k,r}(\la)} X_{-e_k}}
\dots \Psi^q \hr{q^{2(e_k^s,\la)+2\beta_{k,s}(\la)} X_{-e_k}}}
{\Psi^q\left( q^{2\beta_{k,r}(\la)}X_{-e_k} \right)
\dots \Psi^q \hr{q^{2\beta_{k,s}(\la)}X_{-e_k}}}.
\eeq 
Observe that $\beta_{k,r}(\la)=0$ and $\beta_{k,s+1}(\la) = (e_k,\pi(\la))$. Thus, by applying Lemma~\ref{tel-lemma} we see that the product in~\eqref{telescopic-product} telescopes to yield
$$
\pr\hr{\tilde\mu_k(X_\la)} = \pr\hr{X_\la} \Psi^q\big( q^{2\hr{e_k,\pi(\la)}}X_k^{-1} \big) \Psi^q\left( X_{k}^{-1} \right)^{-1} = \mu_k\hr{\pr(X_\la)},
$$
as claimed. 
\end{proof}

\bibliographystyle{alpha}

\end{document}